\documentclass[11pt]{amsart}

\usepackage{amsmath,amsthm,amssymb,amsfonts,ifpdf}
\usepackage{xcolor}

\usepackage{xargs}  
\usepackage[colorinlistoftodos,prependcaption,textsize=tiny,obeyFinal]{todonotes}
\newcommandx{\ebltodo}[2][1=]{\todo[linecolor=red,backgroundcolor=red!25,bordercolor=red,#1]{#2}}
{
  \color{olive}%
}%
{}

\PassOptionsToPackage{hyphens}{url}
\ifpdf
\usepackage[pdftex,pdfstartview=FitH,pdfborderstyle={/S/B/W 1},%
colorlinks=true, linkcolor=blue, urlcolor=blue, citecolor=blue,%
pagebackref=true]{hyperref}
\else
\usepackage[hypertex]{hyperref}
\fi
\usepackage{enumitem}
\setenumerate{leftmargin=*}
\setitemize{leftmargin=*}
\usepackage{amscd}
\usepackage[latin1]{inputenc}
\DeclareMathAlphabet{\mathpzc}{OT1}{pzc}{m}{it}
\usepackage{bbm}

\numberwithin{equation}{section}

\swapnumbers
\newtheorem{thm}[subsection]{Theorem}
\newtheorem{coro}[subsection]{Corollary}
\newtheorem*{cor*}{Corollary}
\newtheorem{lemma}[subsection]{Lemma}
\newtheorem*{sublemma}{Sublemma}

\newtheorem{propos}[subsection]{Proposition}

\newtheorem*{thm*}{Theorem}
\newtheorem*{thma*}{Theorem A}
\newtheorem*{thmb*}{Theorem B}
\newtheorem*{thmc*}{Theorem C}

\swapnumbers
\theoremstyle{definition}

\newtheorem{defin}[subsection]{Definition}

\newcounter{consta}
\renewcommand{\theconsta}{{A_{\arabic{consta}}}}
\newcommand{\consta}{\refstepcounter{consta}\theconsta}

\newcounter{constk}
\renewcommand{\theconstk}{{\kappa_{\arabic{constk}}}}
\newcommand{\constk}{\refstepcounter{constk}\theconstk}

\newcounter{constc}

\newcounter{constE}
\renewcommand{\theconstE}{{{C}_{\arabic{constE}}}}
\newcommand{\constE}{\refstepcounter{constE}\theconstE}

\newcounter{constd}

\makeatletter
\newcommand*\bigcdot{\mathpalette\bigcdot@{.5}}
\newcommand*\bigcdot@[2]{\mathbin{\vcenter{\hbox{\scalebox{#2}{$\m@th#1\bullet$}}}}}
\makeatother

\def\XXint#1#2#3{{\setbox0=\hbox{$#1{#2#3}{\int}$ }
\vcenter{\hbox{$#2#3$ }}\kern-.6\wd0}}

\DeclareMathOperator{\Lip}{Lip}
\DeclareMathOperator{\supp}{supp}

\DeclareMathOperator{\Leb}{Leb}

\DeclareMathOperator{\diff}{d}

\DeclareMathOperator\Mat{Mat}
\DeclareMathOperator\Ad{Ad}

\newcommand\vol{{\rm{vol}}}
\newcommand\SL{{\rm{SL}}}
\newcommand\PSL{{\rm{PSL}}}

\newcommand\PGL{{\rm{PGL}}}
\newcommand\SO{{\rm{SO}}}
\newcommand\Lie{{\rm Lie}}
\newcommand\height{{\rm ht}}

\def\sl{{\mathfrak{sl}}}

\def\bbz{\mathbb{Z}}
\def\bbq{\mathbb{Q}}

\def\bbr{\mathbb{R}}

\def\bbc{\mathbb{C}}

\def\bbn{\mathbb{N}}
\def\Gbf{\mathbb{G}}

\def\Q{\bbq}
\def\Z{\bbz}
\def\R{\bbr}
\def\C{\bbc}

\def\Scal{\mathcal{S}}

\def\hfrak{\mathfrak{h}}

\def\rfrak{\mathfrak{r}}

\def\ufrak{\mathfrak{u}}

\def\gfrak{\mathfrak{g}}

\def\Gbf{\mathbf{G}}
\def\Hbf{\mathbf{H}}
\def\Lbf{\mathbf{L}}

\def\H{\Hbf}
\def\G{\Gbf}

\def\places{{S}}

\def\la{\lambda}

\def\vare{\varepsilon}

\def\zg0{Z_{G_\omega}(s)}

\def\zg{Z_G(s)}

\def\be{\begin{equation}}
\def\ee{\end{equation}}

\def\dist{{\rm dist}}

\def\rank{{\rm rk}}

\def\Sob{{\mathcal S}}

\def\dist{d}
\def\mixexp{\kappa_0}

\def\boxH{\mathsf B^H}
\def\boxHs{\mathsf B^{s,H}}
\def\boxG{\mathsf B^{G}}

\def\boxU{\mathsf B}
\def\rwm{\nu}

\def\rws{t}

\def\conv{\ast}

\def\injr{\eta}

\def\rel{r}
\def\rot{k}
\def\ave{\int_{0}^1}
\def\uvk{u_\rel}
\def\uvkd{\diff\!\rel}

\def\mfm{f}

\def\margI{I}

\DeclareMathOperator{\inj}{inj}

\def\nuni{e}
\def\coneH{\mathsf E}
\def\cone{\mathcal E}

\def\umt{\mathsf Q}

\newcommand{\rhsc}{b}

\newcommand{\mfsc}{b}
\newcommand{\sfh}{\mathsf h}

\newcommand{\sfs}{\mathsf s}

\newcommand{\trct}{\delta}
\newcommand\mfbd{\Upsilon}
\newcommand\eng{\mathcal G}
\newcommand\egbd{\Upsilon}

\newcommand\pvare{\mathsf{c}}

\newcommand\gdh{L}
\newcommand{\icone}{\cone'}
\newcommand{\iconeH}{\coneH'}
\newcommand{\ccone}{\hat{\cone}}
\newcommand{\cconeH}{\hat{\coneH}}
\newcommand{\cmu}{\hat{\mu}}
\newcommand{\density}{\rho}
\newcommand{\ddensity}{\varrho}

\newcommand\adm{\Lambda}
\newcommand{\adl}{\lambda}

\newcommand{\exceptional}{{\rm Exc}}

\newcommand {\absolute}[1] {\left| {#1} \right|}
\newcommand {\norm}[1] {\left\| {#1} \right\|}

\newcommand{\hide}[1]{}

\newcommand\sqf{Q_0}
\newcommand{\cox}{{\mathsf x}}
\newcommand{\coy}{{\mathsf y}}
\newcommand{\coz}{{\mathsf z}}

\newcommand\dn{\mathsf k}

\newcommand\dm{\mathsf m}
\newcommand\dN{\mathsf N}
\newcommand\hdm{\hat{\dm}}

\newcommand{\sfd}{s}
\newcommand{\imp}{\varpi}
\newcommand{\tG}{\tilde G}

\title[Effective equidistribution in rank 2]{Effective equidistribution in rank 2 homogeneous spaces and values of quadratic forms}

\author{E.~Lindenstrauss}
\address[E.L.]{Institute for Advanced Study, 1 Einstein Drive, Princeton, NJ 08540, USA\newline
\emph{and}\newline
The Einstein Institute of Mathematics, Edmund J. Safra Campus, Givat Ram,
The Hebrew University of Jerusalem, Jerusalem 91904, Israel
}
\email{elonl@ias.edu}
\thanks{E.L.\ acknowledges support by ERC 2020 grant HomDyn (grant no.~833423).}

\author{A.~Mohammadi}
\address{A.M.: Department of Mathematics, University of California, San Diego, CA 92093}
\email{ammohammadi@ucsd.edu}
\thanks{A.M.\ acknowledges support by the NSF grants DMS-2055122 and 2350028.}

\author{Z.~Wang}
\address{Z.W.: Department of Mathematics, Johns Hopkins University,
Baltimore, MD 21218}
\email{zhirenw@jhu.edu}
\thanks{Z.W.\ acknowledges support by the NSF grant  DMS-1753042.}

\author{L.~Yang}
\address{L.Y.: Department of Mathematics, National University of Singapore, 119076, Singapore}
\email{lei.yang@nus.edu.sg}
\thanks{L.Y.\ acknowledges support by the Shiing-Shen Chern membership at the Institute for Advanced Study and a startup research funding from National University of Singapore.}

\date{}

\begin{document}

 \begin{abstract}
  We establish effective equidistribution theorems, with a polynomial error rate, for orbits of unipotent subgroups in quotients of quasi-split, almost simple Linear algebraic groups of absolute rank 2.

As an application, inspired by the results of Eskin, Margulis and Mozes, we establish quantitative results regarding the distribution of values of an indefinite ternary quadratic form at integer points, giving in particular an effective and quantitative proof of the Oppenheim Conjecture.
\end{abstract}

\maketitle

\setcounter{tocdepth}{1}
\tableofcontents

\setcounter{part}{-1}
\part{Introduction}

A landmark result of Margulis is his proof of the Oppenheim Conjecture by showing that every $\SO(2,1)$ orbit in $\SL_3(\R)/\SL_3(\Z)$ is either periodic or unbounded (or both --- in our terminology an orbit $L.x$ is periodic if it supports a finite $L$ invariant measure; this implies $L.x$ is closed but does not preclude it being noncompact). Subsequently, Dani and Margulis showed that any $\SO(2,1)$-orbit in $\SL_3(\R)/\SL_3(\Z)$ is either periodic or dense, and also classified possible orbit closures for a one-parameter unipotent subgroup of $\SO(2,1)$.

More precise information regarding the behavior of orbits of one-parameter unipotent subgroups in quotients of real Lie groups was provided by Ratner in \cite{Ratner-Acta, Ratner-measure, Ratner-topological}. These remarkable theorems have been highly influential and have had a plethora of applications, many of them quite unexpected.

This paper is a step in a program to make Ratner's equidistribution theorem effective. Previously, the first three authors proved effective equidistribution results for unipotent flows in $G=\SL_2(\C)$ or $G=\SL_2(\R)\times\SL_2(\R)$ and the last named author for $G=\SL_3(\R)$ and $u_t$ a singular one-parameter unipotent group; here we consider the generic one-parameter group in any quasi-split group of absolute rank 2.

As an application, we give here an effective and quantitative equidistribution result for the values of an indefinite ternary quadratic form at integer points in large balls. This gives in particular an effective and quantitative proof of the Oppenheim Conjecture. An effective proof of the Oppenheim Conjecture was given by Margulis and the first named author in \cite{LM-Oppenheim}, but the rates we give here (say of the smallest nonzero value of $|Q(v)|$ with $v$ an integer vector of norm at most $T$) are polynomial, which is the right kind of dependency, vs.\ a polylogarithmic rate in \cite{LM-Oppenheim}.

Moreover, the quantitative equidistribution result for the values of an indefinite ternary quadratic form is new --- no effective or explicit rate was previously known. Here we follow Eskin, Margulis and Mozes who gave a beautiful argument proving a qualitative\footnote{Eskin, Margulis and Mozes called their result a quantitative version of the Oppenheim conjecture. It is quantitative in the sense that it counts the number of lattice points in a large ball for which $Q(v)$ is in a given interval, but not quantitative in the sense it does not say how large the ball has to be before this asymptotical behaviour begins to hold, and is not effective.}
result of this type for all indefinite quadratic forms of signature $(p,q)$ with $p\geq 3$ and subsequently also for forms of signature $(2,2)$. The techiniques of Eskin, Margulis and Mozes were recently extended to forms of signature $(2,1)$ by Wooyeon Kim in \cite{Kim-Oppenheim}.

\section{Effective equidistribution}\label{sec: equidistribution thm}
In this section we state the main equidistribution theorems of this paper, Theorem~\ref{thm: main} and Theorem~\ref{thm:main unipotent}. These theorems will be proved in Part~1 of the paper.

In \S\ref{sec: Oppenheim statement}, we will discuss applications of these equidistribution theorems to values of quadratic forms, viz.\ Oppenheim conjecture. Part~2 of the paper, is devoted to the proof of these results using the equidistribution result Theorem~\ref{thm: main} as well as an analysis of the cusp excursions for certain orbits which is closely related to the works of Eskin Margulis and Mozes in \cite{EMM-Upp,EMM-22} and Wooyen Kim \cite{Kim-Oppenheim}.

Throughout Part~1 of the paper, $G$ denotes the connected component of identity (as a Lie group) of the real points of an $\R$-algebraic group isogenous to one of the following
\[
\SL_2(\C),\quad \SL_2(\R)\times\SL_2(\R),\quad \SL_3(\R),\quad {\rm SU}(2,1), \quad {\rm Sp}_4(\R), \quad {\bf G}_2(\R).
\]
Put differently, $G$ is the finite index subgroup $\G(\R)^+\subset\G(\R)$ where $\G$ is a semisimple, connected, algebraic $\R$-group which is $\R$-quasi split and has absolute rank $2$, and $\G(\R)^+$ is the subgroup generated by unipotent one-parameter subgroups~\cite[Ch.~I]{Margulis-Book}. 
For $G=\SL_2(\C)$ and $\SL_2(\R)\times\SL_2(\R)$, Theorem~\ref{thm: main} and Theorem~\ref{thm:main unipotent} where established in~\cite{LMW22}; the proof we give here in particular gives a somewhat more streamlined proof of effective equidistribution also in these cases, but up to minor cosmetic improvements in Theorem~\ref{thm:main unipotent} we do not provide any new results in those cases. 

Let $H$ denote the image of a principal $\SL_2(\R)$ in $G$. In particular, $H$ is a maximal connected subgroup which is locally isomorphic to $\SL_2(\R)$, and $\Lie(G)=\Lie(H)\oplus\rfrak$ decomposes as sum of two {\em irreducible} representations of $H$, see \S\ref{sec: principal SL2}. 

For all $t,r\in\R$, let $a_t$ and $u_r$ denote the images of 
\[
\begin{pmatrix}
    e^{t/2} & 0 \\
    0 & e^{-t/2}
\end{pmatrix}
\quad\text{and}\quad \begin{pmatrix}
    1 & r \\
    0 & 1 
\end{pmatrix}.
\]
in $H$, respectively. Then $a_t$ and $u_r$ are {\em regular} one parameter diagonalizable and unipotent subgroups of $G$, respectively. 
With this normalization, if $v\in\rfrak$ is a highest weight vector for $a_t$, then 
\[
\Ad(a_t)v=e^{\dm t}v\qquad\text{where $\dim\rfrak=2\dm+1$}. 
\]

Let $\Gamma\subset G$ be an arithmetic lattice. 
By Margulis' arithmeticity theorem, any lattice $\Gamma\subset G$ is arithmetic, except possibly when $G$ is isogenous to ${\rm SU}(2,1)$ and $\SL_2(\C)$, or $G$ is isogenous to $\SL_2(\R)\times\SL_2(\R)$ and $\Gamma$ is reducible, where non-arithmetic lattices exist. 

Let $X=G/\Gamma$. Let $\dist$ be the right invariant metric on $G$ which is defined using the Killing form and the Cartan involution. This metric induces a metric $\dist_X$ on $X$, \label{d definition page}
and natural volume forms on $X$ and its submanifolds. Let $m_X$ denote the probability Haar measure on $X$

\begin{thm}\label{thm: main}
For every $x_0\in X$ and large enough $R$ (depending explicitly on $x_0$), for any $T \geq R^{\ref{a:main-3-1}}$, at least one of the following holds.
\begin{enumerate}
\item For every $\varphi\in C_c^\infty(X)$, we have 
\[
\biggl|\int_0^1 \varphi(a_{\log T}u_rx_0)\diff\!r-\int \varphi\diff\!m_X\biggr|\leq \Sob(\varphi)R^{-\ref{k:main-3-1}}
\]
where $\Sob(\varphi)$ is a certain Sobolev norm. 
\item There exists $x\in X$ such that $Hx$ is periodic with $\vol(Hx)\leq R$, and 
\[
\dist_X(x,x_0)\leq R^{\ref{a:main-3-1}}(\log T)^{\ref{a:main-3-1}}T^{-\dm}
\] 
where $\dim\rfrak=2\dm+1$. 
\end{enumerate} 
The constants $\consta\label{a:main-3-1}$ and $\constk\label{k:main-3-1}$ are positive, and depend on $X$ but not on $x_0$.
\end{thm}

The strategy for the proof of Theorem~\ref{thm: main} is similar to the general strategy developed in~\cite{LM-PolyDensity, LMW22}. 
A significant simplification is achieved here thanks to the use of higher dimensional energy. 
This in turn is made possible by using stronger projection theorems proved by Gan, Guo, Wang \cite{GGW}. 

Quantitative and effective versions of the aforementioned rigidity theorems in homogeneous dynamics have been sought after for some time, we refer to~\cite{LM-PolyDensity, LMW22, Yang-SL3} for a more detailed discussion of this problem and some recent progress.

As it was done in~\cite{LMW22}, combining Theorem~\ref{thm: main} and the Dani--Margulis linearization method \cite{DM-Linearization} (cf.\ also Shah \cite{Shah-MathAnn}),
that allows to control the amount of time a unipotent trajectory spends near
invariant subvarieties of a homogeneous space, 
we also obtain an effective equidistribution theorem for long pieces of unipotent orbits (more precisely, we use a sharp form of the linearization method taken from~\cite{LMMS}).

\begin{thm}\label{thm:main unipotent}
For every $x_0\in X$ and large enough $R$ (depending explicitly on $X$), for any $T\geq R^{\ref{a:unipotent-1}}$, at least one of the following holds.
\begin{enumerate}
\item For every $\varphi\in C_c^\infty(X)$, we have 
\[
\biggl|\frac{1}{T}\int_0^{T} \varphi(u_rx_0)\diff\!r-\int \varphi\diff\!m_X\biggr|\leq \Sob(\varphi)R^{-\ref{k:main uni}}
\]
where $\Sob(\varphi)$ is a certain Sobolev norm. 
\item There exists $x \in G/\Gamma$ with $\vol(H.x)\leq R^{\ref{a:unipotent-1}}$, 
and for every $r\in [0,T]$ there exists $g(r)\in G$ with $\|g(r)\|\leq R^{\ref{a:unipotent-1}}$ so that  
\[
\dist_X(u_{s}x_0, g(r)H.x)\leq R^{\ref{a:unipotent-1}}\left(\frac{|s-r|}{T}\right)^{1/\ref{a:unipotent-2}}\quad\text{for all $s\in[0,T]$.}
\] 
\item There is a parabolic subgroup $P\subset G$ and some $x \in G/\Gamma$ satisfying that
$\vol(R_u(P).x)\leq R^{\ref{a:unipotent-1}}$, and for every $r\in [0,T]$ there exists $g(r)\in G$ with $\|g(r)\|\leq R^{\ref{a:unipotent-1}}$ so that  
\[
\dist_X(u_{s}x_0, g(r)R_u(P).x)\leq R^{\ref{a:unipotent-1}}\left(\frac{|s-r|}{T}\right)^{1/\ref{a:unipotent-2}}\quad\text{for all $s\in[0,T]$.}
\] 
In particular, $X$ is not compact. 
\end{enumerate} 
The constants $\consta\label{a:unipotent-1}$, $\consta\label{a:unipotent-2}$, and $\constk\label{k:main uni}$ are positive and depend on $X$ but not on $x_0$. 
\end{thm}

Remarks: In options (2) and (3), one can be more explicit about the properties of the element $g(r)$. For instance in option~(2), since $u_{s}x_0$ needs to stay close to $g(r)H.x$ for all $s$ in long intervals around $r$ it follows from (2) that $g(r)$ is very close to $C_G(U)$  (there is of course some play between what ``stay close'' and ``long interval'' means --- if one uses a stricter interpretation of what close means, then one must be more lenient on what long means and vice versa). An analogous statement for $G=\SL_2(\R)\times\SL_2(\R)$ turned out to be useful in the work of Forni, Kanigowski, Radziwi\l{}\l{} studying equidistribution of orbits at nearly prime times \cite{FKR-primes}

\section{Applications to the Oppenheim conjecture}\label{sec: Oppenheim statement}

We now discuss applications of Theorem~\ref{thm: main} to values of quadratic forms in the context of the Oppenheim conjecture. 

\begin{thm}\label{thm: Oppenheim}
 Let $Q$ be an indefinite ternary quadratic form with $\det Q=1$. For all $R$ large enough, depending on $\norm{Q}$, and all $T\geq R^{\ref{a:Oppenheim}}$ at least one of the following holds.
\begin{enumerate}
    \item For every $s\in [-R^{\ref{k:Oppenheim}}, R^{\ref{k:Oppenheim}}]$, there exists a primitive vector $v\in \Z^3$ with $0<\norm{v}\leq T$ so that 
    \[
    \absolute{Q(v)-s}\leq R^{-\ref{k:Oppenheim}}
    \]
    \item There exists some $Q'\in\Mat_3(\R)$ with $\norm{Q'}\leq R$ so that 
    \[
    \norm{Q-\lambda Q'}\leq  R^{\ref{a:Oppenheim}}(\log T)^{\ref{a:Oppenheim}} T^{-2}\qquad\text{where $\lambda=(\det Q')^{-1/3}$.}
    \]
\end{enumerate}
The constants $\consta\label{a:Oppenheim}$ and $\constk\label{k:Oppenheim}$ 
are absolute.
\end{thm}

Since algebraic numbers cannot be well approximated by rationals, one concludes the following corollary from Theorem~\ref{thm: Oppenheim}. 

\begin{coro}\label{cor: Oppenheim}
Let $Q$ be a reduced, indefinite, ternary quadratic form which is not proportional to an integral form but has algebraic coefficients. Then for all $T$ large enough, depending on the degrees and heights of the coefficients of $Q$,  
we have the following: for any 
\[
s\in [- T^{\ref{k:Oppenheim 2}},  T^{\ref{k:Oppenheim 2}}]
\]
there exists a primitive vector $v\in \Z^3$ with $0<\norm{v}\leq T$ so that 
    \[
    \absolute{Q(v)-s}\leq T^{-\ref{k:Oppenheim 2}}.
    \]
The constant $\constk\label{k:Oppenheim 2}$ depends on the degrees of the coefficients of $Q$.  
\end{coro}

As was mentioned, Theorem~\ref{thm: Oppenheim} and Corollary~\ref{cor: Oppenheim} are an effective version of a celebrated theorem of Margulis~\cite{Margulis-Oppenheim}, see also~\cite{DM-Elemntary}. An effective version, with a polylogarithmic rate, was proved by the first named author and Margulis~\cite{LM-Oppenheim}. The proofs in~\cite{Margulis-Oppenheim, DM-Elemntary} and~\cite{LM-Oppenheim} rely on establishing a special case of Raghunathan's conjecture for unipotent flows --- albeit an effective version in the case of~\cite{LM-Oppenheim}. 

Similarly, our proof of Theorem~\ref{thm: Oppenheim} is based on Theorem~\ref{thm: main}. Indeed, using Theorem~\ref{thm: main}, for $G=\SL_3(\R)$, $H=\SO(Q)^\circ$, and adapting the arguments developed by Dani and Margulis~\cite{DM-Linearization} and by Eskin, Margulis, and Mozes~\cite{EMM-Upp, EMM-22}, and recent advances by W.~Kim \cite{Kim-Oppenheim}, we obtain the following theorem. 

\begin{thm}\label{thm: quantitative Oppenheim}
Let $Q$ be an indefinite ternary quadratic form with $\det Q=1$, and put 
\[
\mathsf C_Q=\int_{L}\frac{\diff\!\sigma}{\norm{\nabla Q}}
\]
where $L=\{v\in\R^3: \|v\|\leq 1, Q(v)=0\}$ and $\diff\!\sigma$ is the area element on $L$.

Let $a<b$ and $A\geq 10^3$. There are constants $T_0$ and $C$ depending on $A$ and $\norm{Q}$,  absolute constants $N\geq 1$ and $0<\delta_0<1$, and for every $0<\delta\leq \delta_0$ some $\kappa=\kappa(\delta, A)$ so that the following holds. 

Assume that for $T\geq T_0$ and all integral forms $Q'$ with $\|Q'\|\leq T^\delta$
    \[
    \norm{Q-\lambda Q'}\geq  \|Q'\|^{-A}\qquad\text{where $\lambda=(\det Q')^{-1/3}$,}
    \]
then the following is satisfied: If 
\begin{multline*}
    \biggl|\# \Bigl\{v\in\Z^3: \norm{v}\leq T, a\leq Q(v)\leq b\Bigr\}-\mathsf C_Q(b-a)T \biggr|\geq\\ 
    + C (1+\absolute{a}+\absolute{b})^NT^{1-\kappa},
    \end{multline*}
    there are at most 4 lines $L_1,\ldots,L_4$ and at most 4 planes $P_1,\ldots, P_4$ so that 
    \begin{multline*}
    \# \Bigl\{v\in\Z^3: \norm{v}\leq T, a\leq Q(v)\leq b\Bigr\}=\\
    \mathsf C_Q(b-a)T +\mathcal R_T
    + O\Bigl((1+\absolute{a}+\absolute{b})^NT^{1-\kappa}\Bigr)
    \end{multline*}
    where $\mathcal R_T=\#\{v\in(\cup_i L_i)\textstyle\bigcup(\cup_i P_i): \|v\|\leq T, a\leq Q(v)\leq b\}$. 
    
More precisely, these exceptional lines and planes satisfy the following: 
$L_i\cap\Z^3={\rm Span}\{v_i\}$ satisfying that    
    \be\label{eq: size of v and Qv}
    \|v_{i}\|\leq T^{\delta}\quad\text{and}\quad |Q(v_{i})|\leq T^{-2+\delta}
    \ee
for every $1\leq i\leq 4$.

Also $P_i\cap \Z^3={\rm Span}\{w_{i,1}, w_{i,2}\}$ satisfying  
 \be\label{eq: size of w and Qw}
    \|w_{i,1}\|, \|w_{i,2}\|\leq T^{\delta}\quad\text{and}\quad |Q^*(w_{i,1}\wedge w_{i,2})|\leq T^{-2+\delta}
    \ee
for every $1\leq i\leq 4$. 
     \end{thm}

Note that the main term in part~(1) captures the asymptotic behavior of the volume of the solid given by $Q(v)=a$, $Q(v)=b$, and 
$\norm{v}\leq T$. We also remark that there is a dense family of irrational quadratic forms 
which are very well approximated by rational forms so that  
\[
\# \Bigl\{v\in\Z^3: \norm{v}\leq T, a\leq Q(v)\leq b\Bigr\}\gg T(\log T)^{1-\vare},
\]
see~\cite[\S 3.7]{EMM-Upp}.

\subsection*{Acknowledgment}
We are grateful to the Institute for Advanced Study for its hospitality on multiple occasions during the completion of this project. We thank Hong~Wang for many valuable discussions on projection theorems. We also thank Zuo~Lin for going over early drafts of the elements of \S\ref{sec: Marg func} and providing helpful feedback.

The landmark work of G.A.\ Margulis on the Oppenheim Conjecture and more generally in homogeneous dynamics and its applications to number theory has been a continued source of inspiration for us. In particular, E.L.\ and A.M.\ have had joint works closely related to this paper with Margulis, and are very grateful for all we have learned from him. His insights feature in many places in this paper, sometimes implicitly.


\part{Polynomially effective equidistribution theorems}

In this chapter we prove Theorem~\ref{thm: main} and Theorem~\ref{thm:main unipotent}. 
As noted earlier, the overall strategy for the proofs aligns with the approach developed in~\cite{LMW22}. A key simplification in the proof of Theorem~\ref{thm: main} is accomplished through the use of higher-dimensional energy, made possible by stronger projection theorems established by Gan, Guo, and Wang \cite{GGW}.

\section{Notation and preliminary results}\label{sec: notation}

Let $\G$ be a semisimple, connected, $\R$-group with absolute rank $2$ which is $\R$-quasi split. Let $G=\G(\R)^+$, where $\G(\R)^+$ is the subgroup generated by unipotent one-parameter subgroups. In other words $G$ is the connected component of the identity in the Lie group $\G(\R)$, see~\cite[Ch.~I]{Margulis-Book}. 

More explicitly, $\G$ is isogenous to one the following groups 
\[
\SO(3,1),\quad\SL_2\times \SL_2,\quad\SL_3,\quad {\rm SU}(2,1), \quad {\rm Sp}_4, \quad {\bf G}_2.
\]
Indeed, $\G$ is $\R$-split except when it is isogenous to $\SO(3,1)$ or ${\rm SU}(2,1)$, in which case $\G$ is $\R$-quasi split but not split. 
We assume $G\subset\SL_\dN(\R)$ for some $\dN$ which is fixed throughout the paper.

\subsection{The principal $\SL_2(\R)$ in $G$}\label{sec: principal SL2}
There is a homomorphism $\rho:\SL_2(\R)\to G$ with finite central kernel so that 
\[
\Lie(G)=\Lie(H)\oplus \rfrak
\]
where $H=\rho(\SL_2(\R))$ and $\rfrak$ is an irreducible representation of $H$
with dimension $2\dm+1$ where 
\begin{quote}
$\bullet \;\; \dm=1$ if $\G$ is isogeneous to $\SO(3,1)$ or $\SL_2\times\SL_2$\\
     $\bullet \;\; \dm=2$ if $\G$ is isogeneous to $\SL_3$ or ${\rm SU}(2,1)$\\
     $\bullet \;\; \dm=3$ if $\G$ is isogeneous to ${\rm Sp}_4$\\
     $\bullet \;\; \dm=5$ if $\G$ is isogeneous to ${\G}_2$
\end{quote}  
see~\cite{And-Regular} and the references therein, in particular, see~\cite[Prop.~6]{And-Regular}. 

For all $t,r\in\R$, let $a_t$, $u_r$, and $u^-_s$ denote the images of 
\[
\begin{pmatrix}
    e^{t/2} & 0 \\
    0 & e^{-t/2}
\end{pmatrix},
\qquad\begin{pmatrix}
    1 & r \\
    0 & 1 
\end{pmatrix},
\quad\text{and}\quad \begin{pmatrix}
    1 & 0 \\
    s & 1 
\end{pmatrix}
\]
in $H$, respectively. Then $a_t$ and $u_r$ are {\em regular} one parameter diagonalizable and unipotent subgroups of $G$, respectively.
Similarly, $u_s^-$ is a regular unipotent subgroups of $G$. We let $A=\{a_t\}$, $U=\{u_r\}$, and $U^-=\{u_s^-\}$.  

All regular unipotent one-parameter groups, and hence the corresponding groups $H$, are conjugated to each other under ${\G}^{\rm ad}(\R)$, the adjoint form of $\G$, see~\cite[Thm.~1]{And-Regular}. 

Note also that with the above normalization, if $v\in\rfrak$ is a highest weight vector for $a_t$, then $\Ad(a_t)v=e^{\dm t}v$, recall that 
$\dim\rfrak=2\dm+1$. 

For any $A$-invariant subspace $V\subset\mathfrak{sl}_\dN(\R)$, let $V^0$ denote 
the space of $a_t$-fixed vectors and 
\[
V^\pm=\{z\in V: \textstyle\lim_{t\to\mp\infty}\Ad(a_t)z=0\}.
\]
We assume the embedding $G\subset\SL_\dN(\R)$ is fixed so that 
$\gfrak^0\oplus\gfrak^+=\gfrak\cap{\mathfrak b}_\dN$, the subalgebra of upper triangular matrices in $\mathfrak{gl}_\dN$.

\subsection*{Lie algebras and norms}
Recall that $G\subset \SL_\dN(\R)$.
Let $|\;|$ denote the usual absolute value on $\R$.  
Let $\|\;\|$ denote the maximum norm $\Mat_\dN(\R)$, with respect to the standard basis. 
 
We fix a norm on $\hfrak$ by taking the maximum norm where the coordinates are given by fixing unit basis vectors for the lines $\Lie(U)$, $\Lie(U^-)$, and $\Lie(A)$. 
Since $\rfrak$ is $\Ad(H)$-irreducible, each weight space is one dimensional. Fix a norm on $\rfrak$ by taking the max norm with respect to a basis consisting of unit pure weight (with respect to $a_t$) vectors.  
By taking maximum of these two norms we get a norm on $\gfrak$, which will also be denoted by $\|\;\|$.

Let $\constE\label{E:2ball}\label{E:dist-sheet}\geq 1$ be so that 
\be\label{eq:2ball}
\|hw\|\leq \ref{E:2ball}\|w\|\quad\text{for all $\|h-I\|\leq 2$ and all $w\in \gfrak$.} 
\ee

For all $0<\delta<1$, we define
\be\label{eq:def-BoxH}
\boxH_{\delta}:=\{u_s^-:|s|\leq {\delta}\}\cdot\{a_t: |t|\leq \delta\}\cdot\{u_\rel:|\rel|\leq {\delta}\}
\ee 
Define $\mathsf B_\delta^L$ for $\delta>0$ and $L=U, U^-$ and $A$ similarly, and put 
\[
\boxHs_{\delta}=\{u_s^-:|s|\leq {\delta}\}\cdot\{a_t: |t|\leq \delta\}=\mathsf B_\delta^{U^-}\cdot\mathsf B_\delta^A.
\]
Note that for all $h_i\in(\boxH_{\delta})^{\pm1}$, $i=1,\ldots,5$, we have  
\be\label{eq:B-beta-almost-group}
 h_1\cdots h_5\in \boxH_{100\delta}.
\ee
We also define $\boxG_{\delta}:=\boxH_\delta\cdot\exp(B_\rfrak(0,\delta))$ where $B_\rfrak(0,{\delta})$ denotes the ball of radius $\delta$ in $\rfrak$ with respect to $\|\;\|$. 

Given an open subset $\mathsf B\subset L$, with $L$ any of the above, and $\delta>0$, put $\partial_\delta\mathsf B=\{\sfh\in\mathsf B: \mathsf B_\delta^L.\sfh\not\subset\mathsf B\}$.

We fix Haar measures $m_G$ on $G$ and $m_H$ on $H$. 
Let $\Gamma\subset G$ be an arithmetic lattice, and let $m_X$ denote the probability Haar measure on $X=G/\Gamma$.
 
For all $x\in X$, define $\inj(x)$ as follows 
\be\label{eq:def-inj}
\inj(x)=\min\Bigl\{0.01, \sup\Bigl\{\delta: \text{$g\mapsto gx$ is injective on $\boxG_{10N^2\delta}$}\Bigr\}\Bigr\}.
\ee
For every $\injr>0$, let  $X_\injr=\Bigl\{x\in X: \inj(x)\geq \injr\Bigr\}$.

\subsection*{Commutation relations} Let us record the following two lemmas.

\begin{lemma}[\cite{LM-PolyDensity}, Lemma 2.1]
\label{lem: BCH}
There exist $\delta_0$ and $\constE\label{E:BCH}$ depending on $\dm$ so that the following holds.  
Let $0<\delta\leq \delta_0$, and let $w_1,w_2\in B_\rfrak(0,\delta)$. There are $h\in H$ and $w\in\rfrak$ which satisfy 
\[
\text{$\tfrac23\|w_1-w_2\|\leq \|w\|\leq \tfrac32\|w_1-w_2\|\quad$ and $\quad\|h-I\|\leq \ref{E:BCH}\delta\|w\|$}
\] 
so that $\exp(w_1)\exp(-w_2)=h\exp(w)$. More precisely,
\[
\|w-(w_1-w_2)\|\leq \ref{E:BCH}\delta\|w_1-w_2\|
\]
\end{lemma}

\begin{lemma}[\cite{LM-PolyDensity}, Lemma 2.2]
\label{lem:dist-sheet} 
There exists $\delta_0$ so that the following holds for all $0<\delta\leq \delta_0$.
Let $x\in X_{\delta}$ and $w\in B_\rfrak(0,\delta)$. If there are $h,h'\in \boxH_{2\delta}$
so that $\exp(w')hx=h'\exp(w)x$, then 
\[
\text{$h'=h\quad$ and $\quad w'=\Ad(h)w$}.
\]
Moreover, we have $\|w'\|\leq 2\|w\|$.
\end{lemma}

\subsection*{The set $\coneH_{\eta, t,\beta}$}
For all $\eta,t, \beta>0$, set 
\be\label{eq:def-Ct}
\coneH_{\eta, t,\beta}:=\boxHs_{\beta}\cdot a_t\cdot \big\{u_r: r\in[0,\eta]\big\} \subset H.
\ee
Then $m_H(\coneH_{\eta,t,\beta})\asymp \eta{\beta}^2\nuni^{t}$ where $m_H$ denotes our fixed Haar measure on $H$.

Throughout the paper, the notation $\coneH_{\eta, t,\beta}$ will be used only for $\eta,t, \beta>0$ 
which satisfy $\nuni^{-0.01t}<\beta\leq\eta^2$, even if this is not explicitly stated.

For all $\eta,\beta, \tau>0$, put 
\be\label{eq: def B ell beta}
\umt^H_{\eta,\beta,\tau}=\Bigl\{u^-_s: |s|\leq \beta \nuni^{-\tau}\Bigr\}\cdot\{a_d: |d|\leq \beta\}\cdot\Bigl\{u_r: |r|\leq \eta\Bigr\}.
\ee
Roughly speaking, $\umt_{\eta, \beta,\tau}^H$ is a {\em small thickening} of the $(\beta,\eta)$-neighborhood of the identity in $AU$. 
We write $\umt^H_{\beta,{\tau}}$ for $\umt^H_{\beta,\beta,\tau}$.

The following lemma will also be used in the sequel.

\begin{lemma}[\cite{LM-PolyDensity}, Lemma 2.3]
\label{lem:commutation-rel}
\begin{enumerate}
\item Let $\tau\geq 1$, and let $0<\eta,\beta<0.1$. Then 
\[
\bigg(\Bigl(\umt_{\eta/10\dN^2,\beta/10\dN^2,\tau}^H\Bigr)^{\pm1}\bigg)^3\subset \umt_{\eta,\beta, \tau}^H.
\]
\item For all $0\leq \beta,\eta\leq 1$, $t,\tau>0$, and all $|r|\leq 2$, we have
\be\label{eq:well-rd-tau-1}
\Bigl(\umt_{\eta, \beta^2, \tau}^H\Bigr)^{\pm1}\cdot a_\tau u_r \coneH_{\eta', t,\beta'}\subset a_\tau u_r\coneH_{\eta,t, \beta},
\ee
where $\eta'= \eta(1-10\dN^2\nuni^{-t})$ and $\beta'= \beta(1-10\dN^2\beta)$.
\end{enumerate}
\end{lemma}

\subsection*{Linear algebra lemma}
Recall that $\dim\rfrak=2\dm+1$, and that $\rfrak$ is $\Ad(H)$-irreducible. 
We have the following  

\begin{lemma}[cf.~Lemma A.1,~\cite{LMW22}]
\label{lem: linear algebra 1}
Let $0<\alpha\leq 1/(2\dm +1)$. For all $d>0$ and all $0\neq w\in\rfrak$, we have 
\[
\int_0^1\|a_d u_rw\|^{-\alpha}\diff\!r\leq C e^{-\alpha\dm d}\|w\|^{-\alpha }
\]
where $C$ is an absolute constant.  
\end{lemma}

\section{Avoidance principles in homogeneous spaces}\label{sec: avoidance}

In this section we will collect statements concerning avoidance principles for unipotent flows and random walks on homogeneous spaces; the reader will find this section similar to \cite[\S4]{LMW22} and \cite[\S3]{LM-PolyDensity}. 
The proofs are included in Appendix~\ref{sec: app avoidance} for the convenience of the reader.     

\subsection{Nondivergence results}\label{sec: non-div}\label{sec:SiegelSet}
The results of this subsection are only interesting when $\Gamma$ is a nonuniform lattice, i.e., when $X$ is not compact. 

\begin{propos}
\label{prop:Non-div-main}\label{prop:one-return}\label{lem:one-return}\label{prop: Non-div main}
There exist $\dm_0$ depending only on $\dm$ and $\constE\label{E:non-div-main}\geq 1$ depending on $X$ with the following property. 
Let $0<\delta, \vare<1$, and let $I\subset [-10, 10]$ be an interval with $|I|\geq \delta$. For all $x\in X$, we have 
\[
\Bigl|\Bigl\{r\in I:a_s\uvk x\not\in X_\vare\Bigr\}\Bigr|<\ref{E:non-div-main}\vare^{1/\dm_0} |I|,
\]
so long as $s\geq \dm_0|\log(\delta\inj(x))|+\ref{E:non-div-main}$. 
\end{propos} 

\begin{proof}
There exist $\dm'$, depending on $\dm$, and a function $\omega:X\to [2,\infty)$ so that for all $x\in X$ 
\[
\inj(x)\geq \omega(x)^{-\dm'},
\]
moreover, for all $x\in X$ and all $s\geq \dm'|\log\delta|+B'_0$ 
\[
\int_I\omega(a_s\uvk x)\diff\!r\leq e^{-\star s}\omega(x)+B_0,
\]
where $B'_0$ and $B_0$ depends on $X$. See Proposition~\ref{prop: average of inj} and references there. 
        
        The claim in the proposition follows from these statements and Chebyshev's inequality.  
\end{proof}


\subsection{Inheritance of the Diophantine property}\label{sec: Dioph inheritance}
The following proposition improves the somewhat weak Diophantine property that failure of part~(2) in Theorem~\ref{thm: main} provides to a Diophantine property in terms of volume of periodic orbits involved. This proposition can be compared with the main theorem in~\cite{LMMS}, and will be applied as the first step in the proof of Theorem~\ref{thm: main}.  

\begin{propos}\label{prop: linearization translates}
There exist $D_0$, depending on $\dm$, and $\constE\label{c: linear trans}, s_0$, depending on $X$, so that the following holds. 
Let $R, S\geq 1$. Suppose $x_0\in X$ is so that 
\[
\dist_X(x_0,x)\geq (\log S)^{D_0}S^{-\dm}
\] 
for all $x$ with $\vol(Hx)\leq R$. Then for all 
\[
s\geq \max\Bigl\{\log S, \dm_0|\log(\inj(x_0))|\Bigr\}+s_0
\] 
and all $0<\eta\leq 1$, we have 
\[
\biggl|\biggl\{r\in [0,1]\!:\!\! \begin{array}{c}\inj(a_su_rx_0)\leq \eta \text{ or there is $ x$ with }\\ 
\vol(Hx)\leq R \text{ s.t. }\dist_X(a_{s}u_rx_0,x)\leq \frac{1}{\ref{c: linear trans}R^{D_0}}\end{array}\!\!\biggr\}\biggr|\!\!\leq \ref{c: linear trans}(\eta^{1/\dm_0}+R^{-1}).
\]
\end{propos}

The proof of Proposition~\ref{prop: linearization translates} is postponed to \S\ref{sec: proof linearization}.

\subsection{Closing Lemma}\label{sec: closing lemma}
Let $0<\injr\ll_X1$ and $\beta=\eta^2$. 
For every $\tau\geq0$, put
\be\label{eq: def coneH beta tau}
\coneH_{\tau}=\coneH_{1,\tau,\beta}=\boxHs_{\beta}\cdot a_\tau\cdot \{u_r: r\in [0,1]\} \subset H.
\ee
where $\boxHs_{\beta}=\{u_s^-:|s|\leq {\beta}\}\cdot\{a_d: |d|\leq \beta\}$, see~\eqref{eq:def-Ct}.

If $y\in X$ is so that the map $\sfh\mapsto \sfh y$ is injective over $\coneH_{\tau}$, then $\mu_{\coneH_\tau.y}$ 
denotes the pushforward of the normalized Haar measure on $\coneH_\tau$ to $\coneH_\tau.y\subset X$.

Let $\tau\geq 0$ and $y\in X$. For every $z\in\coneH_\tau.y$, put
\[
\margI(z):=\Bigl\{w\in \rfrak: \|w\|<\inj(z) \text{ and } \exp(w) z\in \coneH_\tau.y\Bigr\};
\]
this is a finite subset of $\rfrak$ since $\coneH_\tau$ is bounded ---  
we will define $\margI_\cone(z)$ for more general sets $\cone$ in the bootstrap phase below.

Let $0<\alpha\leq 1$. Define the function $f:\coneH_\tau.y\to [1,\infty)$ as follows
\[
f(z)=\begin{cases} \sum_{0\neq w\in \margI(z)}\|w\|^{-\alpha} & \text{if $\margI(z)\neq\{0\}$}\\
\inj(z)^{-\alpha}&\text{otherwise}
\end{cases}.
\]

\begin{propos}\label{prop: closing lemma}
There exist $\dm_1$ depending on $\dm$ and $D_1$ depending on $X$ which satisfy the following. 
Let $D\geq D_1$ and $x_1\in X$. Then for all large enough $t$ (depending on $\inj(x_1)$) at least one of the following holds.
 
\begin{enumerate}
\item  There is a subset $J(x_1)\subset [0,1]$ with $|[0,1]\setminus J(x_1)|\ll_X \eta^{1/(2\dm_0)}$ 
such that for all $r\in J(x_1)$ we have the following  
\begin{enumerate}
\item $a_{\dm_1t}u_rx_1\in X_{\eta}$.
\item $\sfh\mapsto \sfh.a_{\dm_1t}u_rx_1$ is injective over $\coneH_{\rws}$.
\item For all $z\in\coneH_\rws.a_{\dm_1t}u_rx_1$, we have 
\[
f(z)\leq  \nuni^{D\rws}.
\]
\end{enumerate}

\item There is $x\in X$ such that $Hx$ is periodic with
\[
\vol(Hx)\leq \nuni^{D_1\rws}\quad\text{and}\quad\dist_X(x,x_1)\leq \nuni^{(-D+D_1)\rws}.
\] 
\end{enumerate} 
\end{propos}

The proof of Proposition~\ref{prop: closing lemma} is postponed to~\S\ref{sec: proof of closing}.

\section{Mixing property and equidistribution}\label{sec: horospheric}

Let us recall the following quantitative decay of correlations for the ambient space $X$: 
There exists $0<\mixexp\leq1$ so that  
\be\label{eq: exp mixing}
\biggl|\int \varphi(gx)\psi(x)\diff\!{m_X}-\int\varphi\diff\!{m_X}\int\psi\diff\!{m_X}\biggr|\ll \Sob(\varphi)\Sob(\psi) \nuni^{-\mixexp d(e,g)}
\ee
for all $\varphi,\psi\in C^\infty_c(X)+\bbc\cdot 1$, where $m_X$ is the $G$-invariant probability measure on $X$ and $d$ is the right $G$-invariant metric on $G$ defined on p.~\pageref{d definition page}. See, e.g., \cite[\S2.4]{KMnonquasi} and references there for~\eqref{eq: exp mixing}. 

Here $\Scal(\cdot)$ is a certain Sobolev norm on $C_c^\infty(X)+\bbc\cdot1$ 
which is assumed to dominate $\|\cdot\|_\infty$ and the 
Lipschitz norm $\|\cdot\|_{\rm Lip}$. Moreover, $\Scal(g.f)\ll\|g\|^\star\Scal(f)$ where the implied constants depend only on $\dm$.


\begin{propos}\label{prop: 1-epsilon N}
There exists $\constk\label{k:mixing}\gg\mixexp$ so that the following holds. Let $\Lambda\geq 1$, 
and let $\nu$ be a probability measure on $B_1^G$ with 
\be\label{eq: RD nu assump}
\frac{\diff\!\nu}{\diff\!m_{G}}(g)\leq \Lambda\qquad\text{for all $g\in\supp\nu$.}
\ee
Let $\ell_1, \ell_2>0$ and $0<\eta<1$ satisfy the following 
\[
\ref{k:mixing}\ell_2\geq \max\{\ell_1,|\log\eta|\}.
\]
Then for all $x\in X_\eta$ and all $\varphi\in C_c^\infty(X)$, we have 
\[
\int_0^1\!\!\int_{G} \varphi(a_{\ell_1}u_{r}a_{\ell_2}gx)\diff\!\nu(g)\diff\!r=\!\int\varphi\diff\!m_X+O\Bigl(\Sob(\varphi)(\eta+\Lambda^{1/2} e^{-\ref{k:mixing}\ell_1})\Bigr).
\]
\end{propos}

\begin{proof}
Put $\mathsf B=B_1^{G}$ and assume, as we may, that $\int\varphi\diff\!m_X=0$.

Applying Fubini's theorem and Cauchy-Schwarz inequality, we have  
\[
\biggl|\int_0^1\!\!\int_{\mathsf B} \varphi(a_{\ell_1}u_{r}a_{\ell_2}gx)\diff\!\nu(g)\diff\!r\biggr|^2\leq
\int_{\mathsf B}\biggl(\int_0^1\varphi(a_{\ell_1}u_{r}a_{\ell_2}gx)\diff\!r\biggr)^2\diff\!\nu(g)
\]
Expanding the inner integral in the right side of the above and using~\eqref{eq: RD nu assump}, we conclude that   
\begin{multline}\label{eq: CS and Fubini}
\biggl|\int_0^1\!\!\int_{\mathsf B} \varphi(a_{\ell_1}u_{r}a_{\ell_2}gx)\diff\!\nu(g)\diff\!r\biggr|^2\leq\\
\Lambda\int_{\mathsf B}\!\int_0^1\!\!\int_0^1\!\!\varphi(a_{\ell_1}u_{r_1}a_{\ell_2}gx)\varphi(a_{\ell_1}u_{r_2}a_{\ell_2}gx)\diff\!r_1\diff\!r_2\diff\!m_{G}(g).
\end{multline}
For all $r_1, r_2\in [0,1]$, let 
\[
\Phi_{r_1,r_2}(z)=\varphi(a_{\ell_1}u_{r_1}z)\varphi(a_{\ell_1}u_{r_2}z).
\]
Then $\Sob(\Phi_{r_1,r_2})\ll e^{M\ell_1}\Sob(\varphi)^2$ for some $M$ depending only on the dimension. 
Moreover, if $|r_1-r_2|\geq e^{-\ell_1/2}$, then~\eqref{eq: exp mixing} implies that 
\[
\biggl|\int\Phi_{r_1,r_2}\diff\!m_X\biggr|=\biggl|\varphi(a_{\ell_1}u_{r_1}z)\varphi(a_{\ell_1}u_{r_2}z)\diff\!m_X\biggr|\ll \Sob(\varphi)^2e^{-\kappa\ell_1} 
\]
where $\kappa=\star\mixexp$. We will assume $\kappa\leq 1/2$. 

We now estimate the second integral in~\eqref{eq: CS and Fubini}. Applying Fubini's theorem,  
\begin{multline*}
\Lambda\int_{\mathsf B}\!\int_0^1\!\!\int_0^1\!\!\varphi(a_{\ell_1}u_{r_1}a_{\ell_2}gx)\varphi(a_{\ell_1}u_{r_2}a_{\ell_2}gx)\diff\!r_1\diff\!r_2\diff\!m_{G}(g)=\\
\Lambda\int_0^1\!\!\int_0^1\!\!\int_{\mathsf B}\!\varphi(a_{\ell_1}u_{r_1}a_{\ell_2}gx)\varphi(a_{\ell_1}u_{r_2}a_{\ell_2}gx)\diff\!m_{G}(g)\diff\!r_1\diff\!r_2
\end{multline*}
Let $\Xi=\{(r_1,r_2)\in[0,1]^2: |r_1-r_2|\geq e^{-\ell_1/2}\}$ and $\Xi'=[0,1]^2\setminus\Xi$. Let now $(r_1,r_2)\in\Xi$, and recall that $x\in X_\eta$. 
Increasing $M$ if necessary and using a partition of unity, 
there exist a collection $\{\psi_i\}\subset C_c^\infty(X)$ satisfying that  
$\Sob(\psi_i)\leq \eta^{-M}$, $\sum \int\psi_i\diff\!m_X=1$, and 
\begin{multline*}
\frac{1}{m(\mathsf B)}\int_{\mathsf B}\!\!\Phi_{r_1,r_2}(a_{\ell_2}gy)\diff\!m_{G}(g)=\\ \sum_i\int\Phi_{r_1,r_2}(a_{\ell_2}z)\psi_i(z)\diff\!m_{X}(z)+O(\eta\|\Phi_{r_1,r_2}\|_\infty)
\end{multline*}
In view of this and~\eqref{eq: exp mixing}, thus
\begin{multline}\label{eq: apply mixng to G+}
\frac{1}{m(\mathsf B)}\int_{\mathsf B}\!\!\Phi_{r_1,r_2}(a_{\ell_2}gy)\diff\!m_{G}(g)=\\
\int\Phi_{r_1,r_2}\diff\!m_X+O\Bigl(\|\varphi\|_\infty^2\eta+\Sob(\Phi_{r_1,r_2})\eta^{-M}e^{-\mixexp\ell_2})\Bigr)
\end{multline}
Using the above observations regarding the Sobolev norm and the integral of 
$\Phi_{r_1,r_2}$,~\eqref{eq: apply mixng to G+} implies that if $(r_1,r_2)\in\Xi$, then 
\[
\frac{1}{m(\mathsf B)}\int_{\mathsf B}\!\Phi_{r_1,r_2}(a_{\ell_2}gy)\diff\!m_{G}(g)=O\Bigl(\Sob(\varphi)^2(\eta+e^{-\kappa\ell_1}+\eta^{-M}e^{M\ell_1}e^{-\mixexp\ell_2})\Bigr)
\]
This, $|\Xi'|\ll e^{-\ell_1/2}$, and~\eqref{eq: CS and Fubini} imply that if $\ell_2>M\max\{\ell_1/\mixexp,|\log\eta|\}$, then  
\[
\biggl|\int_0^1\!\!\int_{\mathsf B} \varphi(a_{\ell_1}u_{r}a_{\ell_2}gy)\diff\!\nu(g)\diff\!r\biggr|=O\Bigl(\Sob(\varphi)(\eta+\Lambda^{1/2} e^{-\kappa\ell_1})\Bigr).
\]
The proposition thus holds with $\ref{k:mixing}=\min\{\kappa,\frac{\mixexp}{2M}\}$. 
\end{proof}

In applying Proposition~\ref{prop: 1-epsilon N}, we consider measures supported on $\rfrak$ with a finitary dimension close to $2\dm+1$. The following lemma, based on standard arguments, establishes the connection to Proposition~\ref{prop: 1-epsilon N}.

\begin{lemma}\label{lem: proj and thickening}
Let $0<\delta_0<1$. 
Let $\ell_1,\ell_2>0$ with $\ref{k:mixing}\ell_2\geq \max\{\ell_1,|\log\eta|\}$ 
and $4\dm \ell_2\leq |\log\delta_0|$, and let $\varrho\leq \varrho_0$, 
where $0<\varrho_0\leq 1$ depends only on the dimension.  
Let $\mu$ be a probability measure on $B_\rfrak(0,\varrho)$ satisfying   
\be\label{eq: mu has dim alpha thick}
\mu(B(w,\delta))\leq \egbd\delta^{2\dm+1}\quad\text{for all $w\in\rfrak$ and all $\delta\geq \delta_0$}.
\ee
Then for all $\varphi\in C_c^\infty(X)$ and all $x\in X_\eta$
\begin{multline*}
\int_0^1\!\!\int_0^1\!\!\int \varphi(a_{\ell_1}u_{r_1}a_{\ell_2}u_{r_2}\exp(w)x)\diff\!\mu(w)\diff\!r_2\diff\!r_1=\\
\int\varphi\diff\!m_X+O\Bigl(\Sob(\varphi)(\varrho^\star+\eta+\egbd^{1/2}\varrho^{-3/2} e^{-\ref{k:mixing}\ell_1})\Bigr) 
\end{multline*}
\end{lemma}

\begin{proof}
We provide the details of the straightforward proof for the convenience of the reader.

As it was mentioned before, we will use Proposition~\ref{prop: 1-epsilon N} to prove the lemma. 
To that end we begin by convolving $\mu$ with a smooth kernel. Let 
\be\label{eq: def mu hat}
\hat\mu=\hat\Phi*\mu,
\ee
where $\hat\Phi(w)=\delta_0^{-2\dm-1}\Phi(\delta_0^{-1} w)$ for a radially symmetric nonnegative smooth function 
$\Phi$ on $\rfrak$ --- recall that $\dim\rfrak=2\dm+1$. 

Then, standard computations imply that 
\be\label{eq: hat mu has dimension 2m+1}
\hat\mu(B_\rfrak(w,\delta))\ll \egbd\delta^{2\dm+1}\qquad\text{for all $w\in\rfrak$ and all $0<\delta\leq1$.}
\ee 
Moreover, since $e^{\dm (\ell_1+\ell_2)}\delta_0\leq e^{-\ell_1}$, we have 
\begin{multline}\label{eq: smooth mu}
\int_0^1\!\!\int_0^1\!\!\int \varphi\Bigl(a_{\ell_1}u_{r_1}a_{\ell_2}u_{r_2}\exp(w)x)\diff\!\mu(w)\diff\!r_2\diff\!r_1= \\
\int_0^1\!\!\int_0^1\!\!\int \varphi(a_{\ell_1}u_{r_1}a_{\ell_2}u_{r_2}\exp(w)x)\diff\!\hat\mu(w)\diff\!r_2\diff\!r_1+O(\Sob(\varphi)e^{-\ell_1})
\end{multline}

In consequence, we now investigate the integral in the second line of~\eqref{eq: smooth mu}. 
The following observation guarantees that we may replace the integral in the second line of~\eqref{eq: smooth mu} by a $\varrho$-thickening of it along $H$. Let $\mathsf B^H=B_{\varrho}^{H}$. Then 
\begin{multline}\label{eq: thicken exp F}
\int_0^1\!\!\int_0^1\!\!\int \varphi\Bigl(a_{\ell_1}u_{r_1}a_{\ell_2}u_{r_2}\exp(w)x)\diff\!\hat\mu(w)\diff\!r_2\diff\!r_1=O\Bigl(\Sob(\varphi)\varrho^{\star}\Bigr)\;\;+\\
\tfrac{1}{m_H(\mathsf B^H)}\!\!\int_0^1\!\!\int_0^1\!\!\int_{\mathsf B^H}\!\!\int\varphi\Bigl(a_{\ell_1}u_{r_1}a_{\ell_2}u_{r_2}\sfh \exp(w)x)\diff\!\hat\mu(w)\diff\!\sfh\diff\!r_2\diff\!r_1
\end{multline}
To see~\eqref{eq: thicken exp F}, recall that for all $\sfh\in \mathsf B^H$, $u_{r_2}\sfh=u_{s}^-a_qu_r$ 
where $|s|, |q|\ll \varrho$ and $\diff\!r=(1+O(\varrho))\diff\!r_2$.
Furthermore, 
$a_{\ell_2}u_{s}^-a_q=u^-_{e^{-\ell_2}s}a_qa_{\ell_2}$. Thus 
\[
a_{\ell_1}u_{r_1}a_{\ell_2}u_{r_2}\sfh=a_{\ell_1}u_{r_1}g''a_qa_{\ell_2}u_r=g'a_qa_{\ell_1}u_{e^{-q}r_1}a_{\ell_2}u_r,
\]
where $\|g'-I\|, \|g''-I\|\ll e^{-\ell_1-\ell_2}$.
This, together with $|e^{-q}-1|\ll \varrho$ and the Folner properties of $\diff\!r_1$ and $\diff\!r_2$, implies 
the claim in~\eqref{eq: thicken exp F}.

In view of~\eqref{eq: thicken exp F}, we will work with the integral on the second line of~\eqref{eq: thicken exp F}. 
Recall that $\gfrak=\hfrak\oplus\rfrak$. 
Assuming $\varrho_0$ is small enough, every $g\in B_{10\varrho_0}^G$ can be uniquely written as 
$g=\sfh\exp(w)\in H\exp(\rfrak)$, moreover, the map $g\mapsto (h,w)$ is a diffeomorphism onto its image. 
Recall that $\varrho\leq \varrho_0$ and put $\diff\!\nu=\frac{1}{m_H(\mathsf B^H)}\diff\!\sfh\diff(\exp\hat\mu)$.
Then~\eqref{eq: hat mu has dimension 2m+1} and $m_H(\mathsf B^H)\asymp \varrho^3$ imply that   
\be\label{eq: hat mu has dimension 2m+1'}
\nu(B^G_\delta(g))\ll \egbd\varrho^{-3}\delta^{\dim G}\quad\text{for all $g\in G$ and all $0<\delta\leq1$.}
\ee
Therefore, $\frac{\diff(u_{r}\nu u_{-r})}{\diff\!m_G}\ll  \egbd$ for all $r\in[0,1]$.
Applying Proposition~\ref{prop: 1-epsilon N}, with $u_{r_2}\nu u_{-r_2}$ and $u_{r_2}x$ for all $r_0\in[0,1]$, thus, 
\begin{multline*}
\int_0^1\!\!\int_0^1\!\!\int_{G}\varphi\Bigl(a_{\ell_1}u_{r_1}a_{\ell_2}u_{r_2}gx)\diff\!\nu(g)\diff\!r_2\diff\!r_1=\\
\!\int\varphi\diff\!m_X+O_c\Bigl(\Sob(\varphi)(\eta+\egbd^{1/2}\varrho^{-3/2} e^{-\ref{k:mixing}\ell_1})\Bigr). 
\end{multline*}
This,~\eqref{eq: thicken exp F}, and~\eqref{eq: smooth mu} complete the proof. 
\end{proof}


\section{Modified energy and projection theorems}\label{sec: proj and energy}

We begin by defining a modified (and localized) $\alpha$-dimensional energy for finite subsets of $\R^d$. Fix a norm $\|\;\|$ on $\bbr^d$ (below we will apply this with $d=5, 7, 11$). Let $\Theta\subset B_{\R^d}(0,1)$ be a finite set. 

For $0\leq \trct<1$ and $0<\alpha<d$, define 
$\eng_{\Theta, \trct}^{(\alpha)}: \Theta\to (0,\infty)$ as follows: 
\[
\eng_{\Theta,\trct}^{(\alpha)}(w)=\sum_{w'\in\Theta\setminus \{w\}}\max(\|w-w'\|,\trct)^{-\alpha}.
\] 
When $\delta=0$, we often write $\eng_{\Theta}^{(\alpha)}(w)$ for $\eng_{\Theta,0}^{(\alpha)}(w)$.

This notation will also be used for finite subsets of $\rfrak$, which is an $2\dm+1$-dimensional $\Ad(H)$-irreducible representation, see~\S\ref{sec: notation}.

The following projection theorem plays a crucial role in our argument.

\begin{thm}\label{thm: proj thm}
Let $0<\alpha< 2\dm+1$, there exists $\imp>0$ so that the following holds. Specifically, the theorem holds with ${\imp}$ as in~\eqref{eq: def alpha'}, see also~\eqref{eq: alpha' is positive}.  

Let $0<\pvare<10^{-4}\alpha$, and $\ell>0$. 
Let $\Theta\subset B_\rfrak(0,1)$, and assume that $\#\Theta$ is large (depending on $\pvare$). Assume further that  
\be\label{eq: energy bd proj thm main}
\eng_{\Theta,\trct}^{(\alpha)}(w)\leq \egbd\quad\text{for every $w\in\Theta$,}
\ee
for some $0<\delta<1$. 

There exists a subset $J\subset [0,1]$ with 
\[
|[0,1]\setminus J|\leq  \mathsf L_\pvare |\log\trct| e^{-\star\pvare^2\ell},
\] 
so that the following holds. 
Let $r\in J$, then there exists $\Theta_{r}\subset \Theta$ with 
\[
\#(\Theta\setminus \Theta_{r})\leq \mathsf L_\pvare |\log\trct| e^{-\star\pvare^2\ell}\cdot (\#\Theta)
\]
such that for all $w\in \Theta_r$, the following is satisfied
\[
\eng_{\Theta(w),\trct'}^{(\alpha)}\bigl(\Ad(a_\ell u_r)w\bigr)\leq \mathsf L_\pvare e^{-{\imp}\ell}\trct^{-\pvare}\egbd
\] 
where $\delta'=e^{\dm\ell}\max(\delta, (\#\Theta)^{-1/\alpha})$, $\mathsf L_\pvare$ is a constant depending on $\pvare$, and 
\[
\Theta(w)=\{\Ad(a_\ell u_r)w': w'\in \Theta_r, \|w-w'\|\leq e^{-\dm\ell}\}.
\] 
\end{thm}

Before proceeding with the proof of this theorem,  
we explicate the value of ${\imp}$ for which we will prove the theorem. Let $0<\alpha< 2\dm+1$, and define 
\be\label{eq: def alpha'}
{\imp}=\begin{cases}
\alpha\dm-\sum_{i=1}^{\lceil\alpha\rceil-1} i & 0< \alpha\leq 2\dm-1\\
\max\Bigl(\dm(1-2\dm+\alpha), \dm(1+2\dm-\alpha)-1\Bigr) & 2\dm-1< \alpha\leq 2\dm\\
\dm(1+2\dm-\alpha) & 2\dm\leq \alpha\leq2\dm+1
\end{cases}
\ee
Note that for any $0<\hat\kappa<\frac{1}{10\dm}$, and any $0<\alpha\leq 2\dm+1-\hat\kappa$, we have 
\be\label{eq: alpha' is positive}
{\imp}\geq \min\{\alpha\dm,\hat\kappa\dm\}. 
\ee
Indeed, if $0<\alpha\leq 1$, then ${\imp}=\alpha\dm$. 
Let us now assume $d-1<\alpha\leq d$ for some $2\leq d\leq 2\dm-1$, then 
\[
{\imp}=\alpha\dm-\tfrac{d(d-1)}{2}\geq \tfrac{(d-1)(2\dm-d)}{2}\geq \tfrac12.
\] 
Now consider $2\dm-1<\alpha\leq 2\dm$. Indeed  
\[
\begin{aligned}
{\imp}&\geq \dm(1+2\dm-\alpha)-1\geq \tfrac{5}{3}\dm-1&&\text{if $2\dm-1\leq \alpha\leq 2\dm-\tfrac23$}\\
{\imp}&\geq \dm(1-2\dm+\alpha)\geq \tfrac\dm3 &&\text{if $2\dm-\tfrac23\leq \alpha\leq 2\dm$}
\end{aligned}
\]
The claim in~\eqref{eq: alpha' is positive} follows.

\medskip

The proof of Theorem~\ref{thm: proj thm} relies primarily on a result by Gan, Guo, and Wang \cite[Thm.~2.1]{GGW}. More specifically, the following lemma forms the crux of the proof of Theorem~\ref{thm: proj thm}. The proof of this lemma, in turn, depends on Theorem~\ref{thm: proj app 1} which is \cite[Thm.~2.1]{GGW} tailored to our application.

\begin{lemma}\label{cor: proj app 2}
Let $0<\alpha< 2\dm+1$, $0<\rhsc_1\leq1$ and $\ell>0$. 
Let $\rho$ be the uniform measure on a finite set $\Theta\subset B_\rfrak(0,1)$ satisfying the following
\be\label{eq: dimension cond app}
\rho(B_\rfrak(w, \rhsc))\leq \egbd\rhsc^\alpha\qquad\text{for all $w$ and all $\rhsc\geq \rhsc_1$}
\ee
Let $0<\pvare<10^{-4}$. For every $1\geq \rhsc\geq e^{\dm\ell}\rhsc_1$, there exists a subset $J_{\ell,\rhsc}\subset [0,1]$ with $|[0,1]\setminus J_{\ell,\rhsc}|\leq C_\pvare' e^{-\star\pvare^2\ell}$ 
so that the following holds. 
Let $r\in J_{\ell,\rhsc}$, then there exists a subset $\Theta_{\ell,\rhsc,r}\subset \Theta$ with 
\[
\rho(\Theta\setminus \Theta_{\ell,\rhsc,r})\leq C_\pvare' e^{-\star\pvare^2\ell}
\]
such that for all $w\in\Theta_{\ell,\rhsc,r}$ we have 
\[
\rho\Bigl(\{w'\in \Theta_{\ell,\rhsc,r}: \|\Ad(a_\ell u_r)w-\Ad(a_\ell u_r)w'\|\leq \rhsc\}\Bigr) \leq C_\pvare' \egbd e^{-{\imp}\ell} \rhsc^{\alpha-\pvare}
\]
where ${\imp}$ is as in~\eqref{eq: def alpha'}. 
\end{lemma}

\begin{proof}
We first establish the claim for $0<\alpha\leq 2\dm-1$, and also obtain one of the bounds in the definition $\imp=\max\Bigl(\dm(1-2\dm+\alpha), \dm(1+2\dm-\alpha)-1\Bigr)$
for $2\dm-1<\alpha\leq 2\dm$, namely $\dm(1-2\dm+\alpha)$.  

To that end, let $k=\lceil \alpha\rceil$ and write
\[
\imp'=\alpha\dm-\textstyle\sum_{i=1}^{k-1} i.
\]

For every $1\leq d\leq 2\dm+1$, let $\rfrak_d$ denote the space spanned by vectors with weight $\dm,\ldots, \dm-d+1$. 
Let $\pi_d:\rfrak\to\rfrak_d$ denote the orthogonal projection. 
Apply Theorem~\ref{thm: proj app 1} with $\rhsc'=e^{-\dm\ell} \rhsc\geq \mfsc_1$ and $\pvare/2$. Then for every $r\in J_{\rhsc'}$ and all $w\in \Theta_{\rhsc',r}$, we have  
\be\label{eq: use thm proj 1}
\rho\Bigl(\{w'\in \Theta: \|\pi_k(\Ad(u_r)w))-\pi_k(\Ad(u_r)w')\|\leq \rhsc'\}\Bigr)\leq C_\pvare\egbd (\rhsc')^{\alpha-\frac\pvare2}.
\ee

Let $P\subset \rfrak_k$ denote the box $\rhsc'\times e^\ell\rhsc'\cdots\times e^{(k-1)\ell}\rhsc'$ centered at the origin, where the directions correspond to the weight spaces for $a_t$ in decreasing order.  
Then $P+\pi_k(\Ad(a_\ell u_r)w)$ can be covered with $\ll e^{\sum_{i=1}^{k-1}i\ell}$ many boxes of size $\rhsc'$. 
Thus~\eqref{eq: use thm proj 1}, applied with $2\rhsc'$, implies that 
\be\label{eq: cover P with boxes}
\begin{aligned}
\pi_k(\rho|_{\Theta_{\ell,\rhsc,r}})(P+\pi_k(\Ad(a_\ell u_r)w))&\leq (C_\pvare\egbd (2\rhsc')^{\alpha-\frac\pvare2})\cdot e^{\sum_{i=1}^{k-1}i\ell}\\
&\leq 2^{\alpha} C_\pvare\egbd e^{-\imp'\ell}\rhsc^{\alpha-\pvare}.
\end{aligned}
\ee 
Note also that 
\begin{multline*}
\{w'\in \Theta_{\rhsc',r}: \|\Ad(a_\ell u_r)w)-\Ad(a_\ell u_r)w'\|\leq \rhsc\}\subset\\
\{w'\in \Theta_{\rhsc',r}: \|\pi_k(\Ad(a_\ell u_r)w))-\pi_k(\Ad(a_\ell u_r)w')\|\leq \rhsc\}\subset P_{\ell, r, w}.
\end{multline*}
This and~\eqref{eq: cover P with boxes} show that  
\begin{multline}\label{eq: estimate for alpha'1}
\rho|_{\Theta_{\ell,\rhsc,r}}\Bigl(\{w'\in\rfrak: \|\Ad(a_\ell u_r)w)-\Ad(a_\ell u_r)w'\|\leq \rhsc\}\Bigr) \leq \\ 2^\alpha C_\pvare \egbd e^{-\imp'\ell} \rhsc^{\alpha-\pvare}
\end{multline}
which establishes the claim when $0< \alpha\leq 2\dm-1$, as well as when $2\dm-1<\alpha\leq 2\dm$ for $\imp'=\dm(1-2\dm+\alpha)$ (when this is positive).   

\smallskip

We now turn to the case when $2\dm-1< \alpha\leq 2\dm+1$.
For a vector $v\in\rfrak$ and $1\leq i\leq 2\dm+1$, let $v_i$ denote the component of $v$ in the weight space $\dm-i+1$. 
Then for every $w\in\Theta$ and all $\rhsc$, we have 
\begin{multline}\label{eq: boxes with el moved}
\{w'\in \Theta: \|\Ad(a_\ell u_r)w)-\Ad(a_\ell u_r)w'\|\leq \rhsc\}= \\
\{w'\in \Theta: \forall i, |(\Ad(u_r)w)_i-(\Ad(u_r)w')_i|\leq e^{(-\dm+i-1)\ell}\rhsc\}. 
\end{multline}

Put $\rhsc'=e^{\dm \ell}\rhsc$, and let $P\subset \rfrak$ be the box
\[
e^{-2\dm \ell}\rhsc'\times e^{(-2\dm+1) \ell}\times \cdots\times e^{-\ell}\rhsc'\times \rhsc'
\]
centered at the origin, i.e., $i$-th weight space has size $e^{-\dm-i}\rhsc'$. 
In view of~\eqref{eq: boxes with el moved}, we will estimate the measure of sets of the form $P+w$ for $w\in\Theta$.
 
To that end, cover $\Theta$ with half-open disjoint boxes $\{B_q: q\in\mathcal Q\}$ of size $\rhsc'$. For all $B_q$, let $\rho_q=\frac{1}{\rho(B_q)}\rho|_{B_q}$, and let $\tilde\rho_q$ denote the image of $\rho_q$ under $z\mapsto \frac{1}{\rhsc'}(z-z_q)$, where $z_q$ is the center of $B_q$.
Then for all $\delta\geq \rhsc_1/\rhsc'$,  
\be\label{eq: dimension of rhoj}
\begin{aligned}
\tilde\rho_q(B_\rfrak(v, \delta))&=\tfrac{1}{\rho(B_q)}\rho (B_q\cap B_\rfrak(v+z_q, \rhsc'\delta))\\
&\leq \tfrac{\egbd}{\rho(B_q)}(\rhsc'\delta)^\alpha=\tfrac{\egbd e^{\alpha\dm\ell}\rhsc^\alpha}{\rho(B_q)}\delta^\alpha
\end{aligned}
\ee
In particular, if $\rhsc\geq e^{\dm\ell}\rhsc_1$, then $e^{-2\dm \ell}\geq \frac{\rhsc_1}{\rhsc'}$, thus~\eqref{eq: dimension of rhoj} holds for $\delta\geq e^{-2\dm \ell}$. 

Let us first assume $2\dm-1< \alpha\leq2\dm$. Put $P'=\frac{1}{\rhsc'}P$. 
Then using Theorem~\ref{thm: proj app 1}, with $k=2\dm-1$ and $\tilde\rho_q$ for all $q\in\mathcal Q$, there exists a subset $J_q\subset [0,1]$ with $|[0,1]\setminus J_q|\ll e^{-\star\pvare^2\ell}$ and for every $r\in J_q$ a subset $\Theta_{q,r}\subset \Theta\cap B_q$ with  
with 
$\rho_q(\Theta\setminus\Theta_{q,r})\ll e^{-\star\pvare^2\ell}$ so that for all $v\in\rfrak$, 
\[
\begin{aligned}
\tilde\rho_q|_{\Theta_{q,r}}(\Ad(u_{-r})P'+v)&\leq C_\pvare \tfrac{\egbd e^{\alpha\dm\ell}\rhsc^\alpha}{\rho(B_q)} e^{(-\sum_{i=2}^{2\dm} i+\pvare)\ell}\\
&\leq C_\pvare\tfrac{\egbd}{\rho(B_q)} e^{(\alpha\dm-\sum_{i=2}^{2\dm} i)\ell}\rhsc^{\alpha-\pvare}\\
&=C_\pvare\tfrac{\egbd}{\rho(B_q)} e^{-(\dm(2\dm+1-\alpha)-1)\ell}\rhsc^{\alpha-\pvare}.
\end{aligned}
\]
Equip $\mathcal Q\times [0,1]$ with $\sigma \times\Leb$ where $\sigma(q)=\rho(B_q)$. By Fubini's theorem, there is $J\subset [0,1]$ with $|[0,1]\setminus J|\ll e^{-\star\pvare^2\ell}$, 
and for all $r\in J$ a subset $\mathcal Q_r\subset\mathcal Q$ satisfying that $\sum_{q\not\in\mathcal Q_r} \rho(B_q)\ll e^{-\star\pvare^2\ell}$ so that $r\in J_q$ for all $q\in\mathcal Q_r$. 

For every $r\in J$, let $\Theta_{\ell, \rhsc,r}=\cup_{q\in\mathcal Q_r} \Theta_{q,r}$. 
Since $P+\Ad(u_r)w$ is contained in $O(1)$ many $B_q$'s, the above and the definition of $\tilde\rho_q$ imply that 
\[
\rho|_{\Theta_{\ell, \rhsc,r}}(\Ad(u_{-r})P+w)\ll C_\pvare\egbd e^{-(\dm(2\dm+1-\alpha)-1)\ell}\rhsc^{\alpha-\pvare}.
\]
In view of~\eqref{eq: boxes with el moved}, this and~\eqref{eq: estimate for alpha'1} establish the claim when $2\dm-1< \alpha\leq2\dm$.

The proof in the case $2\dm\leq \alpha\leq2\dm+1$ is similar.  
Indeed, applying Theorem~\ref{thm: proj app 1}, with $k=2\dm$ and $\{\tilde\rho_q:q\in\mathcal Q\}$, we obtain $J\subset [0,1]$, for all $r\in J$ the subset $\mathcal Q_r\subset \mathcal Q$ and for each $q\in\mathcal Q_r$, the set $\Theta_{q,r}$ as above so that  
\[
\begin{aligned}
\tilde\rho_q|_{\Theta_{q,r}}(P'+\Ad(u_r)v)&\leq C_\pvare \tfrac{\egbd e^{\alpha\dm\ell}\rhsc^\alpha}{\rho(B_q)} e^{(-\sum_{i=1}^{2\dm} i+\pvare)\ell}\\
&\leq C_\pvare\tfrac{\egbd}{\rho(B_q)} e^{(\alpha\dm-\sum_{i=1}^{2\dm} i)\ell}\rhsc^{\alpha-\pvare}\\
&=C_\pvare \tfrac{\egbd}{\rho(B_q)} e^{-\dm(2\dm+1-\alpha)\ell}\rhsc^{\alpha-\pvare},
\end{aligned}
\]
for all $v\in\rfrak$. Put $\Theta_{\ell,\rhsc, r}=\cup_{q\in\mathcal Q_r}\Theta_{q,r}$. Then 
\[
\rho|_{\Theta_{\ell,\rhsc, r}}(P+\Ad(u_r)w)\ll C_\pvare\egbd e^{-\dm(2\dm+1-\alpha)\ell}\rhsc^{\alpha-\pvare}
\]
which, thanks to~\eqref{eq: boxes with el moved}, gives the claim when $2\dm\leq \alpha<2\dm+1$. 
\end{proof}

\subsection{Proof of Theorem~\ref{thm: proj thm}}
Recall the definition of ${\imp}$ from ~\eqref{eq: def alpha'}. 
 
Let us write $\bar\egbd=\frac{\egbd}{\#\Theta}$. 
To simplify the notation, $\Ad(a_\ell u_r)(w)$ will be denoted by $\xi_{\ell, r}(w)$ in the proof. 

Let $\rho$ denote the uniform measure on $\Theta$. By~\eqref{eq: energy bd proj thm main},
\be\label{eq: dimension from energy}
\rho \Bigl(B(w,\rhsc)\cap \Theta\Bigr)\leq 2\bar\egbd\rhsc^\alpha \quad\text{for all $\rhsc\geq \max(\trct, (\#\Theta)^{-\frac1\alpha})=:\mfsc_1$}
\ee
Thus, Theorem~\ref{thm: proj app 1} and Lemma~\ref{cor: proj app 2} are applicable with $\rho$ and $\rhsc\geq \mfsc_1$.

Let $k_1=-\lceil \log\mfsc_1\!\rceil$ and $k_2=\dm\ell$. 
Apply Lemma~\ref{cor: proj app 2} with $\pvare/2$ and $\rhsc=e^{-k}$ for all $k_2\leq k\leq k_1$. 
Let $J=\bigcap_{k=k_2}^{k_1} J_{\ell,e^{-k}}$ where $J_{\ell,\mfsc}$ is as in Lemma~\ref{cor: proj app 2}. 
Then that lemma implies that 
\[
|[0,1]\setminus J|\ll_\pvare \sum_{k_2}^{k_1} e^{-\star\pvare^2\ell} \ll_\pvare |\log\trct| e^{-\star\pvare^2\ell}.
\]

For every $r\in [0,1]$, let 
\be\label{eq: def of Theta r}
\Theta_r=\bigcap_{k=k_2}^{k_1-\dm\ell}\Theta_{\ell, e^{-k},r}
\ee
where $\Theta_{\ell, e^{-k},r}$ is given by Lemma~\ref{cor: proj app 2}. Applying that lemma, thus, we have 
\[
\rho(\Theta\setminus\Theta_r)\ll_\pvare |\log\trct|e^{-\star\pvare^2\ell}.
\]

Define $J$ and for every $r\in J$, $\Theta_r$ as above for this case.   
For every $w\in \Theta_r$ and every $k< k_1-\dm\ell$, let 
\[
\Theta(w,k)=\{w'\in \Theta_r: e^{-k-1}<\|\xi_{\ell, r}(w)-\xi_{\ell,r}(w')\|\leq e^{-k}\};
\]
define $\Theta(w,k_1-\dm\ell)$ similarly but without imposing a lower bound. 
Then by Lemma~\ref{cor: proj app 2}, for all $k\leq k_1-\dm\ell$ we have   
\be\label{eq: rho Theta w k}
\rho(\Theta(w,k))\leq C_\pvare' \bar\egbd \trct^{-\pvare/2}  e^{-{\imp}\ell}e^{-\alpha k}.
\ee

Put $k_1'=k_1-\dm\ell$. From the above we conclude   
\begin{align*}
\eng_{\Theta(w), \trct'}^{(\alpha)}(w)&=\sum_{w'\in \Theta(w)}\max(\|\xi_{\ell,r}(w)-\xi_{\ell,r}(w')\|,\trct')^{-\alpha}\\
&\leq \rho(\Theta(w,k_1'))\cdot e^{k_1'\alpha}+ \sum_{k=k_2}^{k_1'-1}\sum_{\Theta(w,k)}\|\xi_{\ell,r}(w)-\xi_{\ell,r}(w')\|^{-\alpha}\\
&\ll_{\pvare} \sum_{k_2}^{k_1'} \bar\egbd \trct^{-\frac\pvare2} e^{-{\imp}\ell}e^{-\alpha k}e^{\alpha k}\cdot (\#\Theta)\\
&\ll_{\pvare}  |\log\trct| e^{-{\imp}\ell}\trct^{-\frac\pvare2}  \egbd\ll_\pvare e^{-{\imp}\ell} \trct^{-\pvare}\egbd ,
\end{align*}
where we used~\eqref{eq: rho Theta w k} in the third line, $\bar\egbd=\frac{\egbd}{\#\Theta}$ and $k_1\leq |\log\trct|$ in the second to last inequality, and $|\log\trct|\trct^{-\pvare/2}\leq \trct^{-\pvare}$ in the last inequality.   

The proof is complete. 
\qed

\subsection{Linear algebra lemma and the energy}\label{sec: linear algebra proj}
In this section, we detail how Lemma~\ref{lem: linear algebra 1} can be employed to improve the initial estimate provided by Proposition~\ref{prop: closing lemma}. While the formulation of Lemma~\ref{lem: LA implies initial dim} below bears similarities to 
Theorem~\ref{thm: proj thm} albeit with small $\alpha$, a key observation is that, unlike that theorem, this step avoids any loss of scales.

In the proof of Proposition~\ref{propos: imp dim main}, we will use Lemma \ref{lem: LA implies initial dim} to refine the initial estimate provided by Proposition~\ref{prop: closing lemma}, achieving a positive initial dimension. Subsequently, Theorem~\ref{thm: proj thm} will be applied to further improve this dimension, bringing it close to full dimension.

\begin{lemma}\label{lem: LA implies initial dim}
Let $0<\alpha\leq \frac1{2\dm+1}$ and $\ell>0$. 
Let $\Theta\subset B_\rfrak(0,1)$ be a finite set satisfying that 
\be\label{eq: energy bd proj LA}
\eng_{\Theta}^{(\alpha)}(w)\leq \egbd\quad\text{for every $w\in\Theta$,}
\ee
for some $\egbd\geq 1$.

There exists a subset $J\subset [0,1]$ with $|[0,1]\setminus J|\leq  \mathsf L_\alpha e^{-\star\alpha\ell}$, so that the following holds. 
Let $r\in J$, then there exists a subset $\Theta_{r}\subset \Theta$ with 
\[
\#(\Theta\setminus \Theta_{r})\leq \mathsf L_\alpha e^{-\star\alpha\ell}\cdot (\#\Theta)
\]
such that for all $w\in \Theta_r$ we have 
\[
\eng^{{(\alpha)}}_{\Ad(a_\ell u_r)\Theta} \bigl(\Ad(a_\ell u_r)w\bigr)\leq \mathsf L_\alpha e^{-\frac{3\alpha}{4}\ell}\egbd
\] 
where $\mathsf L_\alpha$ is a constant depending on $\alpha$.
\end{lemma}

\begin{proof}
Recall from Lemma~\ref{lem: linear algebra 1} that for all $0\neq v\in \rfrak$, we have 
\[
\int_0^1 \|a_\ell u_rv\|^{-\alpha}\diff\!r\leq C_\alpha e^{-\alpha \ell}\|v\|^{-\alpha}
\]
where $C_\alpha$ depends on $\alpha$. This and the definition of $\eng_\Theta^{(\alpha)}$ imply that 
\begin{align*}
\int_0^1 \eng_{\Ad(a_\ell u_r)\Theta}^{(\alpha)}(\Ad(a_\ell u_r)w)\diff\!r&=\int_0^1\sum_{w\neq w'}\|\Ad(a_\ell u_r)(w-w')\|^{-\alpha}\diff\!r\\
&\leq C_\alpha e^{-\alpha \ell} \sum_{w\neq w'}\|w-w'\|^{-\alpha}\\
&=C_\alpha e^{-\alpha \ell}\eng_{\Theta}^{(\alpha)}(w) \leq C_\alpha e^{-\alpha \ell}\egbd,
\end{align*}
for all $w\in\Theta$. 

Applying Chebyshev's inequality, we conclude that if for $w\in\Theta$ we put 
\[
J(w)=\left\{r\in[0,1]: \eng_{\Ad(a_\ell u_r)\Theta}^{(\alpha)}(\Ad(a_\ell u_r)w)> C_\alpha e^{-3\alpha \ell/4}\egbd\right\},
\] 
then $|J(w)|\leq e^{-\alpha\ell/4}$. Let 
\begin{align*}
   \Xi&=\left\{(w,r)\in\Theta\times [0,1]: \eng_{\Ad(a_\ell u_r)\Theta}^{(\alpha)}(\Ad(a_\ell u_r)w)\leq C_\alpha e^{-3\alpha \ell/4}\egbd\right\}\\
   &=\{(w,r):w \in \Theta, r \in [0,1]\setminus J(w)\}.
\end{align*}
The above discussion implies that $\rho\!\times\! \Leb\,(\Xi)>1-O(e^{-\alpha\ell/4})$, where 
$\rho$ denotes the uniform measure on $\Theta$. This and Fubini's theorem imply that 
there exists $J\subset[0,1]$ with $|[0,1]\setminus J|\ll e^{-\star\alpha \ell}$, so that for every $r\in J$
\[
\rho(\{w: (w,r)\in\Xi\})> 1-O(e^{-\star\alpha\ell}).
\] 
For every $r\in J$, let 
\be\label{eq: Theta r LA}
\Theta_r=\{w: (w,r)\in\Xi\}
\ee 
The claim in the lemma thus follows with $J$ and $\Theta_r$ as above. 
\end{proof}


\section{Margulis functions and projection theorems}\label{sec: Marg func}

Let us put
\[
\coneH=\{u_s^-:|s|\leq {\beta}\}\cdot\{a_d: |d|\leq \beta\}\cdot\{u_r: |r|\leq \eta\},
\]
see~\eqref{eq:def-Ct}. Let $F\subset B_\rfrak(0,\beta)$ be a finite set. Let $y\in X_{\eta}$ be so that 
$(h,w)\mapsto h\exp(w)y$ is injective over $\coneH\times F$ and  
\be\label{eq: def cone}
\cone=\coneH.\{\exp(w)y: w\in F\}\subset X_\eta
\ee

For every $(h,z)\in H\times \cone$, define 
\be\label{eq: def margI h,z}
\margI_{\cone}(h,z):=\Bigl\{w\in \rfrak: \|w\|<\inj(hz),\, \exp(w) h z\in h\cone\Bigr\}.
\ee
Note that $\margI_{\cone}(h,z)$ contains $0$ for all $z\in\cone$,
moreover, since $\mathsf E$ is bounded, $\margI_{\cone}(h,z)$ is a finite set for all $(h,z)\in H\times\cone$. 

Fix some $0< \alpha\leq 2\dm+1$.  
For every $0\leq \trct< 1$, define a modified Margulis function 
$
\mfm_{\cone,\trct}^{(\alpha)}: H\times \cone\to [1,\infty)
$
as follows. 
\[
\mfm_{\cone,\trct}^{(\alpha)}(h,z)=\sum_{0\neq w\in I_\cone(h,z)}\max(\|w\|,\delta)^{-\alpha}
\]

In this paper, we will primarily use the above objects with $h=e$. 
Hence, and in order to simplify the notation, we denote 
$\margI_{\cone}(e, z)$ and $\mfm_{\cone,\trct}^{(\alpha)}(e, z)$ by $\margI_{\cone}(z)$ and $\mfm_{\cone,\trct}^{(\alpha)}(z)$, respectively.

The modified Margulis function and the modified energy discussed in~\S\ref{sec: proj and energy} are closely related. Specifically, the following lemma~\cite[Lemma 9.2]{LMW22} establishes, in a general sense, that bounding one of these quantities implies a bound on the other, with the exception of potential {\em edge effects}. When $\rfrak$ is a Lie subalgebra, this connection becomes more straightforward. However, in the general case, understanding the relationship between $\margI(z)$ and $\margI(\exp(w)z)$ demands further elaboration; see~\cite[Lemma 9.2]{LMW22} for details.

\begin{lemma}\label{lem: margIz energy est}
Let the notation be as above, and assume that 
\[
\mfm_{\cone,\trct}^{(\alpha)}(z)\leq \egbd\qquad\text{for all $z\in\cone$.}
\]
Then for every $z\in \Bigl(\overline{\coneH\setminus \partial_{5\beta^2}\coneH}\Bigr).\{\exp(w)y: w\in F\}$ and all $w\in\margI_{\cone}(z)$, 
\[
\eng^{(\alpha)}_{\margI_{\cone}(z), \trct}(w)\ll\mfbd,
\]
where $\eng$ is defined as in \S\ref{sec: proj and energy}.
\end{lemma}

\subsection*{An iterative dimension improvement lemma}\label{sec: main lemma}
The following lemma outlines a general iterative process for improving the dimension under suitable projection theorems. Readers may compare it to~\cite[Lemmas 9.1 and 10.7]{LMW22} and~\cite[Lemma 7.10]{LM-PolyDensity}. In the next section, we apply Lemma~\ref{lem: main ind lemma} to establish Proposition~\ref{propos: imp dim main}, a central component of our argument.

\smallskip

Let $0<\vare<10^{-10}$ be a small parameter, and $t>0$ a large parameter. 
We also fix some $0<\vare'<10^{-6}\vare$, and put 
\[
\ell=\tfrac{\vare t}{10\dm},\quad \beta=e^{-\vare' t},\quad\text{and}\quad\eta^2=\beta.
\] 

Throughout this section, we fix a large parameter $\mathsf d_0$. The results of this section will be applied with 
$\mathsf d_0=3t+3d_{\rm fn}\ell$, where $d_{\rm fn}$ is the number of steps we need to apply Lemma~\ref{lem: main ind lemma} in the proof of Proposition~\ref{propos: imp dim main}, and is explicated in the next section. 

Recall that $F\subset B_\rfrak(0,\beta)$, and 
\[
\cone=\coneH.\{\exp(w)y: w\in F\}\subset X_\eta.
\]
We assume $\cone$ is equipped with an admissible measure $\mu_\cone$, see \S\ref{sec: cone and mu cone}.

\begin{lemma}\label{lem: main ind lemma}
Let $\frac{1}{2\dm+1}\leq \alpha< 2\dm+1$, and define ${\imp}$ as in~\eqref{eq: def alpha'} ($\imp$ is some explicit function of $\alpha$ that goes to zero as $\alpha \to 2\dm+1$). 
Let $\egbd\leq e^{\mathsf Dt}$ for some $\mathsf D>0$, and assume that  
\[
    \mfm_{\cone,\trct}^{(\alpha)}(z)\leq \egbd\qquad\text{for all $z\in\cone$}
    \]
    where $\delta=0$ if $\alpha=\frac{1}{2\dm+1}$ and $e^{-\mathsf D t}\leq \delta<1$ otherwise. 
        
    The following holds so long as $t$ is large enough, depending on $X$, $\mathsf D$, and $2\dm+1-\alpha$, see~\eqref{eq: choose pvare Dt}. 
    There exists a collection $\{\cone_i: 1\leq i\leq N\}$ of sets 
\begin{align*}
    \cone_{i}=\coneH.\{\exp(w)y_{i}: w\in F_{i}\}\subset X_\eta
\end{align*} 
where for every $1\leq i\leq N$, $F_{i}\subset B_\rfrak(0,\beta)$ with 
\[
\beta^{4\dm+6}\cdot (\#F)\leq \#F_i\leq \beta^{4\dm+4}\#F,
\]
and for every $i$, an admissible measure $\mu_{\cone_{i}}$, 
so that both of the following hold 
\begin{enumerate}
\item For all $1\leq i\leq N$ and all $z\in\coneH.\{\exp(w)y_{i}: w\in F_{i}\}$,
\be\label{eq:improve dim energy est >alpha}
\mfm_{\cone_{i}, {\trct'}}^{(\alpha)}(z)\leq e^{-\frac{{\imp}}{2}\ell}\cdot \egbd+e^{\vare t}\cdot (\#F_i)
\ee
where $\trct'=0$ if $\alpha=\frac{1}{2\dm+1}$ and $\trct'=e^{\dm\ell}\max(\delta,(\#F)^{-\frac1\alpha})$ otherwise.  
\item For all $0<\mathsf d\leq \mathsf d_0-\ell$, all $|s|\leq 2$, and every $\varphi\in C_c^{\infty}(X)$, we have 
\[
\int_0^1\!\!\int \varphi(a_{\mathsf d} u_sa_\ell u_rz)\diff\!\mu_\cone(z)\!\diff\!r=\sum_{i}c_{i}\!\int \varphi(a_{\mathsf d} u_sz)\diff\!\mu_{\cone_{i}}(z)+ O(\Lip(\varphi)\beta^{\ref{k: bootstrap beta exp'}})
\]
where $0\leq c_{i}\leq 1$ and $\sum_{i}c_{i}=1-O(\beta^{\ref{k: bootstrap beta exp'}})$, 
$\Lip(\varphi)$ is the Lipschitz norm of $\varphi$, and $\constk\label{k: bootstrap beta exp'}$ and the implied constants depend on $X$.
\end{enumerate}
\end{lemma}

The proof of Lemma~\ref{lem: main ind lemma} relies on Lemma~\ref{lem: LA implies initial dim} (for $\alpha= 1/(2\dm+1)$) and Theorem~\ref{thm: proj thm} (for $\alpha>1/(2\dm+1)$) 
and will be completed in several steps. The basic idea is straightforward: By applying these results to the sets 
$\margI_{\cone}(z)$ for all $z\in \cone$ and a few applications of Fubini's theorem, we obtain a {\em large} 
subset $\gdh\subset [0,1]$ and for each $r\in \gdh$ a {\em large} subset $\cone(r)\subset \cone$ where the upper bound on the 
modified Margulis function improves under the application of $a_\ell u_r$. 
We then trim and smear $a_\ell u_r\cone(r)$ using the results in~\cite[\S 6--8]{LMW22}, see also \S\ref{sec: cone and mu cone}, to establish the lemma. The details follow.

\subsection*{Proof of Lemma~\ref{lem: main ind lemma}}
Let us write 
\[
\iconeH=\overline{\coneH\setminus \partial_{5\beta^2}\coneH}\quad\text{and}\quad \icone=\iconeH.\{\exp(w)y: w\in F\}.
\] 
Recall from Lemma~\ref{lem: margIz energy est} that our condition
\[
 \mfm_{\cone,\trct}^{(\alpha)}(\bar z)\leq \egbd
\]
implies that for every $\bar z\in\icone$ and all $v\in\margI_{\cone}(\bar z)$, we have $\eng_{\margI_{\cone}(\bar z), \trct}^{(\alpha)}(v)\ll\egbd$.

Let $\bar z\in\icone$, we will estimate 
\[
\eng_{\Ad(a_\ell u_r)\margI, {\trct'}}^{(\alpha)}(\Ad(a_\ell u_r)v)
\] 
for a certain subset $\margI\subset \margI_{\cone}(\bar z)$ which will be explicated below. 

First let us assume $\alpha= \frac1{2\dm+1}$, then ${\imp}=\alpha$ and ${\trct'}=0$. 
In this case, Lemma~\ref{lem: LA implies initial dim} applied with $\Theta=\margI_{\cone}(\bar z)$ implies that 
there is $\theta>0$ (depending on $\dm$) so that for all large enough $\ell$ the following holds. 
Denote by $J_{\bar z}$ the set $J\subset [0,1]$ in Lemma~\ref{lem: LA implies initial dim}, 
and for all $r\in J_{\bar z}$, let $\margI_{\bar z,r}$ be the set $\Theta_r$ in Lemma~\ref{lem: LA implies initial dim} applied with $\Theta=\margI_{\cone}(\bar z)$, see~\eqref{eq: Theta r LA}. Then 
\be\label{eq: number of pts in I z,r LA}
\#(\margI_{\cone}(\bar z)\setminus\margI_{\bar z,r})\leq  e^{-\theta\ell} \cdot (\#\margI_{\cone}(\bar z)),
\ee
and for every $v\in\margI_{\bar z,r}$, we have 
\be\label{eq: use proj thm ind LA}
\begin{aligned}
   \eng_{\Ad(a_\ell u_r)\margI_{\bar z,r}^{\rm int}(v), {\trct'}}^{(\alpha)}(\Ad(a_\ell u_r)v)&\leq \eng_{\Ad(a_\ell u_r)\margI_{\cone}(\bar z)}^{(\alpha)}(\Ad(a_\ell u_r)v)\\
   &\ll e^{-3\alpha\ell/4}\egbd
   \end{aligned}
\ee
where $\margI_{\bar z,r}^{\rm int}(v)=\{v'\in\margI_\cone(\bar z): \|v-v'\|\leq e^{-\dm\ell}\}$ and we used 
${\imp}=\alpha$ and ${\trct'}=0$ in the case at hand.

We now turn to the case $\alpha>\frac{1}{2\dm+1}$. 
The goal is to establish a similar estimate, albeit using Theorem~\ref{thm: proj thm} in this case. 
To that end, first let $\pvare$ be small enough so that $e^{\pvare\mathsf Dt}<e^{{\imp}\ell/4}$. 
\be\label{eq: choose pvare Dt}
e^{-{\imp}\ell} \trct^{-\pvare} \egbd\leq e^{-3{\imp}\ell/4}\egbd, 
\ee
where we used $\delta\geq e^{-\mathsf Dt}$.
Now apply Theorem~\ref{thm: proj thm} with $\Theta=\margI_{\cone}(\bar z)$ and this $\pvare$. 
There is $\theta>0$ so that for all large enough $\ell$ the following holds. 
Denote by $J_{\bar z}$ the set $J\subset [0,1]$ in Theorem~\ref{thm: proj thm}, and for all $r\in J_{\bar z}$, let $\margI_{\bar z,r} \subset \margI_{\cone}(\bar z)$ be the set $\Theta_r$ in Theorem~\ref{thm: proj thm} applied with $\Theta=\margI_{\cone}(\bar z)$, see~\eqref{eq: def of Theta r}. Then 
\be\label{eq: number of pts in I z,r}
\#(\margI_{\cone}(\bar z)\setminus\margI_{\bar z,r})\leq  e^{-\theta\pvare^2\ell} \cdot (\#\margI_{\cone}(\bar z)),
\ee
and for every $v\in\margI_{\bar z,r}$, we have 
\be\label{eq: use proj thm ind}
   \eng_{\Ad(a_\ell u_r)\margI_{\bar z,r}^{\rm int}(v), {\trct'}}(\Ad(a_\ell u_r)v)\ll e^{-3{\imp}\ell/4}\egbd
\ee
where ${\trct'}=e^{\dm\ell}\max(\delta, (\#F)^{-1/\alpha})$, 
\[
\margI_{\bar z,r}^{\rm int}(v)=\{v'\in\margI_{\bar z,r}: \|v-v'\|\leq e^{-\dm\ell}\}, 
\]
and we used~\eqref{eq: choose pvare Dt}.

\subsection*{The set $\gdh_{\mu_\cone}$}
Equip $\cone\times[0,1]$ with $\sigma:=\mu_{\cone}\times\Leb$, where 
$\Leb$ denotes the normalized Lebesgue measure on $[0,1]$. 
Extend the definition of $\margI_{\bar z,r}$ to all $r\in[0,1]$ 
by letting $\margI_{\bar z,r}=\emptyset$ if $r\not\in J_{\bar z}$, and put  
\be\label{eq: define Z}
Z=\Bigl\{(\bar z,r)\in \hat\cone\times[0,1]: \#(\margI_{\cone}(\bar z)\setminus\margI_{\bar z,r})\leq e^{-\theta\pvare^2\ell} \cdot (\#\margI_{\cone}(\bar z))\Bigr\},
\ee
where $\hat\cone=\hat\coneH.\{\exp(w)y: w\in F\}$ and $\hat\coneH=\overline{\coneH\setminus \partial_{200\beta^2}\coneH}$. 

Then, using~\eqref{eq: number of pts in I z,r LA} if $\alpha=\frac1{2\dm+1}$ or~\eqref{eq: number of pts in I z,r} when $\alpha>\frac1{2\dm+1}$, we have that 
\[
\{(\bar z,r): r\in J_{\bar z}\}\subset Z \qquad \text{for all $\bar z\in\hat\cone$.} 
\]
Recall moreover that $\mu_\cone(\cone\setminus\hat\cone)\ll \beta$. 
We conclude that 
\[
\sigma(\cone\times[0,1]\setminus Z)\ll \beta+e^{-\star\pvare^2\ell} \ll e^{-\star\pvare^2\ell},
\]
where we assumed $\pvare\leq \kappa$. This and Fubini's theorem imply that there is a subset 
$\gdh=\gdh_{\mu_\cone}\subset [0,1]$ with 
$|[0,1]\setminus\gdh|\ll e^{-\star\pvare^2\ell}$ so that  
\be\label{eq: measure of Cone-r}
\mu_\cone\Bigl(\cone\setminus Z_r\Bigr)\ll e^{-\star\pvare^2\ell} \qquad\text{for all $r\in\gdh$},
\ee
where $Z_r=\{\bar z\in\hat\cone: (\bar z,r)\in Z\}$.

\subsection*{Decomposing $a_\ell u_r\mu_\cone$ as a convex combination}  
Fix a maximal $e^{-3\mathsf d_0}$ separated subset $\mathcal L$ of $\gdh$, and $\mathsf d$ and $s$ as in the statement of Lemma \ref{lem: main ind lemma}. 
Applying Lemma~\ref{lem: convex comb}, see also~\cite[Lemma 8.9]{LMW22}, for every $r\in\mathcal L$ we can write
\begin{multline}\label{eq: rw ell 1 mathcal L}
\int\varphi(a_{\mathsf d}u_sa_\ell u_{r}z)\diff\!\mu_{\cone}(z)=\\\sum_i c_{i, r}'\int\varphi(a_{\mathsf d} u_sz)\diff\!\mu_{\cone_{i,r}'}(z)+  O(\beta^\star\Lip(\varphi)),
\end{multline}
where the $\cone'_{i,r}$s are sets of the form
\[
\cone'_{i,r}=\coneH.\{\exp(w)y_{i,r}: w\in F'_{i,r}\}\subset X_\eta
\]
with $F'_{i,r}\subset B_\rfrak(0,\beta)$ that satisfies \[\beta^{4\dm+5}\cdot (\#F)\leq F_{i,r}'\leq \beta^{4\dm+4}\cdot(\#F).\] 
We will now further divide the sets $\cone_{i,r}'$. 

Let $\mathcal B$ denote the cubes of size $2^{-n}$ with $2^{-n}\geq e^{(-2\dm-1)\ell}>2^{-n-1}$ which are obtained by 
scaling some fixed partition of $\rfrak$ into unit size cubes by~$2^{-n}$.
Applying Lemma~\ref{lem: regular tree decomposition}, with the collection of cubes $\mathcal B$, we can write    
$F'_{i,r}=F''_{i,r}\sqcup\Bigl(\sqcup_{\varsigma} \hat F'_{i,r,\varsigma}\Bigr)$ where the union is disjoint, 
\[
\#F''_{i,r}\ll \beta^{1/2}(\#F'_{i,r}),
\]
and both of the following hold for all $\varsigma$ 
\begin{enumerate}[label=\alph*.]
\item $\#\hat F'_{i,r,\varsigma}\geq \beta^{0.6}\cdot(\#F'_{i,r})$, and 
\item\label{item: varsigma} There exists $w_{i,\varsigma}\in\rfrak$ so that if $\mathsf C_1\neq \mathsf C_1'\in{\mathcal B}$ intersect $\hat F'_{i,r,\varsigma}+w_{i,\varsigma}$ non-trivially, the distance between $\mathsf C_1\cap(\hat F'_{i,r,\varsigma}+w_{i,\varsigma})$ and $\mathsf C'_1\cap(\hat F'_{i,r,\varsigma}+w_{i,\varsigma})$ is $\gg e^{(-2\dm-1)\ell}\beta$.
\end{enumerate} 

\smallskip

Replacing $\mu_{\cone'_{i,r}}$ by $\sum c_{i,r,\varsigma}\mu_{\cone'_{i,r,\varsigma}}$, where 
\[
\cone'_{i,r,\varsigma}=\coneH.\{\exp(w)y_{i,r}: w\in F'_{i,r,\varsigma}\}\qquad\text{for all $\varsigma$},
\] 
we have that
\[
\int\varphi(a_{\mathsf d}u_sa_\ell u_{r}z)\diff\!\mu_{\cone}(z)=\sum_{i,\varsigma} c_{i, r,\varsigma}'\int\varphi(a_{\mathsf d} u_sz)\diff\!\mu_{\cone_{i,r,\varsigma}'}(z)+  O(\beta^\star\Lip(\varphi)).
\]

By reindexing, may assume that \ref{item: varsigma} above holds already for $F'_{i,r}$. 
In other words, we may assume that there exists a disjoint collection $\mathcal B_i'$ of cubes of size $2^{-n}\geq e^{(-2\dm-1)\ell}\beta\geq 2^{-n-1}$ 
which cover $B_\rfrak(0,\beta)$ so that if $\mathsf C_1\neq \mathsf C_1'\in\mathcal B_i'$ intersect 
$\hat F'_{i,r}$ non-trivially, then  
\be\label{eq: distance C1 and C'1}
\text{the distance between $\mathsf C_1\cap\hat F'_{i,r}$ and $\mathsf C'_1\cap \hat F'_{i,r}$ is $\gg e^{(-2\dm-1)\ell}\beta$.}
\ee 

After this additional refinement, we have \[\beta^{4\dm+6}\cdot (\#F)\leq F_{i,r}'\leq \beta^{4\dm+4}\cdot (\#F).\]

\subsection*{Incremental improvements using $Z_r$}
Let $r\in L$, and let $\bar z\in Z_r$, see~\eqref{eq: measure of Cone-r}.  
Fix~$\bar w\in\margI_{\bar z, r}$; by definition of $\margI_{\bar z, r}$ we have that $z=\exp(\bar w)\bar z\in\hat\cone$.
We will estimate 
\[
\sum_{v\in\margI^{\rm int}_z} \max(\|a_\ell u_r v\|,{\trct'})^{-\alpha},
\]
where $\margI^{\rm int}_z$ is explicated in~\eqref{eq: def of I<}. 

For any small $w \in \rfrak$, we have 
\be\label{eq: recentring I}
\exp(w)z=\exp(w)\exp(\bar w)\bar z=\sfh \exp(\mathsf v_w)\bar z
\ee
with $\sfh \in H$, \  $\mathsf v_w \in \rfrak$, \ $\|\sfh-I\|\ll \beta\|w\|$ and 
\be\label{eq: v_w w w bar}
\|\mathsf v_w-(w+\bar w)\|\ll \|\bar w\|\|w\|,\ee see Lemma~\ref{lem: BCH}. 
On the other hand, if $z\in \hat \cone$ and $w\in I_{\cone}(z)$, we have 
\[
\exp(w)z=\exp(w)h\exp(w_z)y=h\sfh_z\exp(w_{z,w})y
\]
for some $w_z, w_{z,w}\in F$, \  $h\in\hat\coneH$, and $\sfh_z \in H$ with $\|\sfh_z-I\|\ll \beta\|w\|$. Thus~\eqref{eq: recentring I} implies that 
\[
\exp(\mathsf v_w)\bar z=\sfh^{-1}\exp(w)z=\sfh^{-1}h\sfh_z\exp(w_{z,w})y\in \cone.
\]
That is, $\mathsf v_w\in\margI_\cone(\bar z)$. Hence, $w\mapsto \mathsf v_w$ is one-to-one from $\margI_{\cone}(z)$ into $\margI_{\cone}(\bar z)$. 

Moreover, since $e^{-\dm\ell}\|v\| \ll \|a_\ell u_rv\|\ll e^{\dm\ell}\|v\|$ for all $v\in\rfrak$, 
we conclude from~\eqref{eq: recentring I} and \eqref{eq: v_w w w bar} that if $\|\bar w\|\leq e^{(-2\dm-1)\ell}$, then 
\be\label{eq: vw - wbar and w}
\tfrac12\|a_\ell u_r(\mathsf v_w-\bar w)\|\leq \|a_\ell u_r w\|\leq 2\|a_\ell u_r(\mathsf v_w-\bar w)\|.
\ee

Recall that $\bar z\in Z_r$, thus using~\eqref{eq: use proj thm ind LA} if $\alpha=\frac1{2\dm+1}$ and~\eqref{eq: use proj thm ind} if $\alpha>\frac1{2\dm+1}$, we conclude that for every $v\in I_{\bar z, r}$,
\be\label{eq: energy off Psi bar}
\sum_{v\neq v'\in I_{\bar z, r}^{\rm int}(v)}\max(\|a_\ell u_r(v-v')\|,{\trct'})^{-\alpha}\ll e^{-3{\imp}\ell/4}\egbd
\ee
where $\margI_{\bar z,r}^{\rm int}(v)=\{v'\in \margI_{\bar z,r}: \|v-v'\|\leq e^{-\dm\ell}\}$.

Let us now put
\be\label{eq: def of I<}
\margI^{\rm int}_z=\Bigl\{w\in \margI_{\cone}(z): \|w\|\leq e^{(-\dm-1)\ell}, \mathsf v_w\in \margI_{\bar z, r}\Bigr\}.
\ee
Note that for any $w\in \margI^{\rm int}_z$, we have $\|\bar w-\mathsf v_w\|\ll \beta\|w\|$, hence $\mathsf v_w\in I_{\bar z, r}^{\rm int}(\bar w)$.
Thus,~\eqref{eq: vw - wbar and w} and~\eqref{eq: energy off Psi bar} (applied with $v=\bar w$) imply that 
\begin{multline}\label{eq: non-bdry terms in I}
\sum_{0\neq w\in\margI^{\rm int}_z}\max(\|a_\ell u_r w\|,{\trct'})^{-\alpha}\\
\leq 2^{2\dm+1}\!\!\!\sum_{\bar w\neq v'\in I_{\bar z, r}^{\rm int}(\bar w)}\!\!\max(\|a_\ell u_r(\bar w-v')\|,{\trct'})^{-\alpha}
\ll e^{-3{\imp}\ell/4}\egbd. 
\end{multline}

\subsection*{The set $\cone(r)$}
Recall that $\mathcal B$ denotes the cubes of size $2^{-n}$ with $2^{-n}\geq e^{(-2\dm-1)\ell}>2^{-n-1}$ which are obtained by 
scaling some fixed partition of $\rfrak$ into unit size cubes by $2^{-n}$. 
Let $\mathcal C\subset \mathcal B$ denote the collection of those cubes $\mathsf C\in\mathcal B$ with the following property: 
there exists $w_{\mathsf C}\in F\cap\mathsf C$ satisfying that 
\begin{multline*}
\mu_{w_{\mathsf C}}\Bigl\{\bar z\in \hat\coneH\exp(w_{\mathsf C})y\cap Z_r: \tfrac{\#\margI_{\bar z, r}\cap \mathsf B}{\#\margI_{\cone}(\bar z)\cap \mathsf B}\ge (1-O(e^{-\star\pvare^2\ell})\Bigr\}\geq \\ \Bigl(1-O(e^{-\star\pvare^2\ell})\Bigr) \mu_{w_{\mathsf C}}(\hat\coneH\exp(w_{\mathsf C})y)
\end{multline*}
where $\mathsf B=B_\rfrak(0, e^{(-2\dm-1)\ell})$ and $\mu_{w_{\mathsf C}}$ is the restriction of the measure $\mu_{\hat\cone}$ to $\hat\coneH\exp(w_{\mathsf C})y$, see~\ref{sec: cone and mu cone}.  
For every $\mathsf C\in \mathcal C$, fix one such $w_{\mathsf C}$, and let 
\[
Z_{r}(w_{\mathsf C})=\Bigl\{\bar z\in \hat\coneH\exp(w_{\mathsf C})y\cap Z_r: \tfrac{\#\margI_{\bar z, r}\cap \mathsf B}{\#\margI_{\cone}(\bar z)\cap \mathsf B}\ge (1-O(e^{-\star\pvare^2\ell})\Bigr\}.
\]

Fix a covering $\{\mathsf Ph_j'\}_j$ of $\hat\coneH$ with multiplicity bounded by an absolute constant, where 
\[
\mathsf P=\Bigl\{u^-_s: |s|\leq 10\beta^2 \Bigr\}\cdot\{a_\tau: |\tau|\leq 10\beta^2\}\cdot \{u_r: |r|\leq 10e^{-\ell}\eta\};
\]
cf.\ e.g.\ \cite[Lemma 7.9]{LM-PolyDensity}. Note that by our choice of constants, $\beta^2$ is much larger than $e^{-\ell}\eta$.

For every $\mathsf C\in\mathcal C$, let $\mathcal J_{\mathsf C}$ denote the collection of $j$ so that $\mathsf Ph'_j\exp(w_{\mathsf C})\cap Z_r(w_{\mathsf C})\neq \emptyset$. For any $j\in\mathcal J_{\mathsf C}$, fix some $\sfh'_j\in\mathsf P$ so that 
\[
\bar z_{\mathsf C, j}:=\sfh_j'h_j'\exp(w_{\mathsf C})\in Z_{r}(w_{\mathsf C}).
\]

With this notation, set $\margI_{\mathsf C,j}=\margI_{\bar z_{\mathsf C, j}, r}\cap B_\rfrak(0, e^{(-2\dm-1)\ell})$, and let 
\[
\cone^{\mathsf C,j}(r)=\{\sfh \exp(\bar w) \bar z_{\mathsf C,j}: \sfh\in\mathsf P, \bar w\in\margI_{\mathsf C,j}\}\cap \hat\cone.
\]
Define $\cone(r):=\bigcup_{\mathsf C\in \mathcal C}\bigcup_{\mathcal J_{\mathsf C}} \cone^{\mathsf C,j}(r)$; this covering has multiplicity $\leq\mathsf K$, where $\mathsf K$ is absolute. It follows from~\eqref{eq: define Z},~\eqref{eq: measure of Cone-r}, and the above definitions that  
\be\label{eq: measure of Cone-r'}
\mu_\cone(\cone\setminus\cone(r))\ll e^{-\star\pvare^2\ell}.
\ee

\subsection*{Trimming $\cone'_{i,r}$ using $\cone(r)$}
Recall from part~(2) in Lemma~\ref{lem: convex comb} the following property of $F'_{i,r}$
\[
\text{$\umt^H_\ell.\exp(w)y_{i,r}\subset a_\ell u_r\cone\quad$ for all $w\in F'_{i,r}$},
\]
where $\umt_{\ell}^H=\Bigl\{u^-_s: |s|\leq \beta^2 \nuni^{-\ell}\Bigr\}\cdot\{a_\tau: |\tau|\leq \beta^2\}\cdot \{u_r: |r|\leq \eta\}$. 

For all $i$, all $\mathsf C_1\in \mathcal B_i'$, $\mathsf C_2\in\mathcal C$, and all $j\in\mathcal J_{\mathsf C_2}$, let 
\[
F'_{i,\mathsf C_1}(\bar z_{\mathsf C_2,j})=\Bigl\{w\in\mathsf C_1: \exists \bar w\in\margI_{\mathsf C_2,j},\, a_\ell u_r\exp(\bar w)\bar z_{\mathsf C_2,j}\in \umt_\ell^H.\exp(w)y_{i,r}\Bigr\},
\]
roughly speaking, this set captured points in $F'_{i,r}\cap \mathsf C_1$ which may be {\em traced back} to 
$\margI_{\mathsf C_2,j}=\margI_{\bar z_{\mathsf C_2, j}, r}\cap B_\rfrak(0, e^{(-2\dm-1)\ell})$.  

For every $i$, let 
\[
\mathcal B_i:=\Bigl\{\mathsf C_1\in\mathcal B_i': \max_{\mathsf C_2, j} \#F'_{i,\mathsf C_1}(\bar z_{\mathsf C_2,j})\geq (1-10^{-6})\cdot (\#F_{i,r}'\cap\mathsf C_1)\Bigr\},
\] 
and put $\mathcal I_r=\Bigl\{i: \textstyle\bigcup_{\mathsf C_1\in\mathcal B_i}(\mathsf C_1\cap F'_{i,r})\geq \beta^{1/2}\cdot \#(F'_{i,r})\Bigr\}$.

For every $i\in\mathcal I_r$ and all $\mathsf C_1\in\mathcal B_i$, choose $\bar z_{i, \mathsf C_1}\in \hat\coneH\exp(w_{\mathsf C_2, j})y)\cap Z_r$ so that $\#F'_{i,\mathsf C_1}(\bar z_{i,\mathsf C_1})$ realizes the above maximum. Put 
\be\label{eq: def F-ir and E-ir}
F_{i,r}=\textstyle\bigcup_{\mathsf C_1\in \mathcal B_i} F'_{i,\mathsf C_1}(\bar z_{i,\mathsf C_1})\quad\text{and}\quad\cone_{i,r}=\coneH.\{\exp(w)y_{i,r}: w\in F_{i,r}\}. 
\ee
For every $i\in\mathcal I_r$, let $\mu_{i,r}$ denote the restriction of $\mu_{\cone'_{i,r}}$ to $\cone_{i,r}$ normalized to be a probability measure.

\begin{lemma}\label{lem: mfm estimate for cone i,r}
With the above notation, we have 
\[
\mfm_{\cone_{i,r},{\trct'}}^{(\alpha)}(z)\leq e^{-{\imp}\ell/2}\egbd+e^{\vare t }\cdot (\#F_{i,r})
\]
for all $i\in \mathcal I_r$ and all $z\in \cone_{i,r}$.

Moreover, we have 
\begin{multline}\label{eq: rw ell 1 mathcal L mathcal I}
\int\varphi(a_{\mathsf d}u_sz)\diff(a_\ell u_{r}\mu_{\cone})(z)=\\\sum_{\mathcal I_r} c_{i, r}\int\varphi(a_{\mathsf d} u_sz)\diff\!\mu_{\cone_{i,r}}(z)+  O(\beta^\star\Lip(\varphi)).
\end{multline}
\end{lemma}

We postpone the proof of Lemma~\ref{lem: mfm estimate for cone i,r} to after the completion of the proof of Lemma~\ref{lem: main ind lemma}; a task which we now undertake.

\subsection*{Conclusion of the proof of Lemma~\ref{lem: main ind lemma}} 
First recall that  
\begin{align*}
\int\!\!\int_0^1\varphi(a_{\mathsf d} u_sa_\ell u_rz)\diff\!r\diff\!\mu_{\cone}(z)=\iint\varphi(a_{\mathsf d+\ell} u_{r+se^{-\ell}}z)\diff\!r\diff\!\mu_{\cone}(z).
\end{align*}
 
Since $|[0,1]\setminus L_{\mu_\cone}|\ll e^{-\star\pvare^2\ell}=\beta^\star$ and $\mathcal L\subset L_{\mu_\cone}$ 
is a maximal $e^{-3\mathsf d_0}$ separated subset, we have 
\begin{multline*}
\iint\varphi(a_{\mathsf d+\ell} u_{r+se^{-\ell}}z)\diff\!r\diff\!\mu_{\cone}(z)=\\
\sum_{r\in\mathcal L}\int \varphi(a_{\mathsf d+\ell} u_{r+se^{-\ell}}z)\diff\!\mu_{\cone}(z)+O\Bigl(\beta^\star\Lip(\varphi)\Bigr)=\\
\sum_{r\in\mathcal L}\int\varphi(a_{\mathsf d}u_s a_\ell u_{r}z)\diff\!\mu_{\cone}(z)+O\Bigl(\beta^\star\Lip(\varphi)\Bigr) .
\end{multline*}
Combining these and the second assertion in Lemma~\ref{lem: mfm estimate for cone i,r}, we conclude that 
\[
\iint\varphi(a_{\mathsf d} u_sa_\ell u_rz)\diff\!r\diff\!\mu_{\cone}=\sum_{\mathcal I} c_{i,r} \int\varphi(a_{\mathsf d} u_sz)\diff\!\mu_{\cone_{i,r}}(z) + O(\beta^\star\Lip(\varphi));
\]
moreover, by the first claim in Lemma~\ref{lem: mfm estimate for cone i,r}, we have 
\[
\mfm_{\cone_{i,r},{\trct'}}(z)\leq e^{-{\imp}\ell/2}\egbd+e^{\vare t }\cdot (\#F_{i,r}). 
\]
Finally, since $\beta^{4\dm+5}\cdot (\#F)\leq \#F'_{i,r}\leq \#F$ and $\#F_{i,r}\geq \beta^{1/2}\cdot(\#F'_{i,r})$, 
\[
\beta^{4\dm+6}\cdot (\#F)\leq \#F_{i,r}\leq \beta^{4\dm+4}\cdot(\#F).
\]
The proof of Lemma~\ref{lem: main ind lemma} is complete.
\qed


\begin{proof}[Proof of Lemma~\ref{lem: mfm estimate for cone i,r}]
We begin by establishing the first claim in the lemma. Fix some $i$, and denote $F_{i,r}$, $y_{i,r}$, and $\cone_{i,r}$ by $F_{\rm nw}$, $y_{\rm nw}$, and $\cone_{\rm nw}$, respectively.

First note that 
\be\label{eq: mfm cone nw estimate Theta}
\mfm_{\cone_{\rm nw},{\trct'}}^{(\alpha)}(z)\leq \sum_{0\neq w\in\margI_{\cone_{\rm nw}}^{\rm int}(z)}\max(\|w\|,{\trct'})^{-\alpha}+ e^{\vare t}\cdot (\#F_{\rm nw})
\ee
where $\margI_{\cone_{\rm nw}}^{\rm int}(z)=\{w\in\margI_{\cone_{\rm nw}}(z): \|w\|\leq e^{(-2\dm-2)\ell}\}$. 

We also recall the definition of $F_{\rm nw}$ and $\cone_{\rm nw}$ from~\eqref{eq: def F-ir and E-ir}. 
In particular, for every $\mathsf C_1\in\mathcal B_i$, there exists $\mathsf C_2\in\mathcal C$ and $\bar z=\bar z_{\mathsf C_2, j}$ so that for all $w\in F_{\rm nw}\cap \mathsf C_1$, there is some $\bar w\in\margI_{\mathsf C_2, j}=\margI_{\bar z, r}\cap B_\rfrak(0,e^{(-2\dm-1)\ell})$ with
\[
a_\ell u_r\exp(\bar w)\bar z\in \umt_\ell^H.\exp(w)y_{\rm nw}.
\]

For $k=1,2$, let $z_k=a_\ell u_r\exp(\bar w_k)\bar z$, and write
\[
z_k=\sfh_k\exp(w_k)y_{\rm nw},
\]
where $w_k\in F_{\rm nw}$ and $\sfh_k\in\umt_\ell^H$. Then we have 
\be\label{eq: zi hi wi} 
\begin{aligned}
z_2&=\sfh_2\exp(w_2) y_{\rm nw}=\sfh_2\exp(w_2)\exp(-w_1)\sfh_1^{-1}z_1\\
&=\sfh_2\sfh_1^{-1}\exp(\Ad(\sfh_1)w_2)\exp(-\Ad(\sfh_1)w_1)z_1\\
&=\sfh_2\sfh_1^{-1}\hat\sfh\exp(\hat w)z_1
\end{aligned}
\ee
where $\hat\sfh\in H$ and $\hat w\in\rfrak$, moreover, by Lemma~\ref{lem: BCH}, we have
\begin{subequations}
\begin{align}
\label{eq:hi-hat-wi-1}&\|\hat\sfh-I\|\leq \ref{E:BCH}\beta\|\hat w\|\qquad\text{and}\\
\label{eq:hi-hat-wi-2}&\tfrac12\|\Ad(\sfh_1)(w_2-w_1)\|\leq \|\hat w\|\leq2\|\Ad(\sfh_1)(w_2-w_1)\|.
\end{align}
\end{subequations} 

Now let $v\in I_{\cone_{\rm nw}}(z_1)$. 
Then there exist $w_v\in F_{\rm nw}$ and $h'_v\in\coneH$ so that 
\[
\exp(v)z_1=h'_v\exp(w_v)y_{\rm nw}.
\]
Moreover, if $\|v\|\leq e^{(-2\dm-2)\ell}$, then by~\eqref{eq: distance C1 and C'1} $w_v\in\mathsf C_1$, thus 
there exist $\bar w_v\in \margI_{\bar z, r}$ and $\sfh_v\in\umt_\ell^H$ so that 
\[
a_\ell u_r\exp(\bar w_v)\bar z=\sfh_v\exp(w_v) y_{\rm nw}.
\]
Altogether, if $\|v\|\leq e^{(-2\dm-2)\ell}$, then    
\be\label{eq: z-v and z0 relation 0}
a_\ell u_r\exp(\bar w_v)\bar z=h_v\exp(v)z_1
\ee
where $h_v=\sfh_vh_v'^{-1}\in\boxH_{1.1\eta}$.

Applying~\eqref{eq: zi hi wi} with $w_2=w_v$, we get that  
\be\label{eq: z-v and z0 relation}
a_\ell u_r\exp(\bar w_v)\bar z=\sfh_{v}\sfh_1^{-1}\hat\sfh_v\exp(\hat w_v)z_1
\ee
where $\hat \sfh_v$ and $\hat w_v$ satisfy~\eqref{eq:hi-hat-wi-1} and~\eqref{eq:hi-hat-wi-2}, and $\sfh_1,\sfh_v\in\umt_\ell^H$. 

Recall now that $(\hat h,\hat w)\mapsto\hat h\exp(\hat w)z_1$ is injective over $\boxH_{10\eta}\times B_\rfrak(0,10\eta)$. Thus, we conclude from~\eqref{eq: z-v and z0 relation} and~\eqref{eq: z-v and z0 relation 0} that  
\be\label{eq:wij-sublemma}
\hat w_v=v.
\ee

Since $\|v\|\leq e^{(-2\dm-2)\ell}$,~\eqref{eq:hi-hat-wi-1} and~\eqref{eq:wij-sublemma} imply that
\[
\|\hat\sfh_{v}-I\|\leq \ref{E:BCH}\beta\|v\| \leq e^{(-2\dm-2)\ell}\beta\leq \beta^2\nuni^{-\ell}; 
\]
recall that $\beta\geq e^{-\ell}$ and $\dm\geq 2$.
Moreover, $\sfh_1,\sfh_v\in\umt^H_\ell$. Therefore, 
\[
\hat\sfh_v^{-1}\sfh_1\sfh_v^{-1}a_\ell u_r\exp(\bar w_v)\bar z\in a_\ell u_r\cone,
\]
see~\eqref{eq:well-rd-tau-1}. This,~\eqref{eq: z-v and z0 relation}, and~\eqref{eq:wij-sublemma} yield 
\begin{align*}
\exp(v)a_\ell u_r\exp(\bar w_1)\bar z&=\exp(v)z_1\\
&=\hat\sfh^{-1}\sfh_1\sfh_v^{-1}a_\ell u_r\exp(\bar w_v)\bar z\in a_\ell u_r\cone.
\end{align*}
Recall also that $\bar w_v\in \margI_{\bar z, r}$, and note that that since $\|v\|\leq e^{(-2\dm-2)}$, we have $\|\Ad((a_\ell u_r)^{-1})v\|\leq e^{(-\dm-1)\ell}$. Altogether, we conclude that  
\be\label{eq: ad a-ell ur v}
\Ad((a_\ell u_r)^{-1})v\in \margI^{\rm int}_{\exp(\bar w_1)\bar z}, 
\ee
see~\eqref{eq: def of I<}. Let 
\[
\margI_{\cone_{\rm nw}}^{\rm int}(z_1)=\{v\in\margI_{\cone_{\rm nw}}(z_1):\|v\|\leq e^{(-2\dm-2)\ell}\}
\]
Since $\|\bar w_1\|\leq e^{(-2\dm-1)\ell}$, we conclude from~\eqref{eq: ad a-ell ur v} and~\eqref{eq: non-bdry terms in I}, applied with $\exp(\bar w_1)\bar z=:\hat z$, that 
\be\label{eq: mfm at z1}
\begin{aligned}
\mfm_{\cone_{\rm nw}, {\trct'}}^{(\alpha)}(z_1)&\leq \sum_{0\neq v\in\margI_{\cone_{\rm nw}}^{\rm int}(z_1)}\max(\|v\|,{\trct'})^{-\alpha}+e^{\vare t}\cdot (\#F_{\rm nw})\\
&\leq \sum_{0\neq w\in\margI^{\rm int}_{\hat z}}\max(\|a_\ell u_r w\|,{\trct'})^{-\alpha}+ e^{\vare t}\cdot (\#F_{\rm nw})\\
&\leq C e^{-3{\imp}\ell/4}\egbd+ e^{\vare t}\cdot (\#F_{\rm nw}),
\end{aligned}
\ee
where $C$ depends on $X$, see also~\eqref{eq: mfm cone nw estimate Theta} for the first inequality. 

Let now $z=h\exp(w)y_{\rm nw}\in\cone_{\rm nw}$ be arbitrary. Then there exists some $z_1=a_\ell u_r\exp(\bar w_1)\bar z$ so that $z_1=\sfh \exp(w)y_{\rm nw}$, and we have 
\[
\mfm_{\cone_{\rm nw},{\trct'}}^{(\alpha)}(z)\leq 2\mfm_{\cone_{\rm nw},{\trct'}}^{(\alpha)}(z_1).
\]
This and~\eqref{eq: mfm at z1} complete the proof of the first claim, assuming $\ell$ is large enough so that $ 2Ce^{-3{\imp}\ell/4}\leq e^{-{\imp}\ell/2}$.  

The second claim in the lemma follows from~\eqref{eq: measure of Cone-r'} and~\eqref{eq: rw ell 1 mathcal L} as we now explicate. 
Fix some $i$, and suppose that $\mathsf C_1\not\in\mathcal B_i$. Then for all $\mathsf C_2\in\mathcal C$ 
and all $j\in\mathcal J_{\mathsf C_2}$, we have 
\begin{multline*}
\#\bigl\{w\in \mathsf C_1\cap F'_{i,r}: \exp(w)y_{i,r}\not\in\bigl(\umt_\ell^H\bigr)^{-1}.a_{\ell}u_r\exp(\margI_{\mathsf C_2,j})\bar z_{\mathsf C_2, j}\Bigr\}\\ \geq 10^{-6}\cdot (\#(\mathsf C_1\cap F'_{i,r})).
\end{multline*}
Note now that $(a_\ell u_r)^{-1}\bigl(\umt_\ell^H\bigr)^{-1}.a_{\ell}u_r\subset \mathsf P$ and 
\[
m_H((a_\ell u_r)^{-1}\bigl(\umt_\ell^H\bigr)^{-1}.a_{\ell}u_r\asymp m_H(\mathsf P).
\]
Moreover, note that the multiplicity of the covering $\{\cone_r^{j,\mathsf C}\}$ is $\leq \mathsf K$. 
We thus conclude that the contribution of each $\mathsf C_1\not\in\mathcal B_i$ to $\mu_\cone(\cone\setminus\cone_r)$ is 
\[
\asymp (\#(\mathsf C_1\cap F'_{i,r}))m_H(\mathsf P).
\]
Therefore, the total number of $w\in \bigcup_i F'_{i,r}$ that we have omitted is $\ll \beta^\star (\#\bigcup_i F'_{i,r})$, which implies the claim.  
\end{proof}


\section{The main proposition}\label{sec: proj improve dim}
We begin by fixing several parameters that will remain constant throughout this section. Roughly speaking, $\kappa$ is a small parameter representing the incremental improvements in dimension. These improvements are achieved by inductively applying Lemma~\ref{lem: main ind lemma} with various choices of $\alpha$. The constants $d_\cdot$ indicate the number of times the lemma is applied for a fixed $\alpha$, while $p_{\rm fn}$ represents the number of times $\alpha$ is adjusted in increments of $\kappa$. See the discussion proceeding the statement of the theorem for more details.

Let $0<\kappa<1$ be a small parameter --- in our application, we will let 
$\kappa=(\frac{\ref{k:mixing}}{40\dm})^2$, see Proposition~\ref{prop: 1-epsilon N} and Lemma~\ref{lem: proj and thickening}.
Set 
\[
p_{\rm fn}:=\lceil\tfrac{4\dm^2+4\dm}{\kappa(2\dm+1)}\rceil-10,
\] 
and let $\kappa_{\rm fn}=\tfrac{99}{100} \times (\tfrac{3}{4})^{p_{\rm fn}}$ and $\vare=\kappa\cdot \kappa_{\rm fn}$.

Let $t>0$ be a large parameter, and let $D\geq D_1$, see Proposition~\ref{prop: closing lemma}. 
Let 
\[
\begin{aligned}
d_0&=\lceil \tfrac{(10D-9)(4\dm^2+2\dm)}{\vare}\rceil,\\
d_1&=\lceil\tfrac{495}{200\vare}\rceil\quad \text{and}\quad d_{p+1}=\lceil\tfrac34d_p\rceil\;\; \text{ for $0< p< p_{\rm fn}$}
\end{aligned}
\] 
Set $d_{\rm fn}=\sum_{p=0}^{p_{\rm fn}} d_p$, and let 
\[
\vare'=10^{-6}\dm^{-1}d_{\rm fn}^{-2},\quad \beta=e^{-\vare' t},\quad\text{and}\quad\eta^2=\beta.
\]

Put $\ell=\frac{\vare t}{10\dm}$. Let $\rwm_d$ denote the probability measure on $H$ defined by 
\[
\int\varphi\diff\!\rwm_d=\ave\varphi(a_{d}\uvk)\uvkd,
\] 
for any $\varphi\in C_c(H)$, and put 
\[
\mu_{t,\ell, n}=\rwm_{\ell}\conv\cdots\conv\nu_{\ell}\conv\rwm_{t}
\]
where $\rwm_\ell$ appears $n\geq0$ times in the above expression.

The following proposition, whose proof is based on Lemma~\ref{lem: main ind lemma}, is a crucial tool in our argument.

\begin{propos}\label{propos: imp dim main}
Let $x_1\in X$, and assume that Proposition~\ref{prop: closing lemma}(2) does not hold for $x_1$ with parameters $D\geq D_1$ and $t>0$. Let $r_1\in I(x_1)$ and put $x_2=a_{\dm_1 t}u_{r_1}x_1$, see Proposition~\ref{prop: closing lemma}(1). 

There exists a collection $\Xi=\{\cone_i: 1\leq i\leq N\}$ of sets 
\begin{align*}
    \cone_{i}=\coneH.\{\exp(w)y_{i}: w\in F_{i}\}\subset X_\eta
\end{align*} 
satisfying that for all $1\leq i\leq N$, we have $F_{i}\subset B_\rfrak(0,\beta)$ and 
\[
e^{0.98t}\leq \#F_i\leq e^{t}
\]
so that both of the following hold  
\begin{enumerate}
\item For all $1\leq i\leq N$ and all $z\in (\overline{\coneH\setminus\partial_{10\beta}\coneH}).\{\exp(w)y_{i}: w\in F_{i}\}$, 
\be\label{eq:improve dim energy est}
\mfm_{\cone_{i}, \trct}^{(2\dm+1-20\kappa)}(z)\leq e^{2\vare t}\cdot (\#F_i)\qquad\text{where $\trct=e^{-\kappa_{\rm fn}t}$.}
\ee
\item For all $\tau\leq \ell d_{\rm fn}$, all $|s|\leq 2$, and for every $\varphi\in C_c^{\infty}(X)$, 
\[
\int \varphi(a_\tau u_shx_2)\diff\!\mu_{t,\ell, d_{\rm fn}}(h)=\sum_{i}c_{i}\int \varphi(a_\tau u_sz)\diff\!\mu_{\cone_{i}}(z)+ O(\Lip(\varphi)\beta^{\ref{k: bootstrap beta exp}})
\]
\end{enumerate}
where 
\begin{itemize}
\item $0\leq c_{i}\leq 1$ and $\sum_{i}c_{i}=1-O(\beta^{\ref{k: bootstrap beta exp}})$, 
\item $\mu_{\cone_{i}}$ is an admissible measure on $\cone_i$ 
with parameter depending only on $D$ and $X$, see~\S\ref{sec: cone and mu cone},
\item $\Lip(\varphi)$ is the Lipschitz norm of $\varphi$, and 
\item $\constk\label{k: bootstrap beta exp}$ and the implied constants depend on $X$.
\end{itemize}
\end{propos}

The reader may compare this proposition to~\cite[Prop.\ 10.1]{LMW22}. Indeed in~\cite[Prop.\ 10.1]{LMW22}, we established similar bound concerning modified Margulis functions, but for $\alpha<1$ (in the setting considered in~\cite{LMW22}, 
$\dim\rfrak^+=1$) and localized to balls of size $e^{-\sqrt\vare t}$. 
Requiring this localized estimate, albeit for smaller $\alpha$, was responsible for {\em different} stopping times 
--- in~\cite[Prop.\ 10.1]{LMW22} we could not specify the stopping time, 
only guarantee that there will be at least one good stopping time $n$ in an interval of length $O(1/\sqrt\vare)$. 
This resulted in a less transparent endgame analysis. 
The estimate established in~\eqref{eq:improve dim energy est} concerns $\alpha$ dimensional energy for $\alpha$ very close to $2\dm+1=\dim\rfrak$. This is made possible thanks to stronger projection theorems, see Theorem~\ref{thm: proj thm}.

\subsection*{Proof of Proposition~\ref{propos: imp dim main}}
The proof of Proposition \ref{propos: imp dim main} relies on Lemma \ref{lem: main ind lemma}. Indeed with Lemma \ref{lem: main ind lemma} in place, the general strategy is straightforward: 
Let us put $\rho=\tau+t+d_{\rm fn}\ell$. 
Using results in~\cite[\S 6--8]{LMW22} and the fact that $x_2$ satisfies conditions in part~(1) 
of Proposition~\ref{prop: closing lemma}, we can write  
\[
\int \varphi(a_{\rho-t} u_shx_2)\diff\!\mu_{t,\ell, 0}(h)=\sum_{i}c_{i}\!\int \varphi(a_\tau u_sz)\diff\!\mu_{\cone_{i}}(z)+ O(\Lip(\varphi)\beta^{\star})
\]
where the decomposition is similar to the one claimed in the proposition, 
but with $\beta^{4\dm+5}e^{t}\leq \#F_i\leq e^t$, and where we only have the initial estimate 
\[
\mfm_{\cone_i,0}^{(\alpha)}(z)\leq e^{Dt}\qquad\text{for all $0<\alpha\leq 1$}.
\]

Then we apply Lemma~\ref{lem: main ind lemma} with $\alpha_0=\frac{1}{2\dm+1}$ (hence ${\imp}=\alpha_0$ and $\trct'=0$) for $d_0=\lceil \tfrac{(10D-9)(4\dm^2+2\dm)}{\vare}\rceil$ many steps. For every $|s|\leq 2$, thus 
\[
\int \varphi(a_{\rho-t-d_0\ell} u_shx_2)\diff\!\mu_{t,\ell, d_1}(h)=\sum_{i}c_{i}\!\int \varphi(a_{\tau'} u_sz)\diff\!\mu_{\cone_{i}'}(z)+ O(\Lip(\varphi)\beta^{\star})
\]
where the decomposition is as in the proposition, 
but now we have $e^{0.99t}\leq \#F_i'\leq e^{t}$ for all $i$, and the improved estimate  
\[
\mfm_{\cone_i',0}^{(\alpha_0)}(z)\leq 2e^{\vare t}\cdot (\#F_i').
\]
In the next phase, we improve the above estimate for $\mfm^{(\alpha_0)}$ inductively to obtain similar estimate for 
$\alpha_p=\alpha_0+p\kappa$ for all $0\leq p<p_{\rm fn}$: Assume 
\[
\mfm_{\cone_i,\delta_{p}}^{(\alpha_{p})}(z)\leq 2e^{\vare t}\cdot (\#F_i),
\]
where $\trct_1=e^{-0.99 t}$ and $\trct_{p+1}=\trct_p^{3/4}$ for all $p\geq 1$. We conclude that   
\[
\mfm_{\cone_i,\trct_p}^{(\alpha_{p+1})}(z)\leq 2e^{\vare t}\cdot \trct_{p}^{-\kappa}\cdot (\#F_i)
\]
We now apply Lemma~\ref{lem: main ind lemma} with $\alpha=\alpha_{p+1}$ for $d_{p+1}$ many steps. Since $\alpha_0\leq \alpha_{p+1}\leq 2\dm+1-10\kappa\dm$ we have $\imp\geq 9\kappa\dm$, and each application of Lemma~\ref{lem: main ind lemma} improves the bound on $\mfm$ by $e^{-4\kappa\dm\ell}$ while the scale $\trct$ is replaced by $\trct'=e^{\dm\ell}\trct$. 

Applying this, $p_{\rm fn}-1$ many times, we conclude that 
\[
\int \varphi(a_\tau u_shx_2)\diff\!\mu_{t,\ell, d_{\rm fn}}(h)=\sum_{i}c_{i}\int \varphi(a_\tau u_sz)\diff\!\mu_{\cone_{i}}(z)+ O(\Lip(\varphi)\beta^{\star})
\]
where the decomposition is as in the proposition, and  
\[
\mfm_{\cone_i,\trct_{p_{\rm fn}}}^{(\alpha_{p_{\rm fn}})}(z)\leq 2e^{\vare t}\cdot (\#F_i), 
\]
and the proposition follows. 

Let us now turn to the more detailed argument.

\subsection*{Smearing and Folner property}
The following equality which is a simple consequence of commutation relations, will be used throughout the proof 
\be\label{eq: commutation nth time}
a_{t_1} u_{r_1}a_{t_2}u_{r_2}=a_{t_1+t_2}u_{r_2+e^{-t_2}r_1}
\ee

Let $\lambda$ denote the uniform measure on $\boxHs_{\beta+100\beta^2}$, where as before for all $\delta>0$, we put 
\[
\boxHs_{\delta}=\{u_s^-:|s|\leq {\delta}\}\cdot\{a_\tau: |\tau|\leq \delta\}
\]
In view of~\cite[Lemma 7.4]{LMW22}, for all $x\in X$ and all $\varphi\in C_c^\infty(X)$,  
\be\label{eq: Thickening is free}
\int\varphi(hx)\diff(\rwm_{t_1+t_2})(h)=\int\varphi(hx)\diff(\lambda*\rwm_{t_2}*\lambda*\rwm_{t_1})(h)+O(\beta\Lip(\varphi))
\ee
so long as $e^{-t_i}\leq \beta^2$.

\subsection*{Closing lemma and initial separation} 
Recall that $x_2=a_{\dm_1 t}u_{r_1}x_1$ where $r_1\in I(x_1)$. In particular, the map $h\mapsto hx_2$ is injective over 
$\boxHs_\beta\cdot a_t\cdot U_1$, see part~(1) in Proposition~\ref{prop: closing lemma}. 
Put 
\[
\rho=\tau+d_{\rm fn}\ell+t=\tau+\textstyle\sum_{0}^{p_{\rm fn}} d_p\ell+t.
\] 
Recall also that $\mu_{t,\ell,d_{\rm fn}}=\nu_\ell^{(d_{\rm fn})}*\nu_t$. 
An inductive application of~\eqref{eq: commutation nth time} implies that for any 
$h_0\in\supp(\nu_\ell^{d_{\rm fn}})$ and any $|s|\leq 2$, there exists $|s'|\leq 2$ so that $a_\tau u_s h_0=a_{\rho-t}u_{s'}$. Thus 
\be\label{eq: h0 s' and nut}
\begin{aligned}
\int\!\!\varphi(a_{\tau} u_shx_2)\diff\!\mu_{t,\ell,d_{\rm fn}}(h)&=\int\!\!\int_0^1\varphi(a_{\tau} u_sh_0a_tu_rx_2)\diff\!\nu_\ell^{(d_{\rm fn})}(h_0)\diff\!r\\
&=\int\!\!\int_0^1\varphi(a_{\rho-t}u_{s'}a_tu_rx_2)\diff\!\nu_\ell^{(d_{\rm fn})}(h_0)\diff\!r
\end{aligned}
\ee
Now applying~\cite[Lemma 8.4]{LMW22}, see also Lemma~\ref{lem: convex comb}, 
with $x_2$ we get the following: for every $\varphi\in C_c^\infty(X)$, and all $|s'|\leq 2$, 
\be\label{eq: initial dim convex comb'}
\biggl|\int_0^1\!\!\varphi(a_{\rho-t} u_{s'}a_tu_rx_2)\diff\!r-\!\sum_i\! c_i\!\int\!\varphi(a_{\rho-t} u_{s'}z)\diff\!\mu_{\cone_i}(z)\biggr|\ll \Lip(\varphi)\beta^\star
\ee
where $\sum c_i=1-O(\beta^\star)$ and the implied constants depend only on $X$.  

Combining~\eqref{eq: h0 s' and nut} and~\eqref{eq: initial dim convex comb'}, and using $a_\tau u_s h_0=a_{\rho-t}u_{s'}$, we have 
\begin{multline}\label{eq: initial dim convex comb}
\int\!\!\varphi(a_{\tau} u_shx_2)\diff\!\mu_{t,\ell,d_{\rm fn}}(h)=\\
\sum_i c_i\!\int\!\varphi(a_{\tau} u_s h z)\diff\!\mu_{\cone_i}(z)\diff\!\nu_\ell^{(d_{\rm fn})}(h)+ O(\Lip(\varphi)\beta^\star)
\end{multline}

Moreover, $\cone_i=\coneH.\{\exp(w)y_i: w\in F_i\}$ where $y_i\in X_{\eta}$ and $\cone_i\subset X_\eta$. 
Since $\dim\rfrak=2\dm+1$,~\cite[Lemma 8.1 and 8.2]{LMW22}, see also Lemma~\ref{lem: E good h0} and recall that $\dm=1$ in~\cite{LMW22}, we have 
\be\label{eq: num of Fi initial dim}
\beta^{4\dm+7}e^t\leq \#F_i \leq e^{t}
\ee
--- in the references above, the upper bound $\beta^{-3}e^t$ is asserted, 
further subdividing $F_i$, we may replace that by $e^t$ as it is claimed here.  

Furthermore, in view of the definition of $\cone_i$, see~\cite[Lemma 8.4]{LMW22}, and the fact that $x_2$ satisfies properties in part~(1) of Proposition~\ref{prop: closing lemma}, 
\be\label{eq: initial dim'}
\mfm_{\cone_{i},0}^{(\alpha)}(z)\leq e^{Dt}\qquad\text{for all $0<\alpha\leq 1$}
\ee 
for all $i$ and all $z\in \cone_i$.

\subsection*{Applying Lemma~\ref{lem: main ind lemma} with $\alpha=\frac{1}{2\dm+1}$} 
The following lemma will be used to carry out the second phase in the above outline. Before stating the lemma, let us recall that 
$\ell=\frac{\vare t}{10\dm}$, and 
\[
d_0=\lceil\tfrac{(10D-9) (4\dm^2+2\dm)}{\vare}\rceil.
\] 
Also recall that we set $\vare'=10^{-6}\dm^{-1}d_{\rm fn}^{-2}$, $\beta=e^{-\vare' t}$, and $\eta^2=\beta$.

\begin{lemma}\label{lem: main lemma for alpha=m}
Fix some $\cone_i$ as in~\eqref{eq: initial dim convex comb} and some $0\leq k\leq d_0$. 
For every $\varphi\in C_c(X)$ and all $|s|\leq 2$ we have  
\begin{multline*}\label{eq: one step with ell alpha=m}
\int\!\varphi(a_{\tau} u_s h z)\diff\!\mu_{\cone_i}(z)\diff\!\nu_\ell^{(d_{\rm fn})}(h)=\\
\sum_{j}c_{ij}\!\int\!\!\int \varphi(a_{\tau} u_s h z)\diff\!\mu_{\cone_{ij}}(z)\diff\!\nu_\ell^{(d_{\rm fn}-k)}(h)+ O(\Lip(\varphi)\beta^{\star})
\end{multline*} 
and all the following hold
\begin{enumerate}
\item $0\leq c_{ij}\leq 1$ and $\sum c_{ij}=1+O(\beta^\star)$. 
\item $\cone_{ij}=\coneH.\{\exp(w)y_{ij}: w\in F_{ij}\}\subset X_\eta$ and 
\[
\beta^{(4k+4)\dm+6k+7}e^t\leq \#F_{ij} \leq e^{t}\qquad\text{for all $j$}.
\]
\item Let $\alpha_0=\frac{1}{2\dm+1}$. For all $j$ and all $z\in \cone_{ij}$ 
\[
\mfm_{\cone_{ij},0}^{(\alpha_0)}(z)\leq e^{Dt-\frac{k\alpha_0}2 \ell}+e^{\vare t}\cdot (\# F_{ij}).
\] 
\end{enumerate}
\end{lemma}

\begin{proof}
We prove the lemma using induction on $k$. The base case, $k=0$, 
follows from~\eqref{eq: initial dim'} applied with $\alpha=\alpha_0$. That is 
\be\label{eq: initial dim''}
\mfm_{\cone_{i},0}^{(\alpha_0)}(z)\leq e^{Dt}\qquad\text{for all $z\in \cone_i$.}
\ee 
Assume now the statement for some $0\leq k<d_0$:
\begin{multline*}
\int\!\varphi(a_{\tau} u_s h z)\diff\!\mu_{\cone_i}(z)\diff\!\nu_\ell^{(d_{\rm fn})}(h)=\\
\sum_{j}c_{ij}\!\int\!\!\int \varphi(a_{\tau} u_s h z)\diff\!\mu_{\cone_{ij}}(s)\diff\!\nu_\ell^{(d_{\rm fn}-k)}(h)+ O(\Lip(\varphi)\beta^{\star})
\end{multline*} 
and properties (1), (2), and (3) above hold. 

To obtain desired assertions for $k+1$, first note that~\eqref{eq: Thickening is free} implies 
\begin{multline}\label{eq: applying thickening is free 1}
\int\!\!\int \varphi(a_{\tau} u_s h z)\diff\!\mu_{\cone_{ij}}(s)\diff\!\nu_\ell^{(d_{\rm fn}-k)}(h)=\\
\int\!\!\int_0^1\!\!\int \varphi(a_{\tau}u_{s}h a_\ell u_{r}z)\diff\!\mu_{\cone_{ij}}\!\diff\!r\diff\!\nu_\ell^{(d_{\rm fn}-k-1)}(h)+O(\beta^\star\Lip(\varphi))
\end{multline}

Note again that for every $h\in\supp(\nu_\ell^{(d_{\rm fn}-k-1)})$ we have 
\[
a_{\tau}u_{s}h=a_{\tau+(d_{\rm fn}-k-1)\ell}u_{s'},\quad\text{ for some $|s'|\leq 2$}.
\] 
Thus, for all $|s|\leq 2$ and all $h\in\supp(\nu_\ell^{(d_{\rm fn}-k-1)})$, we apply 
Lemma~\ref{lem: main ind lemma} with $\cone=\cone_{ij}$ 
for $\alpha=\alpha_0$ and $\mathsf d=(d_{\rm fn}-k-1)\ell$ and $s'$ as above, and conclude 
\begin{multline}\label{eq: applying main lemma 1}
\int_0^1\!\!\int \varphi(a_{\tau}u_{s}h a_\ell u_{r}z)\diff\!\mu_{\cone_{ij}}(z)\diff\!r=\\ \sum_{\varsigma}c'_{\varsigma}\!\int\!\!\varphi(a_{\tau} u_sh z)\diff\!\mu_{\cone_{\varsigma}'}(z)+ O(\Lip(\varphi)\beta^{\star}),
\end{multline}
where $0\leq c'_{\varsigma}\leq 1$, $\sum c'_{\varsigma}=1+O(\beta^\star)$, and both of the following hold

\begin{itemize}
\item For all $\varsigma$, $F'_{\varsigma}\subset B_\rfrak(0,\beta)$ and $\cone'_\varsigma=\coneH.\{\exp(w)y'_{\varsigma}: w\in F'_{\varsigma}\}\subset X_\eta$, moreover 
\be\label{eq: improve dim energy est >alpha use 1}
\beta^{4\dm+6}\cdot (\#F_{ij})\leq \#F'_\varsigma\leq \#F_{ij}.
\ee
\item For all $\varsigma$ and all $z\in\coneH.\{\exp(w)y_{i}: w\in F'_{\varsigma}\}$, we have 
\begin{align*}
\mfm_{\cone'_{\varsigma}, 0}^{(\alpha_0)}(z)\leq e^{-\alpha_0\ell/2}\cdot \Bigl(e^{(Dt-\frac{k\alpha_0}{2}\ell)}+e^{\vare t}\cdot (\# F_{ij})\Bigr)+e^{\vare t}\cdot (\#F'_{\varsigma})
\end{align*}
where we used $\trct_{0}= 0$ and ${\imp}=\alpha_0$ when $\alpha=\alpha_0$.  
\end{itemize}
Using~(2) in the inductive hypothesis~\eqref{eq: improve dim energy est >alpha use 1}, we have 
\begin{align*}
\beta^{(4(k+1)+4)\dm+6(k+1)+7}e^t&=\beta^{4\dm+6}\beta^{(4k+4)\dm+6k+7}e^t\\
&\leq  \beta^{4\dm+6}\cdot (\#F_{ij})\leq \#F'_\varsigma\leq \#F_{ij}\leq e^t.
\end{align*}
Moreover, recall that $\vare'=10^{-6}\dm^{-1}d_{\rm fin}^{-2}$ and $\beta=e^{-\vare' t}$. 
Thus $e^{-\alpha_0\ell/2} \beta^{-4\dm-7}\leq 1$, and using~\eqref{eq: improve dim energy est >alpha use 2} and $\#F_{ij}\leq e^t$, we conclude  
\[
e^{-\alpha_0\ell/2}(e^{\vare t}\cdot (\#F_{ij}))\leq e^{-\alpha_0\ell/2} \beta^{-4\dm-6} e^{\vare t}\cdot (\#F'_{\varsigma}) \leq e^{\vare t}\cdot (\#F'_{\varsigma}),
\]
Thus one obtains  
\begin{align*}
e^{-\alpha_0\ell/2}\cdot\Bigl(e^{(Dt-\frac{k\alpha_0}{2} \ell)}+e^{\vare t}\!\cdot\! (\# F_{ij})\Bigr)&\leq
 e^{-\alpha_0\ell/2}e^{Dt-\frac{k\alpha_0}{2} \ell}+ e^{\vare t}\!\cdot\! (\#F'_{\varsigma})\\
&\leq e^{Dt-\frac{(k+1)\alpha_0}{2} \ell}+e^{\vare t}\!\cdot\! (\#F'_{\varsigma}).
\end{align*} 

Therefore, integrating~\eqref{eq: applying main lemma 1} over $h\in\supp(\nu_\ell^{(d_{\rm fn}-k-1)})$, 
the proof of the inductive step and of the lemma is complete. 
\end{proof}

\subsection*{Lemma~\ref{lem: main lemma for alpha=m} and initial dimension}
Recall now that $\vare$ is small, $\ell=\frac{\vare t}{10\dm}$, $d_0=\lceil \tfrac{(10D-9) (4\dm^2+2\dm)}{\vare}\rceil$, 
$\vare'=10^{-6}\dm^{-1}d_{\rm fn}^{-2}$, and $\beta=e^{-\vare' t}$. Thus 
\begin{align*}
&e^{0.99t}\leq \beta^{(4d_1+4)\dm+ 6d_1+7}e^t\leq \#F_{ij}\leq e^t,\quad\text{and}\\
&e^{Dt-\frac{d_0\alpha_0}2 \ell}\leq e^{0.9t},\quad\text{where $\alpha_0=\tfrac1{2\dm+1}$.}
\end{align*}

Applying Lemma~\ref{lem: main lemma for alpha=m} with $k=d_0$, hence, we conclude 
\begin{multline}\label{eq: apply main lemma for alpha=m}
\int\!\varphi(a_{\tau} u_s h z)\diff\!\mu_{\cone_i}(z)\diff\!\nu_\ell^{(d_{\rm fn})}(h)=\\
\sum_{j}c_{ij}\!\int\!\!\int \varphi(a_{\tau} u_s h z)\diff\!\mu_{\cone_{ij}}(z)\diff\!\nu_\ell^{(d_1)}(h)+ O(\Lip(\varphi)\beta^{\star})
\end{multline}
where $\cone_i$ is any set appearing in~\eqref{eq: initial dim convex comb}, and all the following are satisfied 
\begin{enumerate}[label={(D-\arabic*)}]
\item\label{(D-1)} $0\leq c_{ij}\leq 1$ and $\sum c_{ij}=1+O(\beta^\star)$. 
\item\label{(D-2)} $\cone_{ij}=\coneH.\{\exp(w)y_{ij}: w\in F_{ij}\}\subset X_\eta$ and 
\[
e^{0.99t}\leq \#F_{ij} \leq e^{t}\qquad\text{for all $j$}.
\]
\item\label{(D-3)} For all $j$ and all $z\in \cone_{ij}$ 
\be\label{eq: mfm final estimate for alpha=m}
\mfm_{\cone_{ij},\trct}^{(\alpha_0)}(z)\leq 2e^{\vare t}\cdot (\# F_{ij}).
\ee
\end{enumerate}

\subsection*{Lemma~\ref{lem: main ind lemma} and incremental dimension improvement}
Recall that $p_{\rm fn}=\lceil\tfrac{4\dm^2+4\dm}{\kappa(2\dm+1)}\rceil-10$ and $\ell=\frac{\vare t}{10\dm}$. 
We let $\trct_1=e^{-0.99 t}$ and let $d_1=\frac{495}{200\vare}$. For all $0<p\leq p_{\rm fn}$, let 
\[
\trct_{p+1}=\trct_p^{3/4}\quad\text{and}\quad d_{p+1}=\tfrac34d_{p}.
\]
Also, for all $0<p\leq p_{\rm fn}$, and all $0<k\leq d_p$, let 
\[
\trct_{p,0}=\trct_p\quad\text{and}\quad\trct_{p,k}=e^{\dm\ell}\trct_{p,k-1}; 
\] 
note that $\trct_{p,d_{p}}=\trct_{p+1}$, and 
\be\label{eq: dp delta_p}
\trct_p^{-\kappa} e^{-{4\dm\kappa}d_p\ell}\leq 1
\ee

Finally, we let $\alpha_p=\alpha_0+p\kappa$, for every $0\leq p\leq p_{\rm fn}$, and let ${\imp}_p$  be ${\imp}$ defined as in~\eqref{eq: def alpha'} with $\alpha=\alpha_p$. In particular,  
\be\label{eq: alpha'-p}
{\imp}_p\geq 9\dm\kappa\quad\text{for all $0<p\leq p_{\rm fn}$.}
\ee

The following lemma, which is an analogue of Lemma~\ref{lem: main lemma for alpha=m}, 
will be used to inductively improve the dimension from $\frac{1}{2\dm+1}$ to $2\dm+1-10\dm\kappa$.

\begin{lemma}\label{lem: main lemma for alpha=m+1}
Fix some $\cone_i$ as in~\eqref{eq: initial dim convex comb}. 
Let $0< p< p_{\rm fn}$ and let $0\leq k\leq d_{p+1}$. 
For every $\varphi\in C_c(X)$ and all $|s|\leq 2$, we have  
\begin{multline*}
\int\!\varphi(a_{\tau} u_s h z)\diff\!\mu_{\cone_i}(z)\diff\!\nu_\ell^{(d_{\rm fn})}(h)=\\
\sum_{j}c_{ij}\!\int\!\!\int \varphi(a_{\tau} u_s h z)\diff\!\mu_{\cone_{ij}}(z)\diff\!\nu_\ell^{(d'(p)-k)}(h)+ O(\Lip(\varphi)\beta^{\star})
\end{multline*} 
where $d'(p)=d_{\rm fn}-d_0-d(p)$, $d(p)=\sum_{q=1}^{p-1} d_{q}$ for $p>1$, and $d(1)=0$. 

Moreover, all the following hold
\begin{enumerate}
\item $0\leq c_{ij}\leq 1$ and $\sum c_{ij}=1+O(\beta^\star)$. 
\item $\cone_{ij}=\coneH.\{\exp(w)y_{ij}: w\in F_{ij}\}\subset X_\eta$ and 
\[
\beta^{(4\dm+6)(d(p)+k)}e^{0.99t}\leq \#F_{ij} \leq e^{t}\qquad\text{for all $j$}.
\]
\item For all $j$ and all $z\in \cone_{ij}$, we have  
\[
\mfm_{\cone_{ij},\trct_{p,k}}^{(\alpha_{p})}(z)\leq 2e^{\vare t}\cdot (\trct_p^{-\kappa} e^{-4\dm\kappa k\ell}+1)\cdot (\# F_{ij})
\] 
\end{enumerate}
\end{lemma}

\begin{proof}
The proof is similar to the proof of Lemma~\ref{lem: main lemma for alpha=m}, and is completed by induction on $p$ and $k$. 
Indeed the case $p=1$ and $k=0$ follows from~\eqref{eq: apply main lemma for alpha=m} and properties \ref{(D-1)}, \ref{(D-2)}, \ref{(D-3)}.
Indeed, the assertions in (1) and (2) in this lemma \ref{(D-1)} and \ref{(D-2)} are the same. To see (3) in the lemma follows    
from~\ref{(D-3)}, note that~\eqref{eq: mfm final estimate for alpha=m} implies 
\[
\mfm_{\cone_{ij},0}^{(\alpha_0)}(z)\leq 2e^{\vare t}\cdot (\# F_{ij}) 
\]
Recall that $\alpha_1=\alpha_0+\kappa$ and $\delta_1=\delta_{1,0}=e^{-0.99 t}$. Thus
\be\label{eq: base case alpha=m+1}
\begin{aligned}
\mfm_{\cone_{ij},\trct_{1,0}}^{(\alpha_1)}(z)&\leq \delta_1^{-\kappa}\sum \max(\|w\|, \delta_1)^{-\alpha_0}\\
&\leq 2e^{\vare t}\cdot \delta_1^{-\kappa} \cdot (\# F_{ij})
\end{aligned}
\ee
as it is claimed in part~(3) for $p=1$ and $k=0$. 

Fix some $p$ and assume now the statement for some $0\leq k<d_{p+1}$: 
\begin{multline*}
\int\!\varphi(a_{\tau} u_s h z)\diff\!\mu_{\cone_i}(z)\diff\!\nu_\ell^{(d_{\rm fn})}(h)=\\
\sum_{j}c_{ij}\!\int\!\!\int \varphi(a_{\tau} u_s h z)\diff\!\mu_{\cone_{ij}}(z)\diff\!\nu_\ell^{(d'(p)-k)}(h)+ O(\Lip(\varphi)\beta^{\star})
\end{multline*} 
and properties (1), (2), and (3) above hold. 

To obtain desired assertions for $k+1$, first note that~\eqref{eq: Thickening is free} implies 
\begin{multline}\label{eq: applying thickening is free 2}
\int\!\!\int \varphi(a_{\tau} u_s h z)\diff\!\mu_{\cone_{ij}}(s)\diff\!\nu_\ell^{(d'(p)-k)}(h)=\\
\int\!\!\int_0^1\!\!\int \varphi(a_{\tau}u_{s}h a_\ell u_{r}z)\diff\!\mu_{\cone_{ij}}\!\diff\!r\diff\!\nu_\ell^{(d'(p)-k-1)}(h)+O(\beta^\star\Lip(\varphi))
\end{multline}

For every $h\in\supp\Bigl(\nu_\ell^{(d'(p)-k)}\Bigr)$ we have 
\[
a_{\tau}u_{s}h=a_{\tau+(d'(p)-k-1)\ell}u_{s'},\quad\text{ for some $|s'|\leq 2$}.
\] 
Thus, for all $|s|\leq 2$ and $h\in\supp\Bigl(\nu_\ell^{(d'(p)-k-1)}\Bigr)$, 
Lemma~\ref{lem: main ind lemma} applied with $\cone=\cone_{ij}$ 
for $\alpha=\alpha_p$ and $\mathsf d=(d'(p)-k-1)\ell$ and $s'$ as above, gives 
\begin{multline}\label{eq: applying main lemma 2}
\int_0^1\!\!\int \varphi(a_{\tau}u_{s}h a_\ell u_{r}z)\diff\!\mu_{\cone_{ij}}(z)\diff\!r=\\ \sum_{\varsigma}c'_{\varsigma}\!\int\!\!\varphi(a_{\tau} u_sh z)\diff\!\mu_{\cone_{\varsigma}'}(z)+ O(\Lip(\varphi)\beta^{\star}),
\end{multline}
where $0\leq c'_{\varsigma}\leq 1$, $\sum c'_{\varsigma}=1+O(\beta^\star)$, and both of the following hold

\begin{itemize}
\item For all $\varsigma$, $F'_{\varsigma}\subset B_\rfrak(0,\beta)$ and $\cone'_\varsigma=\coneH.\{\exp(w)y'_{\varsigma}: w\in F'_{\varsigma}\}\subset X_\eta$, moreover 
\be\label{eq: improve dim energy est >alpha use 2}
\beta^{4\dm+6}\cdot (\#F_{ij})\leq \#F'_\varsigma\leq \#F_{ij}.
\ee
\item For all $\varsigma$ and all $z\in\coneH.\{\exp(w)y_{i}: w\in F'_{\varsigma}\}$, we have 
\[
\mfm_{\cone'_{\varsigma}, \trct_{p,k+1}}^{(\alpha_p)}(z)\leq e^{-{\imp}_p\ell/2}\cdot \Bigl(2e^{\vare t}\cdot (\trct_p^{-\kappa} e^{-4\dm\kappa k\ell}+1)\cdot (\# F_{ij})\Bigr)+e^{\vare t}\cdot (\#F'_{\varsigma})
\]
where, $\trct_{p,k+1}= e^{\dm \ell}\delta_{p,k}$.
\end{itemize}
Using~(2) in the inductive hypothesis~\eqref{eq: improve dim energy est >alpha use 2}, we have 
\begin{align*}
\beta^{(4\dm+6)(d(p)+k+1)}e^{0.99t}&=\beta^{4\dm+6}\beta^{(4\dm+6)(d(p)+k)}e^{0.99t}\\
&\leq  \beta^{4\dm+6}\cdot (\#F_{ij})\leq \#F'_\varsigma\leq \#F_{ij}\leq e^t.
\end{align*}
Moreover, recall from~\eqref{eq: alpha'-p} that $\alpha_p'>9\dm\kappa$, also recall that $\beta=e^{-\vare' t}$ where 
$\vare'=10^{-6}\dm^{-1}d_{\rm fin}^{-2}$. Using~\eqref{eq: improve dim energy est >alpha use 2}, thus 
\begin{multline*}
e^{-{\imp}_p\ell/2}\!\cdot\! \Bigl(2e^{\vare t}\cdot (\trct_p^{-\kappa} e^{-{4\dm\kappa}k \ell}+1)\cdot (\# F_{ij})\Bigr)\\
\leq e^{-{\imp}_p\ell/2}\Bigl(2e^{\vare t}\cdot (\trct_p^{-\kappa} e^{-{4\dm\kappa}k\ell}+1)\cdot \beta^{-4\dm-6}\cdot (\#F'_{\varsigma})\Bigr)\\
\leq \trct_p^{-\kappa} e^{-{4\dm\kappa}(k+1)\ell}\cdot (\#F'_{\varsigma})
\end{multline*}
This and the above estimates imply that 
\[
\mfm_{\cone'_{\varsigma}, \trct_{p,k+1}}^{(\alpha_p)}(z)\leq (\trct_p^{-\kappa} e^{-{4\dm\kappa}(k+1)\ell}+e^{\vare t})\cdot (\#F'_{\varsigma})
\]
where, $\trct_{p,k+1}= e^{\dm \ell}\delta_{p,k}$.

Therefore, integrating~\eqref{eq: applying main lemma 2} over $h\in\supp(\nu_\ell^{(d'(p)-k-1)})$, 
the proof of the inductive step.

After $d_p$ steps, thus, one obtains  
\[
\mfm_{\cone'_{\varsigma}, \trct_{p,d_p}}^{(\alpha_p)}(z)\leq (\trct_p^{-\kappa} e^{-{4\dm\kappa}d_p\ell}+e^{\vare t})\cdot (\#F'_{\varsigma}).
\]
Since $\trct_{p,d_p}=\trct_{p+1}$ and $\trct_p^{-\kappa} e^{-{4\dm\kappa}d_p\ell}\leq 1$, see~\eqref{eq: dp delta_p}, we conclude that 
\be\label{eq: lemma with alphap dp}
\mfm_{\cone'_{\varsigma}, \trct_{p+1}}^{(\alpha_p)}(z)\leq 2e^{\vare t}\cdot (\#F'_{\varsigma}).
\ee
Using $\alpha_{p+1}=\alpha_p+\kappa$, the above yields,
\[
\mfm_{\cone'_{\varsigma}, \trct_{p+1}}^{(\alpha_{p+1})}(z)\leq 2e^{\vare t}\cdot\trct_{p+1}^{-\kappa}\cdot (\#F'_{\varsigma}).
\]
Therefore, we may repeat the above for $0\leq k<d_{p+1}$. 

This completes the proof. 
\end{proof}

\subsection*{Conclusion of the proof of Proposition~\ref{propos: imp dim main}}
Recall that 
\[
\kappa_{\rm fn}=\tfrac{99}{100} \times (\tfrac{3}{4})^{p_{\rm fn}}\quad\text{and}\quad\vare=\kappa\cdot \kappa_{\rm fn}.
\] 
Also note that it follows from $\trct_{p,d_p}=\trct_p^{3/4}=\trct_{p+1}$ and $\trct_1=e^{-0.99t}$ that 
\[
\delta_{p_{\rm fn}, d_{p_{\rm fn}}}=\trct_{p_{\rm fn}}^{3/4}=e^{-\kappa_{\rm fn}t}
\]
Applying Lemma~\ref{lem: main lemma for alpha=m} with $\alpha_{p_{\rm fn}}$ and $k=d_{p_{\rm fn}}$, hence, we conclude 
\begin{multline}\label{eq: apply main lemma for alpha=m+1}
\int\!\varphi(a_{\tau} u_s h z)\diff\!\mu_{\cone_i}(z)\diff\!\nu_\ell^{(d_{\rm fn})}(h)=\\
\sum_{j}c_{ij}\!\int\!\!\int \varphi(a_{\tau} u_s h z)\diff\!\mu_{\cone_{ij}}(z)+ O(\Lip(\varphi)\beta^{\star})
\end{multline}
where $\cone_i$ is any set appearing in~\eqref{eq: initial dim convex comb}, and all the following are satisfied 
\begin{enumerate}
\item $0\leq c_{ij}\leq 1$ and $\sum c_{ij}=1+O(\beta^\star)$. 
\item $\cone_{ij}=\coneH.\{\exp(w)y_{ij}: w\in F_{ij}\}\subset X_\eta$ and 
\[
e^{0.98t}\leq \#F_{ij} \leq e^{t}\qquad\text{for all $j$},
\]
where we used $e^{0.98t}\leq \beta^{(4\dm+6)d(p_{\rm fn})}e^{0.99t}\leq \#F_{ij}\leq e^t$.
\item For all $j$ and all $z\in \cone_{ij}$ 
\[
\mfm_{\cone_{ij},\trct}^{(2\dm+1-20\kappa)}(z)\leq 2e^{\vare t}\cdot (\# F_{ij}) \qquad\text{where $\trct=e^{-\kappa_{\rm fn}t}$},
\]
see~\eqref{eq: lemma with alphap dp} and note that $\alpha_{p_{\rm fn}}\geq 2\dm+1-20\kappa$ 
\end{enumerate}
In view of (1), (2), and (3), Proposition~\ref{propos: imp dim main} follows from~\eqref{eq: apply main lemma for alpha=m+1} and~\eqref{eq: initial dim convex comb}. 
\qed 


\section{From large dimension to equidistribution}\label{sec: large to equi}

The main result of this section is Proposition~\ref{prop: high dim to equid} which will be used in the final step of our proof of Theorem~\ref{thm: main}.  

Let $0<\ref{k:mixing}\leq 1$ be the constant given by Proposition~\ref{prop: 1-epsilon N} ---
this constant is closely related to the spectral gap (or mixing rate) in $G/\Gamma$, c.f.~\eqref{eq: exp mixing}. 
Throughout this section, let $\kappa=(\frac{\ref{k:mixing}}{40\dm})^2$ and $p_{\rm fn}=\lceil\frac{4\dm^2+4\dm}{\kappa(2\dm+1)}\rceil-10$. Let 
\[
\kappa_{\rm fn}=\tfrac{99}{100}\times (\tfrac{3}{4})^{p_{\rm fn}}\quad\text{and}\quad \vare=\kappa\cdot \kappa_{\rm fn}
\] 
We also recall that $\beta=e^{-\vare' t}$ and $\eta^2=\beta$ where $0<\vare'<10^{-6}\vare^{4}$. Let 
\[
\alpha=2\dm+1-20\kappa.
\]

\begin{propos}\label{prop: high dim to equid}
The following holds for all large enough $t$. 
Let $F\subset B_\rfrak(0,\beta)$ be a finite set with $\#F\geq \nuni^{0.9t}$. 
Let 
\[
\cone=\coneH.\{\exp(w)y: w\in F\}\subset X_\eta
\]
be equipped with an admissible measure $\mu_\cone$, see~\S\ref{sec: cone and mu cone}. 
Assume further that the following is satisfied: For all $z=h\exp(w)y$ with $h\in\overline{\coneH\setminus\partial_{10\beta}\coneH}$, 
\be\label{eq: dim m+1 equi prop} 
\mfm_{\cone, \trct_0}^{(\alpha)}(e,z)\leq 2\nuni^{\vare t}\cdot (\#F)\qquad\text{where $\trct_0=e^{-\kappa_{\rm fn}t}$.}
\ee 
Let $\tau$ be a parameter in the range $\frac{\kappa_{\rm fn}t}{8\dm}\leq \tau\leq \frac{\kappa_{\rm fn}t}{4\dm}$. Then 
\[
\bigg|\int_0^1\!\!\int\varphi(a_\tau u_rz)\diff\!\mu_{\cone}(z)\diff\!r-\int \varphi\diff\! m_X\biggr|\ll \Sob(\varphi)\beta^{\star}
\]
for all $\varphi\in C_c^\infty(X)$.
\end{propos}

\begin{proof}
The proof of Proposition~\ref{prop: high dim to equid}, which will be completed in a few steps, is based on Lemma~\ref{lem: proj and thickening}, and in turn on Proposition~\ref{prop: 1-epsilon N}.

\subsection*{Folner property}
Let $\tau$ be as in the statement, and write $\tau=\ell_1+\ell_2$ where  
\be\label{eq: tau and ell-i}
\ell_2=\tau/(1+\ref{k:mixing})\quad\text{and}\quad \ell_1=\ref{k:mixing}\ell_2;
\ee
In particular, $4\dm\ell_2\leq 4\dm\tau\leq \kappa_{\rm fn}t=|\log\delta_0|$ and $\ell_2\geq \tau/2\geq |\log\eta|/\ref{k:mixing}$, 
where the parameter $\eta$ is fixed above, see Lemma~\ref{lem: proj and thickening} for these choices.   

For any $\varphi\in C_c^\infty(X)$, we have 
\begin{multline}\label{eq: int decomp 1}
\int_0^1\!\!\int\varphi(a_\tau u_rz)\diff\!\mu_{\cone}(z)\diff\!r=\\
\int_{0}^1\!\!\int_{0}^1\!\!\int \varphi(a_{\ell_1}u_{r_1}a_{\ell_2}u_{r_2}z)\diff\!\mu_{\cone}(z)\diff\!r_2\diff\!r_1
+O(e^{-\ell_2}\Lip(\varphi))
\end{multline}
where the implied constant depends on $X$. 

In view of~\eqref{eq: int decomp 1}, the proof of Proposition~\ref{prop: high dim to equid} is reduced to investigating the following 
\[
\int_{0}^1\!\!\int_{0}^1\!\!\int \varphi(a_{\ell_1}u_{r_1}a_{\ell_2}u_{r_2}z)\diff\!\mu_{\cone}(z)\diff\!r_2\diff\!r_1.
\]

\subsection{Conditional measures of $\mu_\cone$}\label{sec: conditional mu cone}
Recall that $\cone=\coneH.\{\exp(w)y:w\in F\}$. Fix some $v\in F$ and let $z=\exp(v)y$. Then  
\be\label{eq: chage y to z}
\begin{aligned}
\sfh\exp(w)y&=\sfh\exp(w)\exp(-v)\exp(v)y\\
&=\sfh\sfh_w\exp(v_w)z
\end{aligned}
\ee
where $\|\sfh_w-I\|\ll\beta^2$ and $\frac12\|w-v\|\leq \|v_w\|\leq 2\|w-v\|$, see Lemma~\ref{lem: BCH}.

For our application here, it will be more convenient 
to {\em recenter} $\cone$ from $y$ to $z$. To that end, 
note that $w\mapsto v_w$, see~\eqref{eq: chage y to z}, 
is a one-to-one map. Let $F_v=\{v_w: w\in F\}$, and let $\cconeH=\overline{\coneH\setminus\partial_{20\beta}\coneH}$. Set 
\[
\ccone:=\cconeH.\{\exp(w)z: w\in F_v\}.
\]
Then by~\eqref{eq: chage y to z} and since $\|\sfh_w-I\|\ll\beta^2$, we have $\ccone\subset\cone$; moreover, 
\[
\bar\mu_{\cone}(\cone\setminus\ccone)\ll \beta.
\] 
Thus it suffices to show the claim in the lemma with $\bar\mu_{\cone}$ replaced by 
\[
\hat\mu:=\frac{1}{\bar\mu_{\cone}(\ccone)}\bar\mu_{\cone}|_{\ccone}.
\] 

For later reference, let us also record that $\|\sfh_w-I\|\ll\beta^2$ and~\eqref{eq: chage y to z} imply that indeed
\be\label{eq: ccone in cone'}
\ccone\subset \cone':=\coneH'.\{\exp(w)y: w\in F\}
\ee
where $\coneH'=\overline{\coneH\setminus\partial_{10\beta}\coneH}$. In particular,~\eqref{eq: dim m+1 equi prop} holds for all $z\in\ccone$.

Recall that $\hat\mu$ is the probability measure proportional to 
$\sum_w\hat\mu_{w}$ where $\diff\!\hat\mu_{w}=\hat{\density}_{w}\diff\! m_H$ and 
$\hat{\density}_{w}\asymp 1$. As it was mentioned earlier the proof of Proposition~\ref{prop: high dim to equid} relies on Proposition~\ref{prop: 1-epsilon N}. To set the stage for the latter to be applicable, we will use Fubini's theorem to change the order of disintegration of $\hat\mu$ as follows. 
Let $z\in\ccone$, then  
\[
z=\sfh\exp(v)z=\exp(\Ad(\sfh)v)\sfh z\in\ccone.
\] 
Moreover, $\Ad(\sfh)v\in B_\rfrak(0,8\bar\eta\mfsc)$. Since $\bar\eta/2\leq \inj(z')\leq 2\bar\eta$ for every $z'\in\cone$, we conclude that 
\[
\Ad(\sfh)v\in\margI_{\cone}(\sfh z).
\] 
Let $\pi:\ccone\to\coneH.z$ denote the projection $z'=\sfh\exp(w)z\mapsto \sfh z$.
Using Fubini's theorem, we have  
\[
\hat\mu=\int\hat\mu^\sfh\diff\!\pi_*\hat\mu(\sfh.z),
\]
where $\hat\mu^\sfh$ denotes the conditional measure of $\hat\mu$ for the factor map $\pi$. 

Note that $\hat\mu^\sfh$ is supported on $\margI_{\cone}(\sfh z)$.
In view of the above discussion, $\diff\!\pi_*\hat\mu$ is proportional to $\hat{\density}\diff\! m_H$ restricted to the support of 
$\pi_*\hat\mu$ where $1\ll\hat{\density}\ll 1$, moreover, 
for every $w\in\supp(\hat\mu^\sfh)$,  
\be\label{eq: weight of atoms of cmu}
\hat\mu^\sfh(w)\asymp (\#F)^{-1}
\ee  
where the implied constant depends on $X$. 

Now, using Fubini's theorem, we have 
\begin{multline*}
\int_{0}^1\!\!\int_{0}^1\!\!\int \varphi(a_{\ell_1}u_{r_1}a_{\ell_2}u_{r_2}z)\diff\!\hat\mu(z)\diff\!r_2\diff\!r_1=\\
\int_{\cconeH.z}\!\!\int_{0}^1\!\!\int_{0}^1\!\!\int \varphi(a_{\ell_1}u_{r_1}a_{\ell_2}u_{r_2}\exp(w)\sfh z)\diff\!\hat\mu^\sfh(w)\diff\!r_2\diff\!r_1\diff\!\pi_*\cmu(\sfh.z).
\end{multline*}

Fix some $\sfh\in\cconeH=\overline{\coneH\setminus\partial_{20\mfsc}\coneH}$. 
The proof of the proposition is thus reduced to investigating the following 
\be\label{eq: fix h reduces to this}
\int_0^1\!\!\int_0^1\!\!\int\varphi(a_{\ell_1}u_{r_1}a_{\ell_2}u_{r_2}\exp(w)\sfh z)\diff\!\hat\mu^\sfh(w)\diff\!r_2\diff\!r_1.
\ee

\subsection*{Discretized dimension of $\cmu^\sfh$ and Lemma~\ref{lem: proj and thickening}}
Set
\[
F^\sfh:=\supp(\cmu^\sfh)=\{\Ad(\sfh)w: w\in F\}.
\] 
Note that~\eqref{eq: ccone in cone'} implies 
\[
\exp(\Ad(\sfh)w)\sfh z=\sfh\exp(w)z\in\ccone_i\subset\cone'
\]
Moreover,~\eqref{eq: dim m+1 equi prop} and Lemma~\ref{lem: margIz energy est} imply that for every $w\in F^\sfh$, 
\[
\eng_{F^\sfh, \trct}^{(\alpha)}(w)\ll 2e^{\vare t}\cdot (\#F),
\]
where $\eng$ is defined as in \S\ref{sec: proj and energy}. 
This and~\eqref{eq: weight of atoms of cmu} yield
\be\label{eq: uniform meas on F-h is regular'}
\cmu^\sfh(B(w,\delta))\ll 2e^{\vare t}\delta^\alpha\qquad\text{for all $\delta\geq \delta_0$}
\ee
where the implied constant depends only on $X$. 

Therefore, Lemma~\ref{lem: proj and thickening} applies with $\mu=\cmu^\sfh$, $\varrho=\beta$, $\delta_0=e^{-\kappa_{\rm fn}t}$, 
\[
\egbd=2\delta_0^{-20\kappa}e^{\vare t}=2e^{20\kappa\kappa_{\rm fn} t}e^{\vare t},
\]
and $\ell_1$ and $\ell_2$ as above, see~\eqref{eq: tau and ell-i}. 
By the conclusion of that lemma, thus, 
\begin{multline}\label{eq: apply equidistribution}
\int_0^1\!\!\int_0^1\!\!\int \varphi(a_{\ell_1}u_{r_1}a_{\ell_2}u_{r_2}\exp(w)\sfh z)\diff\!\cmu^\sfh(w)\diff\!r_2\diff\!r_1=\\
\int\varphi\diff\!m_X+O\Bigl(\Sob(\varphi)(\beta^\star+\eta+\egbd^{1/2}\beta^{-3/2} e^{-\ref{k:mixing}\ell_1})\Bigr)
\end{multline}
Recall now that $\tau\geq \frac{\kappa_{\rm fn}t}{8\dm}$ and $\ell_1=\ref{k:mixing}\ell_2=\frac{\ref{k:mixing}\tau}{1+\ref{k:mixing}}$.
Therefore, $\ell_1\geq \frac{\kappa_{\rm fn}\ref{k:mixing}t}{10\dm}$, and 
\[
\egbd^{1/2}\beta^{-3/2} e^{-\ref{k:mixing}\ell_1}\leq e^{10\kappa_{\rm fn}\kappa t}e^{\vare t}e^{-\frac{\kappa_{\rm fn}\ref{k:mixing}^2}{10\dm}t}\leq e^{-\frac{\kappa_{\rm fn}\ref{k:mixing}^2}{20\dm}t},
\]
where in the second to last inequality we used $2e^{\vare t/2}\beta^{-3/2}\leq e^{\vare t}$, and in last inequality, we used 
$\kappa=(\frac{\ref{k:mixing}}{40\dm})^2$ and $\vare=\kappa\cdot\kappa_{\rm fn}$. 

In view of~\eqref{eq: apply equidistribution}, thus, the proof of Proposition~\ref{prop: high dim to equid} is complete.  
\end{proof}


\section{Proof of Theorem~\ref{thm: main}}\label{sec: proof main}

The proof will be completed in some steps 
and is based on various propositions which were discussed so far.

\subsection*{Fixing the parameters}
Let $0<\ref{k:mixing}\leq 1$ be the constant given by Proposition~\ref{prop: 1-epsilon N}. 
Let $\kappa=(\frac{\ref{k:mixing}}{40\dm})^2$ and $p_{\rm fn}=\lceil\frac{4\dm^2+4\dm}{\kappa(2\dm+1)}\rceil-10$. Put 
\be\label{eq: choose theta equi sec final}
\kappa_{\rm fn}=\tfrac{99}{100}\times (\tfrac{3}{4})^{p_{\rm fn}}\quad\text{and}\quad \vare=\kappa\cdot \kappa_{\rm fn}
\ee
Let $D=D_0D_1+2D_1$ where $D_0$ is as in Proposition~\ref{prop: linearization translates} and $D_1$ is as in Proposition~\ref{prop: closing lemma}; we will always assume $D_1, D_0\geq 10\dm$. We will show the claim holds with
\be\label{eq: define a:main-3-1}
\ref{a:main-3-1}=2D_0(2\dm+1)+\dm_0+\dm_1+4,
\ee
where $\dm_0$ is as in Proposition~\ref{prop:Non-div-main} and $\dm_1$ is as in Proposition~\ref{prop: closing lemma}.

Let us also assume (as we may) that  
\be\label{eq: how big is R xi}
R\geq \max\{(10\ref{E:non-div-main})^3\inj(x_0)^{-\dm_0}, e^{\ref{E:non-div-main}}, e^{s_0}, \ref{c: linear trans}\},
\ee
see Proposition~\ref{prop: Non-div main} and Proposition~\ref{prop: linearization translates}. 

Let $T\geq R^{\ref{a:main-3-1}}$, and
suppose that Theorem~\ref{thm: main}(2) does not hold. 
That is, for every $x\in X$ with $Hx$ periodic and $\vol(Hx)\leq R$, 
\be\label{eq: main thm 2 fails}
\dist_X(x,x_0)> R^{\ref{a:main-3-1}}(\log T)^{\ref{a:main-3-1}}T^{-\dm}\geq (\log S)^{D_0}S^{-\dm}
\ee
where $S:=R^{-\ref{a:main-3-1}}T$. 

\subsection*{Folner property and random walks}
We put $d_{\rm fn}=\sum_{p=0}^{p_{\rm fn}} d_p$, where 
\[
\begin{aligned}
d_0&=\lceil \tfrac{(10D-9)(4\dm^2+2\dm)}{\vare}\rceil,\\
d_1&=\lceil\tfrac{495}{200\vare}\rceil\quad \text{and}\quad d_{p+1}=\lceil\tfrac34d_p\rceil\;\; \text{ for $0< p< p_{\rm fn}$}
\end{aligned}
\] 
Let $t=\frac{1}{D_1}\log R$, and let $\ell = \frac{\vare t}{10\dm}$. Then 
\be\label{eq: range of d1ell}
d_{\rm fn}\ell\leq (\tfrac{(10D-9)(2\dm+1)}{5}+\tfrac{99}{100\dm})t +\vare t 
\ee

We now write $\log T=t_3+t_2+t_1+t_0$ where 
\be\label{eq: def t0 t1 t2}
t_1=\dm_1t, \quad t_2=t+d_{\rm fn}\ell, \quad\text{and}\quad t_3=\tfrac{\kappa_{\rm fn}}{8\dm}t
\ee
The choice of $t_2$ is motivated by Proposition~\ref{propos: imp dim main}, the choices of $t_1$ and $t_3$ are motivated by Proposition~\ref{prop: closing lemma} and Proposition~\ref{prop: high dim to equid}, respectively. Thus
\[
\begin{aligned}
t_0&=\log T-(t_1+t_2+t_3)\\
&\geq \log T- (\dm_1+1+(\tfrac{(10D-9)(2\dm+1)}{5}+\tfrac{99}{100\dm})+\vare + \tfrac{\kappa_{\rm fn}}{8\dm})t\\
&\geq \log T -\tfrac{\dm_1+2D(2\dm+1)+2}{D_1}\log R
\end{aligned}
\]
we used~\eqref{eq: def t0 t1 t2} and~\eqref{eq: range of d1ell} in the second line, and used $t=\frac{1}{D_1}\log R$ in the third line. 
Using this and~\eqref{eq: define a:main-3-1}, we conclude that 
\be\label{eq: computation for t0}
\begin{aligned}
t_0&\geq \log T-\ref{a:main-3-1}\log R+(\dm_0+2)\log R\\
&\geq \log S+\dm_0|\log\inj(x_0)|+2\log R.
\end{aligned}
\ee
we used $R\geq \inj(x_0)^{-\dm_0}$ and $\log S=\log T-\ref{a:main-3-1}\log R$ in the last inequality.

Since $a_{\rho_1}u_ra_{\rho_2}=a_{\rho_1+\rho_2}u_{e^{-\rho_2}r}$, for any $\varphi\in C_c^\infty(X)$, we have
\begin{multline}\label{eq: proof integral 1}
\int_0^1\varphi(a_{\log T} u_rx_0)\diff\!r=O(\|\varphi\|_\infty e^{-t})\; +\\
\int_0^1\int_0^1\int_0^1\int_0^1\varphi(a_{t_3} u_{r_3}a_{t_2}u_{r_2}a_{t_1}u_{r_1}a_{t_0}u_{r_0}x_0)\diff\!r_3\diff\!r_2\diff\!r_1\diff\!r_0 
\end{multline}
where the implied constant is absolute and we used $t_0, t_1, t_2\geq t$.

\medskip

Finally, recall that $0<\vare'<10^{-6}\vare^{4}$ and we put $\beta=e^{-\vare' t}$ and $\eta^2=\beta$.

\subsection*{Improving the Diophantine condition}\label{sec:improving diophantine}
Apply Proposition~\ref{prop: linearization translates} with $S=R^{-\ref{a:main-3-1}}T$, then for all
\[
s\geq \max\{\log S,\dm_0|\log\inj(x_0)|\}+s_0,
\] 
we have the following  
\be\label{eq: apply linearization proof}
\biggl|\biggl\{r\in [0,1]: \begin{array}{c}\text{$a_{s}u_rx_0\not\in X_\eta$ or $\exists\, x$ with $\vol(Hx)\leq R$} \\ 
\text{so that }\dist_X(x, a_{s}u_rx_0)\leq R^{-D_0-1}\end{array}\biggr\}\biggr|\ll\eta^{1/\dm_0},
\ee
where we also used $\eta^{1/\dm_0}\geq R^{-1}$ and $R\geq \ref{c: linear trans}$. 

Let $J_0\subset[0,1]$ be the set of those $r_0\in[0,1]$ so that $a_{t_0}u_{r_0}x_0\in X_{\eta}$ and 
\[
\dist_X(x, a_{t_0}u_{r_0}x_0)> R^{-D_0-1}=e^{-D_1(D_0+1)t}
\] 
for all $x$ with $\vol(Hx)\leq R=e^{D_1t}$. 
Then since by~\eqref{eq: computation for t0} and~\eqref{eq: how big is R xi} we have  
\[
t_0\geq \log S+\dm_0|\log\inj(x_0)|+2\log R\geq \max(\log S,\dm_0|\log\inj(x_0)|)+s_0,
\] 
the assertion in~\eqref{eq: apply linearization proof} implies that $|[0,1]\setminus J_0|\ll\eta^{1/\dm_0}$. In consequence, 
\begin{multline}\label{eq: proof integral 2}
\int_{0}^1\varphi(a_{\log T} u_rx_0)\diff\!r=O(\|\varphi\|_\infty \eta^{1/\dm_0})\; +\\
\int_{J_0}\int_0^1\int_0^1\int_0^1\varphi(a_{t_3} u_{r_3}a_{t_2}u_{r_2}a_{t_1}u_{r_1}x(r_0))\diff\!r_3\diff\!r_2\diff\!r_1\diff\!r_0 
\end{multline}
where $x(r_0)=a_{t_0}u_{r_0}x_0$ and the implied constant depends on $X$.

\subsection*{Applying the closing lemma}
For every $r_0\in J_0$, we now apply Proposition~\ref{prop: closing lemma} 
with $x(r_0)$, $D=D_0D_1+2D_1$ and the parameter $t$. 
For any such $r_0$, we have
\[
\dist_X(x, x(r_0))> e^{-D_1(D_0+1)t} =e^{(-D+D_1)t}
\] 
for all $x$ with $\vol(Hx)\leq e^{D_1t}$. Thus Proposition~\ref{prop: closing lemma}(1) holds.
Let 
\[
J_1(r_0)=I(x(r_0))=I(a_{t_0}u_{r_0}x_0)
\] 
Then 
\begin{multline}\label{eq: proof integral 3}
\int_0^1\varphi(a_{\log T} u_rx_0)\diff\!r=O(\|\varphi\|_\infty \eta^{\frac1{2\dm_0}})\; +\\
\int_{J_0}\int_{J_1(r_0)}\int_0^1\int_0^1\varphi(a_{t_3} u_{r_3}a_{t_2}u_{r_2}x(r_0,r_1))\diff\!r_3\diff\!r_2\diff\!r_1\diff\!r_0 
\end{multline}
where $x(r_0, r_1)=a_{t_1}u_{r_1}a_{t_0}u_{r_0}x_0$ and the implied constant is absolute.

\subsection*{Improving the dimension phase}
Fix some $r_0\in J_0$, and let $r_1\in J(r_0)$. Put $x_1=x(r_0, r_1)$. 
Recall that 
\[
\mu_{t,\ell, d_{\rm fn}}=\rwm_{\ell}\conv\cdots\conv\nu_{\ell}\conv\rwm_{t}
\]
where $\rwm_\ell$ appears $d_{\rm fn}$ times in the above expression. Then 
\begin{multline}\label{eq: proof integral 4}
\int_0^1\!\!\int_0^1\varphi(a_{t_3} u_{r_3}a_{t_2}u_{r_2}x_1)\diff\!r_3\diff\!r_2=\\
\int\!\!\int_0^1\!\! \varphi(a_{t_3} u_{r_3}hx_1)\diff\!r_3 \diff\!\mu_{t,\ell, d_{\rm fn}}(h)+ O(\Lip(\varphi)e^{-\ell}),
\end{multline}
where we used $t_2=d_{\rm fn}\ell$ and~\eqref{eq: commutation nth time}, see also~\cite[Lemma 7.4]{LMW22}.  

We now apply Proposition~\ref{propos: imp dim main} with $x_1$ and $\tau=t_3$ and $s=r_{3}$. Then 
\begin{multline}\label{eq: proof integral 5}
\int \varphi(a_{t_3} u_{r_3}hx_1)\diff\!\mu_{t,\ell, d_{\rm fn}}(h)=\\\sum_{i}c_{i}\!\!\int \varphi(a_{t_3} u_{r_3}z)\diff\!\mu_{\cone_{i}}(z)+ O(\Lip(\varphi)\beta^{\star})
\end{multline}
where $0\leq c_{i}\leq 1$ and $\sum_{i}c_{i}=1-O(\beta^{\star})$ and the implied constants depend on $X$. 
Moreover, for all $i$ we have 
\[
    \cone_{i}=\coneH.\{\exp(w)y_{i}: w\in F_{i}\}\subset X_\eta
\]
with $F_{i}\subset B_\rfrak(0,\beta)$, $e^{0.98t}\leq \#F_i\leq e^{t}$, and  
\be\label{eq: proof final energy est}
\mfm_{\cone_{i}, \trct}^{(\alpha)}(z)\leq e^{2\vare t}\cdot (\#F_i)\quad\text{where $\trct=e^{-\kappa_{\rm fn}t}$ and $\alpha=2\dm+1-20\kappa$}
\ee
for all $z\in (\overline{\coneH\setminus\partial_{10\beta}\coneH}).\{\exp(w)y_{i}: w\in F_{i}\}$.   

\subsection*{From large dimension to equidistribution}
We now apply Proposition~\ref{prop: high dim to equid} with $\cone_{i}$ (see~\eqref{eq: proof final energy est} and recall that $t_3=\frac{\kappa_{\rm fn}t}{8\dm}$). Hence, 
\be\label{eq: proof integral 7}
\bigg|\int_0^1\!\!\int\varphi(a_{t_3} u_{r_3}z)\diff\!\mu_{\cone_{i}}(z)\diff\!r_3-\int \varphi\diff\! m_X\biggr|\ll \Sob(\varphi)\beta^\star
\ee
where the implied constant depends on $X$. 

Recall now that $\beta=R^{-\star}$. Thus,~\eqref{eq: proof integral 7},~\eqref{eq: proof integral 5},~\eqref{eq: proof integral 4},~\eqref{eq: proof integral 3}, \eqref{eq: proof integral 2}, and~\eqref{eq: proof integral 1}, imply
\[
\bigg|\int_0^1\varphi(a_{\log T} u_rx_0)\diff\!r-\int \varphi\diff\! m_X\biggr|\ll \Sob(\varphi)R^{-\star},
\]
where the implied constant depends on $X$. The proof is complete.
\qed

\section{Proof of Theorem \ref{thm:main unipotent}}\label{sec: proof long unipotent}
In this section, we use an argument analogous to \cite[\S16]{LMW22} to establish Theorem \ref{thm:main unipotent}. As in \cite[\S16]{LMW22} and previously noted, the proof relies on Theorem~\ref{thm: main} and linearization techniques of Dani and Margulis, albeit in their quantitative form, which were developed in~\cite{LMMS}.

\begin{lemma}[cf.\ \cite{LMW22}, Lemma 16.1]
\label{lem: unipotent linearization}
There exist $\consta\label{a: linearization 1}$, $\consta\label{a: linearization 2}$, and $\constE\label{E: unip lin}$ (depending on $X$) 
so that the following holds. Let $S,M>0$, and $0<\eta<1/2$ satisfy
\[
S\geq M^{\ref{a: linearization 1}}\quad\text{and}\quad M\geq \ref{E: unip lin}\eta^{-\ref{a: linearization 1}}.
\]
Let $x_1\in X_\eta$, and suppose there exists $\exceptional\subset \{r\in [-S,S]: u_rx_1\in X_\eta\}$
with 
\[
|\exceptional|>\ref{E: unip lin}\eta^{1/\ref{a: linearization 2}}S
\]
so that for every $r\in\exceptional$, 
there exists $y_r\in X$ with 
\[
\text{$\vol(H.y_r)\leq M\quad$ and $\quad\dist(u_{r}x_1, y_r)\leq M^{-\ref{a: linearization 1}}$}.
\]
Then one of the following holds  
\begin{enumerate}
    \item There exists $x \in G/\Gamma$ with $\vol(H.x)\leq M^{\ref{a: linearization 1}}$, 
and for every $r\in [-S,S]$ there exists $g\in G$ with $\|g\|\leq M^{\ref{a: linearization 1}}$ so that  
\[
\dist_X(u_{s}x_1, gH.x)\leq M^{\ref{a: linearization 1}}\left(\frac{|s-r|}{S}\right)^{1/\ref{a: linearization 2}}\quad\text{for all $s\in[-S,S]$.}
\] 
\item There is a parabolic subgroup $P\subset G$ and some $x \in G/\Gamma$ satisfying that  
$\vol(R_u(P).x)\leq M^{\ref{a: linearization 1}}$, and for every $r\in [-S,S]$ there exists $g\in G$ with $\|g\|\leq M^{\ref{a: linearization 1}}$ so that  
\[
\dist_X(u_{s}x_1, gR_u(P).x)\leq M^{\ref{a: linearization 1}}\left(\frac{|s-r|}{S}\right)^{1/\ref{a: linearization 2}}\quad\text{for all $s\in[-S,S]$.}
\] 
In particular, $X$ is not compact. 
\end{enumerate}
\end{lemma}

\subsection*{Arithmetic groups}
Recall that $G=\G(\R)$ where $\G$ is an absolutely almost simple $\R$-group, and $H=\H(\R)^\circ$ where $\H\simeq\SL_2$ or $\PGL_2$ is an $\R$-subgroup of $\G$ and the connected component is considered as a Lie group.

Since $\Gamma$ is an arithmetic lattice, there exists a semisimple simply connected $\Q$-almost simple $\Q$-group 
$\tilde\G\subset\SL_{\mathsf D}$, for some $\mathsf D$, and an epimorphism 
\[
\rho:\tilde\G(\R)\to \G(\R)=G
\]
of $\R$-groups with compact kernel so that $\Gamma$ is commensurable with $\rho(\tilde\G(\Z))$. 
Moreover, since $\tilde\G$ is simply connected, we can identify $\tilde\G(\bbr)$ with $G\times G'$ where $G'=\ker(\rho)$ is compact.

We are allowed to choose the parameter $M$ in the lemma to be large depending on $\Gamma$, therefore, by passing to a finite index subgroup, we will assume that $\Gamma\subset \tilde\Gamma:=\rho(\tilde\G(\Z))$, where $\tilde\G(\Z)=\tilde\G(\R)\cap\SL_{\mathsf D}(\Z)$. 

Thus, every $\gamma\in \Gamma$ lifts uniquely to $(\gamma,\sigma(\gamma))\in\tilde\Gamma$, where $\sigma$ is (a collection of) Galois automorphisms.  For every $g\in G$, put 
\[
\hat g=(g,1)\in G\times G'.
\]

If $g\in G$ is so that $Hg\Gamma$ is periodic, let $\Delta_g=\Gamma\cap g^{-1}Hg$ and let $\tilde\Delta_g=\rho^{-1}(\Delta_g)\cap\tilde \Gamma$. Let $\tilde\H_g$ be the Zariski closure of $\tilde\Delta_g$. 
Then $\tilde\H_g$ is a semisimple $\bbq$-subgroup, 
and the restriction of $\rho$ to $\tilde\H_g$ surjects onto $g^{-1}\H g$. 

Let $\tilde H_g=\tilde\H_g(\R)^\circ$, as a Lie group, then  
\[
\overline{\hat g^{-1}\hat H\hat g\tilde\Gamma}=\tilde H_g\tilde\Gamma
\]  

\subsection*{Lie algebras and the adjoint representation}
Recall that $\Lie(G)=\gfrak$ and $\Lie(H)=\hfrak$, which are considered as $\R$-vector spaces. 
Let $v_H$ be a unit vector on the line $\wedge^3\hfrak$. Then 
\[
N_G(H)=\{g\in G: gv_H=v_H\}
\] 
which contains $H$ as a subgroup of finite index. 

Let $\tilde\gfrak=\Lie(\tilde\G(\R))$ and $\tilde\gfrak_\Z:=\tilde\gfrak\cap\sl_N(\Z)$. Then $\tilde\gfrak$ has a natural $\Q$-structure, 
and $\tilde\gfrak_\Z$ is a $\tilde\G(\Z)$-stable lattice in $\tilde\gfrak$.  

Assuming $H$ has periodic orbits in $X=G/\Gamma$, we fix 
${\mathsf g}_1,\ldots,{\mathsf g}_k$ so that $\vol(H{\mathsf g}_i\Gamma)\ll 1$ 
(the implied constant and $k$ depend on $\Gamma$) 
and that every $\tilde H_g$ is conjugate to some $\tilde H_i=\tilde H_{{\mathsf g}_i}$ in $\tilde G$. Let 
\[
{\bf v}_{i}\in \wedge^{\dim \tilde H_i}(\Lie(\tilde H_i))\subset\wedge^{\dim \tilde H_i}\tilde\gfrak
\]
be a primitive integral vector. Then $N_{\tilde G}(\tilde H_i)=\{g\in \tilde G: gv_{i}=v_i\}$,
and $\tilde H_i\subset N_{\tilde G}(\tilde H_i)$ has finite index. 
For all $i$, 
\[
{\bf v}_i=c_i \cdot \Bigl((g_i^{-1}v_H)\wedge v'_i\Bigr) \quad\text{where $v'_i\in\wedge\Lie(G')$ and $|c_i|\asymp 1$}.
\] 

More generally, if ${\bf L}\subset \tilde\G$ is a $\Q$-subgroup, 
we let ${\bf v}_L$ be a primitive integral vector on the line 
$\wedge^{\dim L}\Lie(L)\subset\wedge^{\dim L}\tilde\gfrak$ where $L=\Lbf(\R)^\circ$.  
Recall from~\cite{LMMS} the definition of the height of ${\bf L}$ 
\be\label{eq: def height v-L}
{\rm ht}(\Lbf)=\|{\bf v}_L\|.
\ee

Fix a right invariant metric on $\tilde G$ defined using 
the killing form and the maximal compact subgroup $\tilde K=K\times G'$; 
this metric induces the right invariant metric on $G$ which we fixed on p.~\pageref{d definition page}.

\begin{lemma}\label{lem: volume and height}
Let $Hg\Gamma$ be a periodic orbit, and let $\tilde\H_g$ be as above.  
Both of the following properties hold:
\[
{\rm ht}(\tilde\H_g)^\star\ll\vol(\tilde H_g\tilde\Gamma/\tilde\Gamma)\ll {\rm ht}(\tilde\H_g)^\star
\]
\[
\|g\|^{-\star}\vol(Hg\Gamma)\ll\vol(\tilde H_g\tilde\Gamma/\tilde\Gamma)\ll\|g\|^\star\vol(Hg\Gamma)
\]
\end{lemma}

\begin{proof}
We refer the reader to~\cite[Lemma 16.2]{LMW22} and references there for the proof of this lemma. 
\end{proof}

We now proceed to prove Lemma~\ref{lem: unipotent linearization}. 
This proof follows an adaptation of the argument presented in~\cite[\S16]{LMW22}, which we recount here for the reader's convenience.

\begin{proof}[Proof of Lemma~\ref{lem: unipotent linearization}]
Given our assumption in the lemma, periodic $H$ orbits exist. Let $\tilde H_1,\ldots, \tilde H_{k}$ be defined as above.  
We introduce the constants $\ref{a: linearization 1}$ and $\ref{a: linearization 2}$ which will be chosen as sufficiently large values to be explicated later. Specifically, we will require that $\ref{a: linearization 1}>\max(A, D_2)$, $\ref{a: linearization 2}>D$ and $\ref{E: unip lin}> \max\{kE_1, \ref{c: red th}\}$ 
where $A$, $D$, and $E_1$ are as in \cite[Thm.~1.4]{LMMS} applied with 
$\{\hat u_r\}\subset\tilde G$, and $D_2$ and $\ref{c: red th}$ are as described in Lemma~\ref{lem: reduction theory}. 

Write $x_1=g_1\Gamma$, where 
$\|g_1\|\leq\ref{c: red th} \eta^{-D_2}\leq M$, see Lemma~\ref{lem: reduction theory} and our assumption in this lemma. 
For every $r\in\exceptional$, write $y_r=g(r)\Gamma$ where 
$\|g(r)\|\leq M$, indeed, for every such $r$ there exists $\gamma_r\in\Gamma$ so that 
\be\label{eq: unip lin exceptial gamma r}
 u_rg_1\gamma_r=\epsilon(r)g(r)\quad\text{and}\quad \|u_rg_1\gamma_r\|\leq M+1, 
\ee
where $\|\epsilon(r)\|\ll M^{-\ref{a: linearization 1}}$. 

For every $1\leq i\leq k$, let 
\[
\exceptional_i=\{r\in\exceptional: \text{$\tilde H^r:=\tilde H_{g(r)}$ is a conjugate of $\tilde H_i$}\}.
\]
There is some $i$ so that $|\exceptional_i|\geq |\exceptional|/k$. 
Replacing $\exceptional$ by $\exceptional_i$, we assume that $\tilde H^r$ is a conjugate of $\tilde H_i$ 
for all $r\in\exceptional$. Put $\tilde H^r=\tilde g(r)^{-1}\tilde H_i\tilde g(r)$. Then 
\[
\tilde g(r)=({\mathsf g}_i^{-1}g(r), \tilde g'(r))\in G\times G',
\]
and ${\bf v}^r:=\frac{\|{\bf v}_{\tilde H^r}\|}{\|\tilde g(r)^{-1}{\bf v}_i\|}\tilde g(r)^{-1}{\bf v}_i=\pm {\bf v}_{\tilde H^r}$. 
Moreover, we have  
\be\label{eq: vr and the norm of gr}
{\bf v}^r=c_r\cdot\Bigl((g(r)^{-1}v_H)\wedge (\tilde g'(r)^{-1}v'_i)\Bigr)\quad\text{where $|c_r|\ll {\rm ht}(\tilde\H_g)M\ll M^\star$}
\ee
where we used Lemma~\ref{lem: volume and height} to conclude ${\rm ht}(\tilde\H_g)M\ll M^\star$. 

In view of~\eqref{eq: unip lin exceptial gamma r}, we have 
\be\label{eq: unip lin exceptional gamma r 1}
\hat u_r\hat g_1(\gamma_r,\sigma(\gamma_r)).{\bf v}^r=
c_r\cdot \biggl((\epsilon(r) v_H)\wedge \Bigl((\sigma(\gamma_r)\tilde g'(r)^{-1})v'_i\Bigr)\biggr),
\ee
where $\hat g=(g,1)$ for all $g\in G$. Since $G'$ is compact, ~\eqref{eq: unip lin exceptional gamma r 1} implies  
\be\label{eq: unip lin exceptional gamma r 2}
\|\hat u_r\hat g_1(\gamma_r,\sigma(\gamma_r)).{\bf v}^r\|\leq M^{\hat A},\quad\text{for some $\hat A$.}
\ee
Let $z\in\gfrak$ be a vector so that $u_r=\exp(rz)$. 
Using~\eqref{eq: unip lin exceptional gamma r 1} and associativity of the exterior algebra, we have 
\begin{align}
\notag\|\hat z\wedge \Bigl(\hat u_r\hat g_1(\gamma_r,\sigma(\gamma_r)).{\bf v}^r\Bigr)\|&=|c_r|\Bigl\|\Bigl(z\wedge \epsilon(r) v_H\Bigr)\wedge \Bigl((\sigma(\gamma_r)\tilde g'(r)^{-1})v'_i\Bigr)\Bigr\|\\
\label{eq: unip lin exceptional gamma r 3}&\ll M^\star M^{-{\ref{a: linearization 1}}}<\eta^{A}M^{-A\hat{A}}/E_1. 
\end{align}
where we used $\|\epsilon(r)\|\ll M^{-{\ref{a: linearization 1}}}$ in the second to last inequality, $A$ and $E_1$ are as in \cite[Thm.~1.4]{LMMS}, and we choose ${\ref{a: linearization 1}}$ large enough so that the last estimate holds.

In view of~\eqref{eq: unip lin exceptional gamma r 2} and~\eqref{eq: unip lin exceptional gamma r 3}, 
conditions in \cite[Cor.~7.2]{LMMS} are satisfied. 
Hence, there exist $r\in\exceptional$, $\tilde\gamma=(\gamma, \sigma(\gamma))\in \tilde\Gamma$, and a subgroup 
\[
\tilde\H'\subset \tilde\gamma^{-1}\tilde\H^r\tilde\gamma\cap\tilde\H^r
\]
satisfying that $\tilde\H'(\bbc)$ is generated by unipotent subgroups (see~\cite[p.~3]{LMMS}) 
so that for all $r\in[-S,S]$ both of the following hold 
\begin{subequations}
\begin{align}
\label{eq: tube throughout 1}&\|\hat u_r\hat g_1{\bf v}_{\tilde H'}\|\ll M^\star\\
\label{eq: tube throughout 2}&\|\hat z\wedge (\hat u_r\hat g_1{\bf v}_{\tilde H'})\|\ll S^{-1/D}M^\star.
\end{align}
\end{subequations}
Let $\tilde H'=\tilde\H'(\bbr)^\circ$. Since $\|g_1\|\leq M$, applying~\eqref{eq: tube throughout 1} with $r=0$, we get 
\be\label{eq: height of tilde H'}
\|{\bf v}_{\tilde H'}\|\ll M^\star.
\ee

There are two possibilities for $\tilde H'$:

\subsection*{Case 1} $\rho(\tilde H')$ is a conjugate of $H$.  

This in particular implies that
\[
\rho(\tilde H')=g(r_0)^{-1}Hg(r_0)\quad \text{where $r_0\in\exceptional$ is as above.}
\] 
Put $g'=g(r_0)$. Then $\|g'\|\leq M$ and we have 
\be\label{eq: volume of Hg'Gamma}
\begin{aligned}
\vol(Hg'\Gamma/\Gamma)&\ll \|g'\|^\star\vol(g'^{-1}Hg'\Gamma/\Gamma)\\
&\ll M^\star{\rm ht}(\tilde\H')\ll M^\star
\end{aligned}
\ee
where we used Lemma~\ref{lem: volume and height} in the second 
and~\eqref{eq: height of tilde H'} in the last inequality.

Recall that if we choose a maximal compact $K'\subset G$ associated to a Cartan involution which stabilizes $H$, then $G=K\exp(\rfrak')H$, where $\rfrak'=\rfrak\cap(\Lie(K'))^\perp$, see e.g.~\cite[Thm.~A.1]{EMS-Counting}. This and~\eqref{eq: tube throughout 1} imply that for every $r\in[-S,S]$, we may write  
\[
u_rg_1g'v_H=g_r'g'v_H,\quad\text{where $\|g'_r\|\ll M^\star$.}
\]
Since the map $s\mapsto u_sg_1g'v_H$ is a polynomial with coefficients $\ll M^\star$, 
\[
u_sg_1g'v_H=\epsilon'(s,r)g'_rg'v_H\quad\text{where $\|\epsilon'(s,r)\|\ll M^\star\Bigl({|s-r|}/{S}\Bigr)^{\star}$}.
\]   
Using the fact that $d$ is right invariant, the above implies  
\[
\dist(u_sg_1, g'_rg'^{-1}Hg')\ll M^\star\Bigl({|s-r|}/{S}\Bigr)^{\star};
\]
hence part~(1) in the lemma holds if for every $r\in[-S,S]$ we let $g=g'_rg'^{-1}$.  

\subsection*{Case 2}
$\rho(\tilde H')=g'^{-1}Ug'$ where $U=\{u_r\}$. 

First note that if this holds, then $\Gamma$ is a non-uniform lattice and we may identify $\tilde\G$ and $\G$ as $\R$-groups. 
Thus ${\bf v}_{\tilde H'}\in\Lie(G)$, and we have 
\[
\exp({\bf v}_{\tilde H'})\in \tilde H'\cap\Gamma.
\]
Recall from~\S\ref{sec: notation} that $\gfrak=\hfrak\oplus\rfrak$ where $\hfrak=\Lie(H)$ and $\rfrak$ are $3$ and $2\dm+1$ irreducible representation $\Ad(H)$, respectively. Let us write 
\[
g_1{\bf v}_{\tilde H'}= w+w', \quad\text{where $w_1\in\hfrak$ and $w_2\in\rfrak$}. 
\]
Now~\eqref{eq: tube throughout 1} implies that for every $r\in[-S,S]$ we have
\[
\|\Ad(u_r)w\|, \|\Ad(u_r)w'\|\ll M^\star
\] 
Using standard representation theory of $\SL_2(\R)$ (recall that $H\simeq \SL_2(\R)$ or $H\simeq\PSL_2(\R)$), we may write 
$\Ad(u_r)\bullet$ in the basis consisting of weight spaces for $a_t$ and conclude that for all $t\in[\log M,\log S]$, 
\begin{align*}
&\|(\Ad(a_{-t}u_r)w)_i\|\ll M^\star e^{-t} \quad\text{for $i=-1, 0, 1$},\\
& \|(\Ad(a_{-t}u_r)w')_j\|\ll M^\star e^{-\dm t}\quad\text{for $j=-\dm,\ldots, \dm$}
\end{align*}
We apply this with $t=\star\log M$ large enough so that the above implies 
\be\label{eq: a-t ur v tilde H}
\|u_ra_{-t}g_1{\bf v}_{\tilde H'}\|\leq M^{-1}\qquad\text{for all $r\in[-e^{-t}S, e^tS]$}
\ee
Now~\eqref{eq: a-t ur v tilde H} implies that 
\[
u_ra_{-t}g_1\Gamma\not\in X_{1/M} \qquad\text{for all $r\in[-e^{-t}S, e^tS]$}. 
\]
Thus by Theorem~\ref{thm: non-div parabolic}, there exists a $\Q$-parabolic subgroup ${\bf P}\subset\G$ 
satisfying that ${\bf v}_{\tilde H'}\in \Lie(R_u({\bf P}))$ so that 
\[
\|u_ra_{-t}g_1{\bf v}_{R_u(P)}\|\ll M^{-\star}\qquad\text{for all $r\in[-e^{-t}S, e^tS]$}. 
\]
Conjugating back, and using $t=\star\log M$, we get 
\be\label{eq: parabolic at ur}
\|u_rg_1{\bf v}_{R_u(P)}\|\ll M^\star \qquad\text{for all $r\in[-S, S]$}. 
\ee
Arguing as in the previous case, recall also that $G=KP=KLR_u(P)$ where $L$ is a Levi subgroup of $P$, we get that part~(2) in the lemma holds.  
\end{proof}

\subsection{Proof of Theorem~\ref{thm:main unipotent}} 
Let $\ref{a:main-3-1}$ be as Theorem~\ref{thm: main}, and let $\ref{a: linearization 1}$, $\ref{a: linearization 2}$ and $\ref{E: unip lin}$ 
be as in Lemma~\ref{lem: unipotent linearization}. 
By increasing $\ref{a: linearization 1}$ and $\ref{a: linearization 2}$ if necessary, we may assume $\ref{a: linearization 1}, \ref{a: linearization 2}\geq 10\ref{a:main-3-1}$. Let $E'\geq \ref{a: parabolic linear 2}$ and $F'\geq \ref{a: parabolic linear}$, see Theorem~\ref{thm: non-div parabolic}. In creasing $F'$ if necessary, we will assume $F'\geq \dm_0$ in Proposition~\ref{prop: Non-div main}.  
We will now demonstrate that the theorem holds with
\[
\ref{a:unipotent-1}=\star \ref{a: linearization 1}\quad\text{and}\quad \ref{a:unipotent-2}=\ref{a: linearization 2}
\]

Let $C=\max\{E', (10\ref{E:non-div-main})^3, e^{\ref{E:non-div-main}}, e^{s_0}, \ref{c: linear trans}, \ref{E: unip lin}\}$, 
see~\eqref{eq: how big is R xi}. Let $R\geq C^{2}$, and
\[
d=\ref{a: linearization 1}\log R \quad\text{and}\quad \eta=(C/R)^{\frac1{F'\ref{a: linearization 1}}}.
\] 
Let $T>R^{\ref{a:unipotent-1}}$, and put $T_1=e^{-d}T\geq R^{\ref{a: linearization 1}}$. Then 
\begin{multline}\label{eq: proof main unip 1}
\frac{1}{T}\int_0^T\varphi(u_rx_0)\diff\!r=\frac{1}{T_1}\int_0^{T_1}\varphi(a_d u_{r_1}a_{-d}x_0)\diff\!r_1\\
=\frac{1}{T_1}\int_0^{T_1}\!\!\int_0^1\varphi(a_d u_{r}u_{r_1}a_{-d}x_0)\diff\!r\diff\!r_1+O(\|\varphi\|_\infty T_1^{-1})
\end{multline}
where the implied constant is absolute. 

Put $x_1=a_{-d}x_0$, and define 
\begin{subequations}
\begin{align}
\label{eq: main unip exceptional cusp}\exceptional_1&=\{r_1\in[0,T_1]: u_{r_1}x_1\not\in X_{\eta}\}\\
\label{eq: main unip exceptional periodic}\exceptional_2&=\biggl\{r_1\in[0,T_1]: \begin{array}{c}\text{there exists $x$ with $\vol(Hx)\leq R$}\\\text{and }\dist(u_{r_1}x_1,x)\leq R^{\ref{a:main-3-1}}d^{\ref{a:main-3-1}}e^{-d}\end{array}\biggr\}.
\end{align}
\end{subequations}

Let us first assume that 
\be\label{eq: exceptional sets are small}
|\exceptional_1|\leq C\eta^{1/F'}T_1 \quad\text{and}\quad |\exceptional_2|\leq 2C^2R^{-\kappa}T_1,
\ee
where $\kappa=\min\{1/(F'\ref{a: linearization 1}\ref{a: linearization 2})\}$.

For every 
\[
r_1\in[0,T_1]\setminus\Bigl(\exceptional_1\cup\exceptional_2\Bigr),
\]
put $x(r_1)=u_{r_1}x_1$. Then  
\[
R=C\eta^{-F'\ref{a: linearization 1}}\geq C\inj(x(r_1))^{-\dm_0},
\]
see~\eqref{eq: how big is R xi}; moreover, $e^d=R^{\ref{a: linearization 1}}>R^{\ref{a:main-3-1}}$. 
Thus, conditions of Theorem~\ref{thm: main} are satisfied with the parameters $e^d$, $R$, and $x(r_1)$. 
However, by the definition of $\exceptional_2$, part~(2) in Theorem~\ref{thm: main} does not hold with these choices. 
Consequently, we conclude that for every $r_1$ as described above, 
\[
\biggl|\int_0^1\varphi(a_d u_{r}x(r_1))\diff\!r-\int\varphi\diff\!m_X\biggr|\leq \Sob(\varphi)R^{-\ref{k:main-3-1}}.
\]  
Together with~\eqref{eq: exceptional sets are small} and~\eqref{eq: proof main unip 1}, this implies that 
\[
\biggl|\frac{1}{T}\int_0^T\varphi(u_rx_0)\diff\!r-\int\varphi\diff\!m_X\biggr|\leq (R^{-\ref{k:main-3-1}}+3C^2R^{-\kappa}+2T_1^{-1})\Sob(\varphi),
\]
where we used $C\eta^{1/F'}\leq C^2R^{-\kappa}$. 

Hence, part~(1) in Theorem~\ref{thm:main unipotent} holds with 
$\ref{k:main uni}=\min(\ref{k:main-3-1},\kappa)/2$ if we assume $R$ is large enough.

\medskip

We now assume to the contrary that~\eqref{eq: exceptional sets are small} fails:

\subsection*{Assume that $|\exceptional_1|> C\eta^{1/F'}T_1$.}  
We will show that part~(3) in the theorem holds under this condition.

First note that under this condition $\Gamma$ is non-uniform, thus, $\tilde\G$ may be identified with $\G$ as $\R$ groups. 
Let us write $x_0=g_0\Gamma$. Since $|\exceptional_1|> C\eta^{1/F'}T_1$, then in view of our choices of $F'$ and $C$, 
we conclude from Theorem~\ref{thm: non-div parabolic} that there exists a $\Q$-parabolic subgroup 
${\bf P}\subset\G$ so that 
\[
\|u_ra_{-d}g_0{\bf v}_{R_u(P)}\|\leq \ref{a: parabolic linear 2} \eta^{1/\ref{a: parabolic linear}}\qquad\text{for all $r\in[-T_1, T_1]$}. 
\]
Conjugating back with $a_d$ and using $T=e^{d}T_1$ and $e^d=R^{\ref{a: linearization 1}}$, 
\[
\|u_rg_0{\bf v}_{R_u(P)}\|\leq \ref{a: parabolic linear 2} R^{\star \ref{a: linearization 1}} \qquad\text{for all $r\in[-T, T]$}. 
\]
Arguing as in Case 1 (or Case 2) of the proof of Lemma~\ref{lem: unipotent linearization}, 
we get that part~(3) holds with $\ref{a:unipotent-1}=\star \ref{a: linearization 1}$ and $\ref{a:unipotent-2}=\ref{a: linearization 2}$.

\subsection*{Assume that $|\exceptional_2|> 2C^2R^{-\kappa}T_1$}
If $|\exceptional_1|> C\eta^{1/F'}T_1$, then part~(3) in the theorem holds as we just discussed. 
Thus, we may assume that 
\[
|\exceptional_2|> 2C^2R^{-\kappa}T_1\quad\text{and}\quad |\exceptional_1|\leq  C\eta^{1/F'}T_1.
\] 
Put $\exceptional':=\exceptional_2\setminus \exceptional_1$. Then 
\[
\exceptional'=\biggl\{r_1\in[0,T_1]: \begin{array}{c}\text{$u_{r_1}x_1\in X_\eta$ and there exists $x$ with}\\ \text{$\vol(Hx)\leq R$ 
and $\dist(u_{r_1}x_1,x)\leq R^{\ref{a:main-3-1}}d^{\ref{a:main-3-1}}e^{-d}$}\end{array}\biggr\},
\]
and $|\exceptional'|\geq C^2R^{-\kappa}T_1\geq \ref{E: unip lin}\eta^{1/\ref{a: linearization 2}}T_1$.
Moreover, for $R$ large enough, 
\[
R^{\ref{a:main-3-1}}d^{\ref{a:main-3-1}}e^{-d}=R^{\ref{a:main-3-1}}(\ref{a: linearization 1}\log R)^{\ref{a:main-3-1}}R^{-2\ref{a: linearization 1}}\leq R^{-\ref{a: linearization 1}}.
\]

Fix some $r_1\in\exceptional'$ for the rest of the argument. Put 
\[
\text{$x_2=u_{r_1}x_1=u_{r_0}a_{-d}x_0\;\;$ and $\;\;\exceptional=\exceptional'-r_1 \subset [-T_1,T_1]$}.
\]
Then the conditions in Lemma~\ref{lem: unipotent linearization} are satisfied with $x_2$,
$\exceptional$, $\eta$, $M=R$, and $S=T_1=R^{-\ref{a: linearization 1}}T$. 

First, assume that part~(1) of Lemma~\ref{lem: unipotent linearization} holds.
Then there exists $x \in G/\Gamma$ with $\vol(H.x)\leq R^{\ref{a: linearization 1}}$. 
Moreover, for every $r\in [-T_1,T_1]$ there exists $g\in G$ with $\|g\|\leq R^{\ref{a: linearization 1}}$ so that  
\[
\dist_X(u_{s}x_2, gHx)\leq R^{\ref{a: linearization 1}}\left(\frac{|s-r|}{T_1}\right)^{1/\ref{a: linearization 2}}\quad\text{for all $s\in[-T_1,T_1]$.}
\] 
Since $s-r_1, r-r_1\in[-T_1,T_1]$ for all $s,r\in[0,T_1]$, the above implies 
\begin{align*}
\dist_X(u_{e^ds}x_0, a_dgHx)&=\dist_X(a_du_sa_{-d}x_0, a_dgHx)\\
&=\dist_X(a_du_{s-r_1}u_{r_1}a_{-d}x_0, a_dgHx)\\
&=\dist_X(a_du_{s-r_1}x_2, a_dgHx)\\
&\ll e^{\star d}\dist_X(u_{s-r_1}x_2, gHx)\leq R^{\star \ref{a: linearization 1}}\biggl(\frac{|e^ds-e^dr|}{T}\biggr)^{1/\ref{a: linearization 2}}.
\end{align*}
That is part~(2) in the theorem holds with $\ref{a:unipotent-1}=\star \ref{a: linearization 1}$ and $\ref{a:unipotent-2}=\ref{a: linearization 2}$ for all large enough $R$.  

Assume now that part~(2) in Lemma~\ref{lem: unipotent linearization} holds. Then arguing as above, we conclude that 
part~(3) of the theorem with $\ref{a:unipotent-1}=\star \ref{a: linearization 1}$ and $\ref{a:unipotent-2}=\ref{a: linearization 2}$. 
\qed


\section{Equidistribution of expanding circles}\label{sec: exp circles}
In this section, we record the following theorem concerning equidistribution of large circle; this theorem is a corollary of Theorem~\ref{thm: main} as it was shown in~\cite[\S 5]{LMW23}. 

We keep the notation from Theorem~\ref{thm: main}. In particular, $G$ is any of the following groups 
\[
\SL_3(\R),\quad {\rm SU}(2,1), \quad {\rm Sp}_4(\R), \quad {\bf G}_2(\R),
\]
and $H\subset G$ is the image of the principal $\SL_2(\R)$ in $G$. 
For all $t,r\in\R$ and all $\theta\in[0,2\pi]$, let $a_t$, $u_r$ and $\rot_\theta$ denote the images of 
\[
\begin{pmatrix}
    e^{t/2} & 0 \\
    0 & e^{-t/2}
\end{pmatrix},
\quad \begin{pmatrix}
    1 & r \\
    0 & 1 
\end{pmatrix}, 
\quad\text{and}\quad
\begin{pmatrix}
    \cos\theta & -\sin\theta \\
    \sin\theta & \cos\theta 
\end{pmatrix}
\]
in $H$, respectively. 

We let $\Gamma\subset G$ be an arithmetic lattice, and let $m_X$ denote the probability Haar measure on $X=G/\Gamma$. 

\begin{thm}\label{thm: equid circles}
For every $x_0\in X$, and large enough $R$ (depending explicitly on $x_0$), for any $T\geq R^{\ref{a:circle}}$, at least one of the following holds.
\begin{enumerate}
\item For every $\varphi\in C_c^\infty(X)$ and $2\pi$-periodic smooth function $\xi$ on $\R$, we have 
\[
\biggl|\ave \varphi(a_{\log T}\rot_\theta x_0)\xi(\theta)\diff\!\theta-\int_0^{2\pi}\xi(\theta)\diff\!\theta \int \varphi\diff\!m_X\biggr|\leq \Sob(\varphi)\Sob(\xi)R^{-\ref{k:circle}}
\]
where $\Sob(\cdot)$ denotes appropriate Sobolev norms on $X$ and $\R$, respectively. 
\item There exists $x\in X$ such that $Hx$ is periodic with $\vol(Hx)\leq R$, and 
\[
\dist_X(x,x_0)\leq R^{\ref{a:circle}}(\log T)^{\ref{a:circle}}T^{-\dm}.
\] 
\end{enumerate} 
The constants $\consta\label{a:circle}$ and $\constk\label{k:circle}$ are positive, and depend on $X$ but not on $x_0$.
\end{thm}

\begin{proof}
The proof of~\cite[Thm.1.4]{LMW23} remains valid without modification when \cite[Thm.1.1]{LMW22} is replaced by Theorem~\ref{thm: main} from this paper.
\end{proof}


\part{Application to quadratic forms}

The second part of this article concerns applications to quantitative version of Oppenheim conjecture. In particular, the proof of Theorem~\ref{thm: quantitative Oppenheim} and Theorem~\ref{thm: Oppenheim} will be completed in this part. We employ the strategy pioneered by Eskin, Margulis and Mozes in \cite{EMM-Upp}. This result covered all indefinite quadratic forms in $d\geq 3$ variables except forms of signature $(2,2)$, or $(2,1)$ where a Diophantine condition is needed for the requisite nondivergence results to be true.

The case of $(2,2)$ forms was treated by Eskin, Margulis and Mozes \cite{EMM-22}, and required significant additional ideas. A more precise version, suitable for a quantitative counting result, for the special case of $\Gamma\subset \SL_2(\Z)\times \SL_2(\Z)$ was given by the first three authors of this paper in \cite{LMW23}.

More recently Wooyeon Kim succeeded in giving a treatment of the case of forms of signature $(2,1)$ in \cite{Kim-Oppenheim}. Since the actual results we need are only implicit in \cite{Kim-Oppenheim}, we give a full treatment here; but it is possible to isolate what is needed already from~\cite{Kim-Oppenheim}. 

\section{Upper bound estimates}\label{sec: Oppenheim}

In this section we will state Proposition~\ref{prop: L1+epsilon upp bd} which yields the upper bound estimates needed for the proof of Theorem~\ref{thm: quantitative Oppenheim}.

Let
\[
\sqf(\cox,\coy,\coz)=2\cox\coz-\coy^2.
\]
If $Q$ is an indefinite ternary quadratic form of determinant 1, and
Then there exists some $g_Q\in\SL_3(\R)$ so that $Q(v)=\sqf(g_Qv)$ for all $v\in\R^3$.

Let $H=\SO(\sqf)^\circ\subset \SL_3(\R)=G$; 
note that $H\simeq\PSL_2(\R)$. 
In the case at hand, the groups $a_t$ and $u_r$ featuring in Theorem~\ref{thm: main} can be more explicitly described as follows  
\[
a_t=\begin{pmatrix}
    e^t & 0 & 0\\
    0 & 1 & 0\\
    0 & 0 & e^{-t}
\end{pmatrix}
\quad\text{and}\quad u_r=\begin{pmatrix}
    1 & r & \frac{r^2}{2}\\
    0 & 1 & r\\
    0 & 0 & 1
\end{pmatrix}.
\]
Also let $K=H\cap\SO(3)\simeq\SO(2)$. Indeed, $K=\{\rot_\theta: \theta\in[0,2\pi]\}$ where

\be\label{eq: r-theta Oppenheim}
\rot_\theta=\begin{pmatrix}
    \tfrac{1+\cos\theta}2 & -\tfrac{\sin\theta}{\sqrt 2} & \tfrac{1-\cos\theta}{2}\\
    \tfrac{\sin\theta}{\sqrt 2} & \cos\theta & -\tfrac{\sin\theta}{\sqrt 2}\\
    \tfrac{1-\cos\theta}2 & \tfrac{\sin\theta}{\sqrt 2} & \tfrac{1+\cos\theta}2
    \end{pmatrix}.
\ee

\begin{defin}\label{def:exceptional line}
    A rational line $L\subset\R^3$ will be said to be $(\rho, A, t)$-\emph{exceptional} for $Q$ if the vector $v$ defined up to sign by $L\cap\Z^3={\rm Span}\{v\}$ satisfies  
\[
\|v\|\leq e^{\rho t}\quad\text{and}\quad |Q(v)|\leq e^{-A\rho t}. 
\]
\end{defin}
Existence of 5 exceptional lines yields an approximation of $Q$ with a rational form with low complexity, see Lemma~\ref{lem: five vectors} below. 

We identify $\wedge^2\R^3$ with the dual of $\R^3$, which will be identified with $\R^3$ using the standard Euclidean inner product. 
Then for every $g\in G$
\be\label{eq: action on wedge 2'}
gv_1\wedge gv_2= g^* (v_1\wedge v_2)
\ee
where $g^*=(g^t)^{-1}$. We put $Q^*(v):=\sqf(g^* v)$ for any $v\in\R^3$.

Similar to Definition~\ref{def:exceptional line}, we have the following:
\begin{defin}
    Let $P\subset \R^3$ be a rational plane and suppose $P\cap\Z^3={\rm Span}\{w_1, w_2\}$. Then $P$ will be said to be $(\rho, A, t)$-\emph{exceptional} if the line spanned by $w_1\wedge w_2$ is $(\rho, A, t)$-exceptional for $Q^*$.  
\end{defin}

\noindent
This definition can be equivalently stated in terms of the height of the plane $P$ and the determinant of the rational form induced by $Q$ on $P$.

\medskip

Let $f\in C_c(\R^3)$. For every $h\in H$, define 
\be\label{eq: def tilde f}
\tilde f_{\rho, A, t}(h;g_Q)=\sum_{v\in\mathcal N_{Q,t}}f(hv)
\ee
where $\mathcal N_{Q,t}$ denotes the set of vectors in $\Z^3$ which are not contained in any $(\rho, A, t)$-exceptional line or plane.

\begin{propos}\label{prop: L1+epsilon upp bd}
   For all large enough $A$, for all small enough $\rho$ (depending on $A$), and all large enough $t$ (depending on $A$ and $\rho$), at least one of the following holds:
\begin{enumerate} 
\item Let $\tilde{\mathcal C}_t=\{k\in K: \tilde f_{\rho, A, t}(a_tk; g_Q\Gamma)\geq Ae^{A\rho t}\}$. Then 
\[
\int_{\tilde{\mathcal C}_t}\tilde f_{\rho, A, t}(a_tk;g_Q\Gamma)\diff\!k\ll e^{-\rho^2 t/2}.
\]

\item There exists $P\in\Mat_3(\Z)$ with $\norm{P}\leq e^{D_0\rho t}$ so that
   \[
   \norm{Q-\lambda P}\leq \norm{P}^{-A/2}\quad\text{where $\lambda=(\det P)^{-1/3}$}
   \] 
\end{enumerate}
The implied constant depends on $\norm{g}$ and $D_0$ is absolute.   
\end{propos}

The reader may compare Proposition~\ref{prop: L1+epsilon upp bd} and its proof with~\cite[Prop.~6.1]{LMW23}. 
It is also worth noting that Proposition~\ref{prop: L1+epsilon upp bd} is equivalent to showing that $\tilde f$ belongs to $L^{1+\epsilon}$. 
The proof of this proposition will occupy the rest of this section. 

Let $A$ and $t$ be large parameters ($A$ is absolute and $t>t_0(Q)$), and let $\rho>0$ a small parameter; these will be explicated later. 
In what follows $\eta$ represents a small number. 

\subsection*{Linear algebra lemmas}
We begin with the following lemmas from linear algebra.

\begin{lemma}\label{lem: five vectors}
Let $A, D\geq 1$ be two parameters. Let $v_1,\ldots, v_5\in\Z^3$ be five vectors so that no three of them are co-planar and $|Q(v_i)|\leq \eta^A$ for all $i$. Then the following hold:
\begin{enumerate}[label=\textup{(L-\arabic*)}]
    \item \label{item: P-mat-1}
    Let $\gamma=(v_1, v_2, v_3)$, 
       and assume that 
       \[
       |\det\gamma|\leq \eta^{-D}\quad\text{and}\quad \|\gamma^{-1}v_i\|\leq \eta^{-D}\;\text{ for $i=4,5$}. 
       \]
Then there is a matrix 
$ P\in \Mat_3(\Z)$ with $\det(P)\neq 0$ and $ \|P\|\ll \eta^{-\star D}$
so that 
\[
\norm{\gamma^tQ\gamma - P} \ll \eta^{A-\star D}. 
\]
\item\label{item: P-mat-2} If $\norm{v_i}\leq \eta^{-D}$ for all $i$, then there is a matrix $P$ as in \ref{item: P-mat-1} so that
\[
\norm{Q-\lambda P}\ll \eta^{A-\star D}\quad\text{where $\lambda=(\det P)^{-1/3}$.}
\]
\end{enumerate}
   \end{lemma}

\begin{proof}
Writing $Q$ in the basis $\{v_1,v_2, v_3\}$ implies that 
\[
\gamma^tQ\gamma= B=(b_{ij})
\]
where $|b_{ii}|\ll \eta^A$. 
Write
\begin{align*}
\gamma^{-1}v_4&=(a_1,a_2,a_3)^t\\
\gamma^{-1}v_5&=(a'_1,a'_2,a'_3)^t.
\end{align*}
Then $a_i, a_i'\in\Q$ have height $\leq \eta^{-D}$.

Note that since no three of the $v_i$'s are co-planar, $a_i, a'_i$ are all nonzero.
Suppose $(a_2a_3,a_1a_3,a_1a_2)$ and $(a_2'a_3',a_1'a_3',a_1'a_2')$ are colinear.
Then $\frac{a_2a_3}{a_2'a_3'}=\frac{a_1a_2}{a_1'a_2'}$ so $\frac{a_3}{a_3'}=\frac{a_1}{a_1'}$ and similarly for the other pairs. But this means $v_4$ and $v_5$ are colinear, in contradiction.

Since for $i=4,5$, \ $|Q(v_i)|\leq \eta^A$, we have 
\[
\sum_{0<i<j\leq 3}b_{ij}a_ia_j \ll \eta^{A-2D}\quad\text{and}\quad\sum_{0<i<j\leq 3}b_{ij}a_i'a_j' \ll \eta^{A-2D}
\]
Since $a_i,a'_i\in\Q$ have height $\leq \eta^{- D}$, the above imply \ref{item: P-mat-1}.

The second claim~\ref{item: P-mat-2} follows from \ref{item: P-mat-1} and $\|\gamma\|\ll \eta^{-D}$.
\end{proof}

\begin{lemma}\label{lem: linear algebra}
Let $0<\sigma<1$, and let $0< \delta<1-\sigma$. For every $v\in\R^3$ which satisfies $|\sqf(v)|\geq e^{-2\sigma t}$, we have 
\[
\int_K\norm{a_tk v}^{-1-\delta}\diff\!k\ll \delta^{-1} e^{(-1+\delta+\sigma) t}\norm{v}^{-1-\delta},
\]
where the implied constant is absolute. 
\end{lemma}

\begin{proof}
Recall that $K$ acts transitively on the level sets 
\[
\{v\in\R^3: \norm{v}=c_1, \sqf(v)=c_2\}.
\]
Thus we may assume $v=(0,\vare, 1)$ where $\vare=|\sqf(v)|^{1/2}\geq e^{-\sigma t}$, and have  
    \[
    a_t\rot_\theta v=\Bigl(e^t(\tfrac{1-\cos\theta-\sqrt2\vare \sin\theta}{2}),\tfrac{2\vare\cos\theta-\sqrt2\sin\theta}{2}, e^{-t}(\tfrac{1+\cos\theta+\sqrt2\vare\sin\theta}{2})\Bigr).
    \]
    Since for $0<\theta<1/100$, we have 
    \begin{align*}
    &1-\cos\theta-\sqrt2\vare \sin\theta=-\sqrt2\vare\theta+\tfrac{\theta^2}2+ O(\theta^3)\quad\text{and}\\
    &1+\cos\theta+\sqrt2\vare\sin\theta=2+O(\vare).
    \end{align*}
    Let $\vare'=\min\{\vare, 1/100\}$, the above implies 
    \begin{align*}
    \int_K\|a_tk v\|^{-1-\delta}\diff\!k&\ll \int_{0}^{e^{(-2+\sigma)t}} e^{(1+\delta)t}\diff\!\theta+ \int_{e^{(-2+\sigma)t}}^{\vare'} e^{(-1-\delta)t}\vare^{-1}\theta^{-1-\delta}\diff\!\theta\\
    &\ll e^{(-1+\sigma+\delta)t}+ \delta^{-1}e^{(-1-\delta)t}e^{\sigma t}\theta^{-\delta}\biggr]_{e^{(-2+\sigma)t}}^{\vare'}\\
    &\ll \delta^{-1} e^{(-1+\delta+\sigma)t}.
    \end{align*}
    As we claimed. 
\end{proof}

In the bounds in Lemma~\ref{lem: linear algebra}, the multiplicative coefficient in front of $\norm v ^{-1-\delta}$ tends to zero as $t\to\infty$ (which is what we want), at the expense of giving a restriction on $v$. 
Without the restrictions imposed in Lemma~\ref{lem: linear algebra}, we have the following general (weaker) bound:

\begin{lemma}\label{lem: linear algebra 2}
Let $v\in\R^3$, then
\[
\int_K\|a_tk v\|^{-1-\delta}\diff\!k\ll e^{\delta t}\|v\|^{-1-\delta}
\]
where the implied constant is absolute. 
\end{lemma}

\begin{proof}
    The proof is similar to Lemma~\ref{lem: linear algebra}. Indeed, we have
    \begin{align*}
    \int_K\|a_tk v\|^{-1-\delta}\diff\!k&\ll \int_{0}^{e^{-t}} e^{(1+\delta)t}\diff\!\theta+ \int_{e^{-t}}^{\pi/4} e^{(-1-\delta)t}\theta^{-2-2\delta}\diff\!\theta\\
    &\ll e^{\delta t}\norm{v}^{-1-\delta},
    \end{align*}
    as we claimed. 
\end{proof}

\subsection*{Cusp functions of Margulis}
Recall the functions
\begin{align*}
&\alpha_1(g)=\{1/\norm{gv}: 0\neq v\in\Z^3\}\\
&\alpha_2(g)=\Bigl\{1/\norm{g^*(w_1\wedge w_2)}: 0\neq w_1\wedge w_2\in\wedge^2\Z^3\Bigr\}
\end{align*}
from~\cite{EMM-Upp} and~\cite{EMM-22}.

For all $0\neq v\in \R^3$, all $0<\sigma<1$, and all $\sfd>1$ let 
\[
I_v^\sfd (\sigma)=\{k\in K: \|a_\sfd kg_Qv\|\leq \sigma\}.
\]
We will be interested in $I_v^\sfd(\sigma)$ for vectors where $\absolute{\sqf(g_Qv)}\leq \sigma^{10}$ and $1\ll \|g_Qv\|\leq \sigma e^{\sfd}$. In this range, we have 
\be\label{eq: size of Iv}
|I_v^\sfd(\sigma)|\asymp(e^{-\sfd}\sigma/\|g_Qv\|)^{1/2},
\ee
see e.g.~\cite[Lemma A.6]{EMM-22}. 

Let $\mathcal P$ denote the set of primitive vectors in $\Z^3$, and let 
\[
\mathcal P_\ell=\{v\in\mathcal P: e^{\ell-1}\leq \|v\|<e^{\ell}\},
\]
for every integer $\ell\geq 1$. 

We make the following observation, which is a special case of~\cite[Lemma 5.6]{EMM-Upp} tailored to our applications here. 

\begin{lemma}\label{lem: plane and vectors}
Let $0<\sigma<1$ and $\sfd>1$. Let $v_1, v_2\in\mathcal P_\ell$ be so that $\absolute{\sqf(g_Qv_i)}\leq \sigma^{10}$, $\|g_Qv_i\|\leq \sigma e^{\sfd}$, and $2I_{v_2}^\sfd(\sigma)\cap 2I_{v_1}^\sfd(\sigma)\neq \emptyset $. Then 
\[
\alpha_1(a_\sfd kg_Q\Gamma)\ll_c \sigma\alpha_2(a_\sfd kg_Q)\quad\text{for all $k\in cI_{v_1}^\sfd(\sigma)$},
\]
for every constant $c$.
\end{lemma}

\begin{proof}
Let $v_1$ and $v_2$ be as in the statement. Then, in view of~\eqref{eq: size of Iv}, we have 
\[
\norm{a_\sfd kg_Qv_2}\ll \sigma\quad \text{for all $k\in cI^s_{v_1}(\sigma)$}.  
\]
For every $k\in cI_{v_1}^s(\sigma)$, fix a vector $v_k$ so that 
\[
\alpha_1(a_\sfd kg_Q\Gamma)=\norm{a_\sfd kg_Qv_k}^{-1}.
\]  
Let $P={\rm Span}\{v_k, v\}$ where $v=v_1$ if $v_k=v_2$ and $v=v_2$ if $v_k\neq v_2$.
Let $u$ be the corresponding covector (which is unique up to sign). 
Then 
\begin{align*}
\norm{(a_\sfd kg_Q)^*u}&\leq \|(a_\sfd kg_Q)^*(v_k\wedge v)\|\\
&\leq\norm{a_\sfd kg_Qv_k}\norm{a_\sfd kg_Qv}\ll \sigma\alpha_1(a_\sfd kg_Q)^{-1},
\end{align*}
for all $k\in cI_{v_1}^s(\sigma)$. Taking reciprocal, we have 
\[
\alpha_1(a_\sfd kg_Q\Gamma)\ll \sigma\norm{(a_\sfd kg_Q)^*u}^{-1}\leq \sigma\alpha_2(a_\sfd kg_Q\Gamma),
\]
as it was claimed.
\end{proof}

We will also using the following theorem which is due to Eskin, Margulis and Mozes~\cite{EMM-Upp}.

\begin{thm}[\cite{EMM-Upp}, Thm.~3.2 and Thm.~3.3]
\label{thm: EMM-alpha bound}
Both of the following hold for $i=1,2$ and every $g\in G$.
\begin{enumerate}
\item For every $0<p<1$, we have 
\[
\sup_{s\geq 0}\int \alpha_i^p(a_s kg\Gamma)\diff\!k\ll1 
\]
\item For every $s\geq 0$, we have  
\[
\int \alpha_i(a_s kg\Gamma)\diff\!k\ll s. 
\]
\end{enumerate}
The implied constant is $\ll(1-p)^{-1}\norm{g}^\star$ in part~(1) and $\ll \norm{g}^\star$ in~(2). 
\end{thm}

\begin{proof}
The first statement is \cite[Thm 3.2]{EMM-Upp} and the second is \cite[Thm 3.3]{EMM-Upp}. 
It follows from the proofs of these statements that the implied constant depends polynomially on $\alpha_i(g)\ll\norm{g}^\star$ and $is \ll 1/(1-p)$ in the first case; see~\cite[Lemma 5.1]{EMM-Upp} for the dependence on $p$. 
\end{proof}

We will also use the following elementary consequence of Theorem~\ref{thm: EMM-alpha bound}.    

\begin{coro}\label{cor: Lp functions}
Let $\Theta\subset K$ be a set with $|\Theta|\leq \sigma$. Then 
\[
\int_\Theta\alpha_i^{1/2}(a_dkg\Gamma)\diff\!k\ll \sigma^{1/4}
\] 
\end{coro}

\begin{proof}
 Let $\mathcal B=\{k\in K: \alpha_2(a_dkg_Q\Gamma)\leq \sigma^{-1}\}$. Then  
\begin{align*}
\int_{\Theta}\alpha_i^{1/2}(a_{d}kg\Gamma)\diff\!k&\ll \int_{\Theta\cap\mathcal B}\alpha_i^{1/2}(a_dkg\Gamma)\diff\!k+ \sigma^{1/4} \int_{K\setminus \mathcal B} \alpha_i(a_tkg\Gamma)^{3/4}\diff\!k\\
&\ll \sigma^{1/2}+\sigma^{1/4}\ll \sigma^{1/4}
\end{align*}
as we claimed. 
\end{proof}


\subsection*{Approximation by rational forms}

The following proposition is one of the main ingredients in the proof of Proposition~\ref{prop: L1+epsilon upp bd}.

\begin{propos}\label{prop: main uppbd prop}
    There exists some $D_1, A$ (absolute) so that if $\rho D_1<\tfrac1{100}$, and $\ell>0$ is large enough and the following is satisfied  
\be\label{eq: many almost null}
\#\{v\in\mathcal P_\ell: |\sqf(g_Qv)|\leq e^{-A\rho \ell} \}\geq e^{(1-\rho)\ell},
\ee
    then at least one of the following holds:
\begin{enumerate}
    \item There exist planes $\{P_i: 1\leq i\leq N\}$, for some $N\leq e^{(1-2\rho)\ell}$, so that 
    \[
   \#\{v\in\mathcal P_\ell\setminus(\cup P_i) : |\sqf(g_Qv)|\leq e^{-A\rho \ell} \}\leq e^{(1-2\rho)\ell}.
    \]
    Moreover, there are intervals $\{J_i: 1\leq i\leq N\}$, with multiplicity at most two, so that 
    $I_{v}^{(1+2\rho)\ell}(e^{-\rho \ell})\cap J_i\neq\emptyset$ for every $v\in P_i$, and 
    \[
    \int_{cJ_i}\alpha_1(a_{(1+2\rho)\ell}kg_Q\Gamma)\diff\!k\ll_c e^{-\rho \ell}\int_{ cJ_i}\alpha_2(a_{(1+2\rho)\ell}kg_Q\Gamma)\diff\!k
    \]
    for every constant $c$.
    \item There exists some $P\in\Mat_3(\Z)$ with $\|P\|\leq e^{D_1\rho \ell}$ so that  
   \[
   \|Q-\lambda P\|\leq e^{(-1+D_1\rho)\ell}\quad\text{where $\lambda=(\det P)^{-1/3}$}.
   \] 
\end{enumerate}
\end{propos}

Until the conclusion of the proof of Proposition~\ref{prop: main uppbd prop}, let $\eta=e^{-\rho \ell}$ and $\sfd=(1+2\rho)\ell$, and write $I_v$ for $I_{v}^\sfd(\eta)$. We will also put 
\[
\mathcal W_\ell:=\{v\in\mathcal P_\ell : |\sqf(g_Qv)|\leq e^{-A\rho \ell} \}
\]

\begin{lemma}\label{lem: separation}
    Let the notation be as in Proposition~\ref{prop: main uppbd prop}. 
    Then at least one of the following holds.
    \begin{enumerate}
    \item Part~(1) in Proposition~\ref{prop: main uppbd prop} holds. 
        \item There exist $\mathcal W\subset \mathcal W_\ell$ with $\#\mathcal W\geq  \eta^2 e^\ell$ so that 
\[
I_{v}\cap I_{v'}=\emptyset\quad \text{for all $v\neq v'\in\mathcal W$.}
\]
    \end{enumerate}
    \end{lemma}

\begin{proof}
Suppose the claim in part~(2) fails. Recall from~\eqref{eq: size of Iv} that 
\be\label{eq: eq: size of Iv'}
|I_v|\asymp e^{(-\sfd-\ell)/2}\eta^{1/2}\asymp e^{-\ell}\eta^{3/2},
\ee
where we used $\sfd=(1+2\rho)\ell$, $\eta=e^{-\rho \ell}$, and $\|g_Qv\|\ll \|g_Q\|\|v\|\leq \eta e^{s}$.

In consequence, we may choose $\mathcal W'\subset\mathcal W_\ell$ with $\#\mathcal W'<\eta^{2}e^\ell$ so that $\{J_i\}=\{2I_v: v\in\mathcal W'\}$ covers $\cup_{\mathcal W_\ell} I_v$ with multiplicity at most two. Discard the intervals $J_i$ which intersect only one $I_v$. 
For every remaining interval, $J_i$, we have
\be\label{eq: finding the planes Pi}
\|a_{\sfd}k g_Qv\|\leq \eta\quad\text{for all $v$ with $I_v\cap J_i\neq \emptyset$ and all $k\in I_v$}.
\ee
We draw both conclusions of part~(1) in the lemma from~\eqref{eq: finding the planes Pi}. 
First note that~\eqref{eq: finding the planes Pi} implies that if $v,v'$ so that $I_{v}\cap J_i\neq \emptyset$ and $I_{v'}\cap J_i\neq \emptyset$, and we put $P_i={\rm Span}\{v, v'\}$, then for every $v''$ so that $I_{v''}\cap J_i\neq \emptyset$, $v''\in P_i$. Otherwise 
\[
1\leq |\det(a_{\sfd}k g_Qv, a_{\sfd}k g_Qv', a_{\sfd}k g_Qv''|\ll \eta^3
\]
which is a contradiction assuming $s$ is large enough. Moreover, 
\[
\#\{v: v\not\in\cup P_i\}\leq \eta^{2}e^\ell.
\]

To see the second claim in part~(1), note that if $v,v'$ are so that $I_{v}\cap J_i\neq \emptyset$ and $I_{v'}\cap J_i\neq \emptyset$, the conditions in Lemma~\ref{eq: finding the planes Pi} are satisfied. Hence
\[
 \int_{cJ_i}\alpha_1(a_{\sfd}kg_Q\Gamma)\diff\!k\ll_c e^{-\rho s} \int_{cJ_i}\alpha_2(a_{\sfd}kg_Q\Gamma)\diff\!k,
\]
as we claimed.
\end{proof}

If part~(1) in Lemma~\ref{lem: separation} holds, the proof of Proposition~\ref{prop: main uppbd prop} is complete. Thus We will assume part~(2) in Lemma~\ref{lem: separation} holds, for some $C$ which will be optimized later, for the rest of the proof.

Recall that $K\simeq\SO(2)$. 
Using this isomorphism, we identify $K$ with $[0,2\pi]$, 
and cover $K$ with half open intervals $\{I_j\}$ where $|I_j|=2\pi/N$ for $N=\lceil \eta^{5}e^{\ell}\rceil$.

\begin{lemma}\label{lem: choose many 5 tuples}
   Let $\ell$ and $\mathcal W$ be as in part~(2) in Lemma~\ref{lem: separation}.   
   For every $j$, let 
   \[
   \mathcal W_{\ell, j}=\{v\in\mathcal W_\ell: |I_v\cap I_j|\geq |I_v|/2\},
   \]
   and $\mathcal J_\ell=\{j: \#\mathcal W_{\ell, j}\geq \eta^{-2}\}$. Then 
   \[
   \#\mathcal J_\ell\geq \eta^{9}e^{\ell}
   \]
\end{lemma}

\begin{proof}
Let $\mathcal J'_\ell=\{j: \#\mathcal W_{\ell, j}\leq \eta^{-2}\}$. Then  
\[
\#\Bigl(\textstyle\bigcup_{\mathcal J'}\mathcal W_{\ell, j}\Bigr) \ll \eta^{3}e^\ell.
\]
Moreover, note that since $I_v\cap I_{v'}=\emptyset$ for all $v\neq v'\in\mathcal W$ and $|I_v|\asymp e^{-\ell}\eta^{3/2}$,  
\[
\#\mathcal W_{\ell, j}\ll |I_j|/\eta^{3/2} e^{-\ell}\ll \eta^{-7}. 
\]
These and $\#\mathcal W\geq \eta^2e^{\ell}$ imply that 
\[
\#\mathcal J_\ell =\#\{j: \#\mathcal W_{\ell, j}\geq \eta^{-1}\}\geq \eta^{9}e^\ell,
\]
as we claimed in the lemma.
\end{proof}

\begin{lemma}\label{lem: the matrices for 5 vectors}
    With the notation be as in Lemma~\ref{lem: choose many 5 tuples}, let $j\in\mathcal J_\ell$ and let  $\{v_1,\ldots, v_5\}\subset \mathcal W_{\ell, j}$ be so that $\Bigl\|\frac{v_\alpha}{\|v_\alpha\|}-\frac{v_\beta}{\|v_\beta\|}\Bigr\|\gg \eta^{-1/2}e^{-\ell}$. Then 
    \begin{enumerate}
        \item No three of $\{v_1,\ldots, v_5\}$ are co-planar.
        \item Let $\gamma=(v_1, v_2, v_3)$, then 
        \[
       1\leq |\det\gamma|\leq \eta^{-D}\quad\text{and}\quad \|\gamma^{-1}v_i\|\leq \eta^{-D}\;\text{ for $i=4,5$}. 
       \]
    \end{enumerate}
\end{lemma}

\begin{proof}
In the course of the proof, we will write $g$ for $g_Q$ to simplify the notation.
Let us write $gv_i=v'_i+u_i$ where $\|u_i\|\ll e^{-\ell}\eta^A$ and $\sqf(v'_i)=0$. Now for all $v_\alpha, v_\beta\in\{v_i\}$, we have 
\be\label{eq: sqf v and v'}
\sqf(gv_\alpha, gv_\beta)= \sqf(v'_\alpha,v'_\beta )+O(\eta^A). 
\ee
Moreover, since $\Bigl\|\frac{v_\alpha}{\|v_\alpha\|}-\frac{v_\beta}{\|v_\beta\|}\Bigr\|\gg \eta^{-1/2}e^{-\ell}$ and $v_\alpha, v_\beta\in \mathcal W_{\ell, j}$ we can write
$v'_\beta=\lambda_{\alpha\beta} v'_\alpha+w'_{\alpha\beta}$ where $|\lambda_{\alpha\beta}|=O(1)$ and  $\eta^{-1/2}\ll\|w_{\alpha\beta}\|\ll \eta^{-5}$ and $w'_{\alpha\beta}\perp v'_\alpha$ (with respect to the usual inner product). Hence,    
\[
gv_\beta=\lambda_{\alpha\beta} gv_\alpha+w_{\alpha\beta}\quad\text{where $w_{\alpha\beta}=w'_{\alpha\beta}+O(\eta^A)$.}
\]
More explicit, applying an element of $K$, we will assume $v'_\alpha=(\cox,0,0)$. Then $w'_{\alpha\beta}=(0, \coy,\coz)$ and $\eta^{-1/2}\ll\|w'_{\alpha\beta}\|\ll \eta^{-5}$. Moreover, $2\lambda_{\alpha\beta}\coz\cox-\coy^2=0$ since $\sqf(v'_\beta)=0$, this implies $|\coy|\gg \eta^{-1/4}$, thus $|\coz|\gg \eta^{-1/2}/\|v'_\alpha\|$. This implies $|\sqf(v'_\alpha, v'_\beta)|\gg \eta^{-1/2}$ which in view of~\eqref{eq: sqf v and v'} gives $|\sqf(gv_\alpha, gv_\beta)|\gg \eta^{-1/2}$. 

Now assume contrary to the claim in part~(1), that 
$gv_3=agv_1+bgv_2$ (the argument in other cases is similar). Then 
\[
a\lambda_{31}+b\lambda_{32}=1\quad\text{and}\quad aw_{31}+bw_{32}=0
\]
Solving for $a$ in the first equation, and replacing in the second we conlcude, 
\[
w_{31}=b(\lambda_{32}w_{31}-\lambda_{31}w_{32}).
\]
Since $|\lambda_{\alpha\beta}|=O(1)$ and $\eta^{-1/2}\ll\|w_{\alpha\beta}\|\ll \eta^{-5}$, we conclude that $\eta^{10}\ll |b|\ll \eta^{-10}$. This and $a\lambda_{31}+b\lambda_{32}=1$ imply $\eta^{10}\ll |a|\ll \eta^{-10}$. Recall that 
\[
|\sqf(gv_3)|=|a^2\sqf(gv_1)+b^2\sqf(gv_2)+2ab\sqf(gv_1,gv_2)|\leq \eta^A.
\] 
Since $|\sqf(gv_1, gv_2)|\gg \eta^{-1/2}$ and $\eta^{10}\ll |a|,|b|\ll \eta^{-10}$, we get a contradiction so long as $A\geq 40$ and $\eta$ is small enough. The proof of~(1) is complete.

We now turn to part~(2). Let $I\subset \{w\in\R^3: \|w\|=e^\ell, \sqf(w)=0\}$ be an interval with $|I|\asymp\eta^{-5}$. 
Let $\{w_1, w_2, w_3\}\in I$ satisfy $\|w_i-w_j\|\gg \eta^{-1/2}$. Let $W=(w_1, w_2, w_3)$ which is non-singular by part~(1). A direct computation shows that if 
$v'\in\{\lambda w: \lambda\in[e^{-2}, e^2], w\in I\}$, then $Wx'=v'$ has a solution $\|x'\|\ll \eta^{-\star}$. 

As we did in part~(1), let us write $gv_i=v'_i+u_i$ where $\sqf(v'_i)=0$ and $\|u_i\|\leq e^{-\ell}\eta^A$. Since $gv_1, gv_2, gv_3$ lie in the $e^{-\ell}\eta^A$ neighborhood of an interval $I\subset \{w\in\R^3: \|w\|=e^\ell, \sqf(w)=0\}$ with $|I|\asymp \eta^{-5}$. Thus 
the volume of the tetrahedron spanned by $gv_1, gv_2, gv_3$ is $\ll \eta^{-\star}$. This and part~(1) give 
\be\label{eq: det gamma}
1\leq |\det\gamma|\ll \eta^{-\star}
\ee
Now applying the above discussion with $W'=(v_1', v_2', v_3')$ and $v'=v'_4, v'_5$, there are $x_4',x_5'\in\R^3$ satisfying $\|x'_i\|\ll \eta^{-\star}$ so that $W'x_i'=v'_i$. 

Since $g\gamma=W'+E$ where $\|E\|\ll e^{-\ell}\eta^A$. For $i=4,5$, let 
$x_i=x'_i+\bar u$ where $\bar u=(g\gamma)^{-1}(u_i-Ex'_i)$, then 
\begin{align*}
g\gamma x_i&=g\gamma x'_i+g\gamma\bar u\\
&=W'x'_i+Ex'_i+u_i-Ex'_i=v'_i+u_i=gv_i
\end{align*}
moreover, $\|Ex'_i\|\ll e^{-\ell}\eta^{A-\star}$, $\|u_i\|\leq e^{-\ell}\eta^A$, and 
$\|(g\gamma)^{-1}\|\ll e^{\ell}\eta^{-\star}$. Thus $\|\bar u\|\ll \eta^{A-\star}$. 
Altogether, we conclude that $\|x_i\|=\|\gamma^{-1}v_i\|\ll\eta^{-\star}$. This and~\eqref{eq: det gamma}, complete the proof of part~(2). 
\end{proof}

\begin{lemma}\label{lem: W ell j and periodic H orbits}
    Let the notation be as in Lemma~\ref{lem: choose many 5 tuples}. 
    Then for every $j\in\mathcal J_\ell$, there exists some $k_j\in K$ so that 
    \[
    a_\ell k_j g_Q\Gamma= \epsilon_j g_j
    \]
    where $\|\epsilon_j\|\ll \eta^{A-\star}$ and $\vol(Hg_j\Gamma)\ll \eta^{-\star}$. 
    
    Moreover, we can choose the collection $\{k_j\}$ as above so that the following holds. There exists $\mathcal J'_\ell\subset\mathcal J_\ell$ with $\#\mathcal J'_\ell\gg \eta^{\star}e^\ell$ so that if $j, j'\in\mathcal J'_\ell$ are distinct, then $\|k_j-k_{j'}\|\gg e^{-\ell}$. 
\end{lemma}

\begin{proof}
Fix some $j\in\mathcal J_\ell$, and let $\{v_1,\ldots, v_5\}\subset\mathcal W_{\ell,j}$ be as in Lemma~\ref{lem: the matrices for 5 vectors}.  Then conditions of Lemma~\ref{lem: five vectors} are satisfied with $\{v_1,\ldots, v_5\}$. Let 
\[
\gamma_1=\gamma_1(j)=(v_1, v_2, v_3)
\]
By the conclusion of Lemma~\ref{lem: five vectors} thus
$\gamma_1^tQ\gamma_1=P+E$ where 
\[
P\in \Mat_3(\Z), \quad \|P\|\ll \eta^{-\star},\quad\text{and}\quad \|E\|\ll \eta^{A-\star}.
\]

We conclude that there exists $\hat g_j\in G$, with $\|\hat g_j\|\ll \eta^{-\star}$ so that $(\hat g_j)^{t}\sqf \hat g_j=P$. This implies: $H\hat g_j\Gamma$ is a periodic orbit with $\vol(H\hat g_j\Gamma)\ll \eta^{-\star}$. 

Moreover, we conclude from the above that there exists some $h_j=k_j'\hat ak_j\in\SO(\sqf)$ and $\hat\epsilon_j\in G$ with $\|\hat\epsilon_j-I\|\ll \eta^{A-O(1)}$ so that 
\be\label{eq: k a k and gamma1}
k'_j\hat ak_jg_Q\gamma_1=\hat\epsilon_j \hat g_j.
\ee
Since $\|\hat g_j\|\ll \eta^{-\star}$, we conclude that 
\[
e^{\ell}\eta^{\star}\ll\|\gamma_1\|\eta^{\star}\ll\|\hat a\|\ll \|\gamma_1\|\eta^{-\star}\ll e^{\ell}\eta^{-\star}.
\]
Let us write $\hat a=a_\ell a_{\tau}$ where $|\tau|=\star|\log \eta|$. Multiplying~\eqref{eq: k a k and gamma1} by $(k_j'a_\tau)^{-1}$, we get 
\[
a_\ell k_jg_Q\gamma_1=\epsilon_j g_j\quad\text{where $\|\epsilon_j-I\|\ll \eta^{A-\star}$ and $g_j=(k_j'a_\tau)^{-1}\hat g_j$}. 
\]
This establishes the first claim. 

To see the second claim, let $j,j'\in\mathcal J_\ell$. Repeat the above argument with $j$ and $j'$, and let $k_j, k_{j'}$ and $\gamma_1(j)$ and $\gamma_1(j')$ be the elements as above. If $\|k_j-k_{j'}\|\ll e^{-\ell}$, then 
\[
a_\ell k_{j'}g_Q=g_{jj'} a_\ell k_{j}g_Q\quad\text{ where} \quad \|g_{jj'}\|\ll 1. 
\]
This implies that 
\[
\|(\gamma_1(j))^{-1}\gamma_1(j')\|\ll \eta^{-\star}
\]
Since $\#\mathcal J_\ell\gg \eta^{9} e^{s}$, the collection of $\gamma_1(j)$ for $j\in\mathcal J_\ell$ is distinct, the final claim in the lemma follows. 
\end{proof}

\begin{proof}[Proof of Proposition~\ref{prop: main uppbd prop}]
Recall that $\eta=e^{-\rho \ell}$. Let $\mathcal J'_\ell$ and $\{k_j:j\in\mathcal J'_\ell\}$ be as in Lemma~\ref{lem: W ell j and periodic H orbits}. Then $\#\mathcal J'_\ell\gg \eta^{C}e^\ell$, and for all $j\in \mathcal J'_\ell$, 
\[
a_\ell k_jg_Q\Gamma=\epsilon_j x_j
\]
where $\vol(Hx_j)\leq \eta^{-D_1'}$ and $\|\epsilon_j-I\|\leq \eta^{A-D_1'}$, where $D_1'$ is absolute.  Moreover, $|k_j-k_{j'}|\gg e^{-\ell}$ for all $j\neq j'$. 

Let $\mathcal N=\{\hat g x: \vol(Hx)\leq \eta^{-D_1'}, \|\hat g\|\leq \eta^{A-D_1'}\}$.
We conclude from the above that  
\[
|\{k\in K: a_\ell kg_Q\Gamma\in\mathcal N\}|\gg \eta^{D_1'}.
\]
If $A$ is large enough, we conclude that 
\[
g_Q=\epsilon g
\]
where $\vol(Hg\Gamma)\ll \eta^{-D_1'}$, and $\|\epsilon-I\|\ll \eta^{-\star}e^{-\ell}$, see e.g.~\cite[Prop.\ 4.6]{LMW22}.  

In particular, $g^t\sqf g$ is rational form with height $\leq\eta^{-D_1}=e^{-D_1\rho \ell}$, and
\[
\|Q-g^t\sqf g\|\ll \|g\|^\star\|\epsilon-1\|\leq e^{(-1+D_1\rho)\ell},
\]
as we claimed. 
\end{proof}

\section{Proof of Proposition~\ref{prop: L1+epsilon upp bd}}\label{sec: proof of prop L1+epsilon}
We will complete the proof of Proposition~\ref{prop: L1+epsilon upp bd} in this section; the proof relies on Proposition~\ref{prop: main uppbd prop} and Theorem~\ref{thm: EMM-alpha bound}. 
If part~(2) in Proposition~\ref{prop: main uppbd prop} holds, with some $\ell\geq \rho t-1$, then part~(2) in Proposition~\ref{prop: L1+epsilon upp bd} holds and the proof is complete. Thus we assume part~(2) in Proposition~\ref{prop: main uppbd prop} does not hold for any $\ell\geq \rho t-1$. 

\subsection*{Notation for the proof and Schmidt's Lemma}
Recall that we identify $\wedge^2\R^3$ with the dual of $\R^3$, which will be identified with $\R^3$ using the standard inner product. Then for every $g\in G$
\be\label{eq: action on wedge 2}
gv_1\wedge gv_2= g^* (v_1\wedge v_2)
\ee
where $g^*=(g^t)^{-1}$; note that $H$ is invariant under this involution. Also note that $Q^*(v)=\sqf(g^* v)$ for any $v\in\R^3$.

For every $u\in\wedge^2\R^3$, we will write 
\[
I_u^{*,t}(\sigma)=\{k\in K: \|(a_tkg_Q)^* u\|\leq \sigma\},
\]
and put $I_u^*:=I_u^{*,t}(e^{-\rho t})$. We will also write $I_v$ for $I_v^t(e^{-\rho t})$. 

We emphasize that $I_v$ here is shorthand for $I_v^{t}(e^{-\rho t})$, and is does not represent a shorthand notation for $I_v^{(1+2\rho \ell)}(e^{-\rho\ell})$, which was used in the proof of Proposition~\ref{prop: main uppbd prop}.  

By a variant of Schmidt's Lemma, see also~\cite[Lemma 3.1]{EMM-Upp}, and the definition of $\tilde f$, we have 
\be\label{eq: Schmidt's Lemma}
\tilde f(a_tk; g_Q)\ll \tilde\alpha(a_tk; g_Q)
\ee
where $\tilde\alpha(a_tk; g_Q)=\max\{\tilde\alpha_1(a_tk;g_Q), \tilde\alpha_2(a_tk, g_Q)\}$ and 
\begin{align*}
&\tilde\alpha_1(a_tk;g_Q)=\Bigl\{\norm{a_tkg_Qv}^{-1}: 0\neq v\in \mathcal N_{Q,t}\Bigr\}\\ 
&\tilde\alpha_2(a_tk;g_Q)=\Bigl\{\norm{(a_tkg_Q)^*u}^{-1}: 0\neq u\in \mathcal N_{Q^*,t}\Bigr\}.
\end{align*}

Recall the set $\tilde{\mathcal C}_t=\{k\in K: \tilde f(a_tk; g_Q)\geq Ae^{A\rho t}\}$.
Choosing $A$ large enough to account for the the implied multiplicative constant in~\eqref{eq: Schmidt's Lemma}, 
\[
\tilde{\mathcal C}_t\subset {\mathcal C}_{t,1}\cup \mathcal C_{t,2}.
\]
where $\mathcal C_{t,i}=\{k\in K: \tilde\alpha_i(a_tk; g_Q)\geq e^{A\rho t}\}$ for $i=1,2$. 
In view of~\ref{eq: Schmidt's Lemma} and the above, thus, it suffices to control 
\[
\int_{\mathcal C_{t,1}}\tilde\alpha_1(a_tk; g_Q)\diff\!k\quad\text{and}\quad \int_{\mathcal C_{t,2}}\tilde\alpha_2(a_tk; g_Q)\diff\!k
\]

Let us set   
\begin{align*}
    \mathcal W_{t,1}&:=\{v\in\mathcal N_{Q,t}\cap\mathcal P: \norm{v}\leq e^t, \exists k\in I_v, \|a_tkg_Qv\|\leq e^{-A\rho t}\},\text{ and}\\
    \mathcal W_{t,2}&:=\{u\in\mathcal N_{Q^*,t}\cap\mathcal P: \norm{v}\leq e^t, \exists k\in I_u^*, \|(a_tkg_Q)^*u\|\leq e^{-A\rho t}\}.
\end{align*} 
Note that if $v\in \mathcal W_{t,1}$, then $I_v\cap\mathcal C_t\neq \emptyset$, thus $\|g_Qv\|\leq e^{(1-A\rho)t}$, similarly for covectors in $\mathcal W_{t,2}$. This implies $\|v\|\ll e^{(1-A\rho)t}$. For $i=1,2$ and every integer $\ell\leq t-(A-1)\rho t:=t'$, let 
\[
\mathcal W_{t,i}(\ell)=\mathcal W_{t,i}\cap\mathcal P_\ell.
\]

\subsection*{The sets $\mathcal W_{t,i}'(\ell)$}
We will work with $\mathcal W_{t,1}$, the argument for $\mathcal W_{t,2}$ is similar.

Suppose now $v\in\mathcal W_{t,1}(\ell)$ for some $\ell\leq 1+(1-A\rho)t$. Recall from the definition of $\mathcal W_{t,1}$ that  
\[
\norm{a_tkg_Qv}\leq e^{-A\rho t}, \quad\text{for some $k\in I_v$};
\]
thus, $\absolute{\sqf(a_tkg_Qv)}\ll e^{-2A\rho t}\leq e^{-A\rho t}$.
Since $a_tk$ preserves the form $\sqf$,  
\be\label{eq: Wti ell almost null}
\absolute{\sqf(g_Qv)}\leq e^{-A\rho t}.
\ee

Moreover, since $v$ does not belong to any $(\rho, A,t)$-exceptional line we have $\norm{v}\geq e^{\rho t}$. Since $v\in\mathcal W_{t,1}(\ell)$, we conclude that $\ell\geq \rho t-1$.

For $i=1,2$, let 
\[
\mathcal L_i=\{\rho t-1\leq \ell\leq t': \#\mathcal W_{t,i}(\ell)\geq  e^{(1-\rho)\ell}\},
\]
and let $\mathcal L_i'=\{\rho t-1\leq \ell\leq t': \ell\not\in\mathcal L_i\}$.  

In view of~\eqref{eq: Wti ell almost null}, we have 
\[
\mathcal W_{t,i}(\ell)\subset\{v\in\mathcal P_\ell: \absolute{\sqf(g_Qv)}\leq e^{-A\rho\ell}\}.
\]
This and the definition of $\mathcal L_1$, imply that for every $\ell\in\mathcal L_1$, we have 
\[
\#\{v\in\mathcal P_\ell: \absolute{\sqf(g_Qv)}\leq e^{-A\rho\ell}\}\geq e^{(1-\rho)\ell}. 
\]
Thus, the assumptions of Proposition~\ref{prop: main uppbd prop} holds for any $\ell\in\mathcal L_1$. If part~(2) in Proposition~\ref{prop: main uppbd prop} held, the proof of Proposition~\ref{prop: L1+epsilon upp bd} would be complete. Therefore, we may assume that part~(1) in Proposition~\ref{prop: main uppbd prop} holds.
That is: There exist planes $\{P_i: 1\leq i\leq N\}$, for some $N\leq e^{(1-2\rho)\ell}$, so that 
    \[
   \#\{v\in\mathcal P_\ell\setminus(\cup P_i) : |\sqf(g_Qv)|\leq e^{-A\rho \ell} \}\leq e^{(1-2\rho)\ell}.
    \]
     Moreover, there are intervals $\{J_i: 1\leq i\leq N\}$, with multiplicity at most two, so that $I_{v}^{(1+2\rho)\ell}(e^{-\rho \ell})\cap J_i\neq\emptyset$ for such $v$; thus $I_v\subset cJ_i=:\hat J_i$ for an absolute $c\geq 1$. Furthermore, recall that   
    \[
     \int_{\hat J_i}\alpha_1(a_{(1+2\rho)\ell}kg_Q\Gamma)\diff\!k\ll e^{-\rho \ell}\int_{ \hat J_i}\alpha_2(a_{(1+2\rho)\ell}kg_Q\Gamma)\diff\!k, 
     \]
     and note that $\{\hat J_i\}$ has bounded multiplicity depending on $c$. 

For every $\ell\in\mathcal L_i'$, let $\mathcal W'_{t,i}(\ell)=\mathcal W_{t,i}(\ell)$, and for
every $\ell\in\mathcal L_i$, let $\mathcal W'_{t,i}(\ell)=\mathcal W_{t,i}\setminus (\cup P_i)$. In either case 
\be\label{eq: number of points in W' ell}
\#\mathcal W'_{t,i}(\ell)\leq e^{(1-\rho)\ell}. 
\ee
We put $\mathcal W'_{t,1}=\bigcup_\ell\mathcal W'_{t,i}(\ell)$.

The task is now two folds: we first show that for all $\rho t-1\leq \ell\leq t'$  
\[
\sum_{\mathcal W_{t,1}'}\int_{I_v} \norm{a_tkg_Qv}^{-1-\delta}\diff\!k\ll e^{-\star \rho^2t};
\]
this is an easy consequence of~\eqref{eq: number of points in W' ell} and Lemma~\ref{lem: linear algebra 2}.
  
Then we show that for such $\ell$, 
\[
\int_{\cup \hat J_i}\alpha_1(a_tkg_Q\Gamma)\ll e^{-\star\rho^2t};
\] 
this step is more involved. It relies on Proposition~\ref{prop: main uppbd prop},
Theorem~\ref{thm: EMM-alpha bound}, and ideas similar to~\cite{EMM-Upp}.

Summing these estimates over all $\rho t-1\leq \ell\leq t'$, 
we get $\int_{\mathcal C_{t,1}}\alpha_1\ll e^{-\star \rho^2t}$ as we wanted to show. 
The details follow.

\subsection*{Contribution of $\mathcal W'_{t,i}$}
Applying~\cite[Lemma 5.5]{EMM-Upp} we have  
\[
\sum_{\mathcal W_{t,1}'}\int_{I_v} \norm{a_tkg_Qv}^{-1}\diff\!k\ll \sum_{\lfloor\rho t-1\rfloor}^{\lceil t'\rceil} \sum_{\mathcal W_{t,1}'(\ell)} \norm{v}^{-1}
\]
Using~\eqref{eq: number of points in W' ell}, one obtains 
\[
\begin{aligned}
\sum_{\lfloor\rho t-1\rfloor}^{\lceil t'\rceil} \sum_{\mathcal W_{t,1}'(\ell)} e^{\delta t} \norm{v}^{-1}\ll \sum_{\ell\geq \rho t-1} e^{-\ell} e^{(1-\rho)\ell}\ll e^{-\rho\ell/2}
\end{aligned}
\]
Altogether, we conclude 
\be\label{eq: control W t'}
\sum_{\mathcal W_{t,1}'}\int_{I_v} \norm{a_tkg_Qv}^{-1-\delta}\diff\!k\leq e^{-\rho^2t/2}
\ee
Similarly, $\sum_{\mathcal W_{t,2}'}\int_{I_u^*} \norm{(a_tkg_Q)^*u}^{-1-\delta}\diff\!k\leq e^{-\rho^2t/2}$.

\subsection*{Contribution of $\mathcal W_{t,i}(\ell)\setminus\mathcal W'_{t,i}(\ell)$}
Again we will work with $i=1$, the argument for $i=2$ is similar. 
Let $\ell\in\mathcal L_1$ otherwise $\mathcal W_{t,1}(\ell)=\mathcal W_{t,1}'(\ell)$.  
As it was mentioned above, in view of part~(1) in Proposition~\ref{prop: main uppbd prop}, 
there are intervals $\{J_i: 1\leq i\leq N\}$, with multiplicity at most two, so that 
     $I_{v}^{(1+2\rho)\ell}(e^{-\rho \ell})\cap J_i\neq\emptyset$ for every $v\in \mathcal W_{t,1}(\ell)$. This implies $I_v\subset cJ_i=:\hat J_i$ for every such $v$, where $c\geq 1$ is absolute. Furthermore, recall that 
    \be\label{eq: main prop 1 use}
    \int_{\hat J_i}\alpha_1(a_{(1+2\rho)\ell}kg_Q\Gamma)\diff\!k\ll e^{-\rho \ell}\int_{ \hat J_i}\alpha_2(a_{(1+2\rho)\ell}kg_Q\Gamma)\diff\!k,
    \ee 
     and note that $\{\hat J_i\}$ has bounded multiplicity depending on $c$. 
\begin{sublemma}
We have 
\be\label{eq: alpha 1 at t and s}
\sum_{\hat J_i}\int_{\hat J_i}\alpha_1(a_tkg_Q\Gamma)\diff\!k\ll e^{-\star\rho \ell}\biggl(1+\sum_{\hat J_i}\int_{\hat J_i}\alpha_2(a_{(1+2\rho)\ell}kg_Q\Gamma)\diff\!k\biggr)
\ee
\end{sublemma}

Let us first assume~\eqref{eq: alpha 1 at t and s} and complete the proof of the proposition.
Since $\sum 1_{\hat J_i}\ll 1_{\cup \hat J_i}$,~\eqref{eq: alpha 1 at t and s} implies that  
\begin{align*}
\sum_i\int_{\hat J_i}\alpha_1(a_tkg_Q\Gamma)\diff\!k&\ll e^{-\star\rho \ell}\biggl(1+\int_{\cup \hat J_i}\alpha_2(a_{(1+2\rho)\ell}kg_Q\Gamma)\diff\!k\biggr)\\
&\ll e^{-\star\rho \ell}(1+2\rho)\ell\leq  e^{-\star\rho \ell},  
\end{align*}
where we used part~(2) in Theorem~\ref{thm: EMM-alpha bound} in the second inequality. 

Altogether, for every $\ell\in\mathcal L_1$, 
\be\label{eq: control W t}
\int_{\cup \hat J_i}\alpha_1(a_tk;g_Q)\diff\!k\ll e^{-\star\rho \ell}\leq e^{-\star \rho^2t}. 
\ee
where we used $\ell\geq \rho t-1$ in the last inequality. 

Now summing~\eqref{eq: control W t'} and~\eqref{eq: control W t} over all $\rho t-1\leq \ell\leq t'$, we get 
\[
\int_{\mathcal C_{t,1}}\tilde\alpha_1(a_tk; g_Q\Gamma)\ll te^{-\star \rho^2t}\ll e^{-\star \rho^2t} 
\]
Similarly, $\int_{\mathcal C_{t,2}}\tilde\alpha_2(a_tk; g_Q\Gamma)\ll e^{-\star \rho^2t}$. 

The proposition follows from these in view of~\eqref{eq: Schmidt's Lemma}.  
\qed

\begin{proof}[Proof of the Sublemma] 
It remains to prove~\eqref{eq: alpha 1 at t and s}. 
The argument is similar to the arguments pioneered in~\cite{EMM-Upp}, which are by now well known. 
To simplify the notation, let us write $s=(1+2\rho)\ell$ and put $x=g_Q\Gamma$.

Let $n_1=\lceil 100\rho^{-3}\rceil$ and let $s_1=(t-s)/n_1$.   
We claim that 
\be\label{eq: induction for sublemma}
\sum_i\int_{\hat J_i}\alpha_1(a_{ns_1+s}kx)\diff\!k\leq C^n\sum_i \int_{\hat J_i} \alpha_1(a_skx)\diff\!k+(C+1)^ne^{-\rho\ell/2},
\ee
for all integers $0\leq n\leq n_1$, where $C\geq 1$ is absolute.

Note that the sublemma follows from~\eqref{eq: induction for sublemma} and~\eqref{eq: main prop 1 use}.
Therefore, the task is to establish~\eqref{eq: induction for sublemma}. 

We first make the following observation 
\be\label{eq: estimate square root alpha 2}
\sum_i\int_{\hat J_i}\alpha_2^{1/2}(a_{d}kx)\diff\!k\ll e^{-3\rho\ell/4}\qquad\text{for all $d\geq 0$.}
\ee
To see~\eqref{eq: estimate square root alpha 2}, 
first recall from Proposition~\ref{prop: main uppbd prop} that $\sum 1_{\hat J_i}\ll 1_{\cup\hat J_i}$, hence  
\[
\sum_i\int_{\hat J_i}\alpha_2^{1/2}(a_{d}kx)\diff\!k\ll \int_{\bigcup \hat J_i}\alpha_2^{1/2}(a_{d}kx)\diff\!k.
\]
Moreover, the number of such intervals $\hat J_i$ is $\ll e^{(1-2\rho)\ell}$ and
$|\hat J_i|\asymp e^{-\ell}e^{-3\rho \ell/2}$, see~\eqref{eq: eq: size of Iv'} and the paragraph following that equation. 
Thus  
\[
\Bigl|\bigcup \hat J_i\Bigr|\leq \sum |\hat J_i|\ll (e^{(1-2\rho)\ell}  \cdot (e^{(-1-\frac32\rho)\ell}\leq e^{-3\rho\ell}.
\] 
Hence~\eqref{eq: estimate square root alpha 2} follows from Corollary~\ref{cor: Lp functions}, applied with $\Theta=\bigcup \hat J_i$.

Let us now return to the proof of~\eqref{eq: induction for sublemma}, which will be completed using induction on $n$. 
The base case $n=0$ is trivial. 

Fix some $n\geq 0$ and some $\hat J_i$. Let $M=\lceil e^{ns_1+s}|\hat J_i|\rceil$ and $b={|\hat J_i|}/{M}$. 
For every $k\in L$, let $h_k=a_{ns_1+s}k a_{-ns_1-s}$. Then 
\be\label{eq: decomposing intergral}
\int _{\hat J_i}\alpha_1(a_{(n+1)s_1+s}kx)\diff\!k=\sum_j \int_{0}^{b}\alpha_1(a_{s_1} h_k a_{ns_1+s}k_jx)\diff\!k,
\ee
where $k_j=jb$ for $0\leq j\leq M-1$. 

If $k=\rot_\theta$, see~\eqref{eq: r-theta Oppenheim}, then $h_k=h'_ku_{-e^{ns_1+s}\sin\theta/(1+\cos\theta)}$ where $\|h'_k-I\|\ll b$ is lower triangular. 
Using the change of variable $r=e^{ns_1+s}\frac{\sin\theta}{(1+\cos\theta)}$, we get 
\be\label{eq: change of var from k to r}
\int_{0}^{b}\alpha_1(a_{s_1} h_k a_{ns_1+s}k_jx)\diff\!k\asymp e^{-ns_1-s}\int_0^{b'} \alpha_1(a_{s_1} u_{-r} a_{ns_1+s}k_jx)\diff\!r.
\ee
where $b'=e^{ns_1+s}\frac{\sin b_0}{(1+\cos b_0)}\asymp 1$.

In view of~\cite[Lemma 5.8]{EMM-Upp}, see also~\cite[Lemma 5.5]{EMM-Upp}, we have  
\begin{multline*}
\int_0^{b'} \alpha_1(a_{s_1} u_{-r} a_{ns_1+s}k_jx)\diff\!r\ll \\
\int_0^{b'}\alpha_1(u_{-r} a_{ns_1+s}k_jx)\diff\!r+e^{2s_1}\int_0^{b'}\alpha_2^{1/2}(u_{-r}a_{ns_1+s}kx)\diff\!r
\end{multline*}
Multiplying the above by $e^{-ns_1-s}$ and changing the variables from $u_r$ back to $h_k$ in the second line above, we conclude that 
\begin{multline*}
e^{-ns_1-s}\int_0^{b'} \alpha_1(a_{s_1} u_{-r} a_{ns_1+s}k_jx)\diff\!r\ll\\
\int_0^{b}\alpha_1(h_ka_{ns_1+s}k_jx)\diff\!k+e^{2s_1}\int_0^{b}\alpha_2^{1/2}(h_ka_{ns_1+s}k_jx)\diff\!k=\\
\int_0^{b}\alpha_1(a_{ns_1+s}kk_jx)\diff\!k+e^{2s_1}\int_0^{b}\alpha_2^{1/2}(a_{ns_1+s}kk_jx)\diff\!k.
\end{multline*}
In view of~\eqref{eq: change of var from k to r} and~\eqref{eq: decomposing intergral}, the above implies
\begin{multline*}
\int_{\hat J_i}\alpha_1(a_{(n+1)s_1+s}kx)\diff\!k \ll\sum_j\int_0^{b}\alpha_1(a_{ns_1+s}kk_jx)\diff\!k\quad +\\
e^{2s_1}\sum_j\int_0^{b}\alpha_2^{1/2}(a_{ns_1+s}kk_jx)\diff\!k.
\end{multline*}
Altogether, we have shown 
\[
\int_{\hat J_i}\!\! \alpha_1(a_{(n+1)s_1+s}kx)\diff\!k\ll\!\!\int_{\hat J_i}\!\! \alpha_1(a_{ns_1+s}kx)\diff\!k +e^{2s_1}\!\!\int_{\hat J_i}\!\alpha_2^{1/2}(a_{ns_1+s}kx)\diff\!k.
\]
Summing these over all $\hat J_i$, we get 
\begin{multline}\label{eq: sum Ji for ind}
\sum_{\hat J_i}\int_{\hat J_i}\alpha_1(a_{(n+1)s_1+s}kx)\diff\!k\leq C\sum_{\hat J_i}\int_{\hat J_i}\alpha_1(a_{ns_1+s}kx)\diff\!k\quad +\\
Ce^{2s_1}\sum_{\hat J_i}\int_{\hat J_i}\alpha_2^{1/2}(a_{ns_1+s}kx)\diff\!k
\end{multline}
Recall now that $\ell\geq \rho t-1$ and $s_1\asymp \rho^3t/100$. This and~\eqref{eq: estimate square root alpha 2} imply that
\[
e^{2s_1}\sum_{\hat J_i}\int_{\hat J_i}\alpha_2^{1/2}(a_{ns_1+s}kx)\diff\!k\ll e^{\rho\ell/10} e^{-3\rho/4}\leq e^{-\rho\ell/2}.
\]
This,~\eqref{eq: sum Ji for ind} and the inductive hypothesis finish the proof of~\eqref{eq: induction for sublemma}.  
\end{proof}

\section{Circular averages of Siegel tranforms}\label{sec: equid unbdd}
Let $f\in C_c(\R^3)$. For every $g\in G$, define 
\be\label{eq: def hat f}
\hat f(g)=\sum_{v\in \Z^3}f(gv)
\ee

For any indefinite ternary quadratic form $Q$ with $\det Q=1$, we let $g_Q\in G$ be so that $Q(v)=Q_0(g_Qv)$, where 
$\sqf(\cox,\coy,\coz)=2\cox\coz-\coy^2$.

In this section, we will use Proposition~\ref{prop: L1+epsilon upp bd} and Theorem~\ref{thm: equid circles} to prove the following proposition. The proof of Theorem~\ref{thm: quantitative Oppenheim} will then be completed using this proposition.   

\begin{propos}\label{prop: equi of f hat}
There is an absolute constant $M$, and for every large enough $E$ and $0<\rho\leq 10^{-4}$, there are $\varrho_1, \varrho_2$ 
(depending on $\rho$ and $E$) so that if $Q$ is as above and $t$ is large enough, depending linearly on $\log(\norm{Q})$, the following holds.

Assume that for every $Q'\in\Mat_3(\Z)$ with $\norm{Q'}\leq e^{\rho t}$ and all $\lambda\in\R$,  
\be\label{eq: Diophantine main thm Q0}
\norm{Q-{\lambda}Q'}>\norm{Q'}^{-E/M}.
\ee
There exists some $C$ depending on $E$ and polynomially on $\norm{Q}$ so that the following holds. For any smooth function $\xi$ on $K$, if
\begin{equation*}
\biggl|\int_K\hat f (a_tk g_Q\Gamma)\xi(k)\diff\!k-\int_K\xi\diff\!k\int_X\hat f\diff\!m_X\biggr|
> C\Sob(f)\Sob(\xi)e^{-\varrho_2 t}
\end{equation*}
then there is at least one $(\frac{\varrho_1}{4E}, E, t)$-exceptional line or plane, 
and at most four $(\frac{\varrho_1}{4E}, E, t)$-exceptional lines $\{L_i\}$ and at most four $(\frac{\varrho_1}{4E}, E, t)$-exceptional planes $\{P_i\}$. Moreover  
\[
\int_K\hat f(a_tk g_Q\Gamma)\xi(k)\diff\!k=\int_K\xi\diff\!k\int_X\hat f\diff\!m_X \\
+ \mathcal M + O(\Sob(f)\Sob(\xi)e^{-\varrho_2 t})
\]
where 
\[
\mathcal M=\int_{\mathcal C}\hat f_{\rm sp}(a_tk g_Q\Gamma)\xi(k)\diff\!k
\]
with 
\begin{align*}
\hat f_{\rm sp}(a_tk g_Q\Gamma)&=\!\!\!\sum_{v\in \Z^3 \cap ((\cup_i L_i)\cup (\cup_i P_i))} f(a_tk g_Qv)\\
\mathcal C&=\left\{k\in K: \hat f_{\rm sp}(a_tk g_Q\Gamma)\geq e^{\varrho_1 t}\right\}
\end{align*}
\end{propos}

Before starting the proof, we recall the following well-known fact 

\begin{lemma}\label{lem: periodic orbit rational form}
Suppose there exists $g'\Gamma\in X$ with $\vol(Hg'\Gamma)\leq R$, so that 
\[
d(g_Q\Gamma, g'\Gamma)\leq \beta
\]
There exists an integral form $Q'$ with $\|Q'\|\ll R^\star$ and some $\lambda\in\R$ so that 
\[
\|Q-\lambda Q'\|\ll R^\star \beta
\]
\end{lemma}

\begin{proof} 
Replacing $g'$ by $hg'\gamma$ for some $h\in H$ and $\gamma\in\Gamma$ we may assume 
$\|g'\|\ll 1$. Similarly, we may assume $\|g_Q\|\ll 1$. 

Since $\vol(Hg'\Gamma)\leq R$, $\sqf\circ g'$ is equivalent to an integral form $Q''$ with
$\|Q''\|\ll R^\star$. 

Since $d(g_Q\Gamma, g'\Gamma)\leq \beta$, $g_Q=\epsilon g'\gamma'$ where $\|\epsilon-I\|\ll \beta$ and $\|\gamma'\|\ll1$. Thus  
\[
\|Q-\lambda Q'\|\ll R^\star \beta
\]
for some integral form $Q'$ with $\|Q'\|\ll R^\star$ and some $\lambda\in\R$. 
\end{proof}

\begin{proof}[Proof of Proposition~\ref{prop: equi of f hat}]
Let $E$ and $\rho$ be as in the statement, and let $t>0$ be a large parameter. 
We will use Lemma~\ref{lem: periodic orbit rational form} in the following form:

Let $Q$ satisfy~\eqref{eq: Diophantine main thm Q0}.
There exists $E_1\geq \max(4\ref{a:circle},E)$, where $\ref{a:circle}$ is as in Theorem~\ref{thm: equid circles} so that the following holds. 
For all $t$ so that $t>4\ref{a:circle}\log t$ and for every $x\in X$ with $\vol(Hx)\leq e^{\rho t/E_1}$,  
\[
d(g\Gamma, x)> e^{-t}.
\]
In view of this, part~(1) in Theorem~\ref{thm: equid circles} (applied with $G=\SL_3(\R)$, hence, $\dm=2$) holds with $R=e^{\rho t/E_1}$ and $t$. Indeed, since $\frac{\ref{a:circle}\rho}{E_1}\leq \frac14$ and $t^{\ref{a:circle}}\leq e^{t/4}$,
\[
R^{\ref{a:circle}}t^{\ref{a:circle}} e^{-2t}= e^{\ref{a:circle}\rho t/E_1} t^{\ref{a:circle}} e^{-2t}\leq e^{-t};
\]
hence, part~(2) in Theorem~\ref{thm: equid circles} cannot hold. 

For every $S$, let $1_{X_{1/S}}\leq \varphi_S\leq 1_{X_{1/{(S+1)}}}$ be smooth with $\Sob(\varphi_S)\ll S^\star$.
Put $\hat f_S=\varphi_S\hat f$; we let $L$ be so that $\Sob(\hat f_S)\ll S^L\Sob(f)$. 
Recall that $\alpha_1, \alpha_2\gg S^\star$ in $X_{1/S}$, they belong to $L^2(X,m_X)$, and $\hat f\ll \max(\alpha_1,\alpha_2)$. Thus, 
\be\label{eq: f hat and fS hat}
\int \hat f\diff\!m_X = \int \hat f_S\diff\!m_X+ O(S^{-\star}).
\ee

Set $\vare=\ref{k:circle}\rho/(2LE_1)$, where $\ref{k:circle}$ is as in Theorem~\ref{thm: equid circles}. 
We will show the claim in the theorem holds with 
\[
\varrho_1=\vare\quad\text{and}\quad \varrho_2=\tfrac{\vare^2}{32E^2}.
\]

Apply Lemma~\ref{lem: five vectors} with $\eta^{-D}=e^{\frac\vare{E} t}$, $\eta^A=e^{-\vare t}$, and $Q$. 
If there are at least five $(\frac{\vare}{4E},E, t)$-exceptional lines, then Lemma~\ref{lem: five vectors} implies 
\[
\|Q-\lambda Q'\|\leq \eta^{A-M' D}=e^{(-\vare+ \frac{M'\vare}{E}) t}
\] 
for some $Q'\in \Mat_3(\Z)$ with $\|Q'\|\leq e^{\frac{M'\vare}{E} t}$, where $M'$ is absolute. Assuming $E$ is large compared to this $M'$, however, this contradicts~\eqref{eq: Diophantine main thm Q0} if $M\geq2M'$. Thus, there are at most four $(\frac{\vare}{4E},E, t)$-special lines. Similarly, there are at most four $(\frac{\vare}{4E},E, t)$-exceptional planes.

Denote these lines and the planes (if they exist) by $\{L_i\}$ and $\{P_i\}$, respectively, and put ${\rm Exc}=(\cup_iL_i)\cup (\cup_iP_i)$. 
For every $k\in K$, write 
\[
\hat f(a_tk g_Q\Gamma)=\hat f_S(a_tkg_Q\Gamma)+\hat f_{\rm cusp}(a_tkg_Q\Gamma)+ \hat f_{\rm sp}(a_tkg_Q\Gamma)
\]
where $\hat f_S=\varphi_S\hat f$, $\hat f_{\rm cusp}$ is the contribution of $\Z^3\setminus {\rm Exc}$ 
to $\hat f-\hat f_S$, and $\hat f_{\rm sp}$ is the contribution of $\Z^3\cap{\rm Exc}$ to $\hat f-\hat f_S$. 

By Theorem~\ref{thm: equid circles}, applied with $R=e^{\rho t/E_1}$, 
for any $\xi\in C^\infty(K)$ we have 
\begin{multline}\label{eq: proof main term}
\biggl|\int_K \hat f_S(a_tk g_Q\Gamma)\xi(k)\diff\!k- \int\xi\diff\!k\int_X\hat f_S\diff\!m_X\biggr|\ll\\ 
\Sob(\hat f_S)\Sob(\xi)e^{-\ref{k:circle} \rho t/E'}\ll S^L\Sob(f) \Sob(\xi)e^{-\ref{k:circle} \rho t/E_1}.
\end{multline}
If we choose $S=e^{\vare t}=e^{\ref{k:circle} \rho t/(2LE_1)}$, 
the above is $\ll \Sob(f)\Sob(\xi)e^{-\vare t/2}$. 

Moreover, by Theorem~\ref{thm: EMM-alpha bound} applied with $p=1/2$ and the Chebyshev's inequality, we have 
\be\label{eq: lem L2 bound SL2 used}
\int_{\{k: a_tk g_Q\Gamma\notin X_{1/S}\}} S^{1/4} \diff\!k\ll S^{-1/2} S^{1/4}=S^{-1/4}.
\ee
This and~\eqref{eq: proof main term}, reduce the problem to investigating the integral of $\hat f-\hat f_S=\hat f_{\rm cusp}+\hat f_{\rm sp}$ over 
$\hat{\mathcal C}:=\{k\in K: \hat f-\hat f_S\geq ES^{1/4}\}$.

Let $\tilde f$ be as in~\eqref{eq: def tilde f} with $\vare/4E$, $E$, and $t$. That is:
\[
\tilde f(h;g_Q)=\sum_{v\in \mathcal N_{Q,t}} f(hv)
\]
where $\mathcal N_{Q,t}$ denotes the set of vectors in $\Z^3$ which are not contained in any $(\vare/4E, E, t)$-exceptional line or plane.

Let $\tilde{\mathcal C}_t=\{k\in K: \tilde f(a_tk; g_Q)\geq Ee^{\vare t/4}=ES^{1/4}\}$. 
By the definitions, 
\[
\int_{\hat{\mathcal C}}\hat f_{\rm cusp}(a_tkg_Q\Gamma)\xi(k)\diff\!k\leq \norm{\xi}_{\infty}\int_{\tilde{\mathcal C}_t}\tilde f(a_tk;g_Q)\diff\!k.
\]
Thus Proposition~\ref{prop: L1+epsilon upp bd}, applied with $E$ and $\vare/4E$, implies 
\[
\int_{\tilde{\mathcal C}_t}\tilde f(a_tk ;g_Q)\diff\!k\ll e^{-\frac{\vare^2 t}{32E^2}}.
\]
From these two, we conclude that 
\be\label{eq: estimate for f hat cusp}
\int_{\hat{\mathcal C}}\hat f_{\rm cusp}(a_tk g_Q\Gamma)\diff\!k\ll \norm{\xi}_\infty e^{-\frac{\vare^2 t}{32E^2}}.
\ee

In view of~\eqref{eq: f hat and fS hat},~\eqref{eq: proof main term},~\eqref{eq: lem L2 bound SL2 used} and~\eqref{eq: estimate for f hat cusp}, we have 
\begin{multline*}
\biggl|\int_K\hat f(a_tk g_Q\Gamma)\xi(k)\diff\!k-\int_K\xi\diff\!k\int_X\hat f\diff\!m_X\biggr|\\
= \int_{\mathcal C} \hat f_{\rm sp}(a_tk g_Q\Gamma)\xi(k)\diff\!k+ O(\Sob(f)\Sob(\xi)e^{-\vare^2 t/32E^2})
\end{multline*}
where $\mathcal C=\{k\in K: \hat f_{\rm sp}(a_tk g_Q\Gamma)> Ee^{\vare t/4}\}$.

This completes the proof if we let $\varrho_1=\vare$ and $\varrho_2=\vare^2/32E^2$. 
\end{proof}

The Diophantine condition on $Q$, as stated in~\eqref{eq: Diophantine main thm Q0}, is primarily employed to derive the desired upper bound estimates for circular averages of the Siegel transform, namely the application of Proposition~\ref{prop: L1+epsilon upp bd} in the proof. However, if one assumes the {\em weaker} condition on $Q$ specified in Theorem~\ref{thm: Oppenheim}, it is still possible to establish a lower bound estimate for these averages. This lower bound suffices to prove (a stronger version of) Theorem~\ref{thm: Oppenheim}.

A simplified version of Proposition~\ref{prop: equi of f hat}, which encapsulates this result, is presented below.

\begin{propos}\label{prop: lower bd of ave f hat}
Let $R$ be large, depending on $\norm{Q}$, and let $T\geq R^{A}$.
Assume that for every $Q'\in\Mat_3(\Z)$ with $\norm{Q'}\leq R$ and all $\lambda\in\R$,  
\be\label{eq: Diophantine main R T}
\norm{Q-{\lambda}Q'}>R^A(\log T)^A T^{-2}.
\ee
Then for all $f\in C_c^\infty(\R^3)$ and $\xi\in C_c^\infty(K)$ we have  
\[
\int_K\hat f(a_{\log T}k g_Q\Gamma)\xi(k)\diff\!k\geq \int_K\xi\diff\!k\int_X\hat f\diff\!m_X+O(\Sob(f)\Sob(\xi)T^{-\kappa})
\]
where $A$ and $\kappa$ are absolute. 
\end{propos}

\begin{proof}
The proof of this proposition is contained in the proof of Proposition~\ref{prop: equi of f hat}, we explicate the proof for the convenience of the reader. 

For every $S$, let $1_{X_{1/S}}\leq \varphi_S\leq 1_{X_{1/(S+1)}}$ be a smooth function with $\Sob(\varphi_S)\ll S^\star$.
Put $\hat f_S=\varphi_S\hat f$; we let $L$ be so that $\Sob(\hat f_S)\ll S^L\Sob(f)$. 

In view of Lemma~\ref{lem: periodic orbit rational form}, if for some $R'$ and some $g'$ with $\vol(Hg'\Gamma)\leq R'$, 
\[
d(g_Q\Gamma, g'\Gamma)\leq R'^{\ref{a:circle}}(\log T)^{\ref{a:circle}}T^{-2},
\] 
then there exists $Q'\in\Mat_3(\Z)$ with $\|Q'\|\ll R'^{A'}$ and $\lambda\in\R$ so that  
\[
\|Q-\lambda Q'\|\leq R'^{\ref{a:circle}+A'}(\log T)^{\ref{a:circle}}T^{-2}.
\]
Assuming $R'\asymp R^{1/A'}$ is chosen so that $\|Q'\|\ll R'^{A'}$ implies $\|Q'\|\leq R$, the above 
contradicts~\eqref{eq: Diophantine main R T} so long as $A$ is large enough.  

Therefore, by Theorem~\ref{thm: equid circles}, applied with $G=\SL_2(\R)$ (hence $\dm=2$) and $R'\asymp R^{1/A'}$, 
for any $\xi\in C^\infty(K)$, 
\begin{multline}\label{eq: proof main term'}
\biggl|\int_K \hat f_S(a_{\log T}k g_Q\Gamma)\xi(k)\diff\!k- \int\xi\diff\!k\int_X\hat f_S\diff\!m_X\biggr|\ll\\ 
\Sob(\hat f_S)\Sob(\xi)R'^{-\ref{k:circle}}\ll S^L\Sob(f) \Sob(\xi)R^{-\ref{k:circle}/A'}.
\end{multline}
If we choose $S=R^\star$, the above is $\ll \Sob(f)\Sob(\xi)R^{-\star}$. 

Moreover, recall from~\eqref{eq: f hat and fS hat} that 
\[
\int \hat f\diff\!m_X = \int \hat f_S\diff\!m_X+ O(S^{-\star}) 
\]
Altogether, thus, we conclude 
\begin{align*}
\int_K \hat f(a_{\log T}k g_Q\Gamma)\xi(k)\diff\!k&\geq \int_K \hat f_S(a_tk g_Q\Gamma)\xi(k)\diff\!k\\
&=\int\xi\diff\!k\int_X\hat f\diff\!m_X+\Sob(f) \Sob(\xi)R^{-\star},
\end{align*}
as it was claimed.
\end{proof}

\section{Linear algebra and quadratic forms}\label{sec: LA QF proof}
In this section, we will use transitivity of the action of 
$H$ on level sets $\{v: \sqf(v)=c\}$ to relate the counting problem in Theorem~\ref{thm: quantitative Oppenheim} to averages considered in Proposition~\ref{prop: equi of f hat}. As mentioned before, the argument is similar to~\cite{DM-Linearization, EMM-Upp}. 

Let us begin with the following 

\begin{lemma}[cf.\ \cite{EMM-Upp}, Lemma 3.4]
\label{lem: EMM calculus}
Let $f\geq 0$ be a smooth function supported on $B_{\R^3}(0, R)$ (with $R\geq 1$) and let $\xi$ be a smooth function on $\mathbb S^2$.
Let $v\in\R^3$ and suppose $t\geq2\log R$ and $e^t/2\leq \|v\|\leq e^t$. Then 
\begin{multline*}
\int f(a_tkv)\xi(k^{-1}\mathsf e_3)\diff\!k= \\\tfrac{\sqrt 2}{2\pi(1+O(e^{-t}))\|v\|}J_{f}(\sqf(v), e^{-t}\|v\|)\xi(\tfrac{v}{\|v\|})+O\Bigl(\Sob(f)\Sob(\xi) e^{-2t}\Bigr) 
\end{multline*}
where $\{\mathsf e_1, \mathsf e_2, \mathsf e_3\}$ is the standard basis for $\R^3$, 
\[
J_{f}(c,d):=\frac{1}{d}\int f\Bigl(\tfrac{c+\coy^2}{2d},-\coy, d\Bigr) \diff\!\coy\qquad\text{that is: $\sqf(\tfrac{c+\coy^2}{2d}, \coy, d)=c$,}
\] 
and the implied constants depend polynomially on $R$. 

Moreover, there is some $C>0$ depending on $R$ so that if either $\|v\|>Ce^t$ or $|\sqf(v)|>C$, then both sides of the above equality equal zero without the error term $O\Bigl(\Sob(f) e^{-4t}\Bigr)$.  
\end{lemma}

\begin{proof}
Put $\mathcal K_f(v)=\{k\in K: f(a_tkv)\neq 0\}$. Let $k\in\mathcal K_f(v)$, then  
\be\label{eq: at k v e1}
\begin{aligned}
-R\leq \cox:=&\langle a_tkv, \mathsf e_1\rangle=e^t\langle kv, \mathsf e_1\rangle\leq R\quad\text{and}\\ 
-R\leq \coy:=&\langle a_tkv, \mathsf e_2\rangle=\langle kv, \mathsf e_2\rangle\leq R
\end{aligned}
\ee
The estimates in~\eqref{eq: at k v e1} and $e^{t}/2\leq \|v\|\leq e^t$ thus imply that 
there exists some $C'$ which depends polynomially on $R$ so that if we let 
\[
\mathcal K_R(v)=\{k\in K: \|k^{-1}\mathsf e_3-\tfrac{v}{\|v\|}\|\leq C'e^{-t}\},
\]
then $\mathcal K_f(v)\subset\mathcal K_R(v)$.  
Note that if $\mathcal K_R(v)=\emptyset$ or $\mathcal K_f(v)=\emptyset$, we still have $\mathcal K_f(v)\subset \mathcal K_R(v)$. Altogether, the range of integration is restricted to $\mathcal K_R(v)$. 

We first note that for all $k\in \mathcal K_R(v)$, we have 
\be\label{eq: xi is smooth}
|\xi(k^{-1}\mathsf e_3)-\xi(\tfrac{v}{\|v\|})|\ll \Sob(\xi)e^{-t}\qquad\text{}
\ee
Similarly, for all $k\in \mathcal K_R(v)$,   
\[
\coz:=\langle a_tkv, \mathsf e_3\rangle=e^{-t}\langle kv, \mathsf e_3\rangle= e^{-t}\|v\|+O(e^{-t}).
\]
Let us now put 
\[
\cox'=\tfrac{\sqf(v)+\coy^2}{2e^{-t}\|v\|}=\tfrac{(2\cox\coz-\coy^2)+\coy^2}{2e^{-t}\|v\|}=\tfrac{\cox\coz}{e^{-t}\|v\|}.
\] 
Then $\sqf(\cox',\coy, e^{-t}\|v\|)=\sqf(v)$, $|\cox-\cox'|\ll e^{-t}$, and 
\[
|f(a_tkv)-f(\cox',\coy, e^{-t}\|v\|)|\ll \Sob(f_{c,\delta}) e^{-t}.
\]

In view of~\eqref{eq: xi is smooth}, thus, to complete the proof, we need to compute $\diff\!\coy$. 
First note that 
\[
\int f(a_tkv)\diff\!k=\int f(a_tk'k_vv)\diff\!k'.
\]
Write $k_vv=(v_1, v_2, v_3)$, then $v_3=(1+O(e^{-t}))\|v\|$ and in the notation of~\eqref{eq: r-theta Oppenheim}, 
\[
\coy=\langle k'k_vv, \mathsf e_2\rangle=v_1\tfrac{\sin\theta}{\sqrt2}+v_2\cos\theta-v_3\tfrac{\sin\theta}{\sqrt2}.
\]
Thus $\diff\!k=-\tfrac{\sqrt 2}{2\pi((1+O(e^{-t}))\|v\|)}\diff\!\coy$ on $K_R(v)$, 
where we used $\cos\theta\!=\!1+O(e^{-2t})$ for $k'k_v\in K_R(v)$ and $v_3=(1+O(e^{-t}))\|v\|$. Altogether, we get  
\[
\begin{aligned}
\int f(a_tkv)\diff\!k&=\tfrac{\sqrt 2}{2\pi((1+O(e^{-t}))\|v\|)}\!\!\int f(\cox',-\coy, e^{-t}\|v\|) \diff\!\coy+O\Bigl(\Sob(f) e^{-2t}\Bigr)\\
&=\tfrac{\sqrt 2}{2\pi((1+O(e^{-t}))\|v\|)}J_f(\sqf(v), e^{-t}\|v\|)+O\Bigl(\Sob(f) e^{-2t}\Bigr). 
\end{aligned}
\]
This and~\eqref{eq: xi is smooth} complete the proof of the lemma; note that the last claim in the lemma follows from the above argument. 
\end{proof}

Fix some $0<\vare<10^{-4}$ and cover $\mathbb S^2$ with disjoint half open cubes $\{\Omega_{i,\vare}'\}$ of 
size between $\vare^2$ and $\vare^2/2$. 
For each $i$, let 
\[
\Omega_{i,\vare}=(1-\vare, 1]\cdot (g_Q\Omega_{i,\vare}').
\]
Also set $\sup_{\Omega_{i,\vare}'}\|g_Qw\|=d_i$ and $d_i-C\vare^2=\inf_{\Omega_{i,\vare}'}\|g_Qw\|$, note that $d_i$ and $C$ depend on $\|g_Q\|$. 

\begin{lemma}\label{lem: calculus}
Let $c\in\R$ and $0<\vare<10^{-4}$. 
There exist $f_i^\pm\in C_c^\infty(\R^3)$ and $\xi_i\in C_c^\infty(\mathbb S^2)$ satisfying 
$0\leq f_i^-\leq f_i^+\ll 1$, $\int f^+_i=(1+O(\vare))\int f_i^-$, and $\Sob(\bullet)\ll \vare^{-L}$ for $\bullet=f_i^\pm,\xi_i$ where $L$ is absolute, so that 
\begin{multline*}
J_{f_i^-}(\sqf(v), e^{-t}\|v\|)\xi_i(\tfrac{v}{\|v\|})\leq 1_{\Omega_{i,\vare}}(e^{-t} v)1_{[c-\vare,c+\vare]}(\sqf(v))\leq\\ J_{f_i^+}(\sqf(v), e^{-t}\|v\|)\xi_i(\tfrac{v}{\|v\|})
\end{multline*}
\end{lemma}

\begin{proof}
Let us set
\begin{align*}
\mathsf B^-_i&=\{v=(\cox,\coy,\coz): |\coy|\leq \tfrac12, (1-\vare)d_i\leq \coz\leq d_i-C\vare^2, |\sqf(v)-c|\leq \vare\}\\
\mathsf B^+_i&=\{v=(\cox,\coy,\coz): |\coy|\leq \tfrac12, (1-\vare)(d_i-C\vare^2)\leq \coz\leq d_i, |\sqf(v)-c|\leq \vare\}
\end{align*} 
Define $\mathsf B^{--}_i$ and $\mathsf B^{++}_i$ similarly, by replacing $\vare$ by $\vare-C'\vare^2$ and $\vare+C'\vare^2$ for a large $C'$, depending on $d_i$ and $C$,  respectively.  

Fix smooth functions $\tilde f_i^\pm$ satisfying $1_{\mathsf B_i^{--}}\leq \tilde f_i^-\leq 1_{\mathsf B_i^-}$, 
$1_{\mathsf B_i^+}\leq \tilde f_i^+\leq  1_{\mathsf B_i^{++}}$, and $\Sob(\tilde f_i^\pm)\ll \vare^{-L}$. Put $f_i^\pm=\coz\tilde f_i^\pm$. 

Also fix smooth characteristic functions $\xi_i$ so that 
\[
1_{\Xi_i}\leq \xi_i\leq 1_{\mathcal N_{C'\vare^2}(\Xi_i)}\quad\text{and}\quad \Sob(\xi_i)\ll \vare^{-L}
\] 
where $\mathcal N_{\delta}(\Xi_i)$ is the $\delta$-neighborhood of $\Xi_i=\{\frac{w}{\|w\|}: w\in g_Q\Omega_{i,\vare}'\}$. 

Suppose now that $v$ is so that $J_{f_i^-}(\sqf(v), e^{-t}\|v\|)\xi_i(\frac{v}{\|v\|})\neq 0$, then $v\in g_Q\Omega_{i,\vare}'$.
Moreover, using the definition 
\[
J_{f_i^-}(\sqf(v), e^{-t}\|v\|)=\tfrac{e^t}{\|v\|}\int f_i^- (\cox, -\coy, e^{-t}\|v\|)\diff\!\coy
\]    
where $\sqf(\cox, \coy, e^{-t}\|v\|)=\sqf(v)$, we have $f_i^- (\cox, -\coy, e^{-t}\|v\|)\neq 0$ for some $\cox,\coy$. Since $\frac1\coz f_i^-\leq 1_{B_i^-}$, we thus conclude 
\[
(1-\vare)d_i\leq e^{-t}\|v\|\leq d_i-C\vare^2\quad\text{and}\quad |\sqf(v)-c|\leq \vare
\]
These imply that $e^{-t}v\in \Omega_{i,\vare}$. Therefore,  
\[
1_{\Omega_{i,\vare}}(e^{-t} v)1_{[c-\vare,c+\vare]}(\sqf(v))=1.
\] 
Since $J_{f_i^-}(\sqf(v), e^{-t}\|v\|)\xi(v/\|v\|)\leq 1$ the lower bound follows. 

We now establish the upper bound. Let $e^{-t}v\in\Omega_{i,\vare}$ and $|\sqf(v)-c|\leq \vare$. Suppose
$|\coy|\leq 1/2$, and let $\cox$ be so that $\sqf(\cox,\coy,e^{-t}\|v\|)=\sqf(v)$. 
Then $(\cox,-\coy,e^{-t}v)\in \mathsf B_i^+$, thus $f_i^+(\cox,-\coy,e^{-t}\|v\|)=e^{-t}\|v\|$. Hence 
\[
J_{f_i^+}(\sqf(v), e^{-t}\|v\|)=\tfrac{e^t}{\|v\|}\int f_i^+ (\cox, -\coy, e^{-t}\|v\|)\diff\!\coy=1,
\]  
as we claimed.  
\end{proof}

We also record the following upper bound estimate whose proof relies on similar (but simpler) arguments.
See \cite[Thm.\ 2.3]{EMM-Upp}, also \cite[Lemma 3.9]{LMW23} and a related estimate in~\cite[Thm.~4]{KKL-Mean}.

\begin{lemma}\label{lem: gen upp bd}
 Let $0<\eta<1$. 
 Then for all $s\gg\absolute{\log\eta}$, we have 
  \[
  \#\{v\in \Z^3: \|v\|\leq e^s, |Q(v)|\leq M\}\ll e^{(1+\eta)s}
  \]
  where the implied constants depend polynomially on $M$ and $\|Q\|^{\pm1}$.  
\end{lemma}

\begin{proof}
As indicated above, this is proved in \cite[Thm.\ 2.3]{EMM-Upp}, the polynomial dependence stated above is implicit in loc.\ cit., and is made explicit in \cite[Lemma 3.9]{LMW23} and~\cite[Thm.~4]{KKL-Mean}.     
\end{proof}

Recall that we cover $\mathbb S^2$ with disjoint half open cubes $\{\Omega_{i,\vare}'\}$ of 
size between $\vare^2$ and $\vare^2/2$. For $t>0$ and $a\leq c\leq b$, let 
\[
\mathcal N_i(c,e^t)=\#\{v\in \Z^3: v\in ((1-\vare)e^t, e^t]\cdot\Omega_{i,\vare}', |Q(v)-c|\leq \vare\}
\]
The following lemma relates the lattice point counting to averages considered in Proposition~\ref{prop: equi of f hat}; 
its proof relies on Lemma~\ref{lem: EMM calculus} and Lemma~\ref{lem: calculus}. 

\begin{lemma}\label{lem: siegel quadratic form}
With the notation as in Lemma~\ref{lem: calculus}, let $f_i=f^+_{i}$ also let $\xi_i$ be as in that lemma. 
Then for all $t\gg |\log\vare|$, we have 
\[
\tfrac{2\pi}{\sqrt 2}e^t\int_K\hat f_{i}(a_{t}kg_Q\Gamma)\xi_i(k^{-1}\mathsf e_3)\diff\!k=(1+O(\vare))N_i(c,e^t), 
\]
where the implied constants depend polynomially on $|a|, |b|$, and $\|g_Q\|^{\pm1}$.  
\end{lemma}

\begin{proof}
Let $t=\log T$. By Lemma~\ref{lem: EMM calculus}, we have   
\begin{multline}\label{eq: circle ave J func}
\tfrac{2\pi}{\sqrt2}e^t\int f_i(a_tkv)\xi_i(k^{-1}\mathsf e_3)\diff\!k=\\ 
(1+O(e^{-t}))J_{f_i}(\sqf(v), e^{-t}\|v\|)\tilde\xi_i(\tfrac{v}{\|v\|})+O\Bigl(\Sob(f_i)\Sob(\xi_i) e^{-t}\Bigr) 
\end{multline}
Recall from Lemma~\ref{lem: calculus} that $\Omega_{i,\vare}=(1-\vare, 1]\cdot (g_Q\Omega'_{i,\vare})$. 
By that lemma
\be\label{eq: use calculus lemma}
J_{f}(\sqf(v), e^{-t}\|v\|)\xi_i(\tfrac{v}{\|v\|})=\\(1+O(\vare))1_{\Omega_{i,\vare}}(e^{-t} v)1_{[c-\vare,c+\vare]}(\sqf(v))
\ee

Using the last claim in Lemma~\ref{lem: EMM calculus}, both sides of~\eqref{eq: circle ave J func} 
are zero unless $\|v\|\ll e^t$ and $|\sqf(v)|\ll 1$. Therefore, using Lemma~\ref{lem: gen upp bd}, with $\eta$ small enough so that $\Sob(f_i)\Sob(\xi_i)e^{\eta t}< e^t \vare^{10}$, and summing~\eqref{eq: circle ave J func}, we conclude 
\begin{multline*}
\tfrac{2\pi}{\sqrt 2}e^t\!\int\hat f_i(a_tkg_Q\Gamma)\xi_i(k^{-1}\mathsf e_3)\diff\!k=\tfrac{2\pi}{\sqrt 2}e^t\!\int\!\! \sum_{v\in g_Q\Z^3} f_{i}(a_tkv)\xi_i(k^{-1}\mathsf e_3)\diff\!k=\\(1+O(e^{-t}))\sum J_{f_i}(\sqf(v), e^{-t}\|v\|)\tilde\xi_i(\tfrac{v}{\|v\|})+O\Bigl(\Sob(f_i)\Sob(\xi_i) e^{-t}\Bigr)
\end{multline*}
This and~\eqref{eq: use calculus lemma} imply that
\[
\tfrac{2\pi}{\sqrt 2}e^t\!\int\hat f_i(a_tkg_Q\Gamma)\xi_i(k^{-1}\mathsf e_3)\diff\!k=(1+O(\vare))N_i(c, e^t),
\]
as we claimed. 
\end{proof}


\section{Proofs of Theorem~\ref{thm: Oppenheim} and Theorem~\ref{thm: quantitative Oppenheim}}\label{sec: proof of Oppenheim}
In this section, we will use Proposition~\ref{prop: equi of f hat} and Lemma~\ref{lem: siegel quadratic form} 
to complete the proof of Theorem~\ref{thm: quantitative Oppenheim}.  
A similar, but simpler, argument based on Proposition~\ref{prop: lower bd of ave f hat} and and Lemma~\ref{lem: siegel quadratic form} will also be used to complete the proof of Theorem~\ref{thm: Oppenheim}.

\subsection{Proof of Theorem~\ref{thm: quantitative Oppenheim}}\label{sec: proof of Oppenheim-equidistribution}
Let $A,\delta$ be constants as in the statement of that theorem. 
Fix a large parameter $T$, and put $t=\log T$.  

The following is our standing assumption in this section: for all $Q'\in \Mat_3(\Z)$ with $\norm{Q'}\leq T^\delta$ and all $\lambda\in\R$, 
\be\label{eq: Diophantine cond proof}
\norm{Q-\lambda Q'}\geq \norm{Q'}^{-A}.  
\ee

\begin{lemma}\label{lem: exceptional subspaces ell}
There exists an absolute constant $M'$, so that for all $\varrho<\delta/M'$ and all $\bar E>5A+5M'\varrho$.
There are at most four lines $\mathcal L=\{L_i\}$ and at most four planes $\mathcal P=\{P_i\}$ 
so that if for some $t/5\leq \ell\leq t$ any line $L$ or plane $P$ is $(\varrho, \bar E, \ell)$-exceptional then $L\in\mathcal L$ and $P\in\mathcal P$.   
\end{lemma}

\begin{proof}
We prove this for lines, the proof for planes is similar. Let $t/5\leq \ell\leq t$. Then $e^{\varrho \ell}\leq e^{\varrho t}$ and $e^{-\bar E\ell}\leq e^{-\bar Et/5}$. Recall from Definition~\ref{def:exceptional line} that a line $L$ is $(\varrho, \bar E, \ell)$-exceptional if $L\cap \Z^3$ is spanned by $v$ where 
\[
\|v\|\leq e^{\varrho\ell}\quad\text{and}\quad |Q(v)|\leq e^{-\bar E\ell}
\] 
Hence any such line is $(\varrho, \bar E/5, t)$-exceptional. 

Applying Lemma~\ref{lem: five vectors} with $\eta^{-D}=e^{\varrho t}$, $\eta^A=e^{-\varrho \bar E t/5}$, and $Q$,
now implies that if there are at least five $(\varrho, \bar E/5, t)$-exceptional lines, then  
\[
\|Q-\lambda Q'\|\leq \eta^{A-M' D}=e^{(-\varrho \bar E +5M'\varrho)t/5}
\] 
for some $Q'\in \Mat_3(\Z)$ with $\|Q'\|\leq e^{M'\varrho t}$, where $M'$ is absolute. 
Assuming $\bar E$ and $\varrho$ are chosen as specified in the statement for this $M'$, we derive a contradiction to~\eqref{eq: Diophantine cond proof}.
\end{proof}

In what follows we will apply Proposition~\ref{prop: equi of f hat} with $E=\max(AM, 5A+5)$ and $\rho=\delta$. 
Then Lemma~\ref{lem: exceptional subspaces ell} implies that there are most four lines $\{L_i\}$ and at most four planes $\{P_i\}$ so that for any $t/5\leq \ell\leq t$, any $(\varrho_1/4E, E,\ell)$-exceptional line or plane belongs to $\{L_i\}$ or $\{P_i\}$, respectively. Put $\mathcal L=\cup L_i$ and $\mathcal P=\cup P_i$.

The following basic lattice point count will be used in the argument  
\be\label{eq: contribution of vectors < et}
\#\{v\in \Z^3: \|v\|\leq e^{t/4}\}\leq C_1'e^{3t/4},
\ee
where $C_1'$ is absolute. 

Let $\vare=e^{-\hat\kappa t}$ for some $\hat\kappa$ which will be optimized later. 
For all $t/5\leq \ell\leq t$ and $a\leq c\leq b$, we put
\[
\Upsilon_i(c,e^\ell)=\bigl\{v\in \Z^3: v\in ((1-\vare)e^\ell, e^\ell]\cdot\Omega_{i,\vare}', |Q(v)-c|\leq \vare\bigr\},
\]
and $\mathcal N_i(c,e^\ell)=\#\Upsilon_i(c,e^\ell)$. Similarly, let 
\[
\Upsilon(c,e^\ell)=\bigl\{v\in \Z^3: v\in ((1-\vare)e^\ell, e^\ell]\cdot\mathbb S^2, |Q(v)-c|\leq \vare\bigr\},
\]
and $\mathcal N(c,e^\ell)=\#\Upsilon(c,e^\ell)$.

By Lemma~\ref{lem: siegel quadratic form}, there exist $f_{\ell,i, c}\in C_c^\infty(\R^3)$ and $\xi_i\in C_c^\infty(\mathbb S^2)$ 
so that
\be\label{eq: N-i-c-e-ell}
e^{-\ell}\mathcal N_i(c,e^\ell)=(1+O(\vare))\tfrac{2\pi}{\sqrt2}\!\int \hat f_{\ell,i, c}(a_\ell kg_Q\Gamma)\xi_i(k^{-1}\mathsf e_3)\diff\!k.
\ee
In view of~\eqref{eq: Diophantine cond proof}, applying Proposition~\ref{prop: equi of f hat} with $E$ and $\rho$ as above, 
\begin{multline*}
\int \hat f_{\ell,i, c}(a_\ell kg_Q\Gamma)\xi_i(k^{-1}\mathsf e_3)\diff\!k=\int_K\xi_i(k^{-1}\mathsf e_3)\diff\!k\int_{\R^3} f_{\ell, i, c}\diff\!\Leb \\
+ \mathcal R'_{\ell, i,c} + O(\Sob(f_{\ell, i, c})\Sob(\xi_i)e^{-\varrho_2 \ell})
\end{multline*}
where 
\[
\mathcal R_{\ell, i,c}'=\int\hat f^{\,{\rm sp}}_{\ell, i, c}(a_\ell k g_Q\Gamma)\xi(k)\diff\!k
\]
with 
\begin{align*}
\hat f^{\,{\rm sp}}_{\ell, i, c}(a_\ell k g_Q\Gamma)=\!\!\!\sum_{w\in \Z^3 \cap (\mathcal L\cup \mathcal P)} f(a_\ell k g_Qw)
\end{align*}
and we used $\int_X\hat f_{\ell, i, c}\diff\!m_X=\int_{\R^3} f_{\ell, i, c}\diff\!\Leb$. 

We note that Proposition~\ref{prop: equi of f hat} indeed gives a more precise information where the domain of integration in 
the definition of $\mathcal R'_{\ell, i,c}$ is restricted to  
\[
\mathcal C=\left\{k\in K: \hat f^{\,{\rm sp}}_{\ell, i, c}(a_\ell k g_Q\Gamma)\geq e^{\varrho_1 \ell}\right\}.
\]
However, the integral over the complement of $\mathcal C$ is $O(\Sob(f_{\ell, i, c})\Sob(\xi_i)e^{-\varrho_2 \ell})$, see \eqref{eq: lem L2 bound SL2 used}, and it is more convenient for us here to use the above formulation.  

Using the definitions of $f_{\ell, i, c}$ and $\xi_i$, see Lemma~\ref{lem: calculus}, we have 
\[
\sum_i \int_K\xi_i(k^{-1}\mathsf e_3)\diff\!k\int_{\R^3} f_{\ell, i, c}\diff\!\Leb= C'\cdot \vare^2+ O(\vare^3), 
\]
where the implied constant is $O((1+|c|^\star))$, see also~\cite[Lemma 3.8]{EMM-Upp}. 

\smallskip

Moreover, arguing as in the proof of Lemma~\ref{lem: siegel quadratic form}, but only summing over $v=g_Qw$ for 
$w\in \Z^3 \cap (\mathcal L\cup \mathcal P)$, we conclude that  
\[
e^\ell\cdot \mathcal R'_{\ell, i,c}= (1+O(\vare))\cdot \Bigl(\#(\Upsilon_i(c,e^\ell)\cap(\mathcal L\cup \mathcal P)\Bigr). 
\]
Summing this over all $i$ and using the fact that $\Omega'_{i,\vare}$ are disjoint,  
\be\label{eq: contribution of M}
e^\ell \cdot \mathcal R'_{\ell, c}:=e^\ell \sum_i\mathcal R'_{\ell, i,c}= (1+O(\vare)) \Bigl(\#(\Upsilon(c,e^\ell)\cap(\mathcal L\cup \mathcal P)\Bigr)
\ee

We will, as we may, choose $\hat\kappa$ in the definition $\vare=e^{-\hat\kappa t}$ 
small enough compared to $\varrho_2$, then using~\eqref{eq: N-i-c-e-ell} and the above discussion, 
\[
\mathcal N(c,e^\ell)=C\cdot \vare^2\cdot e^{\ell}+\mathcal R'_{\ell, c}\cdot e^{\ell}+O(\vare^3e^{\ell});
\]
we also used $t/5\leq \ell\leq t$. 

Applying this with $c_j=a+j\vare$ for all $1\leq j\leq \lfloor \frac{b-a}\vare\rfloor$ and all $t/5\leq \ell\leq t$, 
\begin{multline}\label{eq: vare shell num}
\#\bigl\{v\in \Z^3: (1-\vare)e^{\ell}\leq \|v\|\leq e^\ell, a\leq Q(v)\leq b\}=\\ 
\tfrac{C}{2}\cdot \vare\cdot (b-a)\cdot e^\ell+\mathcal R'_\ell\cdot e^\ell+O(\vare^2e^{\ell}),
\end{multline}
where the implied constant is $O((1+|a|+|b|)^\star)$ and $\mathcal R'_\ell=\sum_j\mathcal R'_{\ell, c_j}$. 

Apply~\eqref{eq: vare shell num} with $\ell=(1-j\vare)t$ for $0\leq j\leq \lceil\frac{4}{5\vare}\rceil$. Summing the resulting estimate over   $0\leq j\leq \lceil\frac{4}{5\vare}\rceil$ and using~\eqref{eq: contribution of vectors < et} for $\ell\leq t/4$,
we obtain a bound as stated in Theorem~\ref{thm: quantitative Oppenheim}, but with $e^t\cdot \mathcal R'$ 
instead of $\mathcal R_T$ in the theorem, where $\mathcal R'= \sum_{\ell}\mathcal R'_\ell$.
Note however that in view of~\eqref{eq: contribution of M}, we have 
\[
e^t\cdot \mathcal R' = (1+O(\vare)) \cdot (\#\{v\in\Z^3\cap(\mathcal L\cup \mathcal P): \|v\|\leq e^t, a\leq Q(v)\leq b\})
\]
This establishes the theorem except for the definition of $\mathsf C_Q$, and the estimates~\eqref{eq: size of v and Qv} and~\eqref{eq: size of w and Qw}. See the remark following~\cite[Lemma 3.8]{EMM-Upp} for the definition of $\mathsf C_Q$. 
The proof of~\eqref{eq: size of v and Qv} and~\eqref{eq: size of w and Qw} is the content of the following lemma, which completes the proof. 
\qed

\begin{lemma}\label{lem: exceptional and very exceptional}
Suppose $L$ and $P$ be rational lines and planes, respectively, and let ${\rm Span}\{v\}=L\cap\Z^3$ and ${\rm Span}\{w,w'\}=P\cap\Z^3$. 
Then 
\begin{enumerate}
\item If $|Q(v)|>e^{(-2+4\theta)t}$, then
\[
\#\{u\in \Z^3\cap L: \|u\|\leq e^t, a\leq Q(u)\leq b\}\ll e^{(1-O(\theta))t}
\]
\item If $|Q^*(w\wedge w')|> e^{(-2+4\theta)t}$
\[
\#\{u\in \Z^3\cap P: \|u\|\leq e^t, a\leq Q(u)\leq b\}\ll e^{(1-O(\theta))t}
\]
\end{enumerate} 
\end{lemma} 

\begin{proof}
We will use Lemma~\ref{lem: linear algebra} and an argument based on Lemma~\ref{lem: siegel quadratic form} to prove this. Though, a more hands-on proof is certainly possible. Particularly, for part~(1), we have $Q(u)=n^2Q(v)$ for all $u\in \Z^3\cap L$, which immediately implies the claim in part~(1).    

We will prove part~(2), proof of part~(1) (using the following argument) is similar. 
First note that applying Lemma~\ref{lem: gen upp bd} with $\eta=\theta/5$,  
\begin{multline}\label{eq use non-triv upper bound}
\#\{u\in \Z^3\cap P: \|u\|\leq e^{(1-\frac{\theta}4)t}, a\leq Q(u)\leq b\}\ll\\ e^{(1+\frac\theta5)(1-\frac{\theta}4) t}\ll e^{(1-O(\theta))t}
\end{multline}

In view of~\eqref{eq use non-triv upper bound}, we will consider $(1-\frac\theta2)t\leq \ell\leq t$. 
Let us write $\bar w=w\wedge w'$ and recall that $|Q^*(\bar w)|>e^{(-2+4\theta)t}$. Then  
\[
|Q^*(\bar w)|>e^{(-2+2\theta)\ell}\quad\text{for all $(1-\tfrac\theta2)t\leq \ell\leq t$}.
\]
Therefore, by Lemma~\ref{lem: linear algebra} applied with $\ell$, $\sigma=1-\theta$, and $\delta=\theta/10$, we have 
\be\label{eq: apply not supp null and linear alg} 
\int\|(a_\ell k)^*\bar w\|^{-1}\diff\!k\ll e^{(-2\theta/3)\ell}\ll e^{-\frac\theta2t}.
\ee

Let $\vare=e^{-\hat\kappa t}$ for $\hat\kappa< \theta/10$. Arguing as in the above prove to relate the term $\mathcal R'$ 
to the number of points in $\mathcal L\cup\mathcal P$, we see that 
\[
\#\{u\in \Z^3\cap P: (1-\vare)e^\ell \leq \|u\|\leq e^\ell, a\leq Q(u)\leq b\}
\]
is $\ll \sum_{c,i}\int \hat f_{\ell, i, c}^P(a_tkg_Q\Gamma)\diff\!k$, where 
\[
\hat f^{P}_{\ell, i, c}(a_\ell k g_Q\Gamma)=\!\!\!\sum_{w\in \Z^3\cap P} f(a_\ell k g_Qu).
\] 
Moreover, by a variant of Schmidt's Lemma, we have 
\[
f^{P}_{\ell, i, c}(a_\ell k g_Q\Gamma)\ll \|(a_\ell k)^*\bar w\|^{-1}.
\]
This together with~\eqref{eq: apply not supp null and linear alg}, thus imply 
\[
\#\{u\in \Z^3\cap P: (1-\vare)e^\ell \leq \|u\|\leq e^\ell, a\leq Q(u)\leq b\}\ll e^{-\frac\theta2 t}.
\] 
Summing this over all $\ell=t(1-j\vare)$ for $0\leq j\leq \lfloor \frac\theta{2\vare}\rfloor$, and using $\hat\kappa<\theta/10$ and~\eqref{eq use non-triv upper bound}, the claim in part~(2) follows. 
\end{proof}

\subsection{Proof of Theorem~\ref{thm: Oppenheim}}\label{proof of Oppenheim-lower bd}
We now use a simplified version of the above argument to prove the following 

\begin{thm}\label{thm: Oppenheim-lower bd}
Let $Q$ be an indefinite ternary quadratic form with $\det Q=1$. For all $R$ large enough, depending on $\norm{Q}$, and all $T\geq R^{\ref{a:Oppenheim L}}$ at least one of the following holds.
\begin{enumerate}
    \item Let $a<b$, then we have 
    \begin{multline*}
    \# \Bigl\{v\in\Z^3: \norm{v}\leq T, a\leq Q(v)\leq b\Bigr\}\geq \\  
    \mathsf C_Q(b-a)T + (1+\absolute{a}+\absolute{b})^NTR^{-\ref{k:Oppenheim L}}.
    \end{multline*}
    \item There exists $Q'\in\Mat_3(\R)$ with $\norm{Q'}\leq R$ so that 
    \[
    \norm{Q-\lambda Q'}\leq R^{\ref{a:Oppenheim L}}(\log T)^{\ref{a:Oppenheim L}} T^{-2}\qquad\text{where $\lambda=(\det Q')^{-1/3}$.}
    \]
\end{enumerate}
The constants $N$, $\consta\label{a:Oppenheim L}$, and $\constk\label{k:Oppenheim L}$ 
are absolute, and 
\[
\mathsf C_Q=\int_{L}\frac{\diff\!\sigma}{\norm{\nabla Q}}
\]
where $L=\{v\in\R^3: \|v\|\leq 1, Q(v)=0\}$ and $\diff\!\sigma$ is the area element on $L$.
\end{thm}

\begin{proof}
As noted above, the proof follows steps similar to those in the proof of Theorem~\ref{thm: quantitative Oppenheim} in~\S\ref{sec: proof of Oppenheim-equidistribution}, but replaces the application of Proposition~\ref{prop: equi of f hat} with Proposition~\ref{prop: lower bd of ave f hat}. We will use the notation used in \S\ref{sec: proof of Oppenheim-equidistribution}.    

We will, as we may, assume throughout the argument that for all $Q'\in\Mat_3(\R)$ with $\norm{Q'}\leq R$ and all $\lambda\in\R$
\be\label{eq: Diophantine cond proof 2}
 \norm{Q-\lambda Q'}\leq R^{\bar A}(\log T)^{\bar A} T^{-2}
\ee
for some $\bar A$, which will be determined later in the proof. 
Otherwise, part~(2) in the theorem holds and the proof is complete.

Recall again the following basic lattice point count 
\be\label{eq: contribution of vectors < et 2}
\#\{v\in \Z^3: \|v\|\leq e^{t/4}\}\leq C_1'e^{3t/4},
\ee
where $C_1'$ is absolute. 

Let $\vare=e^{-\hat\kappa t}$ for some $\hat\kappa$ which will be optimized later. 
For all $t/5\leq \ell\leq t$ and $a\leq c\leq b$, we put
\[
\Upsilon_i(c,e^\ell)=\bigl\{v\in \Z^3: v\in ((1-\vare)e^\ell, e^\ell]\cdot\Omega_{i,\vare}', |Q(v)-c|\leq \vare\bigr\},
\]
and $\mathcal N_i(c,e^\ell)=\#\Upsilon_i(c,e^\ell)$. Similarly, let 
\[
\Upsilon(c,e^\ell)=\bigl\{v\in \Z^3: v\in ((1-\vare)e^\ell, e^\ell]\cdot\mathbb S^2, |Q(v)-c|\leq \vare\bigr\},
\]
and $\mathcal N(c,e^\ell)=\#\Upsilon(c,e^\ell)$.

By Lemma~\ref{lem: siegel quadratic form}, there exist $f_{\ell,i, c}\in C_c^\infty(\R^3)$ and $\xi_i\in C_c^\infty(\mathbb S^2)$ 
so that
\be\label{eq: N-i-c-e-ell 2}
e^{-\ell}\mathcal N_i(c,e^\ell)=(1+O(\vare))\tfrac{2\pi}{\sqrt2}\!\int \hat f_{\ell,i, c}(a_\ell kg_Q\Gamma)\xi_i(k^{-1}\mathsf e_3)\diff\!k.
\ee

In view of~\eqref{eq: Diophantine cond proof 2}, and assuming $\bar A$ is large enough, Proposition~\ref{prop: lower bd of ave f hat} implies 
\begin{multline*}
\int \hat f_{\ell,i, c}(a_\ell kg_Q\Gamma)\xi_i(k^{-1}\mathsf e_3)\diff\!k\geq \int_K\xi_i(k^{-1}\mathsf e_3)\diff\!k\int_{\R^3} f_{\ell, i, c}\diff\!\Leb\\+ O(\Sob(f_{\ell, i, c})\Sob(\xi_i)e^{-\kappa \ell})
\end{multline*}
where we used $\int_X\hat f_{\ell, i, c}\diff\!m_X=\int_{\R^3} f_{\ell, i, c}\diff\!\Leb$. 

Using the definitions of $f_{\ell, i, c}$ and $\xi_i$, see Lemma~\ref{lem: calculus}, we have 
\[
\sum_i \int_K\xi_i(k^{-1}\mathsf e_3)\diff\!k\int_{\R^3} f_{\ell, i, c}\diff\!\Leb= C'\cdot \vare^2+ O(\vare^3), 
\]
where the implied constant is $O((1+|c|^\star))$, see also~\cite[Lemma 3.8]{EMM-Upp}. 

We will, as we may, choose $\hat\kappa$ in the definition $\vare=e^{-\hat\kappa t}$ 
small enough compared to $\kappa$, then using~\eqref{eq: N-i-c-e-ell} and the above discussion, 
\[
\mathcal N(c,e^\ell)\geq C\cdot \vare^2\cdot e^{\ell}+O(\vare^3e^{\ell});
\]
we also used $t/5\leq \ell\leq t$. 

Applying this with $c_j=a+j\vare$ for all $1\leq j\leq \lfloor \frac{b-a}\vare\rfloor$ and all $t/5\leq \ell\leq t$,
\begin{multline}\label{eq: vare shell num 2}
\#\bigl\{v\in \Z^3: (1-\vare)e^{\ell}\leq \|v\|\leq e^\ell, a\leq Q(v)\leq b\}\geq\\ 
\tfrac{C}{2}\cdot \vare\cdot (b-a)\cdot e^\ell+O(\vare^2e^{\ell}),
\end{multline}
where the implied constant is $O((1+|a|+|b|)^\star)$. 

Apply~\eqref{eq: vare shell num 2} with $\ell=(1-j\vare)t$ for $0\leq j\leq \lceil\frac{4}{5\vare}\rceil$. Summing the resulting estimate over   $0\leq j\leq \lceil\frac{4}{5\vare}\rceil$ and using~\eqref{eq: contribution of vectors < et 2} for $\ell\leq t/4$,
we obtain a bound as stated in part~(1) of the theorem. For the definition of $\mathsf C_Q$, see the remark following~\cite[Lemma 3.8]{EMM-Upp}. 
\end{proof}

\begin{proof}[Proof of Theorem~\ref{thm: Oppenheim}]
This is a direct consequence of of Theorem~\ref{thm: Oppenheim-lower bd}. Indeed, let 
$\ref{k:Oppenheim}=\ref{k:Oppenheim L}/2N$ and let $\ref{a:Oppenheim}=\ref{k:Oppenheim L}$. 
Now Theorem~\ref{thm: Oppenheim-lower bd} applied with 
$[a, b]=[c-R^{-\ref{k:Oppenheim}}, c+R^{-\ref{k:Oppenheim}}]$ for any $|c|\leq R^{\ref{k:Oppenheim}}$ implies 
Theorem~\ref{thm: Oppenheim}.    
\end{proof}

\begin{proof}[Proof of Corollary~\ref{cor: Oppenheim}]
This is a direct consequence of Theorem~\ref{thm: Oppenheim}. 
For more details, see~\cite[\S12.3]{LM-Oppenheim}.
\end{proof}

\bibliographystyle{halpha}
\bibliography{papers}

\newpage

\part*{Appendices}

\appendix

\section{Unipotent trajectories and parabolic subgroups}\label{sec: parabolic}

In this section $\tilde\G\subset \SL_N$ denotes a semisimple $\Q$-subgroup. 
We let $\tG=\tilde\G(\R)$ and $\tilde\gfrak=\Lie(\tG)$. Also let $\tilde\gfrak_\Z:=\tilde\gfrak\cap\sl_N(\Z)$, 
then $\tilde\gfrak$ has a natural $\Q$-structure, and $\tilde\gfrak_\Z$ is a $\tG\cap \SL_N(\Z)$-stable lattice in $\tilde\gfrak$.
If there is no confusion, we will simply write $gv$ for $\Ad(g)v$; similarly for the natural actions on $\wedge^\ell\tilde\gfrak$.

Fix a Euclidean norm $\|\cdot\|$ on $\Mat_N(\R)$. This induces a norm on $\mathfrak{sl}_N(\R)$ and on $\SL_N(\R)$.
We will also write $\|\cdot\|$ for the induced norms on exterior products of $\mathfrak{sl}_N(\R)$. For $g \in \SL_N(\R)$ we let
$|g| = \max\{\|g\|,\|g^{-1}\|\}$. 

Let $\tilde\Gamma\subset \tG\cap\SL_N(\Z)$ be an arithmetic lattice in $\tG$; let $\tilde X=\tG/\tilde\Gamma$, and  
\[
\tilde X_{\eta} = \{g\tilde\Gamma \in \tilde X: \min_{0 \neq v\in \tilde\gfrak(\Z)} \|gv\| \geq \eta\} \quad\text{for all $\eta>0$.}
\]
These are compact subsets of $\tilde X$, and any compact subset of $\tilde X$ is contained in $\tilde X_\eta$ for some $\eta>0$.

If ${\bf L}\subset \tilde\G$ is a connected $\Q$-subgroup, 
we let ${\bf v}_L$ be a primitive integral vector on the line 
$\wedge^{\dim L}\Lie(L)\subset\wedge^{\dim L}\tilde\gfrak$ where $L=\Lbf(\R)$.  
Recall from~\cite{LMMS} the definition of the height of ${\bf L}$ 
\be\label{eq: def height v-L'}
{\rm ht}(\Lbf)=\|{\bf v}_L\|.
\ee

We also recall the following setting from~\cite{LMMS}. {\em This notation will only be used in this section}. 
Let $U\subset \tG$ be a unipotent subgroup and let $\ufrak=\Lie(U)$.
We fix a basis $\mathcal{B}_U$ of $\ufrak$ consisting of unit vectors 
and set $B_\ufrak(0,\delta) = \{\sum_{z \in \mathcal{B}_U} a_z z: |a_{z}| \leq \delta\}$ for $\delta>0$ as well as 
$B_U(e) = \exp(B_\ufrak(0,1))$.

Let $\lambda: \ufrak \to \ufrak$ be an $\R$-diagonalizable expanding linear map (all eigenvalues have absolute value $>1$).
For any $k\in\Z$ and any $u = \exp(z)\in U$, we set $\lambda_k(u) = \exp(\lambda^k(z))$. We note that $\lambda_k\circ \lambda_\ell=\lambda_{k+\ell}$. 
We shall assume that there exists $k_0 \in \mathbb N$ such that for every integer $k>k_0$,
\be\label{eq: U lambda prop 1}
\exp\big(\lambda_{k-k_0}(B_\ufrak(0,1))\big)\exp\big(\lambda_{k-1}(B_\ufrak(0,1))\big)
\subset \exp\big(\lambda_{k}(B_\ufrak(0,1))\big).
\ee
Since the exponential map $\exp:\ufrak\to U$ pushes the Lebesgue measure on $\ufrak$ to a Haar measure, denoted by $|\cdot|$, on $U$, for any measurable $B\subset U$ 
\be\label{eq: U lambda prop 2}
    |\lambda_k(B)|=|\det(\lambda)|^k|B|\qquad\text{for all $k\in\Z$.}
\ee
To avoid cumbersome statements, we suppose throughout that any constant that is allowed to depend on $\height(\tilde\G)$ is also (implicitly) allowed to depend on $\norm{\lambda}$, $\norm{\lambda^{-1}}$, $\frac{|\lambda_1(B_U(e))|}{|B_U(e)|}=|\det(\lambda)|$, and $k_0$.

\begin{thm}\label{thm: non-div parabolic}
There exist $\consta\label{a: parabolic linear}$ depending on $N$ and $\consta\label{a: parabolic linear 2}$ depending on $N$ and polynomially on $\height(\tilde\G)$ so that for any $g\in\tilde G$, $k\geq 1$, and any $0<\vare\leq 1/2$ at least one of the following holds.
\begin{enumerate}
\item 
\[
|\{u\in B_U(e): \lambda_k(u)g\tilde\Gamma\not\in \tilde X_{\vare}\}|\leq \ref{a: parabolic linear 2} \vare^{1/\ref{a: parabolic linear}}|B_U(e)|.
\]

\item There is a $\Q$-parabolic subgroup ${\bf P}\subset \tilde\G$ with 
$\height(\mathbf P) \leq \ref{a: parabolic linear 2} |g|^{\ref{a: parabolic linear}} \vare^{1/\ref{a: parabolic linear}}$ so that all the following are satisfied 
\begin{subequations}
\begin{align}
\label{eq: parabolic almost fixed}&\|\la_k(u)g{\bf v}_{P}\|\leq \ref{a: parabolic linear 2} \vare^{1/\ref{a: parabolic linear}}\quad &\text{ for all $u\in B_U(e)$},\\
\label{eq: radical parabolic almost fixed}&\|\la_k(u)g{\bf v}_{R_u(P)}\|\leq \ref{a: parabolic linear 2} \vare^{1/\ref{a: parabolic linear}}\quad &\text{ for all $u\in B_U(e)$},\\
\label{eq: class H parabolic almost fixed}&\|\la_k(u)g{\bf v}_{Q}\|\leq \ref{a: parabolic linear 2} \vare^{1/\ref{a: parabolic linear}}\quad &\text{ for all $u\in B_U(e)$},
\end{align}
\end{subequations}
where $Q={\bf Q}(\R)$ for ${\bf Q}=[{\bf P}, {\bf P}]$.  
\end{enumerate}
\end{thm}

This result strengthens~\cite[Thm.~6.3]{LMMS}. While the original theorem only established the existence of a unipotent subgroup, we now demonstrate that this subgroup may, in fact, be taken to be the unipotent radical of a $\Q$-parabolic subgroup. 
Indeed we will prove Theorem~\ref{thm: non-div parabolic}, using ~\cite[Thm.~5.3]{LMMS}, which was also used in the proof of~\cite[Thm.~6.3]{LMMS}, and the following theorem, which is of independent interest.

\smallskip

For every subspace $V\subset\tilde\gfrak$, let $V^\perp$ denote the orthogonal complement of $V$ with respect to the Killing form $\mathsf k_{\tilde\gfrak}$ on $\tilde\gfrak$. 

\begin{thm}\label{thm: generating parabolic}
There are positive constants $\delta, C$, depending on $\tilde\G$, such that if $g \in \tG$ 
is such that there exists a vector $v \in \tilde{\mathfrak{g}}_\mathbb{Z}$ with $\|g v\| \leq \delta$, then 
\begin{enumerate}
\item $\{w\in\tilde\gfrak_\Z: \|gw\|\leq \delta^{1/C}\}$ generates a nilpotent subalgebra, $\mathfrak v$ say.
\item $\mathfrak v^\perp\cap \Ad(g)\tilde\gfrak_Z$ is spanned by vectors of size $\leq C\delta^{-1/C}$.
\item  $\mathfrak v^\perp$ generates a proper parabolic subalgebra of $\tilde\gfrak$. 
\end{enumerate}
\end{thm}

Let us first assume Theorem~\ref{thm: generating parabolic} and complete the proof of Theorem~\ref{thm: non-div parabolic}.

\begin{proof}[Proof of Theorem~\ref{thm: non-div parabolic}]
Let $E$ and $D$ be constants as in~\cite[Thm.~5.3]{LMMS}. In particular, $D$ depends only on $N$, and $E$ on $N$ and $\height (\tilde \G)$. We will show that 
Theorem~\ref{thm: non-div parabolic} holds with appropriately chosen $\ref{a: parabolic linear 2}\geq E$ 
and $\ref{a: parabolic linear}\geq D$.  
Assume
\[
|\{u\in B_U(e): \lambda_k(u)g\tilde\Gamma\not\in \tilde X_{\vare}\}|>E \vare^{1/D}|B_U(e)|. 
\]
Then~\cite[Thm.~5.2]{LMMS} implies that there is a rational subspace ${\bf W}\subset \tilde\gfrak$, say of dimension $\ell$, 
so that
\be\label{eq: ur orbit of vW'}
\sup_{u\in B_U(e)}\|\lambda_k(u)g{\bf v}_{W}\|=\hat\vare\ll\vare^\star, 
\ee
where ${\bf v}_{W}$ denotes the primitive vector corresponding to ${\bf W}$. 

For later use, let us record that Minkowski's theorem and the above imply that for every $u\in B_U(e)$, there exists a nilpotent vector $z_u\in\tilde\gfrak_\Z$ with 
\be\label{eq: the element zu is small}
\|\lambda_k(u)gz_u\|\leq \hat\vare^{1/\dim\tilde\gfrak}.
\ee

Let the partial flag $\Delta_{\ell_1}\subset\cdots \subset\Delta_{\ell_d}$ and the function $\eta$ be given as in~\cite[Thm.~5.3]{LMMS}. 
In particular, $\rank(\Delta_{\ell_i})=\ell_i$ and  
$\eta(\ell_0),\dots,\eta(\ell_{d+1}) \in (0,1]$, with $\ell_0=0$ and $\ell_{d+1}=\dim\tilde\gfrak$, is defined by
\begin{align*}
\eta(\ell_0)&=\eta(\ell_{d+1})=1 \\
\eta(\ell_i) &= \max_ {u \in\mathsf B_U(e)} \|\la_k(u) \Delta_{\ell_i}\| \qquad\text{for $1\leq i \leq d$};
\end{align*}
 as shown in \cite{LMMS}, the function $\eta(\bullet)$ extends to a function $\eta: [0,\dim\tilde\gfrak]\to(0,1]$ so that $-\log \eta: [0,\dim\tilde\gfrak]\to \R^+$ is concave and linear on each interval $[\ell_0, \ell_{1}]$, \dots, $[\ell_{d},\ell_{d+1}]$.

For $0<\tau<1$, put 
\[
{\rm Exc}_\tau=\{u\in B_U(e): \alpha_i(\lambda_k(u)g)^{-1}<\tau^{i}\eta(i)\},
\]
where $\alpha_i$'s are also as in loc.\ cit. Then by~\cite[Thm.~5.3]{LMMS}, we have 
$|{\rm Exc}_\tau|\leq E\tau^{1/D}|B_U(e)|$, where $D$ depends only on $\dim\tilde\gfrak$. 

Moreover, by the last sentence in ~\cite[Thm.~5.3]{LMMS}, we can assume that 
 $\Delta _ {\ell_1} \subset \dots \subset \Delta _ {\ell_d}$ is so that 
\[
\eta(\rank(\Delta)) \leq \max_ {u \in\mathsf B_U(e)} \|(\la_k(u) \Delta)\|,
\]
where $\Delta={\bf W}\cap \tilde\gfrak_\Z$.

We now argue as in the proof of~\cite[Lemma 6.2]{LMMS}. 
Since $-\log\eta$ is concave, we have $\eta(1)\leq \hat\vare^{1/\ell}$. 
Let $\kappa=(4\ell\dim\tilde\gfrak)^{-1}$, 
and let $\hat\vare^{1/\ell}<\rho<1$ be a parameter which will be optimized later.
Let $1\leq s_0<d$ be the largest integer so that $\frac{\eta(s_0)}{\eta(s_0-1)}\leq \rho$. 
The choice of $\kappa$ and $\eta(\dim\tilde\gfrak)=1$ imply that there is some $m\geq 0$ so that
$\frac{\eta(s_0+d)}{\eta(s_0+d-1)}\leq \rho^{1-d\kappa}$ for all $0\leq d\leq m$ and  
\be\label{eq: condition on s0+m+1}
\tfrac{\eta(s_0+m+1)}{\eta(s_0+m)}\geq \rho^{1-(m+1)\kappa}.
\ee 
Put $s=s_0+m$, and note that $s=\ell_j$ for some $j$. Let $\bf V$ be the subspace spanned by $\Delta_{s}$, then 
\[
\|\lambda_k(u)g{\bf v}_{V}\|\leq \eta(s)\leq \rho^{1-m\kappa}\qquad\text{for all $u\in B_U(e)$} 
\] 
Moreover, ${\bf V}$ is a nilpotent subalgebra of $\tilde\gfrak$, 
and we can choose a basis $\{v_1,\ldots, v_s\}$ for ${\bf V}\cap\tilde\gfrak_\Z$ so that 
\be\label{eq: short basis for V}
\|\lambda_k(u)g v_i\|\leq \hat A \frac{\eta(i)}{\eta(i-1)}\leq \hat A \rho^{1-m\kappa},
\ee 
where $\hat A$ depends only on $\dim\tilde\gfrak$, see the proof of~\cite[Lemma 6.2]{LMMS} for all these statements.  

Put $\tau=\rho^{\kappa/(2\dim\tilde\gfrak)}$. 
Then for all $u\in B_U(e)\setminus {\rm Exc}_\tau$ and any $w\in\tilde\gfrak_\Z\setminus{\bf V}$, 
\begin{align*}
\tau^{s+1}\eta(s+1)&\leq \|\lambda_k(u)g(\Delta_s+\Z w)\|\\
&\leq \|\lambda_k(u)g w\|\|\lambda_k(u)g\Delta_{s}\|\leq \|\lambda_k(u)g w\|\eta(s).
\end{align*}
This and~\eqref{eq: condition on s0+m+1} imply that 
\be\label{eq: vectors outside are large}
\begin{aligned}
   \|\lambda_k(u)g w\|&\geq \tau^{s+1}\rho^{1-(m+1)\kappa}\\
   &>\rho^{1-(m+\frac12)\kappa} > \hat A \rho^{1-m\kappa}.  
\end{aligned}
\ee
In view of~\eqref{eq: short basis for V} and~\eqref{eq: vectors outside are large}, for every $u\in B_U(e)\setminus {\rm Exc}_\tau$,
the space 
\[
\Ad(\lambda_k(u)g)(\tilde\gfrak_\Z\cap {\bf V})
\]
is spanned by $\{w\in\tilde\gfrak_\Z: \|\lambda_k(u)gw\|\leq \hat A\rho^{\theta}\}$. 

Apply Theorem~\ref{thm: generating parabolic} with $\hat A \rho^{1-m\kappa}$ and $\lambda_k(u)g{\bf V}$. Then part~(2) of that theorem implies that $\Bigl(\Ad(\lambda_k(u)g)(\tilde\gfrak_\Z\cap {\bf V})\Bigr)^\perp$ generates a proper parabolic subalgebra, $\Ad(\lambda_k(u)g)\Lie(P)$ say. Moreover, if we choose $\rho$ to be a small enough power of $\vare$, then~\eqref{eq: the element zu is small} and part~(2) in Theorem~\ref{thm: generating parabolic} imply that   
\[
\|\lambda_k(u)g{\bf v}_P\|\leq \rho^{-\star}\hat\vare^{1/\dim\tilde\gfrak}\ll \vare^\star\qquad\text{for all $u\in B_U(e)\setminus {\rm Exc}_\tau$}. 
\]
Since $u\mapsto\lambda_k(u)g{\bf v}_P$ is a polynomial map and $|{\rm Exc}_\tau|\ll \tau^{1/D}|B_U(e)|$, 
we conclude that 
\[
\|\lambda_k(u)g{\bf v}_{P}\|\ll\vare^\star\qquad\text{for all $u\in B_U(e)$}.
\] 
This finishes the proof of~\eqref{eq: parabolic almost fixed}. 

To establish the claims in~\eqref{eq: radical parabolic almost fixed} and~\eqref{eq: class H parabolic almost fixed}, we use~\eqref{eq: parabolic almost fixed} and follow the argument in the proof of~\cite[Lemma 4.2]{LMMS}. Specifically, we show that for all $u\in B_U(e)\setminus{\rm Exc}_\tau$,  $\Ad(\lambda_k(u)g)\Lie(Q)$ and $\Ad(\lambda_k(u)g)\Lie(R_u(P))$ are spanned by vectors of size $\ll \rho^{-\star}$. Combined with~\eqref{eq: the element zu is small}, this completes the proof, provided that $\rho$ is chosen to be a sufficiently small power of $\vare$.
\end{proof}

\subsection{Proof of Theorem~\ref{thm: generating parabolic}}
We continue to use the above notation. The proof of Theorem~\ref{thm: generating parabolic}, relies on the following two lemmas. 

\begin{lemma}\label{lem: parabolic 1}
Let $\mathfrak{v} \subset \tilde{\mathfrak{g}}$ be a nilpotent Lie subalgebra. 
Then there exists a parabolic subgroup $P$ such that
\[
\mathfrak{v} \subset \Lie(R_u(P)) \subset \Lie(P) \subset \mathfrak{v}^\perp.
\]
\end{lemma}

\begin{proof}
Let $\mathfrak{v}_0=\mathfrak{v}$, and for $i\geq 1$, define $\mathfrak v_i$ to be the niladical of $N_{\tilde\gfrak}(\mathfrak v_{i-1})$. 
Then there exists some $m$ so that $\mathfrak v_m=\Lie(R_u(P))$ for a parabolic subgroup $P$ of $\tilde G$.    
Since $\mathfrak{v}_i\subset \mathfrak v_{i+1}$, we have
\[
\mathfrak{v}\subset \Lie(R_u(P)).
\]    
This and the definition of the Killing form, $\mathsf k_{\tilde\gfrak}(v,w)={\rm tr}({\rm ad} v\circ{\rm ad} w)$, now imply that $\Lie(P) \subset \mathfrak{v}^\perp$. The proof is complete.   
\end{proof}

\begin{lemma}\label{lem: parabolic 2}
There is a positive constant $C'$, depending on $\dim\tilde\gfrak$, so that for every $A \geq 2$ and $\delta\leq A^{-C'}$ the following holds. 
If there is a vector $v \in \tilde\gfrak_\Z$ with $\|gv\| \leq \delta$, then $\{w\in\tilde\gfrak_\Z: \|gw\|\leq A\}$ 
generates a proper Lie subalgebra of $\tilde{\mathfrak{g}}$. 
\end{lemma}

\begin{proof}
Let $\mathfrak l$ denote the Lie algebra spanned by $\{w\in\tilde\gfrak_\Z: \|gw\|\leq A\}$, and let $g{\bf L}g^{-1}$ be the corresponding subgroup. Then $\mathfrak l$ is spanned by vectors of size $\ll A^\star$, where the implied constants depend only on $\dim\tilde\gfrak$. Moreover $gv\in\mathfrak l$. Thus if $\delta$ is small enough, we have 
\be\label{eq: A and delta parabolic}
\|g{\bf v}_L\|\ll A^\star\delta \leq \delta^{1/2},
\ee
which in particular implies that $\mathfrak l\neq \tilde\gfrak$. 
\end{proof}

\begin{proof}[Proof of Theorem~\ref{thm: generating parabolic}]
Let $\delta_1 =\delta^\star$ be a small power of $\delta$ which will be explicated later. 
Let us also put $\Lambda := \Ad(g)\tilde\gfrak_\Z$. Then so long as $\delta$ is small enough, 
$\{w\in\Lambda: \|w\|\leq \delta_1\}$ generate over $\mathbb{Z}$ 
a lattice in a nilpotent Lie subalgebra $\mathfrak{v}\subset\tilde\gfrak$. This establishes part~(1) in the theorem.  

We now prove part~(2). That is: $\mathfrak{v}^\perp \cap \Lambda$ is spanned by elements of size $\ll \delta_1^{-1}$. 
Consider the vectors $v_1, v_2, \ldots, v_n$ that attain the $i$th successive minima of $\Lambda$, respectively. 
Let $m = \dim\mathfrak{v}$. Then $v_1,\ldots, v_m \in \mathfrak{v}$, and furthermore, $\|v_i\| > \delta_1$ holds for all $i > m$. 

Let $v_1^\ast,\ldots, v_n^\ast$ be the dual basis with respect to the Killing form $\mathsf k_{\tilde\gfrak}$. 
Then $\langle v_1^*,\ldots,v_n^* \rangle \subset \Lambda/M$ where $M \ll 1$. Since $\mathsf k_{\tilde\gfrak}(v_i, v_j^\ast)= \delta_{ij}$, 
we deduce that $\|v_i^\ast\| \gg \|v_i\|^{-1}$. Moreover, $\prod \|v_i\| \asymp \prod \|v_i^\ast\|^{-1}$, therefore, $\|v_i^\ast\| \asymp \|v_i\|^{-1}$.

The orthogonal complement $\mathfrak{v}^\perp$ is spanned by the vectors $v^*_{m+1},\ldots, v^*_n$, and therefore, increasing $M$ if necessary, it is spanned by $\{w\in \Lambda: \|w\| \leq M\|v_{m+1}\|^{-1}\}$. In other words, it is spanned by vectors in $\Lambda$ of size $\ll \delta_1^{-1}$.

We now prove part~(3) in the theorem. 
By applying Lemma~\ref{lem: parabolic 2}, see in particular~\eqref{eq: A and delta parabolic}, 
if $\delta_1$ is a small enough power of $\delta$, the Lie algebra generated by $\mathfrak{v}^\perp$ cannot be the entire algebra. Furthermore, by Lemma~\ref{lem: parabolic 1}, we know it must contain a parabolic subalgebra which contains $\mathfrak{v}$. Therefore, $\langle \mathfrak{v}^\perp \rangle$ is indeed a nontrivial parabolic subalgebra containing $\mathfrak{v}$.
\end{proof}

\subsection{An $S$-arithmetic version of Theorem~\ref{thm: non-div parabolic}}

We conclude this appendix with noting that Theorem~\ref{thm: non-div parabolic} can easily be adapted to the $\places$-arithmetic setting as in \cite{LMMS}, the only difference being that now $\ref{a: parabolic linear 2}$ can depend also on $\#\places$ and polynomially on the primes in $\places$. For the record we state this explicitly; the straightforward modification of the proofs above to this more general context is left to the reader.

 In addition to the semisimple $\Q$-subgroup. $\tilde\G\subset \SL_N$, we now also choose a finite set of places $\places$ of $\Q$ containing the infinite place.
Let $\tG=\tilde\G(\Q_\places)$ and $\tilde\gfrak=\Lie(\tG)$. Also let $\tilde\gfrak_{\Z_\places}:=\tilde\gfrak\cap\sl_N(\Z_\places)$. 
Then $\tilde\gfrak$ has a natural $\Q$-structure, and $\tilde\gfrak_{\Z_\places}$ is a $\tG\cap \SL_N(\Z_\places)$-stable lattice in $\tilde\gfrak$.
If there is no confusion, we will simply write $gv$ for $\Ad(g)v$; similarly for the natural actions on $\wedge^\ell\tilde\gfrak$.

Let $\|\cdot\|_\infty$ denote the Euclidean norm on $\Mat_N(\R)$, and for every finite place $v\in\places$, let $\|\cdot\|_v$ be a maximum norm with respect to the standard basis for $\Mat_N(\Q_v)$ induced by the standard absolute value on $\Q_v$. Let $\|\cdot\|_\places$ or simply $\|\cdot\|$ denote $\max_{v\in\places}\|\cdot\|_v$ This induces a norm on $\mathfrak{sl}_N(\Q_\places)$ and on $\SL_N(\Q_\places)$.
We will also write $\|\cdot\|$ for the induced norms on exterior products of $\mathfrak{sl}_N(\R)$. For $g \in \SL_N(\Q_\places)$ we let
$|g| = \max\{\|g\|,\|g^{-1}\|\}$. 

Let $\tilde\Gamma\subset \tG\cap\SL_N(\Z_\places)$ be an arithmetic lattice in $\tG$; let $\tilde X=\tG/\tilde\Gamma$, and  
\[
\tilde X_{\eta} = \{g\tilde\Gamma \in \tilde X: \min_{0 \neq v\in \tilde\gfrak(\Z)} \|gv\| \geq \eta\} \quad\text{for all $\eta>0$.}
\]
These are compact subsets of $\tilde X$, and any compact subset of $\tilde X$ is contained in $\tilde X_\eta$ for some $\eta>0$.

As it was done in~\cite{LMMS}, 
let $U=\prod_{v\in\places} U_v\subset \tG$ be a unipotent subgroup and let $\ufrak=\Lie(U)$.
We fix a basis $\mathcal{B}_U$ of $\ufrak$ consisting of unit vectors 
and set $B_\ufrak(0,\delta) = \{\sum_{z \in \mathcal{B}_U} a_z z: |a_{z}|_\places \leq \delta\}$ for $\delta>0$ as well as $B_U(e) = \exp(B_\ufrak(0,1))$.

Let $\lambda: \ufrak \to \ufrak$ be a $\Q_\places$-diagonalizable expanding linear map satisfying~\eqref{eq: U lambda prop 1} and~\eqref{eq: U lambda prop 2}.

The proof of Theorem~\ref{thm: non-div parabolic} generalizes {\em mutatis mutandis} to yield the following result.

\begin{thm}\label{thm: non-div parabolic S-arith}
There exist $\consta\label{a: parabolic linear S-arith}$ depending on $N$, and $\consta\label{a: parabolic linear 2 S-arith}$ depending on $N$, $\#\places$ and polynomialy on $\height(\tilde\G)$ and the primes in $\places$ so that for any $g\in\tilde G$, $k\geq 1$, and any $0<\vare\leq 1/2$ at least one of the following holds.
\begin{enumerate}
\item 
\[
|\{u\in B_U(e): \lambda_k(u)g\tilde\Gamma\not\in \tilde X_{\vare}\}|\leq \ref{a: parabolic linear 2 S-arith} \vare^{1/\ref{a: parabolic linear S-arith}}.
\]

\item There is a $\Q$-parabolic subgroup ${\bf P}\subset \tilde\G$ with 
$\height(\mathbf P) \leq \ref{a: parabolic linear 2 S-arith} |g|^{\ref{a: parabolic linear S-arith}} \vare^{1/\ref{a: parabolic linear S-arith}}$ so that all the following are satisfied 
\begin{subequations}
\begin{align}
\label{eq: parabolic almost fixed S-arith}&\|\la_k(u)g{\bf v}_{P}\|\leq \ref{a: parabolic linear 2 S-arith} \vare^{1/\ref{a: parabolic linear S-arith}}\quad &\text{ for all $u\in B_U(e)$},\\
\label{eq: radical parabolic almost fixed S-arith}&\|\la_k(u)g{\bf v}_{R_u(P)}\|\leq \ref{a: parabolic linear 2 S-arith} \vare^{1/\ref{a: parabolic linear S-arith}}\quad &\text{ for all $u\in B_U(e)$},\\
\label{eq: class H parabolic almost fixed S-arith}&\|\la_k(u)g{\bf v}_{Q}\|\leq \ref{a: parabolic linear 2 S-arith} \vare^{1/\ref{a: parabolic linear S-arith}}\quad &\text{ for all $u\in B_U(e)$},
\end{align}
\end{subequations}
where $Q={\bf Q}(\Q_\places)$ for ${\bf Q}=[{\bf P}, {\bf P}]$.  
\end{enumerate}
\end{thm}


\section{Avoidance principles: The proofs}\label{sec: app avoidance}

This section contains proofs of results in \S\ref{sec: avoidance}. 
The proofs are by now standard, and we include them primarily for the convenience of the reader.

\subsection{Proof of Proposition~\ref{prop: linearization translates}}\label{sec: proof linearization}
In this section, we prove Proposition~\ref{prop: linearization translates}. The proof, which is essentially that of~\cite[Prop.\ 4.6]{LMW22} mutatis mutandis, is based on the study of a certain Margulis function, see~\eqref{eq: define f_Y app A}. 
We recall the details to explicate the necessary changes. 

For every $d>0$, define the probability measure $\sigma_d$ on $H$ by
\[
\int \varphi(h)\diff\!\sigma_d(h)=\frac{1}{3}\int_{-1}^2\varphi(a_du_r)\diff\!r.
\]
Let us first remark on our choice of the interval $[-1,2]$: 
We will define a function $f_Y$ in~\eqref{eq: define f_Y app A} below. In Lemmas~\ref{lem: linear algebra}--\ref{lem: Margulis func periodic}, certain estimates for 
\[
\int f_Y(h\,\bigcdot)\diff(\sigma_{d_1}\!\conv\!\cdots\!\conv\!\sigma_{d_n})(h)
\]
will be obtained, then in Lemma~\ref{lem: average over [0,1]}, we will convert these estimates to similar estimates for 
\[
\int_0^1 f_Y(a_{d_1+\cdots+d_n}u_r\,\bigcdot)\diff\!r.
\]
The argument in Lemma~\ref{lem: average over [0,1]} is based on commutation relations between $a_d$ and $u_r$; similar arguments have been used several times throughout the paper. Since the function $f_Y$ can have a rather large Lipschitz constant, we will not appeal to continuity properties of $f_Y$ in Lemma~\ref{lem: average over [0,1]}. 
Instead, we will use the fact that $[0,1]\subset[-1,2]+r$ for any $|r|\leq 1/2$. 

Recall that $\dim\rfrak=2\dm+1$, and that $\rfrak$ is $\Ad(H)$-irreducible. 
Moreover, recall that for every highest weight vector $w\in\rfrak$, we have 
\be\label{eq: highest weight App}
\Ad(a_t)w=e^{\dm t}w. 
\ee

\medskip

We begin with the following linear algebra lemma.

\begin{lemma}[cf.~Lemma A.1,~\cite{LMW22}]
\label{lem: linear algebra App}
Let $0<\alpha\leq 1/(2\dm +1)$. For all $0\neq w\in\rfrak$, we have 
\[
\int\|\Ad(h)w\|^{-\alpha}\diff\!\sigma_d(h)\leq C e^{-\alpha\dm d}\|w\|^{-\alpha }
\]
where $C$ is an absolute constant.  
\end{lemma}

\begin{proof}[Sketch of the proof]
Standard representation theory of $\SL_2$ implies that for every $\|w\|=1$, we have 
\[
\Ad(a_tu_r)w=(e^{\dm t}p_{w}(r),\ldots)
\]
where $p_w(r)=\sum c_{i}r^i$ is a polynomial of degree $2\dm$ with $\max\{|c_i|\}\gg 1$. Thus if we put 
\[
I(\vare)=\{r\in [-1,2]: \vare/2\leq p_w(r)\leq \vare\},
\]
then $|I(\vare)|\ll\vare^{1/(2\dm)}$ where the implied constant is absolute, see~e.g.~\cite[Prop.~3.2]{KM-Nondiv}.

Moreover, $\|\Ad(a_tu_r)w\|\geq e^{\dm t}{\vare}/2$, for every $r\in I(\vare)$. 
The lemma follows if we dyadically decompose the domain and summing the resulting geometric series, which converges independently of $\alpha$ since $\alpha\leq \frac{1}{2\dm+1}<\frac1{2\dm}$.   
\end{proof}

We will also use the following non-divergence result \'a la Eskin, Margulis, and Mozes~\cite{EMM-Upp}.

\begin{propos}[cf.~\cite{EM-RW}]\label{prop: average of inj}
There exist $\consta\label{a:alpha},\constk\label{k:alpha}$ depending only on the dimension, and 
a function $\omega:X\to [2,\infty)$, so that both of the following hold for all $x\in X$ and all $d\geq B_\omega$
\begin{subequations}
    \begin{align}
        \label{eq: alpha and inj}&\inj(x)\geq \omega(x)^{-\ref{a:alpha}} \\
        \label{eq: average alpha} &\int \omega(hx)\diff\!\sigma_d^{(\ell)}(h)\leq e^{-\ref{k:alpha}\ell d}\omega(x)+\bar B e^{\ref{a:alpha} d}
    \end{align}
\end{subequations}
where $\sigma_d^{(\ell)}$ denotes the $\ell$-fold convolution and $\bar B\geq 1$ depends only of $X$.
\end{propos}

\begin{proof}
A function with these properties is constructed in~\cite{EM-RW}. More explicitly, see~\cite[Prop.~26]{SanchezSeong} 
for~\eqref{eq: alpha and inj} and~\cite[Thm.~16]{SanchezSeong} for~\eqref{eq: average alpha}. 
\end{proof}

Let $Y=Hy$ be a periodic orbit. 
For every $x\in X\setminus Y$, define 
\[
I_Y(x)=\{w\in \rfrak: 0<\|w\|< \ref{a:alpha}^{-1}\omega(x)^{-\ref{a:alpha}}, \exp(w)x\in Y\}.
\]  
Increasing $\ref{a:alpha}$ if necessary, we have 
\be\label{eq: number of sheets}
\#I_Y(x)\leq E\vol(Y)
\ee
for a constant $E$ depending only on $X$, see~\cite[Prop.~25]{SanchezSeong} also~\cite[\S9]{LM-PolyDensity}. 

Again increasing $\ref{a:alpha}$ if necessary, for every $h=a_du_r$ with $d\geq 0$, 
all $r\in[-1,2]$, and for all $w\in\mathfrak g$ and all $x\in X$, we have  
\begin{subequations}
    \begin{align}
        \label{eq: cont Ad app B}&\|\Ad(h^{\pm1})w\|\leq \ref{a:alpha} e^{\ref{a:alpha}d/2}\|w\|,\\
        \label{eq: cont inj app B}
&\ref{a:alpha}^{-1}e^{-\ref{a:alpha}d/2}\omega(x)\leq \omega(h^{\pm1}x)\leq \ref{a:alpha}e^{\ref{a:alpha}d/2}\omega(x). 
    \end{align}
\end{subequations}

Let $\alpha=\min\{1/(2\dm+1), 1/\ref{a:alpha}\}$, furthermore, we will replace $\ref{k:alpha}$ with a smaller constant if necessary and assume that $\ref{k:alpha}\leq \alpha\dm/2$. 

Define 
\be\label{eq: define f_Y app A}
f_{Y}(x)=\begin{cases}\sum_{w\in I_Y(x)}\|w\|^{-\alpha}& I_Y(x)\neq \emptyset\\
\omega(x)&\text{otherwise}\end{cases}.
\ee

\begin{lemma}\label{lem: Margulis func periodic 1}
Let $C$ be as in Proposition~\ref{prop: average of inj}, and let  
$d\geq B_\omega$. Then 
\[
\int f(hx)\diff\!\sigma_d(h)\leq
Ce^{-\alpha\dm d}f_Y(x)+\ref{a:alpha}e^{d}E\vol(Y)\cdot (e^{-\ref{k:alpha}d}\omega(x)+\bar B e^{\ref{a:alpha}d})
\]
where $\bar B$ is as in Proposition~\ref{prop: average of inj}.  
\end{lemma}

\begin{proof}
Since $Y$ is fixed throughout the argument, 
we drop it from the index in the notation, e.g., we will denote $f_{Y}$ by $f$ etc.

Let $d\geq 0$ and let $h=a_du_r$ for some $r\in[-1,2]$. 
Let $x\in X$. First, let us assume that there exists some $w\in I(hx)$ with 
\[
\|w\|<\ref{a:alpha}^{-2}e^{-\ref{a:alpha}d}\omega(hx)=:\Upsilon.
\]
This in particular implies that both $I(hx)$ and $I(x)$ are non-empty. 
Hence, 
\be\label{eq: MF-periodic-1}
\begin{aligned}
f(hx)&=\sum_{w\in \margI(hx)}\|w\|^{-\alpha}=\sum_{\|w\|< \Upsilon}\|w\|^{-\alpha}+\sum_{\|w\|\geq \Upsilon}\|w\|^{-\alpha}\\
&\leq \sum_{w\in I(x)}\|\Ad(h)w\|^{-\alpha}+\ref{a:alpha}^{2\alpha}e^{\ref{a:alpha}\alpha d}\Bigl(1+\# \margI(hx)\Bigr)\cdot\omega(hx)^\alpha\\
&\leq \sum_{w\in I(x)}\|\Ad(h)w\|^{-\alpha}+\ref{a:alpha}e^{d}\Bigl(1+\# \margI(hx)\Bigr)\cdot\omega(hx),
\end{aligned}
\ee
in the last inequality, we used $0<\alpha\leq\min\{1/2,1/\ref{a:alpha}\}$ and $\omega(\cdot)\geq 1$.

Note also that if $\|w\|\geq \Upsilon$ for all $w\in \margI(hx)$ 
(which in view of the choice of $\ref{a:alpha}$ includes the case $I(x)=\emptyset$) 
or if $\margI(hx)=\emptyset$, then
\be\label{eq: MF-periodic-2}
f(hx)\leq \ref{a:alpha}e^{d}\Bigl(1+\# \margI(hx)\Bigr)\cdot\omega(hx).   
\ee

Averaging~\eqref{eq: MF-periodic-1} and~\eqref{eq: MF-periodic-2} over $[-1,2]$ and using~\eqref{eq: number of sheets}, we conclude that 
\begin{multline*}
\int f(hx)\diff\!\sigma_d(h)\leq 
\sum_{w\in I(x)}\int \|h w\|^{-\alpha}\diff\!\sigma_d(h)\quad+\\ \ref{a:alpha}e^{d}E\vol(Y)\cdot\int \omega(hx)\diff\!\sigma_d(h);
\end{multline*}
we replace the summation on the right by $0$ if $I(x)=\emptyset$. 

Thus by Lemma~\ref{lem: linear algebra App} and Proposition~\ref{prop: average of inj}, we conclude that  
\begin{multline*}
\int f(hx)\diff\!\sigma_d(h)\leq Ce^{-\dm\alpha d}\cdot \sum_{w\in I(x)}\|w\|^{-\alpha}\quad+\\
\ref{a:alpha}e^{d}E\vol(Y)\cdot (e^{-\ref{k:alpha}d}\omega(x)+\bar Be^{\ref{a:alpha}d})
\end{multline*}
again, the summation on the right by $0$ if $I(x)=\emptyset$. Thus  
\[
\int f(hx)\diff\!\sigma_d(h)\leq 
Ce^{-\dm \alpha d}f(x)+\ref{a:alpha}e^dE\vol(Y)\cdot (e^{-\ref{k:alpha}d}\omega(x)+\bar Be^{\ref{a:alpha}d}).
\]

The proof is complete.
\end{proof}

\begin{lemma}\label{lem: Margulis func periodic}
Let $C$, $\ref{a:alpha}$, and $\ref{k:alpha}$ be as in Proposition~\ref{prop: average of inj}, and let 
\[
\ell= \lceil 10(\ref{a:alpha}+1)/{\ref{k:alpha}}\rceil.
\]
There is an absolute constant $T_0$ so that the following holds. 
Let $T\geq T_0$ and define 
\[
d_i=\frac{\log T}{2^i\ell}
\]
for all $i=1,\ldots, k$ where $k$ is the largest integer so that $d_k\geq \max\Bigl\{B_\omega,\frac{\log(4C)}{\ref{k:alpha}}\Bigr\}$ --- note that 
$\frac{1}{2}\log\log T\leq k\leq 2\log \log T$. Then 
\begin{multline*}
\int f_{Y}(hx)\diff\!\sigma_{d_1}^{(\ell)}\conv\cdots\conv\sigma_{d_k}^{(\ell)}(h)\leq \\
(\log T)^{D'_0}T^{-\dm \alpha} f(x)+B'\vol(Y)\Bigl(T^{-\ref{k:alpha1}}\omega(x)+1\Bigr)
\end{multline*}
where $D'_0$ and $\constk\label{k:alpha1}$ depend on dimension and $B'$ depends on $X$.  
\end{lemma}

\begin{proof}
Again since $Y$ is fixed throughout the argument, 
we drop it from the index in the notation, e.g., we will denote $f_{Y}$ by $f$ etc. 

Let us make some observations before starting the proof. 
Since $d_i\geq \frac{\log(4C)}{\ref{k:alpha}}$ for all $i$, and $\ref{k:alpha}\leq \alpha\dm/2$, the following holds  
\be\label{eq: C-e-alpha-m vs. e k:alpha}
Ce^{-\alpha \dm d_i}\leq e^{-\ref{k:alpha}d_i}\leq 1/4
\ee
Moreover, we have the following two estimates: 
\be\label{eq: Di and di}
5\sum_{j=i+1}^{k} d_{j}\geq 5\ell^{-1}\times 2^{-i-1}\log T\geq (2^i\ell)^{-1}\cdot \log T=d_{i}
\ee

By Lemma~\ref{lem: Margulis func periodic 1}, for all $d\geq B_\omega$, we have 
\begin{multline}\label{eq: main MF periodic app}
\int f(hx)\diff\!\sigma_d(h)\leq \\
Ce^{-\alpha\dm d}f(x)+\ref{a:alpha}Ee^d\vol(Y)\cdot \Bigl(e^{-\ref{k:alpha}d}\omega(x)+\bar Be^{\ref{a:alpha} d}\Bigr).
\end{multline}

Let $\lambda=\ref{a:alpha} E\bar B$. Iterating~\eqref{eq: main MF periodic app}, $\ell$-times, we conclude that
\begin{multline*}
\int f(h_k\cdots h_1x)\diff\!\sigma_{d_1}^{(\ell)}(h_1)\cdots\diff\!\sigma_{d_k}^{(\ell)}(h_k)\leq \\
C^\ell e^{-\alpha\dm \ell d_k}\int\! f(h_{k-1}\cdots h_1x)\diff\!\sigma_{d_1}^{(\ell)}(h_1)\cdots\diff\!\sigma_{d_{k-1}}^{(\ell)}(h_{k-1})\quad +\\
\ref{a:alpha} Ee^{d_k}\vol(Y)(\Xi_k+2\bar Be^{\ref{a:alpha} d_k})
\end{multline*}
we used $Ce^{-\alpha \dm d_k}\leq e^{-\ref{k:alpha}d_k}\leq \frac14$, see~\eqref{eq: C-e-alpha-m vs. e k:alpha}, 
to bound the $\ell$-terms geometric sum by $2\bar Be^{\ref{a:alpha} d_k}$, and 
\[
\Xi_k=\sum_{j=0}^{\ell-1} e^{-\ref{k:alpha} d_k(\ell-j)}\!\!\int\omega(h_kh_{k-1}\cdot\cdot h_1x)\diff\!\sigma_{d_1}^{(\ell)}(h_1)\cdot\cdot\diff\!\sigma_{d_{k-1}}^{(\ell)}(h_{k-1})\diff\!\sigma_{d_k}^{(j)}(h_k),
\]
again we used $Ce^{-\alpha \dm d_k}\leq e^{-\ref{k:alpha}d_k}$, in the definition of $\Xi_k$. 

Note also that in view of the definition of $\lambda$, we have 
\[
\ref{a:alpha} Ee^{d_k}\vol(Y)(\Xi_k+2\bar Be^{\ref{a:alpha} d_k})\leq \lambda \vol(Y) e^{(1+\ref{a:alpha})d_k}(\Xi_k+2),
\]
therefore, we conclude 
\begin{multline}\label{eq: main estimate MF periodic diff di}
\int f(h_k\cdots h_1x)\diff\!\sigma_{d_1}^{(\ell)}(h_1)\cdots\diff\!\sigma_{d_k}^{(\ell)}(h_k)\leq \\
C^\ell e^{-\alpha\dm \ell d_k}\int\! f(h_{k-1}\cdots h_1x)\diff\!\sigma_{d_1}^{(\ell)}(h_1)\cdots\diff\!\sigma_{d_{k-1}}^{(\ell)}(h_{k-1})\quad +\\
\lambda \vol(Y) e^{(1+\ref{a:alpha})d_k}(\Xi_k+2).
\end{multline}

We will apply Proposition~\ref{prop: average of inj}, to bound $\Xi_k$ from above.  
Let us begin by applying Proposition~\ref{prop: average of inj}, $\ell$-times with $d_k$, then 
\[
\Xi_k\leq e^{-\ref{k:alpha} \ell d_k}\int\omega(h_{k-1}\cdot\cdot h_1x)\diff\!\sigma_{d_1}^{(\ell)}(h_1)\cdot\cdot\diff\!\sigma_{d_{k-1}}^{(\ell)}(h_{k-1})+\lambda e^{\ref{a:alpha} d_k}
\]
where we used $e^{-\ref{k:alpha} d_k}\leq 1/4$  and $\lambda=\ref{a:alpha} E\bar B \geq 2\bar B$ to estimate the $\ell$-terms geometric sum. 

The goal now is to inductively apply Proposition~\ref{prop: average of inj}, $\ell$ times with $d_i$ for all $1\leq i\leq k-1$, in order to simplify the above estimate. 
Applying Proposition~\ref{prop: average of inj}, $\ell$-times with $d_{k-1}$, we obtain from the above that
\begin{multline*}
\Xi_k\leq e^{-\ref{k:alpha} \ell(d_k+d_{k-1})}\int\omega(h_{k-2}\cdot\cdot h_1x)\diff\!\sigma_{d_1}^{(\ell)}(h_1)\cdot\cdot\diff\!\sigma_{d_{k-2}}^{(\ell)}(h_{k-2})\quad +\\ 
e^{-\ref{k:alpha}\ell d_k}\cdot (\lambda e^{\ref{a:alpha} d_{k-1}})+\lambda e^{\ref{a:alpha} d_k}.
\end{multline*}
Put $\Theta_k=0$, and for every $0\leq i<k$, let $\Theta_{i}=\sum_{j=i+1}^{k}d_{j}$. Continuing the above inequalities inductively, we conclude 
\begin{align*}
\Xi_k&\leq e^{-\ref{k:alpha} \ell\Theta_0}\omega(x)+\lambda(e^{\ref{a:alpha} d_k}+\sum_{i=1}^{k-1} e^{-\ref{k:alpha} \ell \Theta_{i}}e^{\ref{a:alpha} d_{i}})\\
&\leq e^{-\ref{k:alpha} \ell\Theta_0}\omega(x)+\lambda(e^{\ref{a:alpha} d_k}+\sum_{i=1}^{k-1} e^{-d_i})
\leq e^{-\ref{k:alpha} \ell\Theta_0}\omega(x)+\lambda(e^{\ref{a:alpha} d_k}+M)
\end{align*}
where we used $\ref{k:alpha} \ell \Theta_i\geq \frac{\ref{k:alpha} \ell}5\cdot (5\Theta_i)\geq (\ref{a:alpha}+2)d_i$, see~\eqref{eq: Di and di} and the definition of $\ell$, in the second to last inequality and 
wrote $\sum e^{-d_i}\leq M$ in the last inequality.  

Iterating~\eqref{eq: main estimate MF periodic diff di} and using the above analysis, we conclude 
\begin{multline*}
\int f(h_k\cdots h_1x)\diff\!\sigma_{d_1}^{(\ell)}(h_1)\cdots\diff\!\sigma_{d_k}^{(\ell)}(h_k)\leq\\
C^{\ell k}e^{-\alpha\dm \ell\Theta_0} f(x)+\lambda\vol(Y)\sum_{i=1}^k e^{-\ref{k:alpha} \ell \Theta_i}e^{(\ref{a:alpha}+1)d_i}\Bigl(\Xi_i+ 2\Bigr)
\end{multline*}
where for every $1\leq i\leq k$, we have 
\[
\Xi_i=\sum_{j=0}^{\ell-1} e^{-\ref{k:alpha} d_i(\ell-j)}\!\!\int\omega(h_ih_{i-1}\cdot\cdot h_1x)\diff\!\sigma_{d_1}^{(\ell)}(h_1)\cdot\cdot\diff\!\sigma_{d_{i-1}}^{(\ell)}(h_{i-1})\diff\!\sigma_{d_i}^{(j)}(h_i).
\]
Arguing as above, we have 
\[
\Xi_i\leq e^{-\ref{k:alpha} \ell(\sum_{j=1}^i d_j)}\omega(x)+\lambda(e^{\ref{a:alpha} d_i}+M).
\] 
Recall that $\Theta_i=\sum_{j=i+1}^k d_j$; therefore, we conclude that 
\begin{multline*}
\int f(h_k\cdots h_1x)\diff\!\sigma_{d_1}^{(\ell)}(h_1)\cdots\diff\!\sigma_{d_k}^{(\ell)}(h_k)\leq\\
C^{\ell k}e^{-\alpha\dm \ell\Theta_0} f(x)+ e^{-\ref{k:alpha} \ell\Theta_0}\lambda\vol(Y)\omega(x)\textstyle\sum_{i=1}^k e^{(\ref{a:alpha}+1)d_i} + \\
(M+2)\lambda^2\vol(Y)\sum_{i=1}^k e^{-\ref{k:alpha} \ell \Theta_i}e^{(2\ref{a:alpha}+1)d_i}
\end{multline*}
Recall again from~\eqref{eq: Di and di} and the definition of $\ell$ that $\ref{k:alpha} \ell \Theta_i\geq (2\ref{a:alpha}+2)d_i$. Hence, the last term above is $\leq B'\vol(Y)$ for an absolute constant $B'\geq \lambda$. Similarly, sine $\ell\sum d_i=\log T - O(1)$ 
second to last term is $\leq B'\vol(Y)T^{-\ref{k:alpha1}}$.  

Moreover, $\ell\sum d_i=\log T - O(1)$ where the implied constant is absolute, and $k\leq 2\log\log T$. Hence,
\[
C^{\ell k}e^{-\dm\alpha \ell(\sum_{i=1}^k d_i)}\leq (\log T)^{1+\star\log C}T^{-\dm \alpha }
\]  
so long as $T$ is large enough. The proof of the lemma is complete.  
\end{proof}

\begin{lemma}\label{lem: average over [0,1]}
Let the notation be as in Lemma~\ref{lem: Margulis func periodic}, in particular for every $T\geq T_0$ define 
$d_1,\ldots, d_k$ as in that lemma. Put $d(T)=\ell\sum d_i$, then 
\[
\int_0^1 f_{Y}(a_{d(T)}u_rx)\diff\!r\leq 
3(\log T)^{D'_0} T^{-\alpha\dm}f_{Y}(x) + B\vol(Y)\Bigl(T^{-\ref{k:alpha1}}\omega(x)+1\Bigr)
\]
where $B\geq 1$ depends on $X$.
\end{lemma}

\begin{proof}
Again, since $Y$ is fixed throughout the argument, 
we drop it from the index in the notation, e.g., we will denote $f_{Y}$ by $f$ etc. 

By Lemma~\ref{lem: Margulis func periodic}, we have 
\begin{multline}\label{eq: use Margulis func periodic [0,1]}
\frac{1}{3^{\ell k}}\int_{-1}^2\cdots\int_{-1}^2 f(a_{d_k}u_{r_{k,\ell}}\cdots a_{d_k}u_{r_{k,1}}\cdots a_{d_1}u_{r_{1,1}}x)\diff\!r_{1,1}\cdots\diff\!r_{k,\ell} \leq\\
(\log T)^{D'_0}T^{-\alpha \dm} f(x)+B'\vol(Y)\Bigl(T^{-\ref{k:alpha1}}\omega(x)+1\Bigr)
\end{multline}

Now, for every $(r_{k,\ell},\ldots, r_{1,2}, r_{1,1})\in [-1,2]^{\ell k}$, we have 
\[
a_{d_k}u_{r_{k,\ell}}\cdots a_{d_k}u_{r_{k,1}}\cdots a_{d_1}u_{r_{1,1}}=a_{d(T)}u_{\varphi(\hat r)+r_{1,1}}
\] 
where $\hat r=(r_{k,\ell},\ldots, r_{1,2})$ and $|\varphi(\hat r)|\leq 0.2$. 

In view of~\eqref{eq: use Margulis func periodic [0,1]}, there is $\hat r= (r_{k,\ell},\ldots, r_{1,2})\in [-1,2]^{\ell k-1}$ so that 
\begin{multline}\label{eq: hat r and r1}
\frac{1}{3}\int_{-1+\varphi(\hat r)}^{2+\varphi(\hat r)} f(a_{d(T)}u_{r}x)\diff\!r\leq \\
(\log T)^{D'_0}T^{-\alpha\dm} f(x)+B'\vol(Y)\Bigl(T^{-\ref{k:alpha1}}\omega(x)+1\Bigr).
\end{multline}
Since $|\varphi(\hat r)|\leq 0.2$, we have $[0,1]\subset [-1,2]+\varphi(\hat r)$. 
Therefore,~\eqref{eq: hat r and r1} and the fact that $f\geq 0$ imply that
\[
\frac{1}{3}\int_{0}^1 f(a_{d(T)}u_{r}x)\diff\!x\leq 
(\log T)^{D'_0}T^{-\alpha\dm} f(x)+B'\vol(Y)\Bigl(T^{-\ref{k:alpha1}}\omega(x)+1\Bigr).
\]
The lemma follows with $B=3B'$.
\end{proof}

\begin{proof}[Proof of Proposition~\ref{prop: linearization translates}]
Let $R\geq 1$ be a parameter and assume that $\vol(Y)\leq R$.
Recall that for a periodic orbit $Y$, we put
\[
f_{Y}(x)=\begin{cases}\sum_{w\in I_Y(x)}\|w\|^{-\alpha}& I_Y(x)\neq \emptyset\\
\omega(x)&\text{otherwise}\end{cases}.
\]
Let $\psi(x_0)=\max\{\dist(x_0,Y)^{-\alpha},\omega(x_0)\}$. Then 
\be\label{eq: fYd and dist}
\psi(x_0)\ll f_{Y,d}(x_0)\ll \vol(Y)\psi(x_0),
\ee
where the implied constant depends only on $X$, see~\eqref{eq: number of sheets}. 

With the notation of Lemma~\ref{lem: Margulis func periodic}, let $T\geq T_0$ 
and $d_i=\frac{\log T}{2^i\ell}$ for $1\leq i\leq k$. Then
\be\label{eq: d(T) is almost T}
\log T-\bar b\leq d(T)\leq \log T
\ee
where $\bar b$ is absolute. 

Let $T_1\geq T_0$ be so that $(\log T)^{D'_0}T^{-\alpha\dm}$ is decreasing on $[T_1,\infty)$.
Let
\be\label{eq: choose T_2 periodic}
T_2=\inf\{ T\geq \max(T_1,\omega(x_0)^{{1}/{\ref{k:alpha1}}}): (\log T)^{D'_0}T^{-\alpha\dm}\leq d(x_0, Y)^{\alpha}\}.
\ee 
In other words, for all $T\geq T_2$, we have $T^{-\ref{k:alpha1}}\omega(x_0)\leq 1$ and
\[
(\log T)^{D'_0}T^{-\alpha\dm} d(x_0, Y)^{-\alpha}\leq 1.
\]
Furthermore, in view of~\eqref{eq: fYd and dist} and since $\vol(Y)\leq R$, for all $T\geq T_2$, 
\[
(\log T)^{D'_0}T^{-\alpha\dm } f_Y(x_0)\ll R(\log T)^{D'_0}T^{-\alpha\dm}\psi(x_0)\\
\]
In particular, using~\eqref{eq: fYd and dist} again, we have $(\log T)^{D'_0}T^{-1/3} f_Y(x_0)\ll R$. 

Altogether, we conclude that for all $T\geq T_2$, we have 
\be\label{eq: length of the flow in linearization}
3\log(T)^{D'_0}T^{-\alpha\dm}f_Y(x_0)+B\vol(Y)T^{-\ref{k:alpha1}}(\omega(x_0)+1)\leq B_2R
\ee
where $B_2$ is absolute. 

Let $T\geq T_2$, and let $d(T)=\ell\sum d_i$ where $d_i$'s are as above.
Using~\eqref{eq: length of the flow in linearization} and Lemma~\ref{lem: average over [0,1]}, 
\be\label{lem: use average over [0,1] lemma}
\int_0^1 f_{Y}(a_{d(T)}u_rx)\diff\!r\leq B_2 R.
\ee

Let $D\geq 10$. Then by~\eqref{lem: use average over [0,1] lemma} we have 
\[
|\{r\in[0,1]: f_{Y}(a_{d(T)}u_rx_0)>B_2R^{D}\}|\leq B_2R/B_2R^{D}\leq R^{-D+1}.
\]
In view of~\eqref{eq: fYd and dist}, there some $B_1$ (depending on $X$) so that 
$\dist_X(a_{s}u_rx_0, Y)\leq B_1^{-1}R^{-D/\alpha}$ implies $f_{Y}(a_{s}u_rx_0)>B_2R^{D}$ for all $s\geq 0$ and $r\in[0,1]$.
Therefore, we conclude from the above that 
\be\label{eq: control dist x0 and Y app B}
\Bigl|\Bigl\{r\in [0,1]: \dist_X(a_{d(T)}u_rx_0, Y)\leq (B_1R^{D/\alpha})^{-1}\Bigr\}\Bigr|\leq R^{-D+1}.
\ee

Let now $s\geq \log T_2$, then by~\eqref{eq: d(T) is almost T} there exists some $T\geq T_2$ so that 
\[
d(T)-2\bar b\leq s\leq d(T)+2\bar b
\] 
For every $s\geq \log T_2$, let $T_s$ be the minimum such $T$. 
Then~\eqref{eq: cont Ad app B} implies that is $\hat B\geq 1$ (absolute)
so that if $s\geq \log T_2$ and $r\in[0,1]$ are so that
\[
\dist_X(a_{s}u_rx_0, Y)\leq (\hat BR^{D/\alpha})^{-1},
\] 
then $\dist_X(a_{d(T_s)}u_rx_0, Y)\leq (B_1R^{D/\alpha})^{-1}$.
This and~\eqref{eq: control dist x0 and Y app B}, imply that
\be\label{eq: control dist x0 and Y app B s}
\Bigl|\Bigl\{r\in [0,1]: \dist_X(a_{s}u_rx_0, Y)\leq (\hat BR^{D/\alpha})^{-1}\Bigr\}\Bigr|\leq R^{-D+1}
\ee

Let $\ref{E:non-div-main}$ be as in Proposition~\ref{prop:Non-div-main}, increasing $T_1$ if necessary, we  
will assume $\log T_2\geq \dm_0|\log (\inj(x_0))|+\ref{E:non-div-main}$. 
Using Proposition~\ref{prop:Non-div-main}, thus,  
\be\label{eq: non div app B}
\Bigl|\Bigl\{r\in [0,1]:\inj(a_s\uvk x)< \eta\Bigr\}\Bigr|<\ref{E:non-div-main}\eta^{1/\dm_0}
\ee
for any $\eta>0$ and all $s\geq \log T_2$.

Altogether, from~\eqref{eq: control dist x0 and Y app B s} and~\eqref{eq: non div app B} it follows that 
for any $s\geq \log T_2$, 
\be\label{eq: fYd and inj app B}
\biggl|\biggl\{r\in[0,1]: \begin{array}{c}\inj(a_s\uvk x)< \eta\quad\text{ or }\\
\dist_X(a_{s}u_rx_0, Y)\leq (\hat BR^{D/\alpha})^{-1}\end{array}\biggr\}\biggr|\leq \ref{E:non-div-main}\eta^{1/\dm_0}+R^{-D+1}.
\ee

In view of~\cite[Thm.~5]{SanchezSeong}, the number of periodic $H$-orbits with volume $\leq R$ 
in $X$ is $\leq \hat E R^{\hat D}$ where $\hat D$ depends only on dimension and $\hat E$ depends on $X$. 
Let $D=\hat D+10$ and $\ref{c: linear trans}=\max\{\hat E, \hat B, \ref{E:non-div-main}\}$. Then~\eqref{eq: fYd and inj app B} implies 
\begin{multline}\label{eq: fYd and inj app B final}
\biggl|\biggl\{r\in[0,1]: \!\!\begin{array}{c}\inj(a_s\uvk x)< \eta\text{ or there exists $x$ with}\\
\vol(Hx)\leq R\text{ s.t. }\dist_X(a_{s}u_rx_0, x)\leq (\ref{c: linear trans}R^{D/\alpha})^{-1}\end{array}\!\!\biggr\}\biggr|\\ 
\leq \ref{c: linear trans}(\eta^{1/\dm_0}+R^{-1}).
\end{multline} 

We now show that ~\eqref{eq: fYd and inj app B final} implies the proposition. Suppose 
\[
\dist_X(x_0,x)\geq S^{-\dm}(\log S)^{D'_0}
\]
for every $x$ with $\vol(Hx)\leq R$. Then by~\eqref{eq: choose T_2 periodic}, we have 
\[
T_2\leq \max\{S, \omega(x_0)^{1/\ref{k:alpha1}}, T_1\}.
\]
Therefore, if we let $D_0=(\max\{D'_0, \hat D+10\})/\alpha$ 
and set $s_0=\ref{E:non-div-main}+\log T_1$, then Proposition~\ref{prop: linearization translates} 
follows from~\eqref{eq: fYd and inj app B final}. 
\end{proof}

\subsection{Proof of Proposition~\ref{prop: closing lemma}}\label{sec: proof of closing}
In what follows all the implied multiplicative constants depend only on~$X$.

We begin by recalling some statements and lemmas which will be used in the proof.  
Let $v_H$ be a unit vector on the line $\wedge^3\hfrak\subset \wedge^3\gfrak$.

\begin{lemma}[Cf.~\cite{LM-PolyDensity}, Lemma 6.3]\label{lem:almost-inv}
There exist $\consta\label{a:Eq-proj}$, $\consta\label{a:Eq-proj-2}$, and $\constE\label{E:Eq-proj-mul}$
so that the following holds. Let $\gamma_1,\ldots, \gamma_n\in\Gamma$, and let 
\[
\delta\leq  \ref{E:Eq-proj-mul}^{-1}\Big(\max\{\|\gamma_i^{\pm1}\|: 1\leq i\leq n\}\Big)^{-\ref{a:Eq-proj}}.
\]
Suppose there exists some $g\in G$ so that $\gamma_i g^{-1}v_H=\epsilon_ig^{-1}v_H$ for $i=1,2$ 
where $\|\epsilon_i-I\|\leq \delta$. 
Then, there is some $g'\in G$ such that 
\[
\|g'-g^{-1}\|\leq \ref{E:Eq-proj-mul}  \|g\|^{\ref{a:Eq-proj-2}}\delta \Big(\max\{\|\gamma_i^{\pm1}\|: 1\leq i\leq n\}\Big)^{\ref{a:Eq-proj-2}}
\] 
and $\gamma_ig'v_H=g'v_H$ for $i=1,\ldots, n$.
\end{lemma}

\begin{lemma}[Cf.~\cite{LM-PolyDensity}, Lemma 6.2]\label{lem:non-elementary}
There exist $\constE\label{E:non-el-1}$ and $\consta\label{a:non-el-2}$ depending on $\Gamma$, and $\consta\label{a:non-el-2'}$ (depending on the dimension)
so that the following holds.  
Let $\gamma_1,\gamma_2\in\Gamma$ be two non-commuting elements. 
If $g\in G$ is so that $\gamma_i g^{-1}v_H=g^{-1}v_H$ for $i=1,2$, then $Hg\Gamma$ is a closed orbit with 
\[
\vol(Hg\Gamma)\leq\ref{E:non-el-1}  \|g\|^{\ref{a:non-el-2'}}\Bigl(\max\{\|\gamma_1^{\pm1}\|,\|\gamma_2^{\pm1}\|\}\Bigr)^{\ref{a:non-el-2}}.
\] 
\end{lemma} 

The statements in~\cite[Lemma 6.2, and Lemma 6.3]{LM-PolyDensity} assumed 
$|g_2|\ll 1$. However, the arguments work without any changes and yield Lemmas~\ref{lem:almost-inv} and~\ref{lem:non-elementary}.

We also recall the following lemma which is a consequence of reduction theory. 
In this form, the lemma is a spacial case of~\cite[Lemma 2.8]{LMMS}. 

\begin{lemma}\label{lem: reduction theory}
There exist $D_2$ (depending on $\dm$) and $\constE\label{c: red th}$ (depending on $X$) 
so that the following holds for all $0<\eta<1$. 
Let $g\in G$ be so that $g\Gamma\in X_\eta$. Then there is some $\gamma\in\Gamma$ so that 
\be\label{eq: red theory pf closing}
\|g\gamma\|\leq \ref{c: red th}\eta^{-D_2}.
\ee
\end{lemma}

Finally, we recall Proposition~\ref{prop: Non-div main}: for all positive $\vare$, every interval $J\subset[0,1]$, 
and every $x\in X$, we have 
\be\label{eq:cpct-return}
\Bigl|\Bigl\{r\in J:\inj(a_du_r x)< \vare^{\dm_0}\Bigr\}\Bigr|<\ref{E:non-div-main}\vare|J|, 
\ee
so long as $d\geq \dm_0|\log(|J|\inj(x))|+\ref{E:non-div-main}$.

For the rest of the argument, let  
\be\label{eq: t D2 and eta clsoing}
\rws\geq 100 D_2\dm_0|\log(\eta\,\inj(x_1))|+\ref{E:non-div-main}
\ee
Let $r_1\in[0,1]$ be so that $x_2=a_tu_{r_1}x_1\in X_\eta$. 
Write $x_2=g_2\Gamma$ where $|g_2|\ll \eta^{-D_2}$, see~\eqref{eq: red theory pf closing}.  

Also let $\dn$ be a constant which will be explicated later and will depend only on $\dm$. 

In the course of the proof, we will use $\dn_\cdot$ to denote constants which depend on $\dn$. 
The notation $\hdm_\cdot$ (and the previously used $\dm_\cdot$) will be used for constants 
which depend on $\dm$ but not on the choice of $\dn$ above.

Let $\hdm_0$ be a constant which depends on $\dm$ so that 
\be\label{eq: def dn0}
\|a_s\|\leq e^{\hdm_0s}\quad\text{ for all $s\geq 1$}.
\ee 

We will show that unless part~(2) in the proposition holds, we have the following: for every such $x_2$, there exists $J(x_2)\subset [0,1]$ with $|[0,1]\setminus J(x_2)|\leq 200\ref{E:non-div-main}\eta^{1/(2\dm_0)}$ so that for all $r\in J(x_2)$, we have: 
\begin{itemize}
\item[(a)] $a_{\dn t}u_rx_2\in X_\eta$,  
\item[(b)] the map $\sfh\mapsto \sfh a_{\dn t}u_rx_2$ is injective over $\coneH_{\rws}$, and 
\item[(c)] for all $z\in \coneH_\rws.a_{\dn t}u_rx_2$, we have $f_{\rws,\alpha}(z)\leq\nuni^{D\rws}$.
\end{itemize}
This will imply that part~(1) in the proposition holds as 
\[
a_{\dn t}u_ra_tu_{r'}x_1=a_{\dn+1 t}u_{r'+e^{-t}r}x_1.
\]

Assume contrary to the above claim that for some $x_2$ as above, there exists a subset $I'_{\rm bad}\subset [0,1]$ with  
$|I'_{\rm bad}|>200\ref{E:non-div-main}\eta^{1/(2\dm_0)}$ so that one of (a), (b), or (c) above fails. 
Then in view of~\eqref{eq:cpct-return} applied with $x_2$ and $\dn t$, there is a subset 
$I_{\rm bad}\subset [0,1]$ with $|I_{\rm bad}|\geq 100\ref{E:non-div-main}\eta^{1/(2\dm_0)}$ 
so that for all $r\in I_{\rm bad}$ we have $a_{\dn t}u_rx_2\in X_{\eta}$, but  
\begin{itemize}
\item either the map $\sfh\mapsto \sfh a_{\dn t}u_rx_2$ is not injective on $\coneH_{\rws}$,
\item or there exists $z\in \coneH_\rws.a_{\dn t}u_rx_2$ so that $f_{\rws,\alpha}(z)>\nuni^{D\rws}$.
\end{itemize}
We will show that this implies part~(2) in the proposition holds.

Let us first recall that $f_{\tau}:\coneH_\tau.y\to [1,\infty)$ is defined as follows 
\[
f_{\tau}(z)=\begin{cases} \sum_{0\neq w\in I_\tau(z)}\|w\|^{-\alpha} & \text{if $\margI_\tau(z)\neq\{0\}$}\\
\inj(z)^{-\alpha}&\text{otherwise}
\end{cases}.
\]
where $0<\alpha\leq 1$.

\subsection*{Finding lattice elements $\gamma_r$}
We introduce the shorthand notation 
\[
h_r:=a_{\dn t}u_r,\quad \text{for any $r\in[0,1]$}. 
\]
Let us first investigate the latter situation. That is: for $r\in I_{\rm bad}$ (recall that $h_rx_2\in X_{\eta}$) 
there exists some $z=\sfh_1h_rx_2\in \coneH_\rws.h_rx_2$, so that $f_{\rws,\alpha}(z)>\nuni^{D\rws}$. 
Since $h_rx_2\in X_{\eta}$, we have  
\be\label{eq:Ct-cusp}
\inj(\sfh h_rx_2)\gg \eta\nuni^{-\dm\rws}, \quad\text{for all $\sfh\in\coneH_\rws$}.
\ee
Using the definition of $f_{\rws,\alpha}$, thus, we conclude that 
if $\margI_\rws(z)=\{0\}$, then $f_{t,\alpha}(z)\ll \eta^{-1}\nuni^{\dm\rws}$. Since $t\geq 100D_2\dm_0|\log\eta|$, assuming $\dm_0\geq \dm+1$, $D\geq \dm+1$ and $t$ is large enough, we conclude that $\margI_\rws(z)\neq\{0\}$. 
Moreover, using~\eqref{eq:Ct-cusp}, we have $\#\margI_\rws(z)\ll \eta^{-\hdm_1}\nuni^{\hdm_1\rws}$, see~\cite[Lemma 6.4]{LM-PolyDensity} or the similar estimate~\cite[Lemma 8.1] {LMW22}.

Altogether, if $D\geq\dm+1+2\hdm_1$ and $t$ is large enough, there exists some $w\in I_\rws(z)$ with 
\[
0<\|w\|\leq \nuni^{(-D+\hdm_1+1)\rws}=:\nuni^{-D'\rws}.
\]

The above implies that for some $w\in \rfrak$ with $\|w\|\leq \nuni^{-D'\rws}$ 
and $\sfh_1\neq \sfh_2\in\coneH_\rws$, we have $\exp(w)\sfh_1h_rx_2=\sfh_2h_rx_2$.
Thus 
\be\label{eq:wh-sh}
\exp(w_r)h_r^{-1}\sfs_rh_rx_2=x_2
\ee 
where $\sfs_r=\sfh_2^{-1}\sfh_1$, $w_r=\Ad(h_r^{-1}\sfh_2^{-1})w$. 
In particular, $\|w_r\|\ll \nuni^{(-D'+\hdm_0\dn)\rws}$ where the implied constant depends only on $\dm$.
Assuming $t$ is large enough compared to the implied multiplicative constant,
\be\label{eq:wh-est}
0<\|w_r\|\leq \nuni^{(-D'+\hdm_0\dn+1)\rws}.
\ee
Recall that $x_2=g_2\Gamma$ where $|g_2|\ll \eta^{-D_2}$, thus,~\eqref{eq:wh-sh} implies 
\be\label{eq:gamma-h}
\exp(w_r)h_r^{-1}\sfs_rh_r=g_2\gamma_rg_2^{-1}
\ee
where $1\neq\sfs_r\in H$ with $\|\sfs_r\|\ll \nuni^{2\hdm_0\rws}$ and $e\neq\gamma_r\in \Gamma$.

Similarly, if for some $r\in I_{\rm bad}$, $\sfh\mapsto \sfh h_rx_2$ is not injective, then
\[
h_r^{-1}\sfs_r h_r=g_2\gamma_rg_2^{-1}\neq e.
\]
In this case we actually have $e\neq \gamma_r\in g_2^{-1}Hg_2$ --- we will not use this extra information in what follows.  

\subsection*{Some properties of the elements $\gamma_r$}
Recall that $\|g_2\|\ll \eta^{-D_2}$ and that $t\geq 100D_2\dm_0|\log\eta|$. Therefore, if we put $\dn_1=2\hdm_0(\dn+1)+1$, then 
\be\label{eq:size-gammah}
\|\gamma_r^{\pm1}\|\leq \nuni^{\dn_1\rws}
\ee
where we assumed $t$ is large compared to $\eta$ and used~\eqref{eq:gamma-h}.

We identify $\begin{pmatrix} a_1 & a_2 \\ a_3 & a_4 \end{pmatrix}\in\SL_2(\R)$ with its image in $H$, however, when we write $\|\;\|$ the norm is in $G$ but $|a_i|$ denotes the usual absolute value.  
With this notation, e.g., $h_r^{-1}\sfs_r h_r$ is represented by 
\[
u_{-r}\begin{pmatrix} a_1 & e^{-\dn t}a_2 \\ e^{\dn t}a_3 & a_4 \end{pmatrix}u_r
\]
where $|a_i|\leq 10\nuni^{t}$.

Let $\xi>0$ be so that $\|g\gamma g^{-1}-I\|\geq \xi\eta^{2D_2}$ for all $\gamma\in\Gamma\setminus \{1\}$ 
and $\|g\|\leq \ref{c: red th}\eta^{-D_2}$, see~\eqref{eq: red theory pf closing}.   
Then by~\eqref{eq:gamma-h}, we have 
\[
\biggl\|u_{-r}\begin{pmatrix} a_1 & e^{-\dn t}a_2 \\ e^{\dn t}a_3 & a_4 \end{pmatrix}u_r-I\biggr\|\geq \eta^{D_2'}
\]
for some $D'_2$ depending only on $\dm$ and $D_2$. This implies 
\be\label{eq: closing not unipotent}
\max\{e^{\dn t}|a_3|, |a_1-1|, |a_4-1|\}\geq \eta^{D_2'}.
\ee
Note also that if $e^{\dn t}|a_3|<\eta^{D_2'}$, then 
\[
|a_2a_3|\leq 10\eta^{D_2'} e^{(-\dn+1)t},
\] 
thus $|a_1a_4-1|\ll \eta^{\star}e^{(-\dn+1)t}$, and~\eqref{eq: closing not unipotent} implies $|a_1-a_4|\gg \eta^{D'_2}$. Altogether, 
\be\label{eq: a1-a4}
\max\{e^{\dn t}|a_3|, |a_1-a_4|\}\gg \eta^{D'_2}.
\ee
  
Since $|I_{\rm bad}|\geq 100\ref{E:non-div-main}\eta^{1/(2\dm_0)}$, there are two intervals 
$J,J'\subset [0,1]$ with $\dist(J,J')\geq \eta^{1/(2\dm_0)}$, $|J|, |J'|\geq \eta^{1/(2\dm_0)}$, and
\be\label{eq: Ji cap I bad has many points}
|J\cap I_{\rm bad}|\geq \eta^{1/\dm_0}\quad\text{and}\quad |J'\cap I_{\rm bad}|\geq \eta^{1/\dm_0}.
\ee 
Put $J_\eta=J\cap I_{\rm bad}$.

\subsection*{Claim:} There are $\gg e^{(\dn-2)t/2}$ distinct elements in $\{\gamma_r: r\in J_{\eta}\}$.

Fix $r\in J_{\eta}$ as above, and consider the set of $r' \in J_{\eta}$ so that and $\gamma_r=\gamma_{r'}$.
Then for each such $r'$, 
\begin{align*}
h_r^{-1}\sfs_r h_r&=\exp(-w_r)g_2\gamma_rg_2^{-1}=\exp(-w_r)\exp(w_{r'})h_{r'}^{-1}\sfs_{r'} h_{r'}\\
&=\exp(w_{rr'})h_{r'}^{-1}\sfs_{r'} h_{r'}
\end{align*}
where $w_{rr'}\in \gfrak$ and $\|w_{rr'}\|\ll \nuni^{(-D'+\hdm_0\dn)\rws}$. 

Set $\tau=e^{\dn t}(r'-r)$. Assuming $D'=\frac{D}{\dm+1}-\hdm_1-1$ is large enough, we conclude that
\begin{equation}\label{eq: u_tau equation}
  u_{\tau}\sfs_r u_{-\tau}=h_{r'}h_r^{-1}\,\sfs_r\, h_rh_{r'}^{-1}=\exp(\hat w_{rr'})\sfs_{r'},
\end{equation}
where $\|\hat w_{rr'}\|=\|\Ad(h_{r'})w_{rr'}\|\ll \nuni^{(-D'+\hdm_0\dn+\dm\dn)t}$. Finally, we compute 
\[
u_{\tau}\sfs_r u_{-\tau}=\begin{pmatrix} a_1+a_3\tau& a_2+(a_4-a_1)\tau-a_3\tau^2  \\  a_3& a_4-a_3\tau\end{pmatrix}.
\]

In view of~\eqref{eq: a1-a4}, for every $r\in J_{\eta}$ the set of $r'\in J_{\eta}$ so that 
\begin{equation}\label{eq: define J_r}
  |a_2e^{-\dn t}+(a_4-a_1)(r'-r)-a_3e^{\dn t}(r'-r)^2|\leq 10^{4} e^{(-\dn+1)t}  
\end{equation}
has measure $\ll \eta^{-D_2'/2}e^{(-\dn+1)t/2}$ since at least one of the coefficients of this quadratic polynomial is of size $\gg\eta^{D_2'}$. 
Let $J_{\eta,r}$ be the set of $r'\in J_{\eta}$ for which \eqref{eq: define J_r} holds.

If $r'\in J_{\eta} \setminus J_{\eta,r}$,
then $|a_2+(a_4-a_1)\tau-a_3\tau^2|>10^{4}e^{t}$ (recall that $\tau=e^{\dn t}(r'-r)$), thus for all 
$r'\in J_{\eta} \setminus J_{\eta,r}$, we have 
\[
\|u_{\tau}\sfs_r u_{-\tau}\|> 10^4e^{t}> \|\exp(\hat w_{rr'})\sfs_{r'}\|,
\]
in contradiction to \eqref{eq: u_tau equation}.

In other words, for each $\gamma \in \Gamma$ the set of $r \in J_{\eta}$ for which $\gamma_r = \gamma$ has measure 
$\ll \eta^{-D_2'/2}e^{(-\dn+1)t/2}$ and so the set $\{\gamma_r : r \in J_{\eta}\}$ has at least 
$\gg \eta^{(D_2'+1)/2}e^{(\dn-1)t/2}\gg e^{(\dn-2)t/2}$ distinct elements so long as 
$t\geq 100\dm_0\max(D_2',D_2)|\log\eta|$, see~\eqref{eq: t D2 and eta clsoing}. This establishes the claim.

\subsection*{Zariski closure of the group generated by $\{\gamma_r : r \in I_{\rm bad}\}$} Let ${\bf L}$ denote the Zariski closure of $\langle \gamma_{r}: r\in I_{\rm bad}\rangle$. 

First note that by~\cite[Lemma 2.4]{LMMSW}, we have $[{\bf L}:{\bf L}^\circ]\ll 1$, where the implied constant depends only on the dimension. This and 
\[
\#\{\gamma_r : r \in I_{\rm bad}\}\gg e^{(\dn-2)t/2}
\] 
imply that ${{\bf L}^\circ}\neq \{e\}$. Moreover, from $[{\bf L}:{\bf L}^\circ]\ll 1$ we also conclude that there exists $\gamma_1,\ldots,\gamma_n\subset \{\gamma_r: r \in I_{\rm bad}\}$, where $n$ depends only the dimension, so that $\bf L$ equals the Zariski closure of $\langle \gamma_i: 1\leq i\leq n\rangle$. 

Recall now from~\eqref{eq:gamma-h} that $\exp(w_r)h_r^{-1}\sfs_rh_r=g_2\gamma_rg_2^{-1}$, thus 
\[
\gamma_r .g_2^{-1}v_H=\exp(\Ad(g_2^{-1})w_r).g_2^{-1}v_H.
\]
Moreover, since $\|w_r\|\leq  \nuni^{(-D'+\hdm_0\dn+1)\rws}$,  
\[
\|\Ad(g_2^{-1})w_r\|\ll \eta^{-2\dm D_1} \nuni^{(-D'+\hdm_0\dn+1)\rws}\leq e^{(-D'+\hdm_0\dn+2)t}
\]
for all $r\in I_{\rm bad}$. Recall that $\|\gamma_r^{\pm1}\|\leq \nuni^{\dn_1\rws}$.
If $D'=\frac{D}{\dm+1}-\hdm_1-1$ is large enough, we may apply Lemma~\ref{lem:almost-inv}, with $\{\gamma_1,\ldots,\gamma_n\}$, and conclude that there exists some $g_3\in G$ with 
\be\label{eq: g2 and g3 are close}
\begin{aligned}
\|g_2-g_3\|&\leq \ref{E:Eq-proj-mul}\eta^{-2\dm D_1\ref{a:Eq-proj-2}}\nuni^{(-D'+\hdm_0\dn+2+\ref{a:Eq-proj-2}\dn_1)\rws}\\
&\leq \nuni^{(-D'+\ref{a:Eq-proj-2}\dn_2)\rws},
\end{aligned} 
\ee
so that $\gamma_i .g_3^{-1}v_H=g_3^{-1}v_H$ for all $i$, where $\dn_3$ depends only on $\dm$ and we assumed $t$ is large.  

In view of the choice of $\{\gamma_i: 1\leq i\leq n\}$, this implies $g .g_3^{-1}v_H=g_3^{-1}v_H$ for all $g\in {\bf L}(\R)$. Hence,  
\be\label{eq: L is inside H}
{\bf L}(R)\subset g_3^{-1}N_G(H)g_3
\ee
We also note that $[N_G(H):H]\ll 1$ since $H\subset G$ is a maximal connected subgroup. 

\medskip

We now consider two possibilities for the elements $\{\gamma_r  : r \in I_{\rm bad}\}$. 

\subsection*{Case 1} $\bf L$ is commutative. 
Then ${\bf L}^\circ(\R)\subset g_3^{-1}Hg_3$ is commutative and 
\[
\#\{\gamma\in {\bf L}^\circ(\R): \|\gamma\|\leq e^{\dn_1t}\}\gg e^{(\dn-2) t/2}. 
\]
Since for every torus $T\subset G$, we have $\#(B_T(e,R)\cap \Gamma)\ll (\log R)^2$,
where the implied constant is absolute, ${\bf L}^\circ$ is unipotent and ${\bf L}\subset {\bf L}^\circ\cdot C_G$.   

We also note that since ${\bf L}^\circ$ is a one dimensional unipotent subgroup, 
\[
\#\{\gamma\in {\bf L}: \|\gamma\|\leq 100e^{(\dn-2) t/3}\}\ll e^{(\dn-2) t/3}.
\] 
Furthermore, there are $\gg e^{(\dn-2) t/2}$ distinct elements $\gamma_r$ with $r\in J_{\eta}$.
Thus  
\[
\#\{\gamma_r : \|\gamma_r\|>100e^{(\dn-2) t/3} \text{ and } r\in J_{\eta}\}\gg e^{(\dn-2) t/2}.
\]

For every $r\in I_{\rm bad}$, let
\[
\begin{pmatrix} a_{1,r} & a_{2,r} \\ a_{3,r} & a_{4,r} \end{pmatrix} \quad\text{where $|a_{j,r}|\leq 10\nuni^{t}$}
\] 
denote the element in $\SL_2(\R)$ corresponding to $\mathsf s_r\in H$.  

We will obtain an improvement of~\eqref{eq: closing not unipotent}.  
Let $\xi\eta^{2D_2}\leq \Upsilon\leq e^{(\dn-2) t/3}$ and assume that $\|g_2\gamma_rg_2^{-1}-I\|\geq \Upsilon$ --- by the definition of 
$\xi$, this holds with $\Upsilon=\xi\eta^{2D_2}$ for all $r\in I_{\rm bad}$ and as we have just seen this also holds for with 
$\Upsilon=e^{(\dn-2) t/3}$ for many choices of $r\in J_{\rm bad}$.
Then by~\eqref{eq:gamma-h}, we have 
\be\label{eq: gamma r to Sl2 with r}
\biggl\|u_{-r}\begin{pmatrix} a_{1,r} & e^{-\dn t}a_{2,r} \\ e^{\dn t}a_{3,r} & a_{4,r} \end{pmatrix}u_r-I\biggr\|\geq 10\Upsilon'=O(\Upsilon^{1/\hdm_0})
\ee
where we increase $\hdm_0$ in~\eqref{eq: def dn0} if necessary so that above holds. 

We claim 
\begin{equation}\label{eq: lower bound a(3,r)}
    |a_{3,r}|\geq \Upsilon' e^{-\dn t}.
\end{equation}

To see this, first note that by~\eqref{eq: gamma r to Sl2 with r} $\max\{e^{\dn t}|a_{3,r}|, |a_{1,r}-1|, |a_{4,r}-1|\}\geq \Upsilon'$. 
Assume contrary to our claim that $|a_{3,r}|<\Upsilon' e^{-\dn t}$. 
Then 
\begin{equation}\label{eq:max1,4}
    \max\{|a_{1,r}-1|, |a_{4,r}-1|\}\geq \Upsilon';
\end{equation}
furthermore, we get $|a_{2,r}a_{3,r}|\ll \Upsilon' e^{(-\dn+1)t}$. Thus, 
\begin{equation}\label{eq:max1*4}
    |a_{1,r}a_{4,r}-1|\ll \Upsilon' e^{(-\dn+1)t}\ll e^{-\dn t/2}.
\end{equation}
Moreover, since $h_r^{-1}\sfs_r h_r$ is very nearly $g_2\gamma_{r}g_2^{-1}$, and the latter is either a unipotent element or its minus, we conclude that 
\begin{equation}\label{eq:trace min}
    \min(|a_{1,r}+a_{4,r}-2|, |a_{1,r}+a_{4,r}+2|)\ll e^{(-D'+\star)t}.
\end{equation}
Equations~\eqref{eq:max1*4} and~\eqref{eq:trace min} contradict~\eqref{eq:max1,4} 
if $t$ is large enough. Altogether,~\eqref{eq: lower bound a(3,r)} holds.

We now show that Case 1 cannot occur. 
Since ${\bf L}^\circ$ is unipotent and ${\bf L}\subset {\bf L}^\circ\cdot C_{\bf G}$, we conclude from~\eqref{eq: L is inside H} combined with~\eqref{eq:gamma-h} and~\eqref{eq: g2 and g3 are close} that 
\be\label{eq: us sr and N}
u_{-r}\begin{pmatrix} a_{1,r} & e^{-\dn t}a_{2,r} \\ e^{\dn t}a_{3,r} & a_{4,r} \end{pmatrix}u_r\in \exp(-w_r) (hUh^{-1})\cdot C_{\bf G}
\ee
for all $r\in I_{\rm bad}$, where $h\in H$ and $\|h\|\ll 1$.  
We show that this leads to a contradiction. 

Recall the intervals $J$ and $J'$ from~\eqref{eq: Ji cap I bad has many points}, and let $r_0\in J'\cap I_{\rm bad}$.
then $|r_0-r|\geq \eta^{1/(2\dm_0)}$ for all $r\in J_\eta$. 
Then,~\eqref{eq: us sr and N}, yields that
\be\label{eq: us sr and N'}
u_{-r+r_0}\begin{pmatrix} a_{1,r} & e^{-\dn t}a_{2,r} \\ e^{\dn t}a_{3,r} & a_{4,r} \end{pmatrix}u_{r-r_0}\in \exp(-w_r') (u_{r_0}hUh^{-1}u_{-r_0})\cdot C_{\bf G}
\ee
for all $r\in I_{\rm bad}$.

Let us write $u_{r_0}h=\begin{pmatrix} a & b \\ c & d \end{pmatrix}$, 
then for all $z\in\R$ we have  
\[
u_{r_0}h\begin{pmatrix} 1 & z \\ 0 & 1 \end{pmatrix}h^{-1}u_{-r_0}= \begin{pmatrix}1-acz& a^2z\\ -c^2z & 1+acz\end{pmatrix}.
\]
Let $z_0\in\R$ be so that 
\[
\begin{pmatrix} a_{1,r_0} & e^{-\dn t}a_{2,r_0} \\ e^{\dn t}a_{3,r_0} & a_{4,r_0}\end{pmatrix}=\pm\exp(-w_{r_0})\begin{pmatrix}1-acz_0& a^2z_0\\ -c^2z_0 & 1+acz_0\end{pmatrix}.
\]
By \eqref{eq: lower bound a(3,r)} applied with $\Upsilon'$ corresponding to $\Upsilon=\xi\eta^{2D_2}$, 
we have $|{a_{3,r_0}}|\geq \eta^{\star}e^{-\dn t}$. Since 
\[
|a|, |b|, |c|, |d|\ll 1,
\] 
comparing the bottom left entries of the matrices, we get $|z_0|\gg \eta^{D_3}$. 
Now, since $|a_{2,r_0}|\leq 10e^{t}$, comparing the top right entries we conclude that 
\[
|a|^2\ll \eta^{-D_3}e^{(-\dn+1)t}\leq e^{-(\dn+2) t}.
\]
Since $\det(u_{r_0}h)=1$, it follows that  $|c|$ is also $\gg 1$.  

Let now $r\in J_{\eta}$ be so that $\|\gamma_r\|\geq 100e^{(\dn-2) t/3}$. We write $r_1=r-r_0$, $a'_{2,r}=e^{-\dn t}a_{2,r}$ and $a'_{3,r}=e^{\dn t}a_{3,r}$. 
By~\eqref{eq: lower bound a(3,r)}, applied this time with $\Upsilon'$ corresponding to $\Upsilon=e^{(\dn-2) t/3}$, 
we have $|a'_{3,r}|\geq \Upsilon' \geq e^{\frac{\dn-2}{4\hdm_0}t}$; note also that $|a'_{2,r}|\ll e^{(-\dn+1)t}$. 
In view of~\eqref{eq: us sr and N'}, there exists $z_r\in\R$ so that 
\begin{align*}
u_{-r_1}\begin{pmatrix} a_{1,r} & a'_{2,r} \\ a'_{3,r} & a_{4,r} \end{pmatrix}u_{r_1}&=\begin{pmatrix} a_{1,r}-r_1a'_{3,r}  & a'_{2,r}+(a_{4,r}-a_{1,r})r_1-a'_{3,r}r_1^{2} \\ a'_{3,r} & a_{4,r}+r_1a'_{3,r} \end{pmatrix}\\
&=\pm\exp(-w_r')\begin{pmatrix}1-acz_r& a^2z_r\\ -c^2z_r & 1+acz_r\end{pmatrix}.
\end{align*}
Recall that $|a'_{3,r}|\geq e^{\frac{\dn-2}{4\hdm_0}t}$, $|a_{1,r}|,|a_{4,r}|\ll e^{t}$, and $|a'_{2,r}|\ll e^{(-\dn+1)t}$;
moreover $\eta^{1/(2\dm_0)}\leq |r_1|\leq 1 $ and by~\eqref{eq: t D2 and eta clsoing} $e^{t/10}\geq\eta^{-1}$. 
Thus, so long as $\dn-2>5\hdm_0$, 
\[
0.1|a'_{3,r}|\eta^{1/\dm_0}\leq |a'_{2,r}+(a_{4,r}-a_{1,r})r_1-a'_{3,r}r_1^{2}|\leq 2|a'_{3,r}|.
\]
Hence, since $w_r'$ is small, $|c^2z_r|\eta\ll |a^2z_r|\ll |c^2z_r|$. 
On the other hand, using $r=r_0$, we already established $|a|^2 \leq e^{(-\dn+2)t}$ and $|c|\gg1$, thus $|a^2z_r|\ll e^{(-\dn+2)t}|c^2z_r|$, which is a contradiction, see~\eqref{eq: t D2 and eta clsoing} again.

Altogether, we conclude that Case~1 cannot occur.

\subsection*{Case 2} $\bf L$ is not commutative. In other words, there are $r,r'\in I_{\rm bad}$ so that $\gamma_r$ and $\gamma_{r'}$ do not commute. 

Recall from ~\eqref{eq: g2 and g3 are close} that 
\[
\gamma_r .g_3^{-1}v_H=g_3^{-1}v_H \quad\text{and}\quad 
\gamma_{r'} .g_3^{-1}v_H=g_3^{-1}v_H.
\]
where $\|g_2-g_3\|\leq \nuni^{(-D'+\ref{a:Eq-proj-2}\dn_2)\rws}$.

In view of Lemma~\ref{lem:non-elementary}, thus, we have $Hg_3\Gamma$ is periodic and 
\[
\vol(Hg_3\Gamma)\leq \ref{E:non-el-1}\eta^{-D_2\ref{a:non-el-2'}}\Bigl(\max\{\|\gamma_r^{\pm1}\|,\|\gamma_{r'}^{\pm1}\|\}\Bigr)^{\ref{a:non-el-2}}\leq \ref{E:non-el-1}\nuni^{(1+\ref{a:non-el-2}\dn_1)\rws}.
\]

Then for $t$ large enough,  
\[
\vol(Hg_3\Gamma)\leq\nuni^{(2+\ref{a:non-el-2})\dn_1\rws}\quad\text{ and} \quad d_X(g_2\Gamma,g_3\Gamma)\leq e^{(-D'+\ref{a:Eq-proj-2}\dn_2)t}, 
\]
where $D'=\frac{D}{\dm+1}-\hdm_1-1$.

Since $g_2\Gamma=x_2=a_tu_{r_1}x_1$, part~(2) in the proposition holds with $x'=(a_tu_{r_1})^{-1}g_3\Gamma$ and 
$D_1=\max\{2+\ref{a:non-el-2}\dn_1,1+\hdm_1+\ref{a:Eq-proj-2}\dn_2\}$ so long as $t$ is large enough. 
\qed


\section{Projection theorems}\label{sec: app projection}

We begin by recalling a theorem of  Gan, Guo, and Wang~\cite{GGW}. Let 
\[
\xi(r)=(\tfrac{r}{1!}, \tfrac{r^2}{2!},\ldots, \tfrac{r^{n}}{n!})\subset \R^n.
\]
For all $r\in[0,1]$ and all $1\leq d\leq n$, let $\pi_{r,d}:\R^n\to\R^d$ denote the orthogonal projection into
\[
{\rm Span}\{\xi'(r),\xi^{(2)}(r),\ldots, \xi^{(d)}(r)\}. 
\]

The following theorem follows from~\cite[Thm.2.1]{GGW}, combined with a finitary adaptation of the argument presented in~\cite[\S2]{GGW}.

\begin{thm}
\label{thm: proj app 0}
Let $1\leq d\leq n$ and let $0<\alpha\leq d$.
Let $\rho$ be the uniform measure on a finite set $\Theta\subset B_{\R^n}(0,1)$ satisfying 
\[
\rho(B_\rfrak(w, \rhsc)\cap \Theta)\leq \egbd \rhsc^\alpha\qquad\text{for all $w$ and all $\rhsc\geq \rhsc_1$}
\]
where $\egbd\geq 1$ and $0<\rhsc_1\leq1$.

Let $0<\pvare<10^{-4}\alpha$.
For every $\rhsc\geq \rhsc_1$, there exists a subset 
$J_{\rhsc}\subset [0,1]$ with $|[0,1]\setminus J_{\rhsc}|\leq C_\pvare\rhsc^{\star\pvare^2}$ so that the following holds. 
Let $r\in J_{\rhsc}$, then there exists a subset $\Theta_{\rhsc,r}\subset \Theta$ with 
\[
\rho(\Theta\setminus \Theta_{\rhsc,r})\leq C_\pvare \rhsc^{\star\pvare^2}
\]
such that for all $w\in\Theta_{\rhsc, r}$, we have 
\[
\rho\Bigl(\{w'\in\Theta: \|\pi_{r,d}(w)-\pi_{r,d}(w')\|\leq \rhsc\}\Bigr)\leq C_\pvare\egbd \rhsc^{\alpha-\pvare}
\] 
where the implied constants are absolute.
\end{thm}

\begin{proof}
Since $1\leq d\leq n$ is fixed throughout the argument, we will simplify the notation by writing $\pi_r$ instead of $\pi_{r,d}$. 

For every $r\in[\frac12, 1]$ and all $w\in \Theta$, define 
\[
m^\rhsc(\pi_r(w))=\rho(\{w': \|\pi_r(w)-\pi_r(w')\|\leq \rhsc\}).
\]

Let $E_1, E_2, \ldots$ be large constants to be specified later. Put $\vare=\pvare/E_1$. For all $r\in[\frac12, 1]$, set 
\[
\Theta_{\rm Exc}(r)=\{w: m^\rhsc(\pi_r(w))\geq \egbd\rhsc^{\alpha-E_1\vare}\}.
\]

Suppose, contrary to the claim in Theorem~\ref{thm: proj app 0}
that there exists a subset $I_{\rm Exc}\subset [\frac12,1]$ with 
$|I_{\rm Exc}|\geq E_2\rhsc^{\vare^2/2}$ so that for all $r\in I_{\rm Exc}$, we have 
\be\label{eq: measure of Theta Exc}
\rho(\Theta_{\rm Exc}(r))\geq E_2\rhsc^{\vare^2/2}.
\ee 
We will get a contradiction with~\cite[Thm.~2.1]{GGW}, provided that $E_1$ and $E_2$ are large enough. 

First note that, for every $r\in I_{\rm Exc}$, the number of $\rhsc$-boxes $\{\mathsf B_{i,r}\}$ 
required to cover $\pi_r(\Theta_{\rm Exc}(r))$ is $\leq E_3\egbd^{-1}\rhsc^{-\alpha+E_1\vare}$. 
Let $\mathbb T_{i,r}=\pi_r^{-1}(\mathsf B_{i,r})\cap B_{\R^n}(0,1)$, and put $\mathcal T_r=\{\mathbb T_{i,r}\}$. 
Note that 
\be\label{eq: number of Tr}
\#\mathcal T_r\leq E_3 \egbd^{-1}\rhsc^{-\alpha+E_1\vare}.
\ee
 
Let $\Lambda_\rhsc\subset I_{\rm Exc}$ be a maximal $\rhsc$-separated subset, and extend this to a maximal 
$\rhsc$-separated subset $\hat\Lambda_\rhsc$ of $[\frac12,1]$. 
Equip $\hat\Lambda_\rhsc\times \Theta$ with the product measure $\rho\times \sigma$,
where $\sigma$ denotes the uniform measure on $\hat\Lambda_\rhsc$. Let 
\begin{align*}
F&=\{(r,w)\in \Lambda_\rhsc\times \Theta: m^\rhsc(\pi_r(w))\geq \egbd\rhsc^{\alpha-E_1\vare}\}\\
&=\{(r,w)\in \Lambda_\rhsc\times \Theta: w\in \Theta_{\rm Exc}(r)\}.
\end{align*}
In view of~\eqref{eq: measure of Theta Exc} and $|I_{\rm Exc}|\geq E_2\rhsc^{\vare^2/2}$, we have 
$\sigma\times \rho(F)\geq E_2^2\rhsc^{\vare^2}$.

For every $w\in \Theta$, let $F_w=\{r\in\Lambda_\rhsc: (r,w)\in F\}$, and   
set 
\[
\Theta'=\{w\in \Theta: \sigma(F_w)\geq E_2 \rhsc^{\vare^2}\}.
\] 
Then, using Fubini's theorem, we conclude that $\rho(\Theta')\geq \tfrac12E_2^2\rhsc^{\vare^2}$.

The above definitions thus imply
\[
\sum_{r\in\Lambda_\rhsc}1_{\mathbb T_r}(w)\geq E_3\rhsc^{\vare^2-1}\qquad\text{for all $w\in \Theta'$}
\]
where $E_3=O(E_2)$, for an absolute implied constant. 

Let $\rho'=\rho|_{\Theta'}$. Applying~\cite[Thm.\ 2.1]{GGW} with $\vare^2$, $\rho'$
and $\{\mathcal T_r: r\in\Lambda_\rhsc\}$,  
\[
\begin{aligned}
\sum_{r\in\Lambda_\rhsc} \#\mathcal T_r&\geq E_4(n, \pvare, \alpha) \egbd^{-1}\rho'(\R^n)\rhsc^{-1-\alpha+E\vare}\\
&\geq \tfrac12E_2^2 E_4(n, \pvare, \alpha) \egbd^{-1} \rhsc^{\vare^2}\rhsc^{-1-\alpha+E\vare}
\end{aligned}
\]
where $E=10^{10n}$ and in the second line we used $\rho'(\R^n)=\rho(\Theta')\geq \tfrac12E_2^2\rhsc^{\vare^2}$.

Thus there exists some $r\in\Lambda_\delta$ so that 
\be\label{eq: GGW and pigeonhole}
\#\mathcal T_r \geq \tfrac12\egbd^{-1}E_4(n, \pvare, \alpha) E_2^2 \rhsc^{-\alpha+(D+1)\vare}.
\ee
Now comparing~\eqref{eq: GGW and pigeonhole} and~\eqref{eq: number of Tr} we get a contradiction so long 
as $E_1$ is large enough and $\rhsc$ is small enough. The proof is complete. 
\end{proof}

Recall that the $n$-dimensional ($n\geq 2$) irreducible representation of $\SL_2(\R)$ can be normalized so that for all $w\in\R^n$, 
\[
u_rw=\Bigl(w\cdot\xi'(r), w\cdot\xi^{(2)}(r),\cdots, w\cdot \xi^{(n)}(r)\Big)
\]
Recall also that $\rfrak\simeq \R^{2\dm+1}$ is an irreducible representation of $\Ad(H)$ where 
\begin{quote}
$\bullet \;\; \dm=1$ if $\G$ is isogeneous to $\SO(3,1)$ or $\SL_2\times\SL_2$\\
     $\bullet \;\; \dm=2$ if $\G$ is isogeneous to $\SL_3$ or ${\rm SU}(2,1)$\\
     $\bullet \;\; \dm=3$ if $\G$ is isogeneous to ${\rm Sp}_4$\\
     $\bullet \;\; \dm=5$ if $\G$ is isogeneous to ${\G}_2$
\end{quote}  
Theorem~\ref{thm: proj app 0} thus applies to the adjoint action of $u_r$ on these spaces. 
For every $1\leq d\leq 2\dm+1$, let $\rfrak_d$ denote the space spanned by vectors with weight $\dm,\ldots, \dm-d+1$. 
Let $\pi_d:\rfrak\to\rfrak_d$ denote the orthogonal projection.

\begin{thm}
\label{thm: proj app 1}
Let $0<\alpha, d\leq 2\dm+1$, $d\in\Z$, and $0<\rhsc_1\leq1$. 
Let $\rho$ be the uniform measure on a finite set $\Theta\subset B_\rfrak(0,1)$ with 
\[
\rho(B_\rfrak(w, \rhsc)\cap \Theta)\leq \egbd \rhsc^\alpha\qquad\text{for all $w$ and all $\rhsc\geq \rhsc_1$}
\]
where $\egbd\geq 1$.

Let $0<\pvare<10^{-4}\alpha$.
For every $\rhsc\geq \rhsc_1$, there exists a subset 
$J_{\rhsc}\subset [0,1]$ with $|[0,1]\setminus J_{\rhsc}|\leq C_\pvare'\rhsc^{\star\pvare^2}$ so that the following holds. 
Let $r\in J_{\rhsc}$, then there exists a subset $\Theta_{\rhsc,r}\subset \Theta$ with 
\[
\rho(\Theta\setminus \Theta_{\rhsc,r})\leq C_\pvare \rhsc^{\star\pvare^2}
\]
such that for all $w\in\Theta_{\rhsc, r}$ we have 
\[
\rho\Bigl(\{w'\in\Theta: \|\pi_d(\Ad(u_r)w)-\pi_d(\Ad(u_r)w')\|\leq \rhsc\}\Bigr)\leq C_\pvare\egbd \rhsc^{\min(\alpha,d)-\pvare}.
\] 
\end{thm}


\section{Measures and partitions of unity}\label{sec: app Folner}
In this appendix we collect some of the results from~\cite[\S6--8]{LMW22} which were used in this paper.

\subsection*{Regular tree decomposition}\label{sec: regular tree}
Let us recall the following discussion from \cite[\S6]{LMW22} see also~\cite[Lemma 5.2]{BFLM}. 
Let $t, {\mathsf D}_0\geq 1$ and $0<\vare<1$ be three parameters: $t$ is large and arbitrary, ${\mathsf D}_0$ is moderate and fixed, and $\vare$ is small and fixed; in particular, our estimates are allowed to depend on ${\mathsf D}_0$ and $\vare$, but not on $t$. 
Let $\beta=e^{-\kappa t}$ for some $\kappa$ satisfying $0<\kappa ({\mathsf D}_0+1)\leq 10^{-6}\vare$. 

Let $F\subset B_\rfrak(0,1)$ satisfy that  
\[
e^{t/2}\leq \#F\leq e^{{\mathsf D}_0 t},
\]
and assume that 
\be\label{eq: tree dec eng bd}
\eng_{F,\trct}^{(\alpha)}(w)\leq \egbd \qquad\text{for all $w\in F$}
\ee
where $\egbd>0$ satisfies the following   
\be\label{eq: a priori bound on egbd}
\egbd\leq e^{({\mathsf D}_0+1)t}.
\ee

Fix ${\mathsf L}\in\bbn$, large enough, so that both of the following hold  
\be\label{eq: condition on M}
\text{$2^{-{\mathsf L}}({\mathsf D}_0+1)< \kappa/100\quad$ and $\quad (4\dm+2){\mathsf L}<2^{\kappa {\mathsf L}/100}$}.
\ee
Define $k_0:=\lfloor (-\log_2\beta)/{{\mathsf L}}\rfloor$ and 
$k_1:=\lceil(1+\alpha^{-1}\log_2\egbd) /{\mathsf L} \rceil+1$; note that    
\be\label{eq: F cap the smallest cube k1}
2^{({\mathsf L}k_1-1)\alpha}>\Upsilon.
\ee

For every $k_0\leq k\leq k_1$, let $\mathcal C_{{\mathsf L}k}$ 
denote the collection of $2^{-{\mathsf L}k}$-cubes in $\rfrak$.

\begin{lemma}\label{lem: regular tree decomposition}\label{lem: trimming cone 2}
For all large enough $t$,
we can write $F=F'\bigcup (\bigcup_{i=1}^{N}F_i)$ (a disjoint union) with 
\[ 
\text{$\#F'<\beta^{1/4}\cdot(\#F)\quad$ and $\quad\#F_i\geq \beta^2\cdot(\#F)$}
\] 
so that the following hold. 
\begin{enumerate}
\item For every $i$ and every $k_0-10\leq k\leq k_1$, there exists some $\tau_{ik}$ so that 
for every cube $\mathsf C \in \mathcal C_{{\mathsf L}k}$ we have 
\be\label{eq: regular tree}
2^{{\mathsf L}(\tau_{ik}-2)}\leq \#F_i\cap \mathsf C\leq 2^{{\mathsf L}\tau_{ik}}\quad\text{or}\quad F_i\cap \mathsf C=\emptyset.
\ee
\item For every $i$, we have 
\[
\eng_{F_i,\trct}^{(\alpha)}(w)\leq \beta^{-2\dm-2}\egbd \qquad\text{for all $w\in F$}
\]
\end{enumerate}
\end{lemma}

\begin{proof}
Part~(1) is proved in \cite[Lemma 6.4]{LMW22}. 

For part~(2) see \cite[Lemma 6.5]{LMW22}. 
\end{proof}

\subsection*{Covering lemmas}
Fix $0<\injr\leq0.01$ and let $\beta=\eta^2$. 
For $m\geq0$, we introduce the shorthand notation $\umt^H_m$ for 
\be\label{eq: def BH}
\umt_{\eta,\beta^2,m}^H=\Bigl\{u^-_s: |s|\leq \beta^2 \nuni^{-m}\Bigr\}\cdot\{a_\tau: |\tau|\leq \beta^2\}\cdot U_\eta,
\ee
where for every $\delta>0$, let $U_\delta=\{u_r: |r|\leq \delta\}$, see~\eqref{eq: def B ell beta}.

Define $\umt^G_{m}\subset G$ by thickening $\umt^H_m$ 
in the transversal direction as follows: 
\be\label{eq: def O ell C}
\umt^G_{m}:=\umt^H_m\cdot \exp(B_\rfrak(0,2\beta^2)).
\ee  

\begin{lemma}\label{lem: E good h0}
For every $m\geq 0$, there exists a covering 
\[
\Big\{\umt^G_{m}.y_{j}: j\in \mathcal J_m, y_j\in X_{3\eta/2}\Big\}
\] 
of $X_{2\injr}$ with multiplicity $K$, depending only on $X$.
In particular, $\#\mathcal J_m\ll\eta^{-1}\beta^{-4\dm-6}\nuni^{m}$.
\end{lemma}

\begin{proof}
This is proved in \cite[Lemma 7.1]{LMW22}, we recall the set up to explicate the bound claimed here.
There exists a covering 
\[
\Big\{\Bigl(\boxHs_{\beta^2}\cdot\boxU_{\eta}\cdot\exp\bigl(B_\rfrak(0,\beta^2)\bigr)\Bigr).\hat y_k: k\in \mathcal K, \hat y_k\in X_{2\eta}\Big\}
\] 
of $X_{2\injr}$ with multiplicity $O(1)$ depending only on $X$.

Let us write $\bar{\mathsf B}^G_{\eta,\beta^2}=\boxHs_{\beta^2}\cdot \boxU_{\eta}\cdot\exp(B_\rfrak(0,\beta^2))$. Then 
\be\label{eq: doubling propert Bbar}
\Bigl(\bar{\mathsf B}^G_{0.1\eta, 0.1\beta^2}\Bigr)^{-1}\cdot\Bigl(\bar{\mathsf B}^G_{0.1\eta,0.1\beta^2}\Bigr)\subset 
\Bigl(\bar{\mathsf B}^G_{c\eta,c\beta^2}\Bigr),
\ee
where $c$ depends only on $\dm$, see Lemma~\ref{lem: BCH}. 

Let $\{\hat y_k\in X_{2\eta}: k\in\mathcal K\}$ be maximal with the following property
\[
\bar{\mathsf B}^G_{0.01\eta,0.01\beta^2}.\hat y_i\cap \bar{\mathsf B}^G_{0.01\eta,0.01\beta^2}.\hat y_j=\emptyset\quad\text{ for all $i\neq j$.}
\]
In view of~\eqref{eq: doubling propert Bbar} thus $\{\bar{\mathsf B}^G_{\eta,\beta^2}.\hat y_k:k\in\mathcal K\}$ covers $X_{2\eta}$ with multiplicity $O(1)$. Since $m_G(\bar{\mathsf B}^G_{\eta,\beta^2})\asymp \eta\beta^{4\dm+6}$, we that $\mathcal K\ll \eta^{-1}\beta^{-4\dm-6}$.

The rest of the proof goes through as in \cite[Lemma 7.1]{LMW22}. 
\end{proof}

\subsection*{Boxes and complexity}\label{sec: box complexity} 
Let $\mathsf{prd}:\mathbb R^3\to H$ be the map
\[
\mathsf{prd}(s,\tau,r)= u^-_sa_\tau u_r.
\]
A subset $\mathsf D\subset H$ will be called a {\em box} if there exist intervals 
$I^{\bigcdot}\subset\bbr$ (for $\bigcdot=\pm, 0$) so that 
\[
\mathsf D=\mathsf{prd}(I^-\times I^0\times I^+).
\] 

We say $\Xi\subset H$ has complexity bounded by $L$ (or at most $L$) if 
$\Xi=\bigcup_{1}^L \Xi_i$ where each $\Xi_i$ is a box.

For every interval $I\subset\bbr$, let $\partial I=\partial_{100\eta|I|}I$ (recall that $\eta=\beta^{1/2}$), and put $\mathring I=I\setminus\partial I$.
Given a box $\mathsf D=\mathsf{prd}(I^-\times I^0\times I^+)$, we let 
\begin{subequations}
\begin{align}
\label{eq: def mathring D}&\mathring{\mathsf D}=\mathsf{mul}\Bigl(\mathring{I^-}\times\mathring{I^0}\times\mathring{I^+}\Bigr) \quad\text{and}\\ 
\label{eq: def partial D}&\partial\mathsf D=\mathsf D\setminus \mathring{\mathsf D}.
\end{align}
\end{subequations}

More generally, if $\mathsf D=\mathsf{prd}(I^-\times I^0\times I^+)$ is a box, and 
$\Xi\subset \mathsf D$ has complexity bounded by $L$, we define 
$\partial\Xi:=\bigcup\partial\Xi_i$ and  
\be\label{eq: def mathring Xi}
\mathring\Xi_{\mathsf D}:=\bigcup\mathring\Xi_i
\ee
where the union is taken over those $i$ so that $\Xi_i=\mathsf{prd}(I_i^-\times I_i^0\times I_i^+)$ with 
$|I_i^{\bigcdot}|\geq 100\eta|I^{\bigcdot}|$ for $\bigcdot=\pm,0$.   

\subsection{Admissible measures}\label{sec: cone and mu cone}
Let $\adm>0$. Let $\cone=\coneH.\{\exp(w)y: w\in F\}$. 
A probability measure $\mu_\cone$ on $\cone$ is said to be $\adm$-{\em admissible} if 
\[
\mu_\cone=\frac1{\sum_{w\in F}\mu_w(X)}\sum_{w\in F}\mu_w
\]
where for every $w\in F$, $\mu_w$ is a measure on $\coneH.\exp(w)y$ satisfying that if $\sfh\exp(w)y$ is in the support of $\mu_w$
\[
\diff\!\mu_w(\sfh\exp(w)y)=\adl\ddensity_w(\sfh)\diff\!m_H(\sfh)\quad\text{where $1/\adm\leq \ddensity_w(\bigcdot)\leq \adm$,}
\] 
for some $\adl>0$ independent of $w\in F$. 
Moreover, there is a subset $\coneH_w=\bigcup_{i=1}^{\adm}\coneH_{w,i}\subset \coneH$
so that 
\begin{enumerate}
\item $\mu_w\Bigl((\coneH\setminus \coneH_w).\exp(w)y\Bigr)\leq \adm\beta \mu_w(\coneH.\exp(w)y)$,
\item The complexity of $\coneH_{w,i}$ is bounded by $\adm$ for all $i$, and 
\item $\Lip(\ddensity_w|_{\coneH_{w,i}})\leq \adm$ for all $i$.
\end{enumerate}

\subsection*{Convex combination of measures}
The following lemma was used several times in the paper, in particular, in the proof of Lemma~\ref{lem: main ind lemma}. 

\begin{lemma}\label{lem: convex comb}
Let $\ell, \mathsf d>0$. 
Assume that $\nuni^{-\ell/2}\leq\beta$ and that $h\mapsto hx$ is injective on $\coneH\cdot a_t\cdot \{u_r: r\in[0,1]\}$.
Let 
\[
\cone=\coneH.\{\exp(w)y: w\in F\}\subset X_\eta
\] 
be equipped with an admissible measure $\mu_\cone$. For every $r\in [0,1]$, 
\be\label{eq: convex comb app}
\int\!\varphi(a_{\mathsf d}u_sa_\ell u_{r}z)\diff\!\mu_{\cone}(z)\!=\!\sum_i\! c_{i}\int\!\varphi(a_{\mathsf d} u_sz)\diff\!\mu_{\cone_{i}}(z)+  O(\beta^\star\Lip(\varphi)),
\ee
where $\cone_{i}=\coneH.\{\exp(w)y_{i}: w\in F_{i}\}\subset X_\eta$ and $F_{i}\subset B_\rfrak(0,\beta)$ satisfies  
\begin{enumerate}
\item $\beta^{4\dm+5}\cdot (\#F)\leq F_{i}\leq \beta^{4\dm+4}\cdot (\#F)$, and 
\item $\umt^H_\ell.\exp(w)y_{i}\subset a_\ell u_r\cone$ for all $w\in F_{i}$, where 
\[
\umt_{\ell}^H=\Bigl\{u^-_s: |s|\leq \beta^2 \nuni^{-\ell}\Bigr\}\cdot\{a_\tau: |\tau|\leq \beta^2\}\cdot \{u_r: |r|\leq \eta\}.
\] 
\end{enumerate}
The implied constant depends on $X$. 
\end{lemma}

\begin{proof}
This is proved in~\cite[Lemma 8.9]{LMW22}. 
We provide road map to the argument to elucidate the claims made above regarding $F_{i,r}$. 

To see part~(1), and with the notation as in~\cite[Lemma 8.7]{LMW22}, we have 
$c_{i,r}^j\asymp N_{i,r}^j\Bigl(e^{-\ell}\beta^4\eta\Bigr)\cdot (\#F)^{-1}$.
Therefore, if $c_{i,r}^j\geq \beta^{4\dm+8}e^{-\ell}$, then  
\[
N_{i,r}^j\geq \beta^{4\dm+4}\cdot(\#F).
\]
Moreover, by Lemma~\ref{lem: E good h0},
$\#\mathcal J_\ell\ll \eta^{-1}\beta^{-4\dm-6}e^{\ell}\leq \beta^{-4\dm-7}e^{\ell}$. Thus, 
\[
\sum_{\;\; c_{i,r}^j< \beta^{4\dm+8}e^{-\ell}}c^j_{i,r}\leq \beta. 
\]

One now defines $F_{i}:=F_{i,r}^{j,m}$ where again we used the notation 
which is used after the conclusion of the proof of~\cite[Lemma 8.7]{LMW22}. Then
\[
\beta^{4\dm+5}\cdot(\#F)\leq \#F_{i}\leq \beta^{4\dm+4}\cdot(\#F),
\]
see~\cite[eq.~(8.18)]{LMW22}.

As it is argued in the proof of~\cite[Lemma 8.9]{LMW22}, 
the claim in~\eqref{eq: convex comb app} holds with $\cone_{i}=\coneH.\{\exp(w)y_{i}: w\in F_{i}\}$, see~\cite[eq.~(8.21)]{LMW22}.  

Part~(2) above is~\cite[eq.~(8.14)]{LMW22}. 
\end{proof}

\end{document}